\documentclass[11pt,a4paper,fleqn,final]{article}

\usepackage{preprint-modern}

\crefname{appsec}{Appendix}{Appendices}
\Crefname{appsec}{Appendix}{Appendices}

\usepackage[
backend=biber,
bibencoding=utf8,
giveninits=true,
style=alphabetic,
maxnames=4,
isbn=false,
doi=true,
eprint=false,
url=false,
date=year
]{biblatex}
\AtEveryBibitem{%
  \ifentrytype{online}{}{%
    \clearfield{url}%
    \clearfield{urldate}%
  }%
}

\AtEveryBibitem{\clearfield{pubstate}}
\DeclareFieldFormat{doi}{%
  \href{https://doi.org/#1}{doi.org/#1}}

\DeclareFieldFormat{eprint:arxiv}{%
   \href{https://arxiv.org/abs/#1}{arXiv:#1}}

\DeclareFieldFormat{eprint:hal}{%
    \href{https://hal.science/#1}{#1}}

\renewbibmacro*{doi+eprint+url}{%
  \ifentrytype{online}{%
    \iffieldundef{eprinttype}{%
      \usebibmacro{url+urldate}%
    }{%
      \usebibmacro{eprint}%
    }%
  }{%
    \iftoggle{bbx:eprint}{%
      \iffieldundef{eprint}{%
        \iftoggle{bbx:doi}{\printfield{doi}}{}%
      }{%
        \usebibmacro{eprint}%
      }%
    }{%
      \iftoggle{bbx:doi}{\printfield{doi}}{}%
    }%
  }%
}

\newcommand{\MM}{M\!M}

\newcommand{\LinSD}{\Lfrak}
\newcommand{\LinGL}{\Lcal_{\mathrm{GL}}}

\begin{document}

	\title{Slow-moving pattern interfaces in general directions for a two-dimensional Swift--Hohenberg-type equation}
	
	\author[0000-0002-0329-9402]{Bastian Hilder}{Department of Mathematics, Technische Universität München, Boltzmannstraße 3, 85748 Garching b.\ München, Germany}{bastian.hilder@tum.de}
	\author[0009-0008-0258-5459]{Jonas Jansen}{Institute of Applied Mathematics and Statistics, University of Hohenheim, Schloss Hohenheim 1, 70599 Stuttgart, Germany}{jonas.jansen@uni-hohenheim.de}
	
	\msc{35B36; 35B32; 37L10; 34C37; 35Q56; 35B10; 34E15; 35K58}
	\keywords{pattern formation, Swift--Hohenberg-type equation, spatial dynamics, center manifold theory}
    % msc classification
    % 35B36: Pattern formations in context of PDEs
    % 35B32: Bifurcations in context of PDEs
    % 37L10: Normal forms, center manifold theory, bifurcation theory for infinite-dimensional dissipative dynamical systems
    % 34C37: Homoclinic and heteroclinic solutions to ordinary differential equations
    % 35Q56: Ginzburg-Landau equations
    % 35B10: Periodic solutions to PDEs
    % 34E15: Singular perturbations for ordinary differential equations
    % 35K58: Semilinear parabolic equations
	
	\maketitle

\begin{abstract}
    We rigorously prove the bifurcation of slow-moving pattern interfaces with general direction in a two-dimensional Swift–Hohenberg-type model close to a Turing instability for a large class of nonlinearities. These interfaces describe the invasion of stripe and hexagonal patterns into the spatially homogeneous state and model a possible mechanism for pattern formation, as observed in a wide range of real-world applications. For this, we develop a rigorous framework to establish the existence of such solutions using spatial dynamics and non-standard centre manifold theory. Our approach exploits geometric and algebraic structures generic to $\mathrm{O}(2)$-symmetric pattern-forming systems near a Turing instability, and addresses fundamental technical challenges due to a non-uniform spectral gap around the imaginary axis, quadratic resonances induced by the hexagonal structure, and the high-dimensional phase space of the reduced equations. 
\end{abstract}

\setcounter{tocdepth}{2}
\tableofcontents
    
\section{Introduction}

The formation of spatially regular structures is ubiquitous in natural systems and can be observed in hydrodynamic convection problems \cite{pearson1958-09JournalofFluidMechanics,swift1977-01PhysRevA}, vegetation systems \cite{klausmeier1999-06Science,meron2018-03AnnualReviewofCondensedMatterPhysics}, biological and chemical systems \cite{gierer1972-12Kybernetik,ouyang1991-08Nature}, and engineering applications \cite{cuerno1995-06PhysRevLett,bangsund2019-07NatMater}. These spatial patterns typically arise when a spatially homogeneous state destabilises in response to a change in a system parameter. One of the most famous examples is the Turing instability \cite{turing1952PhilosophicalTransactionsoftheRoyalSocietyofLondon.SeriesBBiologicalSciences}, where the spatially homogeneous state becomes unstable in a neighbourhood of a finite Fourier wave number, which leads to the formation of steady, spatially periodic structures such as hexagonal and stripe patterns. In applications, these spatial patterns typically arise in the wake of an invasion front, which facilitates the spatial transition from the pattern to the unstable homogeneous state, see \Cref{fig:planar-front}. These fronts can appear as radial invasion phenomena, in which the pattern spreads into the unstable state from an initial point defect \cite{bodenschatz1991-11PhysRevLett}, or as planar invasion fronts, in which the pattern and the unstable state are separated by a moving planar interface \cite{nitschke1995-12PhysRevE}.

\begin{figure}[h]
    \centering
    \begin{subfigure}[b]{0.99\textwidth}
        \includegraphics[width=\linewidth]{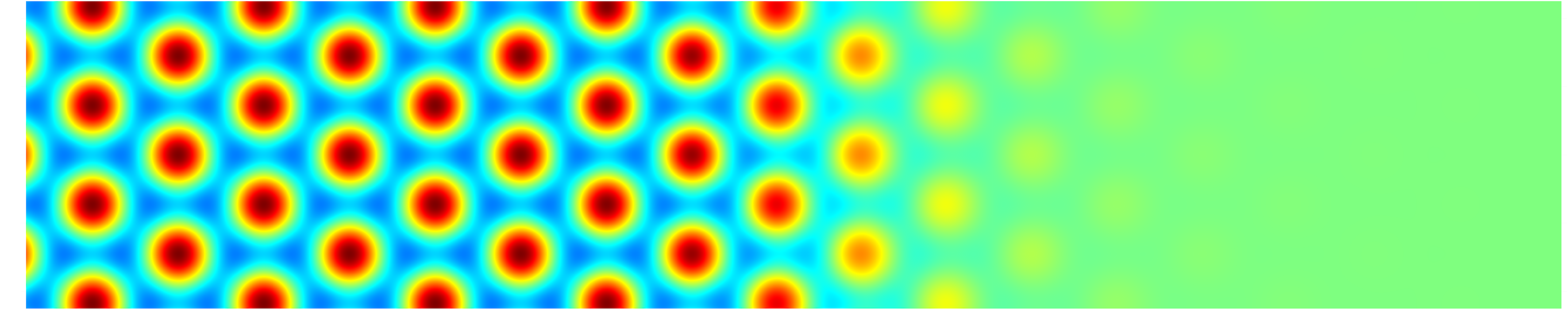}
        \subcaption{Pattern interface for angle $\theta = 0$.}\label{subfig:planar-front-a}
    \end{subfigure}

    \vspace{0.2cm}
    
    \begin{subfigure}[b]{0.99\textwidth}
        \includegraphics[width=\linewidth]{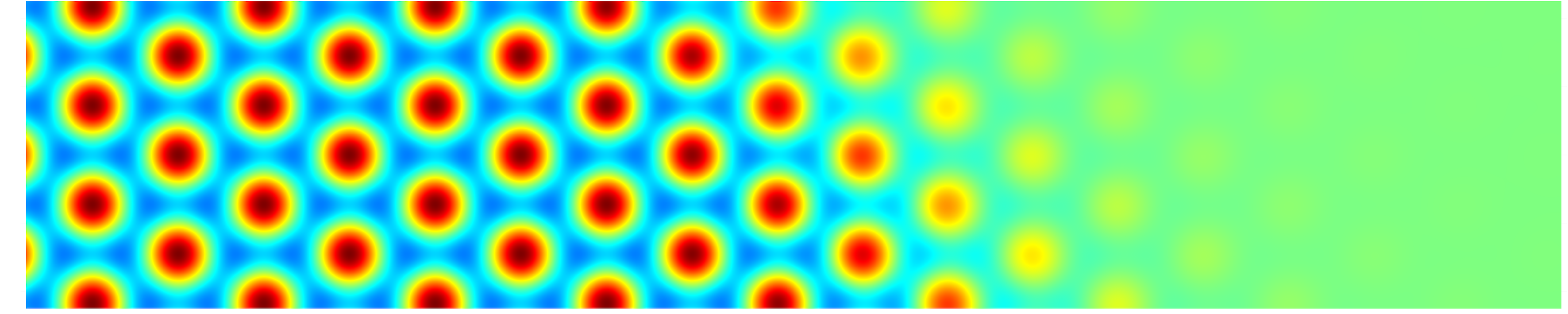}
        \subcaption{Pattern interface for angle $\theta = \tfrac{\pi}{6}$.}\label{subfig:planar-front-b}
    \end{subfigure}
    \caption{Example of two-dimensional planar pattern interfaces showing the invasion of a hexagonal pattern into the spatially homogeneous steady state for different directions.}
    \label{fig:planar-front}
\end{figure}

In this paper, we study the existence of two-dimensional, planar pattern interfaces close to a Turing instability. It is well known that the dynamics close to these instabilities are universal, and the emerging structures are independent of the specific application, cf.~\Cref{sec:amplitude-equations-intro}. Therefore, we analyse the existence of pattern interfaces in a semilinear Swift–Hohenberg-type equation
\begin{equation}\label{eq:Swift-Hohenberg}
    \partial_t u = L(\nabla;\mu) u + N(u,Du,D^2u,D^3u;\vartheta), \qquad L(\nabla;\mu) u = -(1+\Delta)^2 u + \mu u
\end{equation}
with $u(t,\x) \in \R$, a set of parameters $\vartheta \in \Theta$, time $t \geq 0$ and space $\x \in \R^2$. The classical Swift–Hohenberg equation with $N(u) = -u^3$ was first proposed as a model for the formation of thermally-driven convection cells in fluid applications \cite{swift1977-01PhysRevA}. The equation \eqref{eq:Swift-Hohenberg} undergoes a Turing instability at $\mu = 0$, that is, there exists a critical absolute wave number $k_c = 1$, such that, for $\mu > 0$, solutions to the linearised equation of the form $e^{\lambda t - i \k \cdot \x}$ grow exponentially for $|\k|$ close to $k_c$ and decay otherwise, cf.~\Cref{app:general-spectral-analysis}. Therefore, this is a prototypical model to study the dynamics close to such an instability. Planar pattern interfaces for a similar model with $N(u) = -\beta|\nabla u|^2 - u^3$ were first constructed in \cite{doelman2003-02EuropeanJournalofAppliedMathematics}, where the authors obtained solutions travelling in $x_1$-direction under the assumption that $|\beta| \ll 1$.

Here, we develop a fully rigorous framework for two-dimensional pattern interfaces connecting periodic solutions on a hexagonal Fourier lattice and travelling in general directions $\d = (\cos(\theta), \sin(\theta))$ satisfying $\cot(\theta) \in \sqrt{3} \Q$ for a Swift–Hohenberg-type equation with a generic class of nonlinearities. The framework relies on a spatial-dynamics approach with a non-standard normal-form transformation and centre manifold reduction with a non-uniform spectral gap. Specifically, we address two fundamentally new mathematical obstacles: quadratic resonances inherent to the hexagonal lattice structure induce singular scalings in the normal-form analysis, and the insufficient gain of regularity in the front direction, which requires a functional-analytic framework using mixed-regularity Sobolev spaces. This approach reduces the construction of pattern interfaces to proving the existence of travelling fronts in a system of degenerate Ginzburg–Landau equations. Here, we prove the existence of selected heteroclinic orbits, which completes the fully rigorous construction of specific planar fronts. We refer to \Cref{sec:main-results} for a detailed list of the main results.

While the specific analysis is restricted to \eqref{eq:Swift-Hohenberg}, our framework applies much more broadly to generic rotational- and reflectional-invariant, semilinear pattern-forming systems close to a Turing instability. Indeed, wherever possible, we exploit generic geometric and algebraic properties of such systems in our analysis rather than relying on explicit computations.

\subsection{Formal amplitude system and nonlinear waves}\label{sec:amplitude-equations-intro}

As mentioned above, the dynamics of small-amplitude solutions close to a Turing instability are universal, that is, the dynamics are governed by a system of simpler PDEs whose form is independent of the specific application. To make this precise, we introduce the Fourier lattice $\Gamma$ given by
\begin{equation}\label{eq:fourier-lattice}
    \Gamma := \Bigl\{\sum_{j=1}^{3} n_j \k_j : n_j\in \Z\Bigr\},
\end{equation}
which is spanned by $\k_j$, $j=1,2,3$, defined by
\begin{equation}\label{eq:hex-lattice-generators}
    \k_1 = \begin{pmatrix}
        1 \\ 0
    \end{pmatrix}, \quad \k_2 = \frac{1}{2}\begin{pmatrix}
        -1 \\ \sqrt{3}
    \end{pmatrix}, \quad \k_3 = -\frac{1}{2} \begin{pmatrix}
        1 \\ \sqrt{3}
    \end{pmatrix}.
\end{equation}
Functions, which have a Fourier series with modes in $\Gamma$, are periodic with respect to a hexagonal lattice and, in particular, contain stripe and hexagonal patterns. Close to the onset of instability, i.e.~for $\mu = \eps^2 \mu_0$ with $0 < \eps \ll 1$ and $\mu_0 > 0$, solutions to \eqref{eq:Swift-Hohenberg} are, to leading order, concentrated at critical Fourier modes $\Gamma_0 := \Gamma \cap \{|\gammab| = 1\} = \{\pm \k_j \,:\, j=1,2,3\}$, where the solutions to the linear equation grow exponentially. Therefore, they are of the form
\begin{equation}\label{eq:amplitude-ansatz}
    u(t,\x) = \eps \sum_{j = 1}^3 A_j(\eps^2 t, \eps \x) e^{i \k_j \cdot \x} + c.c. + h.o.t.,
\end{equation}
where $c.c.$ and $h.o.t.$ denote complex-conjugated and higher-order terms, respectively. The amplitude modulations $A_j(T,\X) \in \C$ depend on the slow temporal variable $T = \eps^2 t$ and the slow spatial variable $\X = \eps \x$. To formulate the evolution equations for the amplitude modulations, we make additional assumptions on the nonlinearity.

\begin{assumption}\label{ass:rotation-invariant}
    Assume that the nonlinearity is rotation-invariant and reflection-invariant, in the sense that, if $u$ is a solution to \eqref{eq:Swift-Hohenberg}, then also $u \circ R$ is a solution for any $R\in \mathrm{O}(2)$.
\end{assumption}

Furthermore, we split $N(u,Du,D^2u,D^3u;\vartheta) = N_2(u,u;\vartheta) + N_3(u,u,u;\vartheta) + R(u;\vartheta)$ with a bilinear form $N_2$, a trilinear form $N_3$ and a polynomial remainder $R$, which contains all terms of degree at least four. Since the hexagonal lattice is resonant, i.e. $\k_1 + \k_2 + \k_3=0$, quadratic terms of leading-order modes can generate large contributions at critical modes. To control these interactions, we need to assume that they gain an additional order of $\eps$ through the nonlinearity.

\begin{assumption}\label{ass:smallness}
    Assume that there exists a submanifold $\Theta_0$ of the parameter space $\Theta$ such that $N_2(e^{i\k_j \cdot \x}, e^{i \k_\ell \cdot \x};\vartheta) = 0$ for all $|\k_j + \k_\ell| = 1$ and $\vartheta \in \Theta_0$. Furthermore, we assume that there are curves $\eps \mapsto \vartheta = \vartheta(\eps)$ in a neighbourhood of $\Theta_0$ such that $N_2(e^{i\k_j \cdot \x}, e^{i \k_\ell \cdot \x};\vartheta) = \eps \beta_2 + \Ocal(\eps^2)$ for all $|\k_j + \k_\ell| = 1$.
\end{assumption}

\begin{remark}
    We point out that this is a weaker assumption compared to the construction of \cite{doelman2003-02EuropeanJournalofAppliedMathematics}, where $N_2(u) = -\beta|\nabla u|^2$ and $\beta = \Ocal(\eps)$ is used.
\end{remark}

\begin{remark}
    We note that both \Cref{ass:rotation-invariant,ass:smallness} hold true in relevant physical examples such as asymptotic thin-film models of the Bénard–Marangoni problem, cf.~\cite{hilder2025-08JNonlinearSci}.
\end{remark}

Under \Cref{ass:rotation-invariant,ass:smallness}, the amplitude modulations $A_j$ satisfy the Ginzburg–Landau system
\begin{equation}\label{eq:amplitude-system-hex}
    \begin{split}
        \partial_T A_1 &= 4 (\k_1 \cdot \nabla_{\X})^2 A_1 + \mu_0 A_1 + \beta_2 \bar{A}_2 \bar{A}_3 + (K_0 |A_1|^2 + K_2 (|A_2|^2 + |A_3|^2))A_1, \\
        \partial_T A_2 &= 4 (\k_2 \cdot \nabla_{\X})^2 A_2 + \mu_0 A_2 + \beta_2 \bar{A}_1 \bar{A}_3 + (K_0 |A_2|^2 + K_2 (|A_1|^2 + |A_3|^2))A_2, \\
        \partial_T A_3 &= 4 (\k_3 \cdot \nabla_{\X})^2 A_3 + \mu_0 A_3 + \beta_2 \bar{A}_1 \bar{A}_2 + (K_0 |A_3|^2 + K_2 (|A_1|^2 + |A_2|^2))A_3 \\
    \end{split}
\end{equation}
with coefficients $\beta_2 \in \R$ and $K_0, K_2 \in \R$. We point out that the coefficients are the same in all three equations due to the rotational invariance and they are real-valued due to the reflectional invariance, cf.~\cite[Sec.~5.4]{hoyle2007book}.

The system \eqref{eq:amplitude-system-hex} can formally be derived from \eqref{eq:Swift-Hohenberg}, or, in fact, a general $\mathrm{O}(n)$-symmetric pattern-forming system close to a Turing instability, through multiscale expansion. In spatially one-dimensional systems, or more generally, systems on an infinite cylinder, there is a well-developed justification theory, which shows that the resulting amplitude equation, the classical real Ginzburg–Landau equation, makes good predictions for the dynamics over a long time interval. Moreover, there are approximation results stating that any small solution will, to leading order, eventually be localised in Fourier space at the critical wave numbers, and the corresponding amplitudes are again determined by the one-dimensional amplitude equation. It is expected that these results directly extend to the two-dimensional case if the solution is sufficiently localised in Fourier space at the lattice points in $\Gamma$ and the nonlinearity satisfies the smallness condition in \Cref{ass:smallness}. We refer to \cite[Ch.~10]{schneider2017book} for a more detailed discussion of this topic.

As discussed above, the main goal of this paper is to show the existence of pattern interfaces moving in direction $\d \in S^1$. These solutions correspond to travelling-wave solutions in the Ginzburg–Landau system \eqref{eq:amplitude-system-hex} of the form $A_j(T,X) = A_j(\Xi)$ with $\Xi = \d \cdot \X - c_0 T$ and $c_0 > 0$. These waves are solutions to the system
\begin{equation}\label{eq:travelling-wave-equation}
    \begin{split}
        0 &= 4 (\k_1 \cdot \d)^2 \partial_{\Xi}^2 A_1 + c_0 \partial_{\Xi} A_1 + \mu_0 A_1 + \beta_2 \bar{A}_2 \bar{A}_3 + (K_0 |A_1|^2 + K_2 (|A_2|^2 + |A_3|^2))A_1, \\
        0 &= 4 (\k_2 \cdot \d)^2 \partial_{\Xi}^2 A_2 + c_0 \partial_{\Xi} A_2 + \mu_0 A_2 + \beta_2 \bar{A}_1 \bar{A}_3 + (K_0 |A_2|^2 + K_2 (|A_1|^2 + |A_3|^2))A_2, \\
        0 &= 4 (\k_3 \cdot \d)^2 \partial_{\Xi}^2 A_3 + c_0 \partial_{\Xi} A_3 + \mu_0 A_3 + \beta_2 \bar{A}_1 \bar{A}_2 + (K_0 |A_3|^2 + K_2 (|A_1|^2 + |A_2|^2))A_3.
    \end{split}
\end{equation}
We highlight that the dynamics of the system \eqref{eq:travelling-wave-equation} depend on the chosen direction since the diffusion coefficients explicitly depend on it. Moreover, this shows that there is a set of special directions $\d$ which satisfy $\d \cdot \k_j = 0$ for some $j = 1,2,3$. In this case, the travelling-wave system \eqref{eq:travelling-wave-equation} has a degenerate-elliptic structure.

Such travelling waves with finite speed $c_0$ correspond to \emph{slow-moving} pattern interfaces in the Swift–Hohenberg equation \eqref{eq:Swift-Hohenberg} moving with speed $\eps c_0$ in direction $\d$. Indeed, we find that $\Xi = \eps(\d\cdot \x - \eps c_0 t)$. These fronts are the relevant class of solutions for describing invasion processes in real-world systems, in the sense that, heuristically, front speeds of order $\eps$ are selected, cf.~\Cref{sec:front-speed}. Moreover, the direction-dependence is a distinct feature of this class of solutions. In particular, the construction of \emph{fast-moving} fronts, considered for example in \cite{collet1986-03CommunMathPhys,hilder2025-08JNonlinearSci}, is independent of the direction. This can also be seen by considering the formal limit $c_0 \to \infty$ in \eqref{eq:amplitude-system-hex}, which can be made rigorous using fast-slow methods, cf.~\Cref{sec:fast-slow-results}.

\subsection{Selected speeds and marginal stability analysis}\label{sec:front-speed}

As is typical for invasion phenomena into unstable states, our analysis yields a continuous family of fronts with different speeds for each direction. This raises the natural question: which front is selected, i.e., which front forms in the time-evolution problem starting from steep or step-like initial data? This question has attracted much attention recently, and for invading fronts in spatially one-dimensional reaction-diffusion equations where the spatially homogeneous invading state is exponentially stable, a robust framework for stability and front selection is available, see e.g.~\cite{avery2022-07CommAmerMathSoc,avery2025-12preprint}. In addition, there are recent advances in generalising this framework to steady pattern-forming fronts where the front is stationary in an appropriate co-moving frame, see \cite{avery2026-03preprint}. In contrast, the pattern interfaces constructed in this paper are non-steady as they describe a connection of two stationary patterns via a moving front. To the best of our knowledge, a rigorous framework is not available for such non-steady pattern-forming fronts, even in one spatial dimension, or spatially two-dimensional reaction-diffusion systems.

Although a rigorous framework is lacking, there are formal conjectures about front selection in situations where the front dynamics are determined by the linear behaviour about the unstable state, which are referred to as pulled fronts. Here, the conjectured selected front speed can be determined through the \emph{marginal stability conjecture}, see e.g.~\cite{vansaarloos2003-11PhysicsReports} and \cite[Def.~6.8 and Conj.~6.9]{avery2025-12preprint}. For this, one observes that the unstable state, when viewed in a moving reference frame, transitions from absolutely unstable to convectively unstable as the reference frame speed increases. That is, at low speeds, a perturbation of the unstable state will grow exponentially at each fixed point, while at high speeds it will decay exponentially at each fixed point. The marginal stability conjecture states that the selected speed is the one at which the system transitions from absolute to convective instability.

In practice, the conjectured selected speed is determined from the dispersion relation of the linearisation about the unstable state in a co-moving frame $\xib = \x - \d c t$, which is obtained by making an ansatz of the form $u(t,\xib) = e^{\lambda t + \nub \cdot \xib}$ with $\lambda \in \C$ and $\nub \in \C^2$. In our case, the resulting dispersion relation reads as
\begin{equation}\label{eq:dispersion-relation}
    d(\lambda, \nub, c) := L(\nub;\mu) + c \d \cdot \nub - \lambda = - (1 + \nub \cdot \nub)^2 + \mu + c \d \cdot \nub - \lambda.
\end{equation}
Following \cite{holzer2014-08JNonlinearSci}, we now assume periodicity in the transverse direction of the front, i.e., in $\d^\perp$-direction. This is reflected by writing $\nub = \nu \d + i k_\perp \d^\perp$ with $\nu \in \C$, $k_{\perp}\in \R$. In addition, since $\lambda$ is the growth rate of the linear problem in a moving frame with speed $c$, to determine the selected speed, we set $\lambda = i\omega$ with $\omega \in \R$. Then, we obtain the spatial wave number $\nu \in \C$, the speed $c > 0$ and the temporal oscillation $\omega \in \R$ by assuming that $\nu \mapsto d(i \omega, \nu \d + i k_{\perp} \d^\perp, c)$ has a pinched double root, cf.~\cite[Def.~2.3]{avery2025-12preprint}. A necessary condition for this is
\begin{equation*}
    d(i \omega, \nu \d + i k_{\perp} \d^\perp, c) = 0, \quad \text{and} \quad \partial_\nu d(i \omega, \nu \d + i k_{\perp} \d^\perp, c) = 0
\end{equation*}
for fixed $k_\perp \in \R$. While this is not a sufficient condition for a pinched double root, it is often used in practice to determine the critical speed.

Applied to our case, where the dispersion relation is given by \eqref{eq:dispersion-relation}, we first determine $\Re\nu$ and $\Im\nu$ by solving
\begin{equation*}
    \begin{split}
        \Im(\partial_{\nu} L(\nu\d + i k_{\perp}\d^{\perp},\mu)) & = 0, \\
        \Re(L(\nu\d + i k_{\perp}\d^{\perp},\mu)) & = c  \Re \nu = \Re(\partial_{\nu} L(\nu\d + ik_{\perp}\d^{\perp},\mu)) \Re \nu,
    \end{split}
\end{equation*}
cf.~\cite[Eq.~(2.14)]{avery2025-12preprint}. Since we are looking for fronts which invade the unstable state and travel from left to right, we look for solutions to the linear equation which decay as $\d \cdot \xib \to \infty$. Therefore, assuming $\Re(\nu) < 0$, we find
\begin{equation*}
\begin{split}
    \nu & = -\frac{1}{2 \sqrt{3}}\sqrt{\sqrt{\left(1-k^2_{\perp}\right)^2+6 \mu }-(1-k_{\perp}^2)} \pm i \frac{1}{2} \sqrt{\sqrt{\left(1-k^2_{\perp}\right)^2+6 \mu }+3(1-k^2_{\perp})},\\
    c &= \frac{4}{3 \sqrt{3}}\left(\sqrt{\left(1-k^2_{\perp}\right)^2+6 \mu } +2(1-k^2_{\perp})\right) \sqrt{\sqrt{\left(1-k^2_{\perp}\right)^2+6 \mu }- (1-k^2_{\perp})},\\
    \omega &= \mp\frac{1}{6 \sqrt{3}}\left(5 \sqrt{\left(1-k^2_{\perp}\right)^2+6 \mu }+ 7(1-k^2_{\perp})\right) \sqrt{\sqrt{\left(1-k^2_{\perp}\right)^2+6 \mu }+3(1- k^2_{\perp})} \sqrt{\sqrt{\left(1-k^2_{\perp}\right)^2+6 \mu }-(1-k^2_{\perp})}.
\end{split}
\end{equation*}
Assuming that $\mu = \eps^2 \mu_0$, i.e., that the problem is close to the onset of instability, we find, by expanding with respect to $\eps$, that
\begin{equation}\label{eq:results-marginal-stability-analysis}
    \begin{split}
        \nu & = \pm i \sqrt{1-k_{\perp}^2} + \eps \frac{1}{2}\sqrt{\frac{\mu_0}{1-k_{\perp}^2}} + \Ocal(\eps^2), \\
        c & = \eps 4 \sqrt{(1-k_\perp^2) \mu_0} + \Ocal(\eps^2), \\
        \omega & =  \mp \eps 4 (1-k_\perp^2) \sqrt{\mu_0} + \Ocal(\eps^2).
    \end{split}
\end{equation}
From these calculations, we obtain two crucial observations. First, we note that the selected speed is of order $\eps$ and therefore, slow-moving fronts are selected close to the onset of instability, cf.~\cite{collet2002-09JournalofStatisticalPhysics,avery2025-03ProcAmerMathSoc}. Second, as already pointed out in \cite{holzer2014-08JNonlinearSci}, we find that the maximal speed is obtained at $k_\perp = 0$. Therefore, it is conjectured in \cite{holzer2014-08JNonlinearSci} that at the leading edge one finds stripe patterns orthogonal to the direction of the planar front, which seems to be confirmed by numerical experiments, see e.g.~\cite{pismen1994-08EurophysLett,csahok1999-08EurophysLett}. For a more detailed discussion, we refer to \cite{holzer2014-08JNonlinearSci,avery2025-12preprint} and \Cref{sec:discussion}.

\subsection{General strategies and main challenges}\label{sec:general-strategy}

To construct the pattern interfaces, we use a spatial-dynamics approach, \cite{kirchgassner1982-07JDiffEq}, where an elliptic system is written as a dynamical system by interpreting a spatial variable as 'time'. This approach has been used successfully to construct, primarily one-dimensional, pattern interfaces \cite{eckmann1991-02CommunMathPhys,haragus-courcelle1999-01ZangewMathPhys,doelman2003-02EuropeanJournalofAppliedMathematics,faye2015-04JournalofDifferentialEquations,hilder2020-08JournalofDifferentialEquations,hilder2022-09JournalofMathematicalAnalysisandApplications,hilder2025-08JNonlinearSci}. For this, we make the ansatz
\begin{equation}\label{eq:modfront-ansatz}
    u(t,\x) = U(\d \cdot \x - \eps c_0 t , \x)
\end{equation}
and assume that $U(\xi,\p)$ is periodic in $\p$ with respect to a hexagonal lattice, and that $U$ connects two different planar pattern solutions $u_1$ and $u_2$ (including the trivial one) to \eqref{eq:Swift-Hohenberg}, that is,
\begin{equation*}
    \lim_{\xi \to -\infty} U(\xi,\p) = u_1(\p) \quad \text{and} \quad \lim_{\xi \to \infty} U(\xi,\p) = u_2(\p).
\end{equation*}
Inserting the ansatz \eqref{eq:modfront-ansatz} into \eqref{eq:Swift-Hohenberg} yields
\begin{equation*}
    - c \partial_{\xi} U = \Fcal(\d \partial_\xi + \nabla_\p;\mu,\vartheta)(U)
\end{equation*}
for some smooth function $\Fcal(\cdot;\mu,\vartheta)$, where we used that $\nabla u = (\d \partial_\xi + \nabla_\p) U = \d (\partial_\xi + \d \cdot \nabla_\p) U + \d^\perp (\d^\perp \cdot \nabla_\p) U$. Since the nonlinearity $N$ in \eqref{eq:Swift-Hohenberg} contains at most third-order derivatives, the highest-order derivative in $\xi$ comes from the linear part and we can write
\begin{equation}\label{eq:spat-dyn-step1}
    (\partial_\xi + \d \cdot \nabla_\p)^4 U = c \partial_\xi U + (\Fcal(\d \partial_\xi + \nabla_\p;\mu,\vartheta)(U) + (\partial_\xi + \d \cdot \nabla_\p)^4 U)
\end{equation}
to isolate the $\partial_\xi^4 U$-term. Due to the periodicity in $\p$, we can write $U$ as a Fourier series in $\p$, that is
\begin{equation*}
    U(\xi,\p) = \sum_{\gammab \in \Gamma} \hat{U}(\xi,\gammab) e^{i\gammab \cdot \p}.
\end{equation*}
Inserting this, we can write the resulting systems as a first-order system in $\xi$, which reads as
\begin{equation}\label{eq:spat-dyn-Fourier}
    \partial_\xi \hat{W}(\cdot, \gammab) = \hat{\Lcal}(\gammab;\mu,c) \hat{W}(\cdot,\gammab) - i \d \cdot \gammab \hat{W}(\cdot,\gammab) + \hat{\Ncal}(\hat{W};\gammab;\vartheta)
\end{equation}
for all $\gammab \in \Gamma$, where $\hat{W}(\xi,\gammab) = (\hat{U}(\xi,\gammab), (\partial_{\xi} + i\d \cdot \gammab) \hat{U}(\xi,\gammab), (\partial_{\xi} + i \d \cdot \gammab)^2 \hat{U}(\xi,\gammab), (\partial_{\xi} + i \d \cdot \gammab)^3 \hat{U}(\xi,\gammab))$ and $\hat{\Lcal}(\mu,c) = \bigoplus_{\gammab\in \Gamma} \hat{\Lcal}(\gammab;\mu,c) $ explicitly given in \eqref{eq:spat-dyn-linear}.
The nonlinearity $\Ncal$ is given by
\begin{equation*}
    \hat{\Ncal}(\hat{W};\gammab;\vartheta) = (0,0,0,\hat{N}(\hat{W};\gammab;\vartheta))^T,
\end{equation*}
where $\hat{N}(\hat{W};\cdot;\vartheta)$ is the Fourier transform of $N(U,(\d \partial_\xi + \nabla_p)U,(\d \partial_\xi + \nabla_p)^2U,(\d \partial_\xi + \nabla_p)^3U)$ with respect to $\p$. We specifically note that $\hat{\Lcal}$ and $\hat{\Ncal}$ in \eqref{eq:spat-dyn-Fourier} only contain Fourier multipliers of the form $(\d^\perp \cdot \gammab)$ which correspond to transverse derivatives. This will be crucial in our analysis as discussed in \Cref{sec:function-spaces}.

As in \Cref{sec:amplitude-equations-intro}, we introduce a small scaling parameter $\eps > 0$ by $\mu = \mu_0 \eps^2$ and $c = c_0 \eps$ with $\mu_0,c_0 = \Ocal(1)$. As discussed in \Cref{sec:front-speed}, this scaling mimics the expected spreading speed of a front evolving from compactly supported initial data close to the onset of instability, see also \cite{collet2002-09JournalofStatisticalPhysics,avery2025-07JEurMathSoc} and \Cref{app:front-speed-numerics}. Using this rescaling, we introduce the following notation for the linear operators
\begin{equation*}
    L^\eps(\nabla) := L(\nabla;\eps^2 \mu_0)
    \quad\text{and}\quad
    \hat{\Lcal}^\eps := \hat{\Lcal}(\eps^2 \mu_0, \eps c_0).
\end{equation*}

The next step is to reduce the infinite-dimensional dynamical system \eqref{eq:spat-dyn-Fourier} to a finite-dimensional system using centre manifold theory. For this, we need to analyse the spectrum of the full linear operator 
\begin{equation*}
    \tilde{\Lcal}^{\eps} := \bigoplus_{\gammab\in \Gamma} \tilde{\Lcal}^\eps(\gammab) \quad \text{with} \quad \tilde{\Lcal}^\eps(\gammab) := \hat{\Lcal}^{\eps}(\gammab) - i\d\cdot \gammab I.
\end{equation*}
The main challenge in the centre-manifold analysis is that $\tilde{\Lcal}^0$ has infinitely many, purely imaginary eigenvalues, which move away from the imaginary axis at different speeds for $\eps > 0$, creating an $\eps$-dependent splitting between finitely many \emph{more-central} and infinitely many \emph{less-central} eigenvalues. Although a similar problem appeared in previous papers, see e.g.~\cite{haragus-courcelle1999-01ZangewMathPhys}, in two-dimensional systems with a general class of quadratic nonlinearities, additional obstructions arise due to the resonance of the hexagonal lattice. In addition, the spectrum depends crucially on the direction $\d$. Writing $\d = (\cos(\theta),\sin(\theta))$, it turns out that the hyperbolic eigenvalues satisfy a spectral gap condition only under the condition that $\cot(\theta)\in \sqrt{3}\Q$ and, to avoid further technicalities related to small-devisor problems, see \Cref{sec:discussion} for a discussion, we adopt this assumption throughout the paper. Note that this condition is satisfied for $\theta = 0$ and $\theta =\tfrac{\pi}{6}$.

The final step of the analysis is the derivation and analysis of the system of reduced equations on the centre manifold. As expected, this system is indeed given by the travelling-wave system \eqref{eq:travelling-wave-equation} obtained from the amplitude system \eqref{eq:amplitude-system-hex}. Due to the high dimensionality of the reduced system, constructing heteroclinic orbits between equilibria is another main challenge, which we at least partially address in this paper.

\subsection{Main results of the paper}\label{sec:main-results}

We now summarise the main results and techniques of this paper. We show that
\begin{itemize}
    \item for $\eps > 0$ sufficiently small, $\cot(\theta) \in \sqrt{3}\Q$, $\mu_0\in \R\setminus\{0\}$, and $c_0 > 0$, the system \eqref{eq:spat-dyn-Fourier} has a finite-dimensional, locally invariant centre manifold, which contains all small bounded solutions to \eqref{eq:spat-dyn-Fourier} of size $\Ocal(\eps^{3/4 + \delta})$ with $0 < \delta < \tfrac{1}{4}$ and the reduced equations on the centre manifold are, to leading order, given by \eqref{eq:travelling-wave-equation}, see \Cref{thm:centre-manifold};
    \item for $\mu_0 > 0$, $K_0 < 0 $, $K_2 < -\tfrac{K_0}{2}$, and $c_0 > 0$, the reduced equations on the centre manifold have persistent heteroclinic orbits connecting a non-trivial equilibrium to the trivial equilibrium, which correspond to an invasion of the spatially homogeneous state by a planar pattern in \eqref{eq:Swift-Hohenberg}, see \Cref{fig:pattern-interfaces} and \Cref{thm:main-theorem}.
\end{itemize}

\begin{figure}[h]
    \centering
    \begin{subfigure}[b]{0.99\textwidth}
        \includegraphics[width=\linewidth]{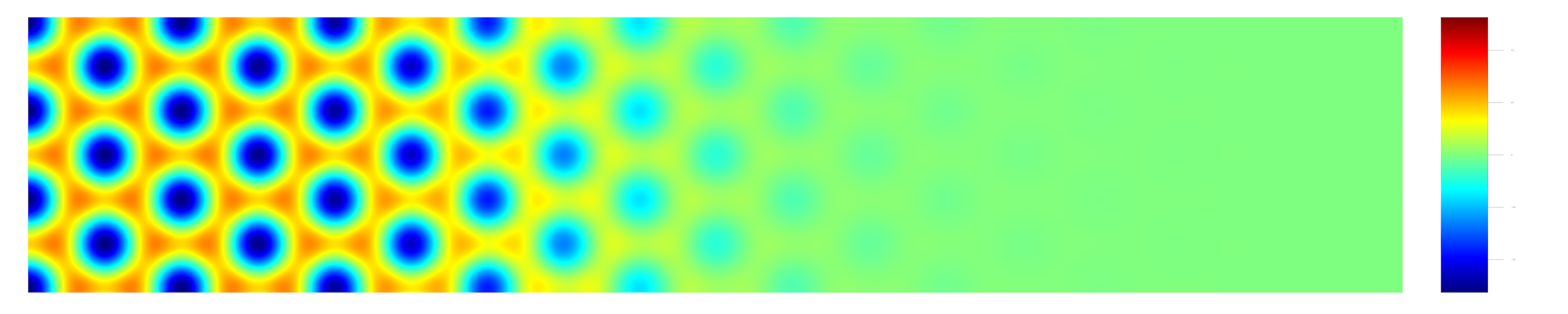}
        \includegraphics[width=\linewidth]{img/fronts/rw-rigorous_theta_0.png}
        \subcaption{Pattern interfaces for angle $\theta = 0$.}\label{subfig:pattern-interfaces-a}
    \end{subfigure}

    \vspace{0.2cm}
    
    \begin{subfigure}[b]{0.99\textwidth}
        \includegraphics[width=\linewidth]{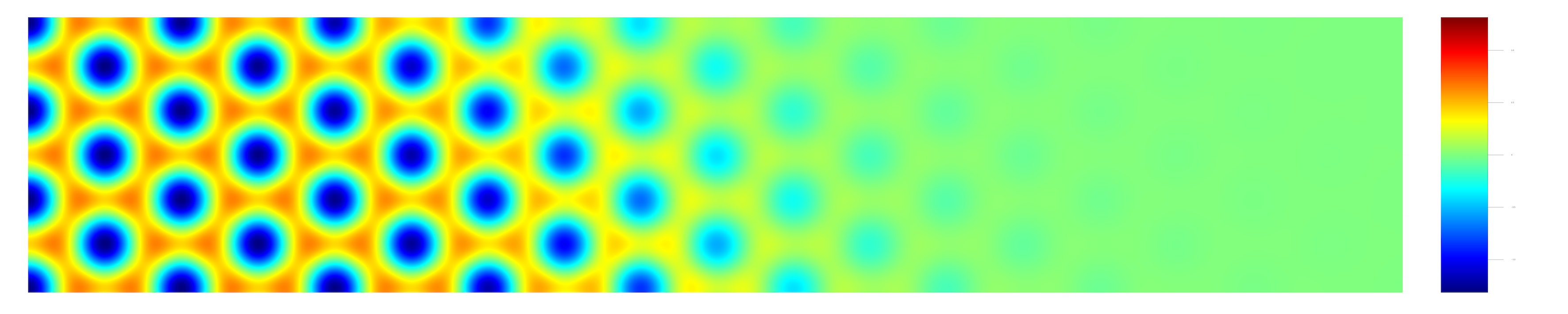}
        \includegraphics[width=\linewidth]{img/fronts/rw-rigorous_theta_30.png}
        \subcaption{Pattern interfaces for angle $\theta = \tfrac{\pi}{6}$.}\label{subfig:pattern-interfaces-b}
    \end{subfigure}
    \caption{The pattern interfaces in panels \sref{subfig:pattern-interfaces-a} and \sref{subfig:pattern-interfaces-b} show two examples of rigorously constructed solutions to \eqref{eq:Swift-Hohenberg} in \Cref{thm:main-theorem}. They describe the invasion of the spatially homogeneous state by down-hexagons and roll waves, respectively.}
    \label{fig:pattern-interfaces}
\end{figure}

\begin{remark}
    We focus on the analysis for patterns, which are periodic with respect to a hexagonal lattice. However, the same analysis can also be done for square patterns, see \Cref{sec:square}. In fact, this case is less challenging since the corresponding Fourier lattice is non-resonant.
\end{remark}

\begin{remark}
    We note that due to rotation and reflection symmetry, it is sufficient to consider directions with angle $\theta \in [0,\tfrac{\pi}{6}]$.
\end{remark}

\paragraph{Spectral Analysis}

The key insight for the spectral analysis is that the eigenvalues of $\tilde{\Lcal}$ can be characterised through the linearisation of the Swift–Hohenberg equation \eqref{eq:Swift-Hohenberg}, specifically, through roots of $L^\eps(\d \lambda + i \gammab) + \eps c_0 \lambda$ and their multiplicity, cf.~\Cref{lem:eigenvalue-correspondence,lem:multiplicity}. This is not a new observation and has been pointed out, for example, in \cite{afendikov1995-06ArchRationalMechAnal}. For two-dimensional fronts, this allows for a direct geometric interpretation: if the line $\ell_{\gammab} := \{\gammab + \d \lambda \,:\, \lambda \in \R\}$ intersects the unit circle, the corresponding $\tilde{\Lcal}^0(\gammab)$ has only purely imaginary eigenvalues with imaginary parts given by the distance between the intersection points and $\gammab$, see \Cref{fig:geometric-argument}. In particular, the critical Fourier modes yield eigenvalues at zero. Moreover, if $\ell_{\gammab}$ has no intersection points, the eigenvalues of $\tilde{\Lcal}^0(\gammab)$ have real part with absolute value given by the minimal distance of $\ell_{\gammab}$ to the circle. In particular, this shows that there is a spectral gap around the imaginary axis for $\eps = 0$ if and only if the lattice is periodic in direction $\d$, which is equivalent to the condition $\cot(\theta)\in \sqrt{3}\Q$. In addition to the value of the eigenvalues, the geometric picture also characterises their algebraic multiplicity, which is given by the multiplicity of the root of $\lambda \mapsto L^0(\d \lambda + i \gammab)$. So, the eigenvalues belong to two Jordan blocks of size two if the line $\ell_{\gammab}$ has two or zero intersections with the circle or a Jordan block of size four if $\ell_{\gammab}$ is tangent to the circle, using the generic structure of the Turing instability. 

For $\eps > 0$, the eigenvalues on the imaginary axis depart at different speeds in $\eps$. Indeed, any eigenvalues emerging from zero have real part $\Ocal(\eps)$ and any eigenvalues emerging from a non-zero point on the imaginary axis have real part $\Ocal(\sqrt{\eps})$. The only exception is the case $\theta = \tfrac{\pi}{6}$, where there are Jordan blocks of size four at zero corresponding to $\gammab = \pm \k_2$ and which split into one eigenvalue with real part $\Ocal(\eps)$ and three eigenvalues with real part $\Ocal(\eps^{1/3})$. Again, this is a generic effect, since the system is close to a (rotation-symmetric) Turing instability, cf.~\Cref{app:general-spectral-analysis}.

The spectral picture for $\eps > 0$ is thus as follows: there are finitely many more-central eigenvalues with real part $\Ocal(\eps)$, and the remaining eigenvalues split into less-central eigenvalues with real part $\Ocal(\sqrt{\eps})$ and hyperbolic eigenvalues.

\paragraph{Centre manifold reduction}

The next step is to prove that there is a locally invariant centre manifold associated with the finitely many more-central eigenvalues. As discussed above, the main obstruction is the non-uniform spectral gap. Therefore, one has to carefully track the dependence on $\eps$ throughout the proof of the centre manifold theorem. This has been done before, mainly for systems with one unbounded spatial direction, see e.g.~\cite{haragus-courcelle1999-01ZangewMathPhys}, and the main steps are as follows: first, we split the unknown variable into a more-central, less-central, and hyperbolic part. Second, we introduce rescaled variables, where the more-central, less-central, and hyperbolic parts are rescaled with suitable fractional powers of $\eps$, and which allow us to determine the size of the centre manifold. 

The third step is then to show that the resulting rescaled system has a centre manifold of size $\Ocal(1)$. For this, we first apply a normal-form transformation to remove problematic quadratic nonlinear terms with a singular scaling in $\eps$. Here, new effects emerge in the two-dimensional case since the typical non-resonance condition in the normal-form analysis, cf.~\cite[Sec.~3.3]{guckenheimer1983book}, is not satisfied uniformly in $\eps$. Nevertheless, it is possible to obtain a near-identity normal-form transformation, cf.~\eqref{eq:normal-form-lc} and \Cref{lem:normal-form-analysis} by exploiting the $\sqrt{\eps}$-behaviour of the less-central eigenvalues. The second step is to obtain sufficient semigroup estimates, cf.~\Cref{lem:semigroup-bounds-full}, and to construct a centre manifold using the usual fixed-point construction. Here, the main obstruction in two dimensions is that the semigroups are analytic in $\d^\perp$-direction, but do not gain enough regularity in $\d$-direction to compensate for more than one derivative in $\d$-direction. Therefore, the fact that the nonlinearities of the spatial-dynamics system \eqref{eq:spat-dyn-Fourier} only contain $\d^\perp$-derivatives becomes crucial. To exploit this structure, we perform the fixed-point argument in periodic Sobolev spaces with mixed dominating smoothness, see \eqref{eq:function-space-mixed-smoothness}. Carefully tracking the dependence on $\eps$ then yields a centre manifold of size $\Ocal(\eps^{3/4+\delta})$ for $\delta \in (0,\tfrac{1}{4})$, cf.~\Cref{thm:centre-manifold}. This specifically is large enough to contain solutions to \eqref{eq:Swift-Hohenberg} of size $\eps$, which covers the pattern-dynamics near onset, cf.~\eqref{eq:amplitude-ansatz}. Notably, the resulting centre manifold is smaller than in the one-dimensional case, cf.~\cite{haragus-courcelle1999-01ZangewMathPhys}, where the size is of order $\eps^{2/3 + \delta}$. This difference originates directly from the degenerate normal-form transformation, which yields an additional condition $2 \beta - \tfrac{1}{2} > 1$ from the Lipschitz constant of the less-central nonlinearity, cf.~\Cref{lem:Lipschitz}, which is not present in the one-dimensional case.

On the centre manifold, we then derive the reduced equation
\begin{equation}\label{eq:reduced-equations-introduction}
\begin{split}
    4(\d\cdot\k_1)^2 \partial_{\Xi}^2 A_1 + c_0 \partial_{\Xi} A_1 + \mu_0 A_1 + \beta_2\bar{A}_2\bar{A}_3 + (K_0|A_1|^2 + K_2(|A_2|^2 + |A_3|^2)) A_1 + \Ocal(\eps) & = 0, \\
    4(\d\cdot\k_2)^2 \partial_{\Xi}^2 A_2 + c_0 \partial_{\Xi} A_2 + \mu_0 A_2 + \beta_2\bar{A}_1\bar{A}_3 + (K_0|A_2|^2 + K_2(|A_1|^2 + |A_3|^2)) A_2 + \Ocal(\eps) & = 0, \\
    4(\d\cdot\k_3)^2 \partial_{\Xi}^2 A_3 + c_0 \partial_{\Xi} A_3 + \mu_0 A_3 + \beta_2\bar{A}_1\bar{A}_2 + (K_0|A_3|^2 + K_2(|A_1|^2 + |A_2|^2)) A_3 + \Ocal(\eps) & = 0,
\end{split}
\end{equation}
which in particular agrees to leading order with the travelling-wave amplitude equations \eqref{eq:travelling-wave-equation}. Although the derivation is mostly standard, we introduce a systematic change of variables exploiting the Jordan-block structure at $\eps = 0$, see \eqref{eq:clever-basis}. This agrees with the ad-hoc change of variables used, for example, in \cite{eckmann1991-02CommunMathPhys}, but makes the connection to the algebraic structure of the eigenvalues more transparent.

Note that there are two special cases, which require separate treatment. The first case is $\theta = \tfrac{\pi}{6}$. As discussed above, in this case, $\tilde{\Lcal}^0(\pm \k_2)$ has an eigenvalue at zero corresponding to a Jordan block of size four, which splits into one more-central and three less-central eigenvalues. Therefore, in contrast to the generic case $\theta \neq \tfrac{\pi}{6}$, we cannot define uniformly bounded projections separating the more-central and less-central eigenvalues. However, it is still possible to construct a centre manifold of the same size, cf.~\Cref{sec:special-case}, by defining non-normalised 'projection-like' maps, see \eqref{eq:not-projections}, which still map into the relevant eigenspace. With this, the rest of the proof works the same as in the generic-angle case and the reduced equations on the centre manifold are again given by \eqref{eq:reduced-equations-introduction} with $\d\cdot\k_2 = 0$.

The second case is $c_0 = c_{\crit}(\k_j)$, where
\begin{equation}\label{eq:c-crit}
    c_{\crit}(\k_j) := 4|\d\cdot \k_j|\sqrt{\mu_0}
\end{equation}
is the selected speed predicted by marginal stability analysis, cf.~\Cref{sec:front-speed}. Notably, it can be obtained from \eqref{eq:results-marginal-stability-analysis} where the transverse wave number $k_{\perp}$ is given by $\d^{\perp}\cdot \k_j$. In the spectral analysis, this speed yields a different $\eps$-expansion of the more-central eigenvalues. Specifically, for $c_0 \neq c_{\crit}(\k_j)$ for all $j = 1,2,3$, the more-central eigenvalues coming from $\tilde{\Lcal}^{\eps}(\k_j)$ split at leading order $\eps$ for all $j=1,2,3$. In contrast, for $c_0 = c_{\crit}(\k_j)$ for fixed $j = 1,2,3$ the leading-order behaviour of the more-central eigenvalues from $\tilde{\Lcal}^\eps(\k_j)$ is the same and they only split at order $\eps^{3/2}$. While this does not affect the construction of the centre manifold, since the more-central eigenvalues still are of order $\eps$, it changes the derivation of the reduced equations, see \Cref{sec:reduced-equations-critical}. However, the resulting leading-order system on the centre manifold is still \eqref{eq:travelling-wave-equation}. Indeed, this is expected since the derivation of the amplitude system \eqref{eq:amplitude-system-hex} is done in the stationary frame. Nevertheless, as pointed out in \Cref{sec:stability-of-equilibria}, the dynamics of the reduced equation still change when $c_0$ passes through $c_{\crit}(\k_j)$ for some $j = 1,2,3$.

\paragraph{Dynamics of reduced equations and pattern interfaces}

Finally, we analyse the dynamics of the reduced system \eqref{eq:reduced-equations-introduction} on the centre manifold. For this, we restrict the corresponding twelve-dimensional (note the complex conjugate equations) first-order system to the invariant, six-dimensional subspace of real-valued solutions. For the special direction $\theta = 0$, there is an additional four-dimensional invariant subspace, see \cite{doelman2003-02EuropeanJournalofAppliedMathematics}. However, to deal with general directions, we analyse the full six-dimensional system. We first collect non-trivial equilibria, which correspond to planar patterns in \eqref{eq:Swift-Hohenberg} and which can be obtained using standard results from equivariant bifurcation theory. To analyse their stability, we establish a rigorous relation between the temporal stability of an equilibrium as a stationary state in the Ginzburg–Landau system \eqref{eq:amplitude-system-hex} without diffusion, and the stability in the spatial dynamics system on the centre manifold, see \Cref{lem:characterisation-spat-dyn-stable-vs-pde-stable}. Specifically, temporally fully unstable equilibria are stable in the reduced equation on the centre manifold, and temporally stable equilibria are saddles with an even splitting between stable and unstable directions. In particular, this implies that for $\mu_0 > 0$, the trivial equilibrium is stable. While this characterisation was previously noted, see e.g.~\cite{doelman2003-02EuropeanJournalofAppliedMathematics}, to the best of our knowledge, a rigorous proof is not available in the literature. Moreover, we discuss that the result also applies to reaction-diffusion systems when the linearisation about an equilibrium is symmetric; otherwise, it might fail, see \Cref{rem:spectral-characterisation-for-RD}.

To obtain heteroclinic orbits, we use that the leading-order first-order system has a strictly decreasing Lyapunov function, see \eqref{eq:hamiltonian}. Restricting to the case $\mu_0 > 0$, $c_0 > 0$, $K_0 < 0$ and $K_2 < -\tfrac{K_0}{2}$, we use this to prove the existence of heteroclinic orbits between the non-trivial equilibrium with lowest energy and the (stable) trivial equilibrium, see \Cref{thm:heteroclinic} for the case $\theta \neq \tfrac{\pi}{6}$. The proof relies on constructing an appropriate bounded trapping region around the trivial equilibrium using the Lyapunov function. Their persistence then follows using the transverse intersection of stable and unstable manifolds, see \Cref{thm:persistence}. Finally, using fast-slow analysis, we show that these orbits converge to a heteroclinic orbit connecting the same patterns along a sequence of angles $\theta$ satisfying $\cot(\theta) \in \sqrt{3}\Q$ and converging to $\tfrac{\pi}{6}$. Similarly, the fast-moving pattern interfaces obtained in \cite{hilder2025-08JNonlinearSci} can be recovered in the limit $c_0 \to \infty$ as expected.

In addition to the rigorous analysis, we find heteroclinic orbits using a numerical shooting algorithm, which yields orbits connecting the other non-trivial equilibrium to the trivial state, see \Cref{fig:planar-front,fig:planar-fronts-numerics} for the corresponding planar interfaces. In particular, we find orbits from up-hexagons to the trivial state, which pass close to the roll waves and thus correspond to a two-stage invasion described in \Cref{sec:front-speed} where the primary and secondary invasion fronts have the same speed, see \Cref{subfig:two-stage}.

\subsection{Related results}\label{sec:related-results}

We now discuss related results. As mentioned above, the existence of pattern interfaces using spatial dynamics and centre manifold theory is well-known for systems in an infinite cylinder. Specifically, they have been constructed for the one-dimensional Swift–Hohenberg equation \cite{eckmann1991-02CommunMathPhys}, the Taylor–Couette problem \cite{haragus-courcelle1999-01ZangewMathPhys}, and a non-local reaction-diffusion system \cite{faye2015-04JournalofDifferentialEquations}. In the case of a Turing(–Hopf) instability with additional conservation law, which for example appears in the Bénard–Marangoni problem, one-dimensional pattern interfaces have been found in phenomenological models \cite{hilder2020-08JournalofDifferentialEquations,hilder2022-09JournalofMathematicalAnalysisandApplications}. Recently, the existence, stability, and selection of related (steady) front solutions have also been studied in the FitzHugh–Nagumo system \cite{avery2025-07JEurMathSoc,avery2026-03preprint}.

In systems with two unbounded spatial directions, results are sparser. Besides the existence result for planar pattern interfaces in \cite{doelman2003-02EuropeanJournalofAppliedMathematics}, much attention has been focused on stationary interfaces between differently oriented roll waves, which are also known as grain boundary or domain walls. Their existence has been obtained for a two-dimensional Swift–Hohenberg equation \cite{haragus2007-01InternationalJournalofDynamicalSystemsandDifferentialEquations,haragus2012-12EuropeanJournalofAppliedMathematics} and, more recently, for the Rayleigh–Bénard problem \cite{haragus2021-02ArchRationalMechAnal,haragus2022-12JDynDiffEquat,buffoni2023-05JournalofDifferentialEquations,iooss2025JMathFluidMech,iooss2026-03JDynDiffEquat}. While these results are restricted to semilinear systems, fast-moving pattern interfaces with $c = \Ocal(1)$ have recently been obtained in a quasilinear asymptotic thin-film model for the Bénard–Marangoni problem \cite{hilder2025-08JNonlinearSci}. We expect that the generic approach of this paper can also be transferred to construct slow-moving pattern interfaces in these models. However, there are substantial additional difficulties which we discuss in more detail in \Cref{sec:discussion}.

All of the previously mentioned results used a spatial-dynamics approach to obtain the existence of pattern interfaces. This approach does not directly translate to systems that exhibit multiple, simultaneous pattern-forming instabilities which occur at resonant wave numbers. Here, difficulties arise from additional interactions of the critical modes. For these systems, the existence of solutions in one spatial dimension that are close to pattern interfaces at least on a long time interval has been established using modulation theory \cite{gauss2021-05ChaosAnInterdisciplinaryJournalofNonlinearScience,hilder2025StudiesinAppliedMathematics}.

The study of steady localised solutions to pattern-forming systems has received much interest in recent years, and we refer to \cite{bramburger2025-03preprint} for a recent overview. Many of these results rely on numerical continuation methods, for example, obtained using pde2path \cite{uecker2021-01book}, see \cite{uecker2014-01SIAMJApplDynSyst} for an application. While these references mainly consider stationary solutions and the existence of radial front solutions remains a largely open question, we highlight the recent derivation of radial amplitude equations \cite{hill2024-12SIAMJApplMath}.

Numerical continuation methods in combination with far-field-core decompositions have also been used recently to study the behaviour of invading one-dimensional pattern interfaces beyond onset \cite{lloyd2025-02preprint}.

\subsection{Outline of the paper}

The paper is structured as follows: In \Cref{sec:spectral-analysis}, we perform a spectral analysis of the linearity in the spatial-dynamics system \eqref{eq:spat-dyn-Fourier} by first analysing the spectrum for $\eps = 0$, and then characterising the behaviour of the eigenvalues for $\eps > 0$. 

In \Cref{sec:centre-manifold-theory}, we prove a centre manifold theorem for the spatial-dynamics system \eqref{eq:spat-dyn-Fourier} and derive the reduced equations. For this, we first introduce an appropriate functional-analytic setup based on mixed-regularity Sobolev spaces and then perform a normal-form analysis in a rescaled system. Then, we prove the necessary bounds on the semigroups to obtain a sufficiently large centre manifold with a standard fixed-point argument. In the analysis, we first consider the generic case $\theta \neq \tfrac{\pi}{6}$ and $c_0 \neq c_{\crit}(\k_j)$. Then, we provide a proof of the centre manifold theorem in the special case $\theta = \tfrac{\pi}{6}$, and a derivation of the reduced equation in the case $c_0 = c_{\crit}(\k_j)$ separately.

\Cref{sec:interfaces} then contains the analysis of the dynamics of the reduced equations on the centre manifold. Here, we first discuss the equilibria which correspond to planar patterns in \eqref{eq:Swift-Hohenberg}, and characterise their stability. Afterwards, we prove the existence of persistent heteroclinic orbits and use this to establish the existence of slow-moving pattern interfaces in \eqref{eq:Swift-Hohenberg}, see \Cref{thm:main-theorem}.

In \Cref{sec:square}, we discuss how our results transfer from the hexagonal lattice case to the square lattice case. We conclude the main body of the paper by discussing related questions in \Cref{sec:discussion}. In \Cref{app:general-spectral-analysis}, we then provide a detailed discussion of how the spectral results apply more generally to generic, rotation-invariant pattern-forming systems close to a Turing instability. Finally, \Cref{app:front-speed-numerics} contains numerical results on the selected front speed.

\section{Spectral analysis}\label{sec:spectral-analysis}

We now analyse the spectrum of the linear part of \eqref{eq:spat-dyn-Fourier}. Although the calculations are done for the linear operator given by the Swift–Hohenberg-type equation \eqref{eq:Swift-Hohenberg}, the results are fully generic for rotationally symmetric pattern-forming systems close to a Turing instability, albeit with more involved notation. We outline this in detail in \Cref{app:general-spectral-analysis}. Using the specific form of $\hat{\Lcal}(\gammab)$ and that $i \d \cdot \gammab \hat{W}$ only introduces a purely imaginary shift, we can relate the spectrum of $\tilde{\Lcal}^\eps$ to the spectrum of $L^{\eps}$, cf.~\cite{eckmann1991-02CommunMathPhys}.

\begin{lemma}\label{lem:eigenvalue-correspondence}
    Let $\gammab \in \Gamma$. Then, $\lambda \in \C$ is an eigenvalue of $\tilde{\Lcal}^{\eps}(\gammab)$ if and only if $L^\eps(\d \lambda + i \gammab) + \eps c_0 \lambda = 0$. The corresponding eigenvector is then given by $(1, \lambda + i \d \cdot \gammab, (\lambda + i \d \cdot \gammab)^2,(\lambda + i \d \cdot \gammab)^3)^T$.
\end{lemma}
\begin{proof}
    We first note that $\lambda \in \C$ is an eigenvalue of $\tilde{\Lcal}^\eps(\gammab)$ if and only if $\lambda + i\d \cdot \gammab$ is an eigenvalue of $\hat{\Lcal}^\eps(\gammab)$. By standard construction, see e.g.~\cite{eckmann1991-02CommunMathPhys}, the corresponding eigenvector for the latter eigenvalue problem is $(1, \lambda + i \d \cdot \gammab, (\lambda + i \d \cdot \gammab)^2,(\lambda + i \d \cdot \gammab)^3)^T$. Then, since the problems only differ by a multiple of the identity, the eigenvectors are identical.
    
    Since \eqref{eq:spat-dyn-Fourier} and \eqref{eq:spat-dyn-step1} are equivalent by construction, we also find that the linearised equations
    \begin{equation}\label{eq:spat-dyn-Fourier-linearised}
        \partial_\xi \hat{W}(\cdot, \gammab) = \hat{\Lcal}^{\eps}(\gammab) \hat{W}(\cdot,\gammab) - i \d \cdot \gammab \hat{W}(\cdot,\gammab)
    \end{equation}
    and
    \begin{equation}\label{eq:spat-dyn-step1-linearised}
        (\partial_\xi + i \d \cdot \gammab)^4 \hat{U}(\cdot, \gammab) = \eps c_0 \partial_\xi \hat{U}(\cdot, \gammab) + L^{\eps}(\d \partial_\xi + i \gammab) \hat{U}(\cdot, \gammab) + (\partial_\xi + i \d \cdot \gammab) \hat{U}(\cdot, \gammab)
    \end{equation}
    are equivalent. Therefore, $\lambda \in \C$ is an eigenvalue of $\tilde{\Lcal}^\eps(\gammab) = \hat{\Lcal}^{\eps}(\gammab) - i \d\cdot\gammab I$ if and only if $\hat{W} = e^{\lambda \xi} (1, \lambda + i \d \cdot \gammab, (\lambda + i \d \cdot \gammab)^2,(\lambda + i \d \cdot \gammab)^3)^T$ is a solution to \eqref{eq:spat-dyn-Fourier-linearised} if and only if $\hat{U} = e^{\lambda \xi}$ is a solution to \eqref{eq:spat-dyn-step1-linearised} if and only if
    \begin{equation*}
        0 = \eps c_0 \partial_\xi \hat{U}(\cdot, \gammab) + L^{\eps}(\d \partial_\xi + i \gammab) \hat{U}(\cdot, \gammab) = \eps c_0 \lambda + L^{\eps}(\d \partial_\xi + i \gammab).
    \end{equation*}
    This completes the proof.
\end{proof}

\subsection{Spectral analysis for $\eps = 0$}\label{sec:spectral-analysis-eps-zero}

We start the spectral analysis with the case $\eps = 0$ and then analyse the behaviour for $\eps > 0$. Since we aim to apply centre manifold theory, we first characterise the purely imaginary eigenvalues of $\tilde{\Lcal}^0$. Applying \Cref{lem:eigenvalue-correspondence}, we find that $\lambda = i \lambda_i$ with $\lambda_i \in \R$ is in the spectrum of $\tilde{\Lcal}^{0}(\gammab)$ for some $\gammab \in \Gamma$ if and only if $L^0(i(\d\lambda_i + \gammab)) = \hat{L}^0(\d \lambda_i + \gammab) = 0$. Here $\hat{L}^0(\k) = -(1-|\k|^2)^2$, $\k\in\R^2$, is the Fourier symbol of $L^0$. In particular, $i \lambda_i$ is in the spectrum of $\tilde{\Lcal}^0(\gamma)$ if and only if $|\d \lambda_i + \gammab| = 1$. Therefore, the purely imaginary eigenvalues of $\tilde{\Lcal}^0(\gammab)$ can be determined geometrically, see \Cref{fig:spec-real-part-geometric-interp,fig:spectral-eps=0}.

This geometric interpretation in particular yields that there are infinitely many purely imaginary eigenvalues of $\tilde{\Lcal}^0$, which is given as the direct sum of the $\tilde{\Lcal}^0(\gammab)$ over the Fourier modes $\gammab\in \Gamma$. Due to the discreteness of the Fourier lattice $\Gamma$, we also obtain that these imaginary eigenvalues are discrete, that is, there is a uniform neighbourhood around every imaginary eigenvalue not containing any other eigenvalue on the imaginary axis.

Finally, we discuss the multiplicity of the imaginary eigenvalues. For this, we use the following result, cf.~\cite[Lem.~5.4]{hilder2025-08JNonlinearSci}.

\begin{lemma}\label{lem:multiplicity}
    Let $\gammab \in \Gamma$ and let $i\lambda_i \in i \R$ be an eigenvalue of $\tilde{\Lcal}^{\eps}(\gammab)$. Then, the geometric multiplicity of $i\lambda_i$ is one. Additionally, its algebraic multiplicity is $m_0$ if $\hat{L}^0(\d (\lambda_i + \tilde{\lambda}) + \gammab) = \kappa \tilde{\lambda}^{m_0} + \Ocal(|\tilde{\lambda}|^{m_0+1})$ with $\kappa \neq 0$.
\end{lemma}

For eigenvalues on the imaginary axis, \Cref{lem:eigenvalue-correspondence} allows for a convenient geometric interpretation of the eigenvalue problem for $\tilde{\Lcal}^0$. That is, $\tilde{\Lcal}^0(\gammab)$ has an imaginary eigenvalue $i\lambda_i$ for some $\gammab \in \Gamma$ if the line $\ell_{\gammab} := \{\gammab + \d \lambda \,:\, \lambda \in \R\}$ intersects the unit circle $S^1$. Since $|\d| = 1$, $\lambda_i$ is the distance between $\gammab$ and the intersection point. Note that, for fixed $\gammab\in \Gamma$, three different scenarios can occur: the line $\ell_{\gammab}$ can have zero, one or two intersection points with the unit circle. In the first case, $\tilde{\Lcal}^0(\gammab)$ has no purely imaginary eigenvalue. In the second case, $\ell_\gamma$ intersects the unit circle tangentially and thus, $\tilde{\Lcal}^0(\gammab)$ has one purely imaginary eigenvalue. Finally, in the third case, $\ell_\gamma$ has two intersections with the unit circle and $\tilde{\Lcal}^0(\gammab)$ has two distinct purely imaginary eigenvalues.

Together with \Cref{lem:multiplicity}, this geometric point of view also allows us to determine the algebraic multiplicity of the purely imaginary eigenvalues. For this, we calculate using $\tilde{\gammab} = \gammab + \d \lambda \in S^1$
\begin{equation*}
    \hat{L}^0(\tilde{\gammab} + \d \tilde{\lambda}) = -(1-|\tilde{\gammab}+\d \tilde{\lambda}|^2)^2 =  -(\tilde{\lambda}^2 + 2 \tilde{\lambda} \d \cdot \tilde{\gammab})^2
\end{equation*}
% When doing the same computation for thin films, note that the radial derivative of the line is 0, when the line is tangential, but the second radial derivative is non-zero. Hence, the result follows from the expansion of the eigenvalue combined with the chain rule.
Therefore, if $\ell_{\gammab}$ has two intersection points with the unit circle, then $\tilde{\gammab} \cdot \d \neq 0$ and $\hat{L}^0(\tilde{\gammab} + \d \tilde{\lambda})$ is quadratic to lowest order in $\tilde{\lambda}$. Hence, in this case, $i\lambda_i$ is an eigenvalue with algebraic multiplicity two. In particular, the purely imaginary spectrum of $\tilde{\Lcal}^0(\gammab)$ consists of two Jordan blocks of size two. If the intersection of $\ell_{\gammab}$ with the unit circle is tangential, then $\tilde{\gammab}$ is orthogonal to $\d$. Thus, $i\lambda_i$ is an eigenvalue with algebraic multiplicity four of $\tilde{\Lcal}^0(\gammab)$ and belongs to a single Jordan block of size four. By applying these observations to the Fourier lattice $\Gamma$, we obtain the following result.

\begin{proposition}\label{prop:imaginary-spectrum-discrete}
    The imaginary spectrum of $\tilde{\Lcal}^\eps$ is discrete, and all imaginary eigenvalues belong to a Jordan block of size at least two. The origin $\lambda = 0$ is an eigenvalue of $\tilde{\Lcal}^0(\gammab)$ for $\gammab\in \Gamma$ if and only if $\gammab \in \Gamma_0$. In particular, $\lambda = 0$ is an eigenvalue of $\tilde{\Lcal}^\eps$ with finite multiplicity. Furthermore, $\tilde{\Lcal}^{0}(\gammab)$ has a Jordan block of size four at $\lambda = 0$ if and only if $\gammab = \pm \k_2$ and $\theta = \tfrac{\pi}{6}$. Otherwise, $\tilde{\Lcal}^{0}(\pm\k_j)$ has two Jordan blocks of size two at $\lambda = 0$ and $0 \neq \lambda \in i\R$. Up to rotational symmetry, there is no other angle $\theta$ for which a Jordan block of size four occurs for $\lambda = 0$.
\end{proposition}

Since the spectrum on the imaginary axis is now fully characterised, it remains to understand the spectrum outside the imaginary axis. Specifically, we are interested in the presence of a spectral gap. It turns out that this is highly dependent on the angle $\theta$ as the following result states.

\begin{proposition}\label{prop:real-spectral-gap}
    The operator $\tilde{\Lcal}^\eps$ has a spectral gap around the imaginary axis if and only if $\cot(\theta) \in \sqrt{3}\Q$. Furthermore, if $\cot(\theta) \notin \sqrt{3}\Q$, any sequence of eigenvalues $\lambda_n \in \C \setminus i\R$ with $\Re(\lambda_n) \rightarrow 0$ satisfies $|\Im(\lambda_n)| \rightarrow \infty$.
\end{proposition}

\begin{proof}
    We recall from \Cref{lem:eigenvalue-correspondence} that for any eigenvalue $\lambda \in \C$ of $\tilde{\Lcal}^\eps$, there exists a $\gammab \in \Gamma$ such that
    \begin{equation*}
        L^0(\d \lambda + i \gammab) = - (1 + \lambda^2 - |\gammab|^2 + 2i \lambda \d \cdot \gammab)^2 = 0.
    \end{equation*}
    Solutions to this can be explicitly given by
    \begin{equation}\label{eq:formula-lambda-pm}
        \lambda_\pm = - i \d \cdot \gammab \pm \sqrt{-|\d\cdot \gammab|^2 + |\gammab|^2 - 1}.
    \end{equation}
    Therefore, $\lambda_\pm$ has a non-trivial real part if and only if $-|\d\cdot \gammab|^2 + |\gammab|^2 - 1 > 0$. In this case, 
    \begin{equation*}
        \Re(\lambda_\pm) = \pm \sqrt{-|\d\cdot \gammab|^2 + |\gammab|^2 - 1} = \pm \sqrt{|\d^{\perp}\cdot \gammab|^2 - 1},
    \end{equation*}
    where $\d \perp \d^{\perp}\in S^1$ and we used $|\gammab|^2 = |\d\cdot \gammab|^2 + |\d^{\perp}\cdot\gammab|^2$. This allows for the following geometric interpretation: $|\d^\perp \cdot \gammab|$ gives the distance of $\gammab\in \Gamma$ to the line $\{\d s \,:\, s \in \R\}$. Therefore, if $|\d^\perp \cdot \gammab| > 1$ the shifted line $\{\d s + \gammab \,:\, s \in \R\}$ has no intersections with the unit circle, see \Cref{fig:spec-real-part-geometric-interp}. Hence, it suffices to understand the set $\{|\d^{\perp}\cdot \gammab| \, : \, \gammab\in \Gamma\}$. Since $\d^{\perp} = (-\sin(\theta),\cos(\theta))$, for  $\gammab = n_1\k_1 + n_2\k_2\in \Gamma$ we obtain
    \begin{equation*}
        \d^{\perp} \cdot \gammab = - n_1 \sin(\theta) + n_2(-\sin(\theta) \cos(\tfrac{2\pi}{3}) + \cos(\theta)\sin(\tfrac{2\pi}{3})) = -n_1 \sin(\theta) - n_2\sin(\theta - \tfrac{2\pi}{3})
    \end{equation*}
    using the addition theorems. We now make the following claim: the set $\{n_1 \sin(\theta) + n_2\sin(\theta - \tfrac{2\pi}{3}) \, : \, k_1,k_2\in \Z\}$ is discrete in $\R$ if and only if $\theta\in \sqrt{3}\Q$. Otherwise, this set is dense in $\R$. Observe that this implies the first part of the proposition.

    To prove the claim, we note that we can equivalently prove the claim for the set $\{n_1  + n_2\tfrac{\sin(\theta - 2\pi/3)}{\sin(\theta)} \, : \, k_1,k_2\in \Z\} = \Z + a \Z$ with $a=\tfrac{\sin(\theta - 2\pi/3)}{\sin(\theta)}$, since in the case $\theta = 0$ there is nothing to prove. Recall that $\Z + a\Z$ is discrete in $\R$ if and only if $a\in \Q$, and otherwise it is dense. Finally, we observe that $\tfrac{\sin(\theta - 2\pi/3)}{\sin(\theta)} = -\frac{1}{2} - \sqrt{3}\cot(\theta) \in \Q$ if and only if $\cot(\theta) \in \sqrt{3}\Q$.
\end{proof}

\begin{figure}[H]
    \begin{subfigure}[b]{0.45\textwidth}
        \includegraphics[width = \linewidth]{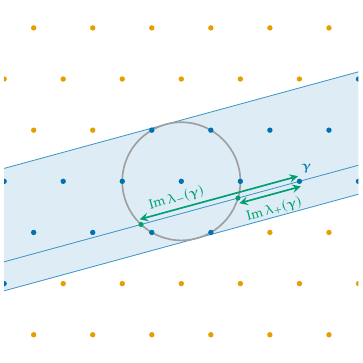}
        \subcaption{}
        \label{subfig:geo-characterisation-imaginary-eigvals}
    \end{subfigure}
    \hfill
    \begin{subfigure}[b]{0.45\textwidth}
        \includegraphics[width = \linewidth]{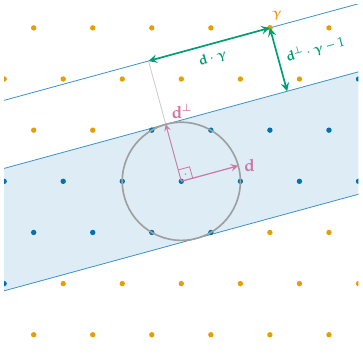}
        \subcaption{}
        \label{subfig:geo-characterisation-hyperbolic-eigvals}
    \end{subfigure}
    \caption{Geometric characterisation of the spectrum of $\tilde{\Lcal}^0$: lattice points in the blue shaded strip generate eigenvalues on the imaginary axis \sref{subfig:geo-characterisation-imaginary-eigvals}, while lattice points outside generate hyperbolic eigenvalues \sref{subfig:geo-characterisation-hyperbolic-eigvals}.}
    \label{fig:spec-real-part-geometric-interp}
\end{figure}

\begin{figure}[H]
    \centering
    \begin{subfigure}[b]{0.45\textwidth}
        \includegraphics[width = \linewidth]{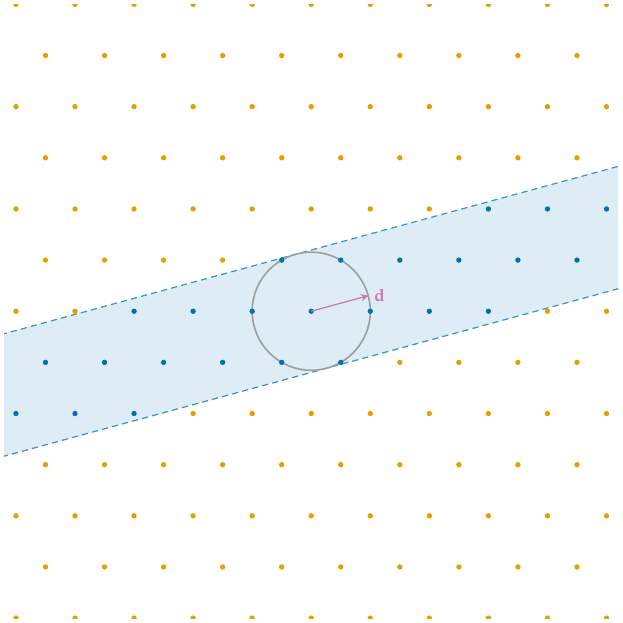}
        \subcaption{}
        \label{subfig:spec-gap-1}
    \end{subfigure}
    \hfill
    \begin{subfigure}[b]{0.45\textwidth}
        \includegraphics[width = \linewidth]{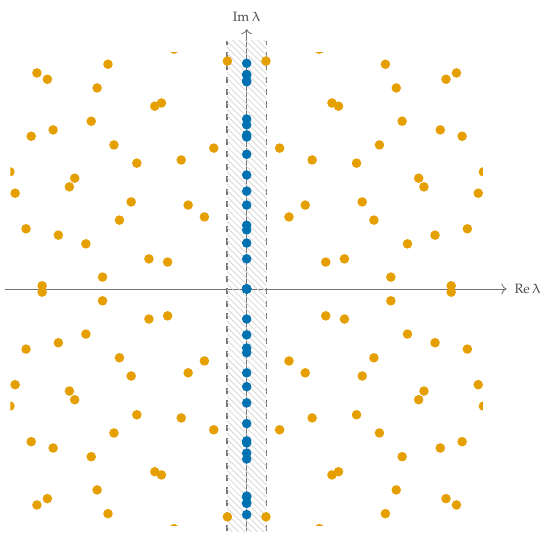}
        \subcaption{}
        \label{subfig:spec-gap-2}
    \end{subfigure}    
    \caption{Spectrum of $\tilde{\Lcal}^0$ for $\theta$ satisfying $\cot(\theta) \in \sqrt{3}\Q$. Here, all hyperbolic lattice points have a positive distance to the blue shaded strip \sref{subfig:spec-gap-1}, which yields a spectral gap around the imaginary axis \sref{subfig:spec-gap-2}.}
    \label{fig:spectral-eps=0}
\end{figure}

\subsection{Spectral analysis for $\eps > 0$}\label{sec:spectral-analysis-eps-positive}

We now want to understand the spectrum of the operator $\tilde{\Lcal}^\eps$ for $\eps >0$. In particular, we want to show that there is a strip with a size of order $\sqrt{\eps}$ around the imaginary axis that contains only finitely many eigenvalues, which emerge precisely from the eigenvalue curves for $\gammab \in \Gamma_0$ and have a real part of order $\eps$. It turns out that the analysis presented here strongly depends on the presence of a spectral gap around the imaginary axis at $\eps = 0$. Therefore, we restrict to angles $\theta$ satisfying $\cot(\theta) \in \sqrt{3}\Q$, see \Cref{prop:real-spectral-gap}.

The analysis is split into three cases: first, we handle the eigenvalues of $\hat{\Lcal}^{\eps}(\pm \k_j) \mp i \d\cdot\k_j I$ for $j \in 1,2,3$, second, we consider the eigenvalues which are on the imaginary axis at $\eps = 0$, and finally, we deal with the eigenvalues with a non-trivial real part at $\eps = 0$. While the first case can be handled with an implicit function theorem argument similar to \cite{hilder2020-08JournalofDifferentialEquations}, the latter two cases require a more subtle argument to guarantee that there is a uniform gap of order $\sqrt{\eps}$.

For this, we recall that eigenvalues of $\tilde{\Lcal}^{\eps}(\gammab)$ with $\gammab \in \Gamma$ are given by the roots of the fourth-order polynomial
\begin{equation}\label{eq:polynomial-spectrum}
    p(\lambda;\gammab,\eps) := L^{\eps}(\d \lambda + i\gammab) + \eps c_0 \lambda = L^0(\d\lambda + i\gammab) + \eps^2 \mu_0 + \eps c_0 \lambda,
\end{equation}
with $L^0(\vb) = -(1+\vb \cdot \vb)^2$. We note that $p(\lambda;\gammab,0)$ has two double roots or one fourth root. The latter always lies on the imaginary axis. Note that for $\gammab\notin\Gamma_0$ all roots are of the form $\lambda_\pm = - i \d\cdot \gamma \pm \Delta(\d^\perp \cdot \gammab)$, see \eqref{eq:formula-lambda-pm}, where $\Delta$ is real if $|\d^\perp \cdot \gammab| > 1$ and purely imaginary otherwise. Under the assumption that $\cot(\theta)\in \sqrt{3}\Q$, it follows that there is $\delta>0$ such that $|\Delta(\d^\perp \cdot \gammab)| > \delta$ for all $\gammab\in \Gamma\setminus \Gamma_0$ with $\gammab\cdot \d^{\perp} \neq 1$.

The analysis of the spectrum of $\tilde{\Lcal}^\eps$ relies on a careful expansion using the implicit function theorem, which is only valid locally. This is combined with two results about the location of the roots of the polynomial $p$ that guarantee that eventually sufficient uniform spectral gaps hold true. The following lemma guarantees that the eigenvalue curves starting at purely imaginary eigenvalues will eventually be sufficiently far away from the imaginary axis. This includes the fourth-order Jordan blocks.

\begin{lemma}\label{lem:polynom-roots-1}
    For every $c_1>0$ sufficiently small there are $\eps_0 >0$ and $r_1>0$ such that for every $\eps \in (0,\eps_0)$ and every $\gammab \in \Gamma$ with $|\d^\perp \cdot \gammab| \leq  1$ with $\eps|\d \cdot \gammab|> c_1$ it holds that $p(\lambda;\gammab,\eps) = 0$ implies that $|\Re(\lambda)| > r_1$.
\end{lemma}
\begin{proof}
    Using that $L^0$ has roots at $\lambda_\pm = -i \d\cdot \gammab \pm \Delta$ for some $\Delta = i \Delta_i \in i\R$ with $|\Delta_i| \leq 1$, we can write
    \begin{equation*}
        p(\lambda;\gammab,\eps) = -(\lambda + i(\d \cdot \gammab) + i\Delta_i)^2 (\lambda + i(\d \cdot \gammab) - i\Delta_i)^2 + \eps^2 \mu_0 + \eps c_0 \lambda.
    \end{equation*}
    We now replace $\lambda = \tilde{\lambda} - i(\d \cdot \gammab)$ to obtain
    \begin{equation*}
        g(\tilde{\lambda}) := -(\tilde{\lambda} + i\Delta_i)^2(\tilde{\lambda} - i\Delta_i)^2 + \eps^2 \mu_0 + \eps c_0 \tilde{\lambda} - i \eps c_0 (\d \cdot \gammab).
    \end{equation*}

    Now, fix any $c_1 > 0$. Then, we show the claim in two steps. First, we show that for any $C_i > 0$ sufficiently large, there exist $\eps_0 > 0$ and $\tilde{r}_1$ such that for any $\eps \in (0,\eps_0)$ it holds that $g(\tilde{\lambda}) \neq 0$ if $\tilde{\lambda} \in \{|\Re(\tilde{\lambda})| < \tilde{r}_1\} \cap \{|\Im(\tilde{\lambda})| > C_i\}$. Second, we show that for any $C_i > 0$ sufficiently large, there exist $\eps_0>0$ and $\tilde{r}_2 > 0$ such that for any $\eps \in (0,\eps_0)$ it holds that $g(\tilde{\lambda}) \neq 0$ if $\tilde{\lambda} \in \{|\Re(\tilde{\lambda})| < \tilde{r}_2\} \cap \{|\Im(\tilde{\lambda})| < C_i\}$ provided $\eps|\d \cdot \gammab|> c_1$. This proves the claim by choosing $r_1 = \min(\tilde{r}_1, \tilde{r}_2)$.

    For the first step, we fix $C_i > 0$ sufficiently large. Then, we note that $\Re(g(\tilde{\lambda}))$ is independent of $\d\cdot \gammab$ and $|\Re(-(\tilde{\lambda} + i\Delta_i)^2(\tilde{\lambda} - i\Delta_i)^2)| \sim \Im(\tilde{\lambda})^4$ for $|\Im(\tilde{\lambda})| \rightarrow \infty$. In particular, since $|\Delta_i| \leq 1$ for all $\gammab\in \Gamma$ with $|\d^\perp \cdot \gammab|\leq 1$, the coefficients of $\Re(g)$ are also uniformly bounded for all such $\gammab\in \Gamma$. Therefore, by choosing $\tilde{r}_1$ and $\eps_0$ sufficiently small, we can guarantee that $\Re(g(\tilde{\lambda})) \neq 0$ and thus $g(\tilde{\lambda}) \neq 0$ if $|\Im(\tilde{\lambda})| > C_i$ and $|\Re(\tilde{\lambda})| < \tilde{r}_1$.

    For the second step, we notice that for $\tilde{\lambda} = i \tilde{\lambda}_i$ it holds
    \begin{equation*}
        g(i\tilde{\lambda}_i) = -(\tilde{\lambda}_i - \Delta_i)^2(\tilde{\lambda}_i + \Delta_i)^2 + \eps c_0 i\tilde{\lambda}_i - i c_0 \eps \d\cdot \gammab + \eps^2\mu_0.
    \end{equation*}
    Hence, we obtain $|g(i\tilde{\lambda}_i)| > |\Im(g(i\tilde{\lambda}_i)| > c_0 c_1 - \eps_0 C_i > \tfrac{c_0 c_1}{2}$ for all $|\tilde{\lambda}_i| < C_i$ by choosing $\eps_0$ sufficiently small and provided $\eps |\d\cdot \gammab| > c_1$. Now, a compactness argument guarantees that there is $\tilde{r}_2 > 0$ such that $g(\tilde{\lambda}) \neq 0$ for all $\tilde{\lambda}$ with $|\Re(\tilde{\lambda})| < \tilde{r}_2$ and $|\Im(\tilde{\lambda})| < C_i$. This concludes the proof.
\end{proof}

In order to guarantee that the real spectral gap persists, we need another lemma in a similar spirit:

\begin{lemma}\label{lem:polynom-roots-2}
    Assume $\cot(\theta) \in \sqrt{3}\Q$. There exist $\eps_0  >0 $ and $r_2 > 0$ such that for all $\gammab \in \Gamma$ with $|\d^\perp \cdot \gammab| > 1$ it holds that $p(\lambda;\gammab,\eps) = 0$ implies that $|\Re(\lambda)| > r_2$ for all $\eps \in (0,\eps_0)$. 
\end{lemma}
\begin{proof}
    Since $|\d^\perp \cdot \gammab| > 1$, we find that $p(\lambda;\gammab,0)$ has two double roots at $\lambda_0 = -i \d \cdot\gammab \pm \Delta$ for $\Delta\in \R$. Furthermore, we note that due to $\cot(\theta) \in \sqrt{3}\Q$, there exists a $\delta > 0$ such that $|\Delta| > \delta$. Similar to the proof of \Cref{lem:polynom-roots-1}, we write $\lambda = \tilde{\lambda} - i (\d \cdot \gamma)$ and define
    \begin{equation*}
        f(\tilde{\lambda}) := - (\tilde{\lambda} + \Delta)^2 (\tilde{\lambda} - \Delta)^2 + \eps^2 \mu_0 + \eps c_0 \tilde{\lambda} - i \eps c_0 (\d \cdot \gammab).
    \end{equation*}
    We now show that there exist $\eps_0, r_2$ such that $\Re(f(\tilde{\lambda})) \neq 0$ for $\eps \in (0,\eps_0)$ and $|\Re(\tilde{\lambda})| < r_2$. For this, we note that
    \begin{equation}
        \Re(f(\tilde{\lambda})) = \Re(- (\tilde{\lambda} + \Delta)^2 (\tilde{\lambda} - \Delta)^2) + \eps^2 \mu_0 + \eps c_0 \Re(\tilde{\lambda}).
    \end{equation}
    By choosing $r_2$ and $\eps_0$ sufficiently small, we can thus treat $\eps^2 \mu_0$ and $\eps c_0 \Re(\tilde{\lambda})$ as small perturbations. Therefore, it is sufficient to guarantee that $\Re((\tilde{\lambda} + \Delta)^2 (\tilde{\lambda} - \Delta)^2) \neq 0$ for all $\tilde{\lambda}$ with $|\Re(\tilde{\lambda})| < r_2$. As in the proof of \Cref{lem:polynom-roots-1}, this is true for $\Im(\tilde{\lambda})$ sufficiently large. Hence, it remains to prove that the real part is non-zero in a compact strip around the imaginary axis, which follows from $|\Re(f(i\tilde{\lambda}_i))| = (\tilde{\lambda}_i^2 + \Delta^2)^2 > \delta^4$. Using a compactness argument concludes the proof.
\end{proof}

We are now in a position to prove the main result on the behaviour of the spectrum of $\tilde{\Lcal}^\eps$ for $\eps > 0 $ sufficiently small.

\begin{proposition}\label{prop:spectrum-eps-positive}
    Let $\cot(\theta) \in \sqrt{3}\Q$. There is $\eps_0 > 0$ and a constant $C_0>0$ such that for $\eps \in (0,\eps_0)$ the spectrum of $\tilde{\Lcal}^\eps$ looks like this:
    \begin{enumerate}[label=(\alph*), ref=\theproposition(\alph*)]
        \item if $\theta \in [0,\tfrac{\pi}{6})$, there are twelve eigenvalues $\lambda_{\pm}^{\eps}(\k_j) = \eps \nu_{\pm}(\k_j)  + \Ocal(\eps^2)$, $j = \pm 1, \pm 2, \pm 3$ of $\tilde{\Lcal}^{\eps}(\k_j)$ with
        \begin{equation}\label{eq:spectrum-eps-pos-perturbation-block-size-two}
            \nu_{\pm}(\k_j) = \begin{dcases}
                \frac{-c_0 \pm \sqrt{c_0^2 - 16 (\d \cdot \k_j)^2\mu_0}}{8(\d \cdot \k_j)^2} & \text{if } c_0 \neq c_{\crit}(\k_j), \\[6pt]
                -\frac{c_0}{8(\d \cdot \k_j)^2} \pm \sqrt{\eps} \frac{i^{3/2}\mu_0^{3/4}}{2\sqrt{2}(\d\cdot\k_j)^2} & \text{if } c_0 = c_{\crit}(\k_j),
            \end{dcases}
        \end{equation}
        where $c_{\crit}(\k_j)$ is given by \eqref{eq:c-crit};
        \item\label[proposition]{it:spectrum-pi-over-six} if $\theta = \tfrac{\pi}{6}$, there are ten eigenvalues $\lambda_{\pm}^{\eps}(\k_j) = \eps \nu_{\pm}(\k_j)  + \Ocal(\eps^2)$ for $j = \pm 1, \pm 3$ and $\lambda_\perp^{\eps}(\pm\k_2) = \eps \nu_\perp(\pm \k_2) + \Ocal(\eps^2)$ of $\tilde{\Lcal}^{\eps}(\k_j)$ with
        \begin{equation}\label{eq:spectrum-eps-pos-perturbation-block-size-four}
            \nu_\perp(\pm \k_2) = - \dfrac{\mu_0}{c_0};
        \end{equation}
        \item all other eigenvalues $\lambda_h(\eps)$ of the operator $\tilde{\Lcal}^\eps$ satisfy
        \begin{equation*}
            |\Re(\lambda_h(\eps))| > C_0\sqrt{\eps}.
        \end{equation*}
    \end{enumerate}
\end{proposition}

\begin{remark}
    Note that we need to distinguish the case that $c_0 = c_\mathrm{crit}(\k_j)$ is the conjectured selected front speed, cf.~\Cref{sec:front-speed}, in the expansion \eqref{eq:spectrum-eps-pos-perturbation-block-size-two}. In this case, the Jordan block splits only at higher order, which creates additional difficulties in the centre manifold analysis and the derivation of the corresponding reduced equations. We refer to \Cref{sec:reduced-equations-critical} for additional details.
\end{remark}

\begin{proof}
    We first recall from \Cref{lem:eigenvalue-correspondence} that $\lambda \in \C$ is an eigenvalue of $\tilde{\Lcal}^\eps$ if there exists a $\gammab \in \Gamma$ such that $L^{\eps}(\d\lambda + i \gammab) + \eps c_0 \lambda = 0$. To prove the first statement, we assume that $\theta \neq \tfrac{\pi}{6}$ and thus $\d \cdot \k_j \neq 0$ for all $j = \pm 1, \pm 2, \pm 3$. Then, we look for solutions to
    \begin{equation*}
        L^{\eps}(\d\lambda + i\k_j) + \eps c_0 \lambda = 0
    \end{equation*}
    for $\lambda$ close to zero. Using that $|\k_j| = 1$ we find that 
    \begin{equation*}
        L^{\eps}(\d\lambda + i\k_j) = - (\lambda^2 + 2 i\lambda \d \cdot \k_j)^2 + \eps^2 \mu_0.
    \end{equation*}
    We note that for $\eps = 0$, eigenvalues are roots of the fourth-order polynomial $L^0(\d \lambda + i \k_j)$, which has a double root at $\lambda = 0$ and a double root at $-2i\d \cdot \k_j$. Hence, for $\eps > 0$ there are exactly two solutions close to zero. To find these, we may consider the lowest-order powers of $\lambda$ and make the ansatz $\lambda^{\eps}(\k_j) = \eps \nu$, which yields
    \begin{equation}\label{eq:equation-for-correction-more-central-eigenvalues}
        4 (\d \cdot \k_j)^2 \nu^2 + c_0 \nu + \mu_0 + \Ocal(\eps |\nu|^3) = 0.
    \end{equation}
    By solving the leading-order equation, we thus find that if $c_0 \neq c_{\crit}(\k_j)$, the solutions are given by
    \begin{equation*}
        \nu_\pm(\k_j) = \frac{-c_0 \pm \sqrt{c_0^2 - 16 (\d \cdot \k_j)^2\mu_0}}{8(\d\cdot \k_j)^2}.
    \end{equation*}
    Applying the implicit function theorem as in \cite{hilder2020-08JournalofDifferentialEquations}, these solutions persist, and we find two eigenvalues of the form $\lambda_\pm^{\eps}(\k_j) = \eps \nu_\pm(\k_j) + \Ocal(\eps^2)$.

    If $c_0 = c_{\crit}(\k_j)$, we make the extended ansatz $\lambda^\eps(\k_j) = \eps \nu_1 + \eps^{3/2} \nu_2$. The leading-order equation is then solved by $\nu_1 = \nu_1(\k_j) = -\tfrac{c_0}{8(\d\cdot\k_j)}$. Additionally, noting that the $\eps^{5/2}$-terms vanish due to $c_0 = c_{\crit}(\k_j)$ and equating terms of order $\eps^3$ to zero, we find that
    \begin{equation*}
        \nu_{2,\pm}(\k_j) = \pm \dfrac{i^{3/2}\mu_0^{3/4}}{2\sqrt{2} (\d \cdot \k_j)^2}.
    \end{equation*}
    Again, the implicit function theorem yields that these solutions persist, and we obtain two eigenvalues of the form $\lambda_\pm^\eps(\k_j) = \eps \nu_1(\k_j) + \eps^{3/2} \nu_{2,\pm}(\k_j) =: \eps \nu_\pm(\k_j)$.

    To prepare the proof of the third statement, we now also show that the eigenvalues of $\hat{\Lcal}^{\eps}(\k_j) - i\d\cdot \k_j I$ close to $-2i\d\cdot \k_j$ have real part of order $\sqrt{\eps}$. Again, there are exactly two solutions close to $-2i\d\cdot \k_j$ and we therefore make the ansatz $\lambda_h(\k_j) = -2i\d\cdot \k_j + \sqrt{\eps} \tilde{\nu}(\k_j)$. Inserting this yields
    \begin{equation*}
        0 = - (2i \d\cdot \k_j + \sqrt{\eps} \tilde{\nu})^2 \eps \tilde{\nu}^2 + \eps c_0(-2i\d\cdot \k_j + \sqrt{\eps} \tilde{\nu})+ \eps^2 \mu_0 = \eps (4 (d\cdot \k_j)^2 \tilde{\nu}^2 - 2i\d\cdot\k_j c_0) + \Ocal(\eps^{\tfrac{3}{2}}).
    \end{equation*}
    Therefore, we find to leading order that $\tilde{\nu}_\pm(\k_j) = \pm \sqrt{\tfrac{i c_0}{2 \d\cdot \k_j}}$. Again, an application of the implicit function theorem shows that these persist under higher-order perturbations, which shows the desired behaviour of the eigenvalue curves since $|\Re(\tilde{\nu}_\pm(\k_j))| > 0$.

    We now assume that $\theta = \tfrac{\pi}{6}$ and thus $\d \cdot (\pm \k_2) = 0$. The only difference from the case $\theta \neq \tfrac{\pi}{6}$ is the eigenvalues of $\hat{\Lcal}^{\eps}(\pm \k_2) \mp i\d\cdot \k_2 I$, which satisfy
    \begin{equation*}
        -\lambda^4 + \eps c_0 \lambda + \eps^2 \mu_0 = 0.
    \end{equation*}
    Therefore, as already shown in \Cref{lem:multiplicity}, $\lambda = 0$ is a fourth-order root at $\eps = 0$. We now show that this splits into one eigenvalue $\lambda_\perp^{\eps}(\pm \k_2)$ with real part of order $\eps$ and three eigenvalues with real part of order $\eps^{1/3}$. Indeed, making the ansatz $\lambda = \eps \nu_\perp(\pm\k_2)$ we find, to leading order,
    \begin{equation*}
        \nu_\perp(\pm\k_2) = - \dfrac{\mu_0}{c_0},
    \end{equation*}
    while we obtain a third-order polynomial, to leading order, when making the ansatz $\lambda = \eps^{1/3}\tilde{\nu}$, which therefore yields three solutions $\tilde{\nu}_{1,2,3}$. Again, following a similar strategy as in \cite{hilder2020-08JournalofDifferentialEquations}, the use of the implicit function theorem shows the persistence of all solution branches and therefore the claim.

    It thus remains to show the third statement. We first consider the case $|\d^{\perp}\cdot \gammab| \leq 1$. We first restrict to the case $|\d^\perp \cdot \gammab| < 1$ which yields that $\Delta_i > 0$. We start by showing that for $\eps |\d \cdot \gammab| < c_1$ for $c_1$ sufficiently small the expansions for the four eigenvalues of $\hat{\Lcal}^{\eps}(\gammab) - i\d\cdot \gammab I$
    \begin{equation}\label{eq:expansion-imaginary-less-central}
        \lambda_{h,+}(\gammab) = - i \d \cdot \gammab + i\Delta_i \pm\sqrt{\eps (\d \cdot \gammab - \Delta_i)} \dfrac{\sqrt{ic_0}}{2\Delta_i} + \Ocal(\eps |\d \cdot \gammab + 1|)
    \end{equation}
    hold. Similarly for $\lambda_{h,-}$. For this, we replace $\lambda_h(\gammab) = -i \d \cdot \gammab + i\Delta_i + \nu_h = i\lambda_0 + \nu_h$ which yields
    \begin{equation*}
        0 = - \nu_h^2 (\nu_h + 2 i \Delta_i)^2 + \eps^2 \mu_0 + \eps c_0 (i \lambda_0 + \nu_h) = - \nu_h^2 (\nu_h + 2 i \Delta_i)^2 + \eps^2 \lambda_0^2 \dfrac{\mu_0}{\lambda_0^2} + \eps c_0 i\lambda_0 + \eps \lambda_0 c_0 \dfrac{\nu_h}{\lambda_0}.
    \end{equation*}
    Again, we replace $\nu_h = \sqrt{\eps \lambda_0} \tilde{\nu}_h$ and consider the lowest power in $\eps \lambda_0$. This gives
    \begin{equation*}
        0 = 4 \Delta_i^2 \tilde{\nu}_h^2 + i c_0 + \Ocal(\eps \lambda_0)
    \end{equation*}
    and therefore $\tilde{\nu}_h = \pm \tfrac{\sqrt{-i c_0}}{2 \Delta_i}$. Again, applying the implicit function theorem and using $\Ocal(\eps\lambda_0) = \Ocal(\eps|\d\cdot\gammab+1|)$, we obtain the desired expansion \eqref{eq:expansion-imaginary-less-central}. In particular, this implies that there exists a constant $C_0 > 0$, which is independent of $\gamma$ such that $|\Re(\lambda_{h,+}(\gammab))| > C_0 \sqrt{\eps}$ for $\eps |\d \cdot \gammab| < c_1$ with $c_1$ sufficiently small.

    Next, we consider the case that $|\d^\perp \cdot \gammab| = 1$ and $\gammab\neq \pm \k_2$, which corresponds to a Jordan block of size four. In this case, we find that $\nu_h = \lambda_h(\gammab) + i\d \cdot \gammab$ satisfies
    \begin{equation*}
        0 = -\nu_h^4 + \eps^2 \mu_0 + \eps c_0(\lambda_0 + \nu_h).
    \end{equation*}
    As expected from the Jordan block of size four, this suggests the rescaling $\nu_h = (\eps \lambda_0)^{1/4} \tilde{\nu}_h$. Proceeding with this as in the case $|\d^\perp \cdot \gammab| < 1$ we thus find to lowest order in $\eps \lambda_0$ that $\tilde{\nu}_h^4 = i c_0$. Again, the implicit function theorem implies persistence of these solutions, and we obtain that $|\Re(\lambda_{h,+}(\gammab))| > C_0 \eps^{1/4} > C_0 \sqrt{\eps}$ for $\eps | \d \cdot \gammab| < c_1$ and $\eps > 0$ sufficiently small. This completes the proof in the case $\eps | \d \cdot \gammab| < c_1$.

    If, on the other hand, $\eps |\d \cdot \gammab| > c_1$, we can use \Cref{lem:polynom-roots-1} to obtain that $|\Re(\lambda_{h,+}(\gammab))| > r_1$ for some $r_1 >0$ independent of $\gammab$.

    Finally, if $|\d^\perp \cdot \gammab| > 1$, then \Cref{lem:polynom-roots-2} yields that $|\Re(\lambda_h(\gammab))| \geq r_2$ for some $r_2 > 0$ independent of $\gammab$. This concludes the proof of the third statement.
\end{proof}

\section{Centre manifold theory and reduced equations}\label{sec:centre-manifold-theory}

We now establish a centre manifold theorem following the analysis in \cite{haragus-courcelle1999-01ZangewMathPhys} and pay particular attention to the size of the centre manifold. We first restrict to the case $\cot(\theta) \in \sqrt{3}\Q$ with $\theta \neq \tfrac{\pi}{6}$, where the more-central eigenvalues belong to Jordan blocks at zero of size 2. We then discuss the extension to the case $\theta = \tfrac{\pi}{6}$, when there are more-central eigenvalues belonging to a Jordan block of size 4 in \Cref{sec:special-case}. The derivation of the reduced equations in the case $c_0 = c_{\crit}(\k_j)$ is done in \Cref{sec:reduced-equations-critical}.

The analysis is split into different steps: first, we define the necessary projections onto the more-central, less-central, and hyperbolic parts and introduce rescaled variables by using different rescalings for the different parts of the solution. These rescaled variables make the $\eps$-dependence in the centre manifold analysis explicit and allow us to show that the centre manifold is indeed large enough to contain interesting solutions. Next, we perform a normal-form analysis to remove the terms with the lowest-order scaling in $\eps$. Finally, we show that we can close a fixed-point argument to obtain a centre manifold, which is of size $\Ocal(1)$ in the rescaled variables. Returning to the original variables then yields a centre manifold which is sufficiently large to contain the relevant solutions.

\subsection{Functional analytic set-up and projections}\label{sec:function-spaces}

We start by introducing the proper function spaces to perform the centre manifold analysis. For the following analysis, we need additional notation to distinguish regularity in $\d$-direction and $\d^\perp$-direction. The reason is that the real part of the eigenvalues of $\tilde{\Lcal}^\eps$ emerging from the imaginary axis only behaves like $(\eps|\d \cdot \gammab|)^{1/4}$ as $|\d \cdot \gammab| \rightarrow \infty$ and $|\d^\perp \cdot \gammab| < C$. Therefore, we cannot expect that the corresponding stable and unstable semigroups can compensate for the loss of more than one derivative in $\d$-direction. However, it turns out that this is not necessary as the spatial-dynamics formulation \eqref{eq:spat-dyn-Fourier} is designed so that only Fourier multipliers in $\d^\perp \cdot \gammab$ occur in the nonlinearity and it therefore suffices to gain regularity in $\d^{\perp}$-direction. This motivates the following definitions. Given a strip $S_{[a,b)} := \{\gammab \in \Gamma \,:\, \d \cdot \gammab \in [a,b)\}$, we define the perpendicular function space

\begin{equation*}
    H_{S_{[a,b)}}^\ell = \Big\{W = \sum_{\gammab \in S_{[a,b)}} \hat{W}(\gammab) e^{i \gammab\cdot \p}\,:\, \|W\|_{H_{S_{[a,b)}}^\ell}^2 := \sum_{\gammab \in S_{[a,b)}} (1+ |\d^{\perp} \cdot \gammab|^2)^\ell |\hat{W}(\gammab)|^2  < \infty, \, \hat{W}(\gammab)\in \C^4\Big\}.
\end{equation*}
and the mixed-regularity function spaces 
\begin{equation}\label{eq:function-space-mixed-smoothness}
    H_{\d}^{\ell_1}H_{\d^{\perp}}^{\ell_2} := \Big\{W = \sum_{\gammab \in \Gamma} \hat{W}(\gammab) e^{i \gammab\cdot \p}\,:\, \|W\|_{H_{\d}^{\ell_1}H_{\d^{\perp}}^{\ell_2}}^2 := \sum_{j\in \Z} (1+ |j|^2)^{\ell_1} \|\hat{W}|_{S_{[j,j+1)}}\|^2_{H^{\ell_2}_{S_{[j,j+1)}}}  < \infty, \, \hat{W}(\gammab)\in \C^4\Big\},
\end{equation}
where for $W = \sum_{\gamma \in \Gamma} \hat{W}(\gammab) e^{i \gammab\cdot \p}$ we denote $\hat{W}|_{S_{[j,j+1)}} = \sum_{\gammab \in S_{[j,j+1)}} \hat{W}(\gammab) e^{i \gammab\cdot \p}$. Note that the function space $H_{\d}^{\ell_1}H_{\d^{\perp}}^{\ell_2}$ can also be characterised by the equivalent norms
\begin{equation*}
    \begin{split}
        \|W\|_{H_{\d}^{\ell_1}H_{\d^{\perp}}^{\ell_2}} & \asymp \left(\sum_{\gammab\in \Gamma} (1+ |\d\cdot \gammab|^2)^{\ell_1}(1+ |\d^\perp \cdot \gammab|^2)^{\ell_1} |\hat{W}(\gammab)| \right)^{\tfrac{1}{2}}\\ 
         &\asymp \|W\|_{L^2_{\per}} + \|(\d\cdot \nabla)^{\ell_1}W\|_{L^2_{\per}} + \|(\d^\perp \cdot \nabla)^{\ell_2}W\|_{L^2_{\per}} + \|(\d\cdot \nabla)^{\ell_1}(\d^{\perp}\cdot \nabla)^{\ell_2}W\|_{L^2_{\per}}.
    \end{split}
\end{equation*}
Spaces of the type \eqref{eq:function-space-mixed-smoothness} are known as spaces with mixed dominating smoothness \cite{schmeisser1987book,triebel2019book} since they carry more regularity in mixed derivatives. It holds that $H_{\d}^{\ell_1}H_{\d^{\perp}}^{\ell_2} \hookrightarrow C_\per(\R^2)$ for any $\ell_1,\ell_2 \geq 1$ \cite[Rem.~2.4.1/2]{schmeisser1987book}. Moreover, they are multiplication algebras, and the following result is proved in \cite[Thm.~3.1]{nguyen2017-08JournalofMathematicalAnalysisandApplications} for $\ell_1 = \ell_2$, and the proof can be directly adapted to $\ell_1 \neq \ell_2$, cf.~\cite{hansen2010thesis}. Although the proof in \cite{nguyen2017-08JournalofMathematicalAnalysisandApplications} is only provided in function spaces on $\R^d$, the results can be directly transferred to spaces of periodic functions via the extension map $(\Ecal f)(x) = \phi(x) f(x \mod \Z^2)$ for $\phi \in C^\infty_c$ which is equal to one on a periodic cell, see the proof of \cite[Theorem A.7]{stevenson2025-02preprint}.
\begin{lemma}
    Assume $\ell_1,\ell_2 \geq 1$ are integers. Then $H_{\d}^{\ell_1}H_{\d^{\perp}}^{\ell_2}$ is an algebra and there exists a constant $C>0$ such that
    \begin{equation}\label{eq:product-estimate-mixed-smoothness}
        \|W_1W_2\|_{H_{\d}^{\ell_1}H_{\d^{\perp}}^{\ell_2}} \leq C \|W_1\|_{H_{\d}^{\ell_1}H_{\d^{\perp}}^{\ell_2}}\|W_2\|_{H_{\d}^{\ell_1}H_{\d^{\perp}}^{\ell_2}}
    \end{equation}
    holds for all $W_1,W_2 \in H_{\d}^{\ell_1}H_{\d^{\perp}}^{\ell_2}$.
\end{lemma}

\begin{remark}
    Note that the product estimate \eqref{eq:product-estimate-mixed-smoothness} is not obvious. In particular, the standard approach to estimate the product in $H^k(\R^d)$ for $k> \frac{d}{2}$ via $\|W_1W_2\|_{H^k}\leq C \|W_1\|_{H^k}\|W_2\|_{L^{\infty}} +\|W_1\|_{L^{\infty}}\|W_2\|_{H^k}$ fails, see \cite[Thm.~3.5]{nguyen2017-08JournalofMathematicalAnalysisandApplications}.
\end{remark}

With this framework, we now have the mapping properties
\begin{equation*}
    \tilde{\Lcal}^{\eps} : H^{\ell_1+1}_{\d}H^{\ell_2+4}_{\d^\perp} \to H^{\ell_1}_{\d}H^{\ell_2}_{\d^\perp} \quad \text{and} \quad \Ncal : H^{\ell_1}_{\d}H^{\ell_2+3}_{\d^\perp} \to H^{\ell_1}_{\d}H^{\ell_2}_{\d^\perp}
\end{equation*}
for all integers $\ell_1, \ell_2 \geq 1$. Therefore, we define the spaces
\begin{equation*}
    \Xcal = H^{\ell_1}_{\d}H^{\ell_2}_{\d^\perp}, \quad \Ycal = H^{\ell_1+1}_{\d}H^{\ell_2+1}_{\d^\perp}, \quad \Zcal = H^{\ell_1+1}_{\d}H^{\ell_2+4}_{\d^\perp}.
\end{equation*}

Next, we define the spectral projections onto the different parts of the spectrum of $\tilde{\Lcal}^{\eps}$. In particular, for every $\gammab \in \Gamma$ let $J_{mc}^{\eps}(\gammab)$, $J_{lc}^{\eps}(\gammab)$ and $J_h^{\eps}(\gammab)$ be positively oriented Jordan curves around the more-central, less-central, and hyperbolic eigenvalues of $\hat{\Lcal}^{\eps}(\gammab) - i \d \cdot \gammab$. Then, we define the projections
\begin{equation}\label{eq:def-projections}
    \Pcal_j^{\eps}(\gammab) := \int_{J_{j}^{\eps}(\gammab)} (\lambda - \hat{\Lcal}^{\eps}(\gammab) + i \d \cdot \gammab)^{-1} \dd \lambda \quad \text{ and } \quad  \Pcal_j^\eps := \sum_{\gammab \in \Gamma} e^{i\gammab \cdot \p} \Pcal_j^{\eps}(\gammab)
\end{equation}
for $j = mc, lc, h$. In particular, we find the following result.

\begin{lemma}\label{lem:estimate-projections}
    Let $\cot(\theta) \in \sqrt{3}\Q$ and $\ell_1, \ell_2 \geq 1$ be integers. If $\theta \neq \tfrac{\pi}{6}$, the Jordan curves $J_{mc}^{\eps}(\gammab)$, $J_{lc}^{\eps}(\gammab)$ and $J_h^{\eps}(\gammab)$ can be chosen independently of $\eps$ and there exists an $\eps$-independent constant $C$ such that
    \begin{equation}\label{eq:bound-projection}
        \|\Pcal_j^\eps\|_{H^{\ell_1}_{\d}H^{\ell_2}_{\d^\perp} \rightarrow H^{\ell_1}_{\d}H^{\ell_2}_{\d^\perp}} \leq C
    \end{equation}
    for $j = mc, lc, h$. In the case $\theta = \tfrac{\pi}{6}$, the Jordan curves $J_{mc}^{\eps}(\pm\k_2)$ and $J_{lc}^{\eps}(\pm\k_2)$ have to be chosen $\eps$-dependent and the estimates
    \begin{equation}\label{eq:bound-projection-Jordan4}
        \|\Pcal_{mc}^\eps(\pm \k_2)\|_{\C^4 \rightarrow \C^4} \leq C \eps^{-1} \text{ and } \|\Pcal_{lc}^\eps(\pm \k_2)\|_{\C^4 \rightarrow \C^4} \leq C \eps^{-1}
    \end{equation}
    hold. All other curves can be chosen independently of $\eps$ and the projections satisfy \eqref{eq:bound-projection}.
\end{lemma}

\begin{proof}
    If $\theta \neq \tfrac{\pi}{6}$, we note that for each $\gammab \in \Gamma$, the different parts of the spectrum of $\hat{\Lcal}^{\eps}(\gammab) - i \d \cdot \gammab$ are uniformly separated at $\eps = 0$, see \Cref{fig:spectral-eps=0}. Hence, the Jordan curves can be chosen independently of $\eps$ and the separation persists for $\eps > 0$. This proves the estimate \eqref{eq:bound-projection} since $\Pcal_j^\eps$ is given as the direct sum of $\Pcal_j^\eps(\gammab)$.

    In the case $\theta = \tfrac{\pi}{6}$, $\hat{\Lcal}^0(\pm \k_2) \mp i \d \cdot \k_2$ has a Jordan block of size 4, which separates into one more-central $\lambda_1(\eps)$ and three less-central eigenvalues $\lambda_j(\eps)$, $j=2,3,4$, for $\eps > 0$, see the proof of \Cref{prop:spectrum-eps-positive}. Let $J_{mc}^\eps(\pm\k_2)$ be a circle with radius of order $\eps$, which separates $\lambda_1(\eps)$ from the other eigenvalues. We may write
    \begin{equation*}
        \begin{split}
            \int_{J_{mc}^\eps(\pm\k_2)} (\lambda- \hat{\Lcal}^{\eps}(\pm\k_2) \pm i \d \cdot \k_2)^{-1} \dd \lambda & = \int_{J_{mc}^\eps(\pm\k_2)} \frac{\operatorname{adj}(\lambda-\hat{\Lcal}^{\eps}(\pm\k_2)\pm i \d \cdot \k_2)}{\det(\lambda-\hat{\Lcal}^{\eps}(\pm\k_2)\pm i \d \cdot \k_2)} \dd\lambda \\
            & = \int_{J_{mc}^\eps(\pm\k_2)} \frac{\operatorname{adj}(\lambda-\hat{\Lcal}^{\eps}(\pm\k_2)\pm i \d \cdot \k_2)}{\prod_{j=1}^4 (\lambda-\lambda_j(\eps))} \dd\lambda.
        \end{split}
    \end{equation*}
    Since $\hat{\Lcal}^{\eps}(\pm\k_2)$ is uniformly bounded in $\eps$ and using the residue theorem, we find that
    \begin{equation*}
        \|\Pcal^{\eps}_{mc}(\pm \k_2)\|_{\C^4\rightarrow\C^4} = \left\|\int_{J_{mc}^\eps(\pm\k_2)} (\lambda- \hat{\Lcal}^{\eps}(\pm\k_2)\pm i \d \cdot \k_2)^{-1} \dd \lambda\right\|_{\C^4\rightarrow\C^4} \lesssim \left|\textrm{Res}_{\lambda = \lambda_1(\eps)} \frac{1}{\prod_{j=1}^4 \lambda-\lambda_j(\eps)}\right| \lesssim \eps^{-1},
    \end{equation*}
    where the last estimate is due to $\lambda_j(\eps)\sim \eps^{1/3}$, $j=2,3,4$. Using that $\Pcal_{lc}^{\eps}(\pm \k_2) = I - \Pcal_{mc}^{\eps}(\pm \k_2)$ shows \eqref{eq:bound-projection-Jordan4} and completes the proof.
\end{proof}

\begin{remark}\label{rem:projection-is-optimal}
    We point out that in the case $\theta=\tfrac{\pi}{6}$, the estimates for the projections in \eqref{eq:bound-projection-Jordan4} are optimal. Consider, for example, the matrix
    \begin{equation*}
        \begin{pmatrix}
            0 & 1 & 0 & 0 \\
            0 & 0 & 1 & 0 \\
            0 & 0 & 0 & 1 \\
            - \eps^2 & \eps & 0 & \eps
        \end{pmatrix}
    \end{equation*}
    whose characteristic polynomial is given by $(\lambda-\eps)(\lambda^3-\eps)$. Then, the projection $\Pcal^{\eps}$ onto the more-central eigenvector $\phib^{\eps}= (1,\eps,\eps^2,\eps^3)$ can explicitly be computed as
    \begin{equation*}
        \Pcal^{\eps} W = \frac{\phib^{\eps,*} \cdot W}{\phib^{\eps,*}\cdot \phib^{\eps}} \phib^{\eps} = \frac{\phib^{\eps,*} \cdot W}{\eps + \eps^3} \phib^{\eps},
    \end{equation*}
    where $\phib^{\eps,*}$ is the corresponding eigenvector to the adjoint matrix, normalised so that the last component equals one.
\end{remark}

With these projections, we can also define the projected spaces 
\begin{equation*}
    \Xcal_j^\eps := \Pcal_j^\eps \Xcal, \quad \Ycal_j^\eps := \Pcal_j^\eps \Ycal \quad\text{and} \quad\Zcal_j^\eps := \Pcal_j^\eps\Zcal \quad\text{for } j = mc, lc, h.
\end{equation*}

\subsection{The rescaled system and normal-form analysis}\label{sec:rescaling-normal-form}

We now construct a centre manifold by following the general strategy in \cite{haragus-courcelle1999-01ZangewMathPhys}: First, we split the variables into a more-central, less-central, and hyperbolic part. The main idea for quantifying the $\eps$-dependence is to introduce rescaled variables that use different $\eps$-rescalings for the more-central part compared to the less-central and hyperbolic parts. We then prove that the resulting rescaled system has a centre manifold of size $\Ocal(1)$, and therefore, that the $\eps$-dependence for the centre manifold in the original system is fully characterised by the rescaling. The construction follows the standard approach, except for one main obstruction: the quadratic terms generate nonlinear contributions which prevent a direct fixed-point construction. However, these can be handled by invoking \Cref{ass:smallness} and performing an $\eps$-dependent normal-form transformation.

Using the projections \eqref{eq:def-projections}, we split $W$ into the more-central, less-central, and hyperbolic parts
\begin{equation}\label{eq:projected-variables}
    W_{mc} := \Pcal_{mc}^\eps W, \quad W_{lc} := \Pcal_{lc}^\eps W, \quad W_h := \Pcal_h^\eps W
\end{equation}
and define the projected linear operators
\begin{equation*}
\begin{split}
    \tilde{\Lcal}_{mc}^{\eps}(\gammab) & := \Pcal_{mc}^\eps(\gammab) (\hat{\Lcal}^{\eps}(\gammab)-i\d\cdot\gammab) \Pcal_{mc}^\eps(\gammab), \qquad \tilde{\Lcal}_{lc}^{\eps}(\gammab) := \Pcal_{lc}^\eps(\gammab) (\hat{\Lcal}^{\eps}(\gammab)-i\d\cdot\gammab) \Pcal_{lc}^\eps(\gammab), \quad \text{ and}\\
    \tilde{\Lcal}_{h}^{\eps}(\gammab) & := \Pcal_{h}^\eps(\gammab) (\hat{\Lcal}^{\eps}(\gammab)-i\d\cdot\gammab) \Pcal_{h}^\eps(\gammab).
\end{split}
\end{equation*}
We then introduce the rescaled variables
\begin{equation}\label{eq:rescaling}
    \shortunderline{W}_{mc} := \eps^{-\beta} W_{mc}, \quad \shortunderline{W}_{lc} := \eps^{-\gamma} W_{lc}, \quad \shortunderline{W}_{h} := \eps^{-\gamma} W_{h}
\end{equation}
with $0 < \beta < 1 < \gamma < 2$ to be determined later, see \eqref{eq:beta-gamma-final}.

The goal now is to show that the system for the rescaled variables $(\shortunderline{W}_{mc},\shortunderline{W}_{lc},\shortunderline{W}_{h})$ has an $\Ocal(1)$-centre manifold. Applying the projections to the spatial-dynamics system \eqref{eq:spat-dyn-Fourier}, we obtain the system
\begin{equation}\label{eq:spatial-dynamics-split}
    \begin{split}
        \partial_{\xi} \shortunderline{\hat{W}}_{mc}(\cdot, \gammab)  & = \tilde{\Lcal}_{mc}^{\eps}(\gammab)\shortunderline{\hat{W}}_{mc}(\cdot,\gammab) + \shortunderline{\hat{\Ncal}}_{mc} (\shortunderline{\hat{W}}_{mc},\shortunderline{\hat{W}}_{lc} + \shortunderline{\hat{W}}_{h};\gammab,\vartheta), \\
        \partial_{\xi} \shortunderline{\hat{W}}_{lc}(\cdot, \gammab)  & = \tilde{\Lcal}_{lc}^{\eps}(\gammab)\shortunderline{\hat{W}}_{lc}(\cdot,\gammab) + \shortunderline{\hat{\Ncal}}_{lc} (\shortunderline{\hat{W}}_{mc},\shortunderline{\hat{W}}_{lc} + \shortunderline{\hat{W}}_{h};\gammab,\vartheta), \\
        \partial_{\xi} \shortunderline{\hat{W}}_{h}(\cdot, \gammab)  & = \tilde{\Lcal}_{h}^{\eps}(\gammab)\shortunderline{\hat{W}}_{h}(\cdot,\gammab) + \shortunderline{\hat{\Ncal}}_{h} (\shortunderline{\hat{W}}_{mc}, \shortunderline{\hat{W}}_{lc} + \shortunderline{\hat{W}}_{h};\gammab,\vartheta),
    \end{split}
\end{equation}
where the rescaled nonlinearities are defined as
\begin{equation*}
\begin{split}
    &\shortunderline{\hat{\Ncal}}_{mc} (\shortunderline{\hat{W}}_{mc},\shortunderline{\hat{W}}_{lc} + \shortunderline{\hat{W}}_{h};\gammab,\vartheta)  : = \eps^{-\beta} \Pcal_{mc}^\eps(\gammab) \hat{\Ncal}(\eps^\beta \shortunderline{\hat{W}}_{mc} + \eps^\gamma(\shortunderline{\hat{W}}_{lc} + \shortunderline{\hat{W}}_{h}) ;\gammab,\vartheta) \\
    &= \eps^\beta \Pcal_{mc}^\eps(\gammab) \hat{\Ncal}_2(\shortunderline{\hat{W}}_{mc};\gammab,\vartheta) + 2 \eps^{\gamma}\Pcal_{mc}^\eps(\gammab) \hat{\Ncal}_2(\shortunderline{\hat{W}}_{mc}, \shortunderline{\hat{W}}_{lc} + \shortunderline{\hat{W}}_{h};\gammab,\vartheta) + \eps^{2\gamma - \beta } \Pcal_{mc}^\eps(\gammab) \hat{\Ncal}_2(\shortunderline{\hat{W}}_{lc} + \shortunderline{\hat{W}}_{h};\gammab,\vartheta) \\
    &\quad + \eps^{2\beta} \Pcal_{mc}^\eps(\gammab) \hat{\Ncal}_3(\shortunderline{\hat{W}}_{mc} + \eps^{\gamma-\beta} (\shortunderline{\hat{W}}_{lc} + \shortunderline{\hat{W}}_{h});\gammab,\vartheta) + \eps^{3\beta} \Pcal_{mc}^\eps(\gammab)\hat{\Rcal}(\shortunderline{\hat{W}}_{mc} + \eps^{\gamma-\beta} (\shortunderline{\hat{W}}_{lc} + \shortunderline{\hat{W}}_{h});\gammab,\vartheta),
\end{split}
\end{equation*}
\begin{equation*}
\begin{split}
    &\shortunderline{\hat{\Ncal}}_{lc} (\shortunderline{\hat{W}}_{mc},\shortunderline{\hat{W}}_{lc} + \shortunderline{\hat{W}}_{h};\gammab,\vartheta)  := \eps^{-\gamma} \Pcal_{lc}^\eps(\gammab) \hat{\Ncal}(\eps^\beta \shortunderline{\hat{W}}_{mc} + \eps^\gamma(\shortunderline{\hat{W}}_{lc} + \shortunderline{\hat{W}}_{h}) ;\gammab,\vartheta), \\
    &= \eps^{2 \beta - \gamma} \Pcal_{lc}^\eps(\gammab) \hat{\Ncal}_2(\shortunderline{\hat{W}}_{mc};\gammab,\vartheta) + 2 \eps^{\beta} \Pcal_{lc}^\eps(\gammab)\hat{\Ncal}_2(\shortunderline{\hat{W}}_{mc}, \shortunderline{\hat{W}}_{lc} + \shortunderline{\hat{W}}_{h};\gammab,\vartheta) + \eps^{\gamma} \Pcal_{lc}^\eps(\gammab) \hat{\Ncal}_2(\shortunderline{\hat{W}}_{lc} + \shortunderline{\hat{W}}_{h};\gammab,\vartheta) \\
    &\quad + \eps^{3\beta - \gamma} \Pcal_{lc}^\eps(\gammab)\hat{\Ncal}_3(\shortunderline{\hat{W}}_{mc} + \eps^{\gamma-\beta} (\shortunderline{\hat{W}}_{lc} + \shortunderline{\hat{W}}_{h});\gammab,\vartheta) + \eps^{4\beta - \gamma} \Pcal_{lc}^\eps(\gammab){\Rcal}(\shortunderline{\hat{W}}_{mc} + \eps^{\gamma-\beta} (\shortunderline{\hat{W}}_{lc} + \shortunderline{\hat{W}}_{h});\gammab,\vartheta),
\end{split}
\end{equation*}
\begin{equation*}
\begin{split}
    &\shortunderline{\hat{\Ncal}}_{h} (\shortunderline{\hat{W}}_{mc},\shortunderline{\hat{W}}_{lc} + \shortunderline{\hat{W}}_{h};\gammab,\vartheta) := \eps^{-\gamma} \Pcal_{h}^\eps(\gammab) \hat{\Ncal}(\eps^\beta \shortunderline{\hat{W}}_{mc} + \eps^\gamma(\shortunderline{\hat{W}}_{lc} + \shortunderline{\hat{W}}_{h}) ;\gammab,\vartheta) \\
    &= \eps^{2 \beta - \gamma} \Pcal_{h}^\eps(\gammab)\hat{\Ncal}_2(\shortunderline{\hat{W}}_{mc};\gammab,\vartheta) + 2 \eps^{\beta} \Pcal_{h}^\eps(\gammab)\hat{\Ncal}_2(\shortunderline{\hat{W}}_{mc}, \shortunderline{\hat{W}}_{lc} + \shortunderline{\hat{W}}_{h};\gammab,\vartheta) + \eps^{\gamma} \Pcal_{h}^\eps(\gammab) \hat{\Ncal}_2(\shortunderline{\hat{W}}_{lc} + \shortunderline{\hat{W}}_{h};\gammab,\vartheta) \\
    &\quad + \eps^{3\beta - \gamma} \Pcal_{h}^\eps(\gammab)\hat{\Ncal}_3(\shortunderline{\hat{W}}_{mc} + \eps^{\gamma-\beta} (\shortunderline{\hat{W}}_{lc} + \shortunderline{\hat{W}}_{h});\gammab,\vartheta) + \eps^{4\beta - \gamma} \Pcal_{h}^\eps(\gammab)\hat{\Rcal}(\shortunderline{\hat{W}}_{mc} + \eps^{\gamma-\beta} (\shortunderline{\hat{W}}_{lc} + \shortunderline{\hat{W}}_{h});\gammab,\vartheta).
\end{split}
\end{equation*}

Following the standard construction \cite{haragus2011book}, we aim to obtain a centre manifold as a fixed point of the variation of constants formula in exponentially weighted spaces, see \eqref{eq:weighted_space}. This requires estimating the product of a bound on the semigroups generated by $\Pcal_j^\eps\tilde{\Lcal}^\eps$ and the Lipschitz constant of the projected nonlinearities locally around $W\equiv 0$. Since the nonlinearities are polynomials, they are locally Lipschitz with a local Lipschitz constant of the order of the corresponding $\eps$-scaling. If $\theta \neq \tfrac{\pi}{6}$, the $\eps$-scaling of the semigroup bounds is inherited from the spectral behaviour, that is, the semigroups generated by $\tilde{\Lcal}_{mc}^{\eps}$ and $\tilde{\Lcal}_{lc}^{\eps}$ scale like $\eps^{-1}$ and the semigroup generated by $\tilde{\Lcal}_{h}^{\eps}$ is bounded independent of $\eps$, see \Cref{lem:semigroup-bounds} for more details. Note that although the eigenvalues of $\tilde{\Lcal}_{lc}^{\eps}$ have real parts of order $\sqrt{\eps}$, it still generates a scaling of $\eps^{-1}$ since the eigenvalues appear as Jordan blocks on the imaginary axis at $\eps = 0$, which split into stable and unstable eigenvalues. The case $\theta = \tfrac{\pi}{6}$ is more complicated and will be discussed later, see \Cref{sec:special-case}. However, it turns out that the relevant nonlinearities are the same.

This allows us to identify problematic nonlinearities in \eqref{eq:spatial-dynamics-split}, which generate singular terms in $\eps$ in the fixed-point argument. In the more-central equation, the only problematic nonlinearity is $\eps^\beta \Pcal_{mc}^\eps(\gammab) \hat{\Ncal}_2(\shortunderline{\hat{W}}_{mc};\gammab,\vartheta)$ since $0<\beta < 1 < \gamma$. In contrast to the one-dimensional case \cite{haragus-courcelle1999-01ZangewMathPhys} and the square case discussed in \Cref{sec:square}, this term is in general non-zero due to the resonance of the hexagonal lattice. Therefore, we need the additional smallness \Cref{ass:smallness} to guarantee that this term gains an additional $\eps$ and can be handled in the fixed-point argument. In the less-central equation, the problematic terms are $\eps^{2 \beta - \gamma} \Pcal_{h}^\eps(\gammab)\hat{\Ncal}_2(\shortunderline{\hat{W}}_{mc};\gammab,\vartheta)$ and $\eps^{\beta} \Pcal_{lc}^\eps(\gammab)\hat{\Ncal}_2(\shortunderline{\hat{W}}_{mc}, \shortunderline{\hat{W}}_{lc} + \shortunderline{\hat{W}}_{h};\gammab,\vartheta)$, where we point out that $2\beta-\gamma < 1$ due to $\beta < 1 < \gamma < 2$. As in \cite{haragus-courcelle1999-01ZangewMathPhys}, we show that there is a normal-form transformation to remove these terms. Note that, in contrast to \cite{haragus-courcelle1999-01ZangewMathPhys}, the resonances of the hexagonal lattice cause additional obstructions. However, we still obtain a uniformly bounded, near-identity transformation under the additional, non-restrictive assumption $\beta > \tfrac{1}{2}$, cf.~\Cref{lem:normal-form-analysis}. Finally, since the semigroup generated in the hyperbolic equation is uniformly bounded, there are no problematic nonlinearities.

We now discuss the normal-form transformation in the less-central equation. Therefore, we make the near-identity transformation
\begin{equation}\label{eq:normal-form-lc}
    \Tcal^\eps\shortunderline{\hat{W}}_{lc}(\gammab) = \shortunderline{\hat{W}}_{lc}(\gammab) + \eps^{2\beta-\gamma} \Bcal_1(\shortunderline{\hat{W}}_{mc},\shortunderline{\hat{W}}_{mc};\gammab) + \eps^{\beta} \Bcal_2(\shortunderline{\hat{W}}_{mc},\shortunderline{\hat{W}}_{lc};\gammab) + \eps^{\beta} \Bcal_3(\shortunderline{\hat{W}}_{mc},\shortunderline{\hat{W}}_{h};\gammab)
\end{equation}
with bilinear forms $\Bcal_1$, $\Bcal_2$, and $\Bcal_3$. Inserting \eqref{eq:normal-form-lc} into the less-central equation and setting the contributions of order $\eps^{2\gamma - \beta}$ and $\eps^\beta$ to zero to remove the problematic terms, we find that $\Bcal_1$, $\Bcal_2$, and $\Bcal_3$ are determined by the equations
\begin{equation}\label{eq:normal-form-bilinear}
\begin{split}
    &\Bcal_1(\tilde{\Lcal}_{mc}^{\eps}\shortunderline{\hat{W}}_{mc},\shortunderline{\hat{W}}_{mc};\gammab) + \Bcal_1(\shortunderline{\hat{W}}_{mc},\tilde{\Lcal}_{mc}^{\eps}\shortunderline{\hat{W}}_{mc};\gammab) - \tilde{\Lcal}_{lc}^{\eps}(\gammab)\Bcal_1(\shortunderline{\hat{W}}_{mc},\shortunderline{\hat{W}}_{mc};\gammab)\\
    & \qquad+ \Pcal_{lc}^\eps(\gammab)\hat{\Ncal}_2(\shortunderline{\hat{W}}_{mc};\gammab,\vartheta) = 0, \\
    &\Bcal_2(\tilde{\Lcal}_{mc}^{\eps}\shortunderline{\hat{W}}_{mc},\shortunderline{\hat{W}}_{lc};\gammab) + \Bcal_2(\shortunderline{\hat{W}}_{mc},\tilde{\Lcal}_{lc}^{\eps}\shortunderline{\hat{W}}_{lc};\gammab) - \tilde{\Lcal}_{lc}^{\eps}(\gammab)\Bcal_2(\shortunderline{\hat{W}}_{mc},\shortunderline{\hat{W}}_{lc};\gammab)\\
    & \qquad+ 2\Pcal_{lc}^\eps(\gammab)\hat{\Ncal}_2(\shortunderline{\hat{W}}_{mc}, \shortunderline{\hat{W}}_{lc};\gammab,\vartheta) = 0\\
    &\Bcal_3(\tilde{\Lcal}_{mc}^{\eps}\shortunderline{\hat{W}}_{mc},\shortunderline{\hat{W}}_{h};\gammab) + \Bcal_3(\shortunderline{\hat{W}}_{mc},\tilde{\Lcal}_h^{\eps}\shortunderline{\hat{W}}_{h};\gammab) - \tilde{\Lcal}_{lc}^{\eps}(\gammab)\Bcal_3(\shortunderline{\hat{W}}_{mc}, \shortunderline{\hat{W}}_{h};\gammab)\\
    & \qquad+ 2\Pcal_{lc}^\eps(\gammab)\hat{\Ncal}_2(\shortunderline{\hat{W}}_{mc},\shortunderline{\hat{W}}_{h};\gammab,\vartheta) = 0.
\end{split}
\end{equation}
To solve these equations, we note that $\Bcal_1$ is a linear map from $\Zcal_{mc}^{\eps} \otimes \Zcal_{mc}^{\eps} \to \Zcal$ and hence we need to solve
\begin{equation*}
    \Bcal_1 (\tilde{\Lcal}_{mc}^\eps(\gammab) \otimes \Id + \Id \otimes \tilde{\Lcal}_{mc}^\eps(\gammab)) - \tilde{\Lcal}_{lc}^{\eps}(\gammab)\Bcal_1 = -  \Pcal_{lc}^\eps(\gammab)\hat{\Ncal}_2.
\end{equation*}
It suffices to solve this on the basis $\phib_i^{\eps} \otimes \phib_j^{\eps}$ consisting of tensor products of eigenvectors $\phib_i$ to $\tilde{\Lcal}_{mc}^{\eps}$. So, equation \eqref{eq:normal-form-bilinear} simplifies to solving
\begin{equation*}
    (\lambda_i^{\eps} + \lambda_j^{\eps} - \tilde{\Lcal}_{lc}^{\eps}(\gammab)) \Bcal_1(\phib_i^{\eps},\phib_j^{\eps};\gammab) = -  \Pcal_{lc}^\eps(\gammab)\hat{\Ncal}_2(\phib_i^{\eps},\phib_j^{\eps};\gammab,\vartheta).
\end{equation*}
Note that this has a solution for all tensors $\phi_i^{\eps}\otimes \phi_j^{\eps}$ if $\lambda_i^{\eps}+\lambda_j^{\eps}$ is in the resolvent set of $\tilde{\Lcal}_{lc}^{\eps}(\gammab)$. The solution has a bound of order one, provided that $\lambda_i^{\eps}+\lambda_j^{\eps}$ is uniformly bounded away from the spectrum of $\tilde{\Lcal}_{lc}^{\eps}(\gammab)$. The same argument applies to $\Bcal_2$, which is a linear map from $\Zcal_{mc}\otimes \Zcal_{lc} \to \Zcal$. So, $\Bcal_2$ is defined provided that $\lambda_i^{\eps}+\lambda_j^{\eps}$ is in the resolvent set of $\tilde{\Lcal}_{lc}^{\eps}(\gammab)$, where $\lambda_i^{\eps}$ is an eigenvalue of $\tilde{\Lcal}_{mc}^\eps(\gammab_i)$ and $\lambda_j^{\eps}$ is an eigenvalue of $\tilde{\Lcal}_{lc}^\eps(\gammab_j)$ where $\gammab_i + \gammab_j = \gammab$. Then, $\Bcal_2$ has an order one bound if $\lambda_i^{\eps}+\lambda_j^{\eps}$ is uniformly bounded away from the spectrum of $\tilde{\Lcal}_{lc}^{\eps}(\gammab)$. The same applies to $\Bcal_3$. However, while such a uniform bound exists for $\Bcal_1$ and $\Bcal_3$, it turns out that we can only obtain an $\eps$-dependent lower bound for $\Bcal_2$. This is in contrast to the one-dimensional case discussed in \cite{haragus-courcelle1999-01ZangewMathPhys}, where a uniform bound is obtained. Nevertheless, and due to the specific geometry of the eigenvalues, cf.~\Cref{fig:spectral-eps=0}, we still obtain the following result.

\begin{lemma}\label{lem:normal-form-analysis}
    There exist maps $\Bcal_1$ and $\Bcal_3$, which satisfy \eqref{eq:normal-form-bilinear} and are bounded independently of $\eps$. Additionally, there are maps $\Bcal_{2,1}$ and $\Bcal_{2,2}$, which are bounded independently of $\eps$ such that 
    \begin{equation*}
        \Bcal_2 = \Bcal_{2,1} + \eps^{-\tfrac{1}{2}} \Bcal_{2,2}
    \end{equation*}
    is a solution to \eqref{eq:normal-form-bilinear}.
\end{lemma}

Before we prove this result, we observe that this is sufficient to choose $\beta > \tfrac{1}{2}$ to guarantee that the normal-form transformation is indeed a near-identity transformation. This does not pose additional challenges to the subsequent centre manifold reduction.

\begin{proof}
    For $\Bcal_1$ we note that $\lambda_i^\eps$ and $\lambda_j^\eps$ as eigenvalues of $\tilde{\Lcal}_{mc}^\eps(\gammab)$ lie in a neighbourhood of size $\Ocal(\eps)$ of the origin. In contrast, the spectrum of $\tilde{\Lcal}_{lc}^{\eps}(\gammab)$ has an imaginary part which is strictly bounded away from zero. This allows us to define $\Bcal_1$ with an $\Ocal(1)$-bound. 

    Similarly, for $\Bcal_3$ we find that the spectrum of $\tilde{\Lcal}_{h}^\eps(\gammab)$ is strictly bounded away from the imaginary axis uniformly in $\eps$. Therefore, we find a map $\Bcal_3$ with an $\Ocal(1)$-bound.

    It remains to construct $\Bcal_2$. The main issue is that, under the assumption that $\cot(\theta)=\sqrt{3}\Q$, there exist wave-number pairs $\gammab_1$ and $\gammab_2$ with $|\gammab_1 - \gammab_2| = 1$ such that $\tilde{\Lcal}_{lc}^\eps(\gammab_1)$ and $\tilde{\Lcal}_{lc}^\eps(\gammab_2)$ have eigenvalues with the same imaginary part for $\eps = 0$, see \Cref{rem:same-imaginary-part}. Therefore, we cannot expect to obtain a transformation $\Bcal_2$, which is bounded as $\eps \to 0$. 
    
    Instead, we split $\Bcal_2$ into two parts $\Bcal_{2,1}$ and $\tilde{\Bcal}_{2,2}$. The first part $\Bcal_{2,1}$ is constructed to remove all quadratic combinations, where the eigenvalues of $\tilde{\Lcal}_{lc}^\eps(\gammab_1)$ and $\tilde{\Lcal}_{lc}^\eps(\gammab_2)$ have different imaginary parts at $\eps = 0$. In this case, the existence of a bounded map $\Bcal_{2,1}$ follows as for $\Bcal_1$ and $\Bcal_3$. In particular $\Bcal_{2,1}$ is uniformly bounded in $\eps > 0$ and $(\gammab_1,\gammab_2)$. The latter follows from the fact that the eigenvalues of $\tilde{\Lcal}^0$, which lie on the imaginary axis, are discrete and have a positive minimal distance, see \Cref{prop:imaginary-spectrum-discrete,prop:spectrum-eps-positive}. The second part deals with the resonant case. The key observation is that, while $\tilde{\Lcal}_{lc}^\eps(\gammab_1)$ and $\tilde{\Lcal}_{lc}^\eps(\gammab_2)$ share one eigenvalue for $\eps = 0$, for $\eps > 0$ they have different real parts. In particular, the difference is of order $\sqrt{\eps}$, which follows from \eqref{eq:expansion-imaginary-less-central}. Moreover, the difference is also bounded away from zero uniformly in $(\gammab_1,\gammab_2)$. Indeed, \eqref{eq:expansion-imaginary-less-central} implies that for the eigenvalues $\lambda(\gammab_1)^{\eps}$ and $\lambda(\gammab_2)^{\eps}$ of $\tilde{\Lcal}_{lc}^\eps(\gammab_1)$ and $\tilde{\Lcal}_{lc}^\eps(\gammab_2)$, respectively, the expansions
    \begin{equation*}
        \lambda(\gammab_1)^{\eps} = \lambda(\gammab_1)^{0} + \sqrt{\eps} \sqrt{-i\lambda(\gammab_1)^{0}} \frac{ic_0}{2\Delta_i(\gammab_1)} + \Ocal(\eps)
    \end{equation*}
    with $i\Delta_i(\gammab_1) = i \d\cdot \gammab_1 - \lambda(\gammab_1)^{0}$ hold. We observe that $\Delta_i(\gammab_1) - \Delta_i(\gammab_2) = \d \cdot (\gammab_1 - \gammab_2) = \pm \d\cdot \k_j$ for some $j=1,2,3$, which is non-zero if $\theta\neq \frac{\pi}{6}$. Furthermore, this also shows that the difference $\lambda(\gammab_1)^{\eps} - \lambda(\gammab_2)^{\eps} \sim \sqrt{\eps}$ with a constant uniform in $\gammab\in \Gamma$ with $|\d^{\perp}\cdot \gammab| \leq 1$. Therefore, we make the ansatz $\tilde{\Bcal}_{2,2} = \eps^{-\tfrac{1}{2}} \Bcal_{2,2}$ and it remains to show that $\Bcal_{2,2}$ is bounded. This follows similarly to the construction of $\Bcal_{2,1}$ by noting that $\eps^{-\tfrac{1}{2}}(\lambda_i^\eps+ \lambda_j^{\eps} - \tilde{\Lcal}_{lc}^{\eps}(\gammab))^{-1}$ is a uniformly bounded operator for $\eps > 0$ and $\gammab\in \Gamma$ with $|\d^{\perp}\cdot \gammab| \leq 1$. This concludes the proof.
\end{proof}

\begin{remark}\label{rem:same-imaginary-part}
    If $\cot(\theta) \in \sqrt{3}\Q$, there exist infinitely many wave-number pairs $\gammab_1$ and $\gammab_2$ with $|\gammab_1 - \gammab_2| = 1$ such that $\tilde{\Lcal}_{lc}^\eps(\gammab_1)$ and $\tilde{\Lcal}_{lc}^\eps(\gammab_2)$ have eigenvalues with the same imaginary part $\lambda$ for $\eps = 0$. Indeed, $\|\gammab_1 + \lambda \d\| = 1 = \|\gammab_2 + \lambda \d\|$ and $\gammab_1-\gammab_2 = \pm \k_j$ infers that $\gammab_1 + \lambda \d = \pm \k_{\ell}$ since $(\k_n,-\k_{\ell})$ and $(-\k_n,\k_{\ell})$ are the only two pairs of points with distance one and lying on a line intersecting the circle with direction $\k_j$ where $\{j,\ell,n\}=\{1,2,3\}$, see \Cref{fig:geometric-argument}. The fact that there are infinitely many follows from the same intuition that lies behind the proof of \Cref{prop:real-spectral-gap}.
\end{remark}

\begin{figure}[h]
    \centering
    \includegraphics[width=0.6\textwidth]{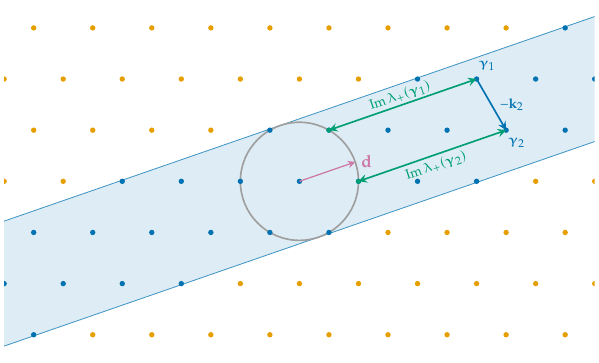}
    \caption{Geometric picture of two Fourier modes $\gammab_1$ and $\gammab_2$ with $\gammab_2 - \gammab_1 = -\k_2$ which generate two less-central eigenvalues of $\tilde{\Lcal}^0$ with the same imaginary part.}
    \label{fig:geometric-argument}
\end{figure}

We call the new variables under the normal-form transform \eqref{eq:normal-form-lc} 
\begin{equation}\label{eq:normal-form-W-to-V}
    \hat{V} = (\hat{V}_{mc},\hat{V}_{lc},\hat{V}_{h}) = (\shortunderline{\hat{W}}_{mc},\Tcal^{\eps}\shortunderline{\hat{W}}_{lc},\shortunderline{\hat{W}}_h).
\end{equation}
Then equation \eqref{eq:spatial-dynamics-split} in the new variables is given by
\begin{equation}\label{eq:spatial-dynamics-after-normal-form}
    \begin{split}
        \partial_{\xi} \hat{V}_{mc}(\cdot, \gammab)  & = \tilde{\Lcal}_{mc}^\eps(\gammab) \hat{V}_{mc}(\cdot, \gammab) + \hat{\Ncal}_{mc} (\hat{V}_{mc},\hat{V}_{lc} + \hat{V}_{h};\gammab,\vartheta), \\
        \partial_{\xi} \hat{V}_{lc}(\cdot, \gammab)  & = \tilde{\Lcal}_{lc}^{\eps}(\gammab) \hat{V}_{lc}(\cdot,\gammab) + \hat{\Ncal}_{lc} (\hat{V}_{mc},\hat{V}_{lc} + \hat{V}_{h};\gammab,\vartheta), \\
        \partial_{\xi} \hat{V}_{h}(\cdot, \gammab)  & = \tilde{\Lcal}_{h}^\eps(\gammab) \hat{V}_{h}(\cdot,\gammab) + \hat{\Ncal}_{h} (\hat{V}_{mc}, \hat{V}_{lc} + \hat{V}_{h};\gammab,\vartheta),
    \end{split}
\end{equation}
where, by a slight abuse of notation, we have reused the notation $\Ncal$ for the nonlinearity of the rescaled and normal-form-transformed variables to avoid introducing more letters. 

\begin{lemma}\label{lem:Lipschitz}
    For every $\ell >0$ and for $\eps >0$ sufficiently small, the map $\Ncal = (\Ncal_{mc},\Ncal_{lc},\Ncal_h)$ is locally Lipschitz continuous as a map from $\Zcal \to \Ycal^3$. That is, for every $r>0$ and $V_1,V_2 \in B_r(0)\subset \Zcal$ there is a constant $C_r>0$ independent of $\eps$ such that the following bounds hold true
    \begin{equation}\label{eq:Lipschitz-bounds}
    \begin{split}
        \|\Ncal_{mc}(V_1) - \Ncal_{mc}(V_2)\|_{\Ycal} & \leq C_r \eps^{\kappa_{mc}} \|V_1-V_2\|_{\Zcal}, \\
        \|\Ncal_{lc}(V_1) - \Ncal_{lc}(V_2)\|_{\Ycal} & \leq C_r \eps^{\kappa_{lc}} \|V_1-V_2\|_{\Zcal}, \\
        \|\Ncal_{h}(V_1) - \Ncal_{h}(V_2)\|_{\Ycal} & \leq C_r \eps^{\kappa_{h}} \|V_1-V_2\|_{\Zcal},
    \end{split}
    \end{equation}
    where $\kappa_{mc} = \min(\beta + 1, \gamma, 2 \beta)$, $\kappa_{lc} = \min(3\beta - \gamma, 2\beta - \tfrac{1}{2},\gamma)$, and $\kappa_h = \min(2\beta - \gamma, \beta, \gamma)$.
\end{lemma}

\begin{proof}
    We observe that $\Ncal$ is still a polynomial nonlinearity in $(V,(\d^{\perp}\cdot\nabla_p) V,(\d^{\perp}\cdot\nabla_p)^2V,(\d^{\perp}\cdot\nabla_p)^3_{\p}V)$ since we assumed the nonlinearity $N$ in \eqref{eq:Swift-Hohenberg} to be polynomial and the normal-form transformation \eqref{eq:normal-form-lc} only increases the degree. So, $\Ncal$ is locally Lipschitz, and it suffices to compute the lowest-order scaling contributions in $\eps$.

    For this, we first note that the normal-form transform only acts on the less-central modes and there, we have $\hat{V}_{lc} = \hat{W}_{lc} + \eps^{\tilde{\kappa}} \Bcal$ with $\tilde{\kappa} = \min(2\beta - \gamma, \beta, \beta - \tfrac{1}{2})$. Therefore, for the more-central and hyperbolic part of the nonlinearity, the conjectured scaling follows from the $\eps$-scaling of the leading-order terms in the rescaled nonlinearities \eqref{eq:spatial-dynamics-split}. The scaling for $\Ncal_{mc}$ follows by \Cref{ass:smallness}, which gives an additional order of $\eps$ for the term $\Pcal_{mc}^\eps(\gammab) \hat{\Ncal}_2(\shortunderline{\hat{W}}_{mc};\gammab,\vartheta)$. After applying the normal-form transform, the first two terms are replaced by a higher-order polynomial with leading term given by $2 \eps^{\beta} \Pcal_{lc}^\eps(\gammab)\hat{\Ncal}_2(\shortunderline{\hat{W}}_{mc}, \eps^{\tilde{\kappa}}\Bcal;\gammab,\vartheta)$. The last relevant term is given by the third term $\eps^{\gamma} \Pcal_{lc}^\eps(\gammab) \hat{\Ncal}_2(\shortunderline{\hat{W}}_{lc} + \shortunderline{\hat{W}}_{h};\gammab,\vartheta)$. Using that $\tilde{\kappa} = \min(\beta-\tfrac{1}{2},2\beta-\gamma)$, gives $\kappa_{lc} = \min(3\beta - \gamma, 2 \beta - \tfrac{1}{2}, \gamma)$. This completes the proof.
\end{proof}

\subsection{Fixed-point argument and centre manifold theorem}\label{sec:proof-centre-manifold-theorem}

We now prove the existence of a centre manifold for the transformed system \eqref{eq:spatial-dynamics-after-normal-form} via a standard fixed-point argument in exponentially weighted spaces using the integral formulation of \eqref{eq:spatial-dynamics-after-normal-form}, see e.g.~\cite{haragus2011book}. To formulate this, we introduce the projected semigroups
\begin{equation}\label{eq:projected-semigroups}
\begin{split}
    S_{mc}^\eps(t;\gammab) &:= \dfrac{1}{2\pi} \int_{J_{mc}^\eps(\gammab)} (\lambda I - \hat{\Lcal}^{\eps}(\gammab))^{-1} e^{\lambda t} \dd\lambda, \\
    S_{lc,s}^\eps(t;\gammab) &:= \dfrac{1}{2\pi} \int_{J_{lc,s}^\eps(\gammab)} (\lambda I - \hat{\Lcal}^{\eps}(\gammab))^{-1} e^{\lambda t} \dd\lambda, \\
    S_{lc,u}^\eps(t;\gammab) &:= \dfrac{1}{2\pi} \int_{J_{lc,u}^\eps(\gammab)} (\lambda I - \hat{\Lcal}^{\eps}(\gammab))^{-1} e^{\lambda t} \dd\lambda, \\
    S_{h,s}^\eps(t;\gammab) &:= \dfrac{1}{2\pi} \int_{J_{h,s}^\eps(\gammab)} (\lambda I - \hat{\Lcal}^{\eps}(\gammab))^{-1} e^{\lambda t} \dd\lambda, \\
    S_{h,u}^\eps(t;\gammab) &:= \dfrac{1}{2\pi} \int_{J_{h,u}^\eps(\gammab)} (\lambda I - \hat{\Lcal}^{\eps}(\gammab))^{-1} e^{\lambda t} \dd\lambda,
\end{split}
\end{equation}
where $J_{mc}$ is the positively oriented Jordan curve introduced before \eqref{eq:def-projections} and $J_{j,s}(\gammab)$ and $J_{j,u}(\gammab)$ are positively oriented Jordan curves around the stable and unstable eigenvalues of $\Pcal_{j}^\eps\hat{\Lcal}^{\eps}(\gammab)$ for $j = lc,h$. We then obtain the corresponding semigroups for $\hat{\Lcal}^{\eps}(\gammab) - i\d \cdot \gammab$ via $\tilde{S}_j^\eps(t,\gammab) = S_j^\eps(t,\gammab) e^{-i \d \cdot \gammab t}$ for $j = mc, lc, h$. Note that we obtain the full semigroups $S^{\eps}_{mc}$, $S^{\eps}_{lc,s}$, $S^{\eps}_{lc,u}$, $S^{\eps}_{h,s}$ and $S^{\eps}_{h,u}$ by taking the direct sum over $\gammab \in \Gamma$. Equivalently, one could also construct the full semigroups as in \eqref{eq:projected-semigroups} by replacing $\hat{\Lcal}^{\eps}(\gammab)$ by $\hat{\Lcal}^{\eps}$ and integrating over Jordan curves $J_{mc}$, $J_{lc,s}^\eps$, $J_{lc,u}^\eps$, $J_{h,s}^\eps$, $J_{h,u}^\eps$ which enclose the more-central, stable less-central, unstable less-central, stable hyperbolic, and unstable hyperbolic part of the spectrum of $\hat{\Lcal}^{\eps}$. Furthermore, as the full semigroups operate diagonally on the Fourier modes, we can restrict them to any subset of Fourier modes. Specifically, $\|\tilde{S}^\eps(t)\|_{H^\ell_{S_{[j,j+1)}} \to H^{\ell+3}_{S_{[j,j+1)}}}$ denotes the operator norm of the semigroup restricted to Fourier modes in $S_{[j,j+1)}$.

We introduce a smooth cut-off function $\chi_r : \R \mapsto [0,1]$ for $r > 0$, which is given by $\chi_r(x) =1$ if $|x| \leq r$, $\chi_r(x) = 0$ if $|x| \geq 2r$ and $\chi_r(x) \in (0,1)$ for $r < |x| < 2r$. We then define the cut-off nonlinearities $\shortunderline{\Ncal}(\,\cdot\,) = \Ncal(\,\cdot\,\chi_r(\|\cdot\|_{\Zcal}))$, which are globally Lipschitz continuous, and consider the system
\begin{equation}\label{eq:spatial-dynamics-after-normal-form-and-cut-off}
    \begin{split}
        \partial_{\xi} \hat{V}_{mc}(\cdot, \gammab)  & = \tilde{\Lcal}_{mc}^\eps(\gammab) \hat{V}_{mc}(\cdot, \gammab) + \shortunderline{\hat{\Ncal}}_{mc} (\hat{V}_{mc},\hat{V}_{lc} + \hat{V}_{h};\gammab,\vartheta), \\
        \partial_{\xi} \hat{V}_{lc,s}(\cdot, \gammab)  & = \tilde{\Lcal}_{lc,s}^\eps(\gammab) \hat{V}_{lc,s}(\cdot,\gammab) + \shortunderline{\hat{\Ncal}}_{lc,s} (\hat{V}_{mc},\hat{V}_{lc} + \hat{V}_{h};\gammab,\vartheta), \\
        \partial_{\xi} \hat{V}_{lc,u}(\cdot, \gammab)  & = \tilde{\Lcal}_{lc,u}^\eps(\gammab) \hat{V}_{lc,u}(\cdot,\gammab) + \shortunderline{\hat{\Ncal}}_{lc,u} (\hat{V}_{mc},\hat{V}_{lc} + \hat{V}_{h};\gammab,\vartheta), \\
        \partial_{\xi} \hat{V}_{h,s}(\cdot, \gammab)  & = \tilde{\Lcal}_{h,s}^\eps(\gammab) \hat{V}_{h,s}(\cdot,\gammab) + \shortunderline{\hat{\Ncal}}_{h,s} (\hat{V}_{mc}, \hat{V}_{lc} + \hat{V}_{h};\gammab,\vartheta), \\
        \partial_{\xi} \hat{V}_{h,u}(\cdot, \gammab)  & = \tilde{\Lcal}_{h,u}^\eps(\gammab) \hat{V}_{h,u}(\cdot,\gammab) + \shortunderline{\hat{\Ncal}}_{h,u} (\hat{V}_{mc}, \hat{V}_{lc} + \hat{V}_{h};\gammab,\vartheta).
    \end{split}
\end{equation}
Here, $\Pcal_j^\eps = \Pcal_{j,s}^\eps + \Pcal_{j,u}$ is the decomposition into the projections onto the stable and unstable eigenvalues, and we decompose $\tilde{\Lcal}_{j}^\eps$ into $\tilde{\Lcal}_{j,s}^\eps$ and $\tilde{\Lcal}_{j,u}^\eps$ accordingly.
Following the standard construction of a centre manifold, we then show that the variation-of-constants formula
\begin{equation}\label{eq:spatial-dynamics-after-normal-form-and-cut-off-variation-of-constants}
    \begin{split}
        \hat{V}_{mc}(t;\gammab) & = \tilde{S}_{mc}^\eps(t;\gammab) V_{mc}(0;\gammab) + \int_0^t \tilde{S}_{mc}^\eps(t - \tau;\gammab) \shortunderline{\hat{\Ncal}}_{mc}(\hat{V}_{mc}(\tau;\gammab),\hat{V}_{lc}(\tau;\gammab) + \hat{V}_{h}(\tau;\gammab);\gammab,\vartheta) \dd \tau,\\
        \hat{V}_{lc,s}(t;\gammab) & = \int_{-\infty}^t \tilde{S}_{lc,s}^\eps(t - \tau;\gammab) \shortunderline{\hat{\Ncal}}_{lc,s}(\hat{V}_{mc}(\tau;\gammab),\hat{V}_{lc}(\tau;\gammab) + \hat{V}_{h}(\tau;\gammab);\gammab,\vartheta) \dd \tau, \\
        \hat{V}_{lc,u}(t;\gammab) & = - \int_t^{\infty} \tilde{S}_{lc,u}^\eps(t - \tau;\gammab) \shortunderline{\hat{\Ncal}}_{lc,u}(\hat{V}_{mc}(\tau;\gammab),\hat{V}_{lc}(\tau;\gammab) + \hat{V}_{h}(\tau;\gammab);\gammab,\vartheta) \dd \tau, \\
        \hat{V}_{h,s}(t;\gammab) & = \int_{-\infty}^t \tilde{S}_{h,s}^\eps(t - \tau;\gammab) \shortunderline{\hat{\Ncal}}_{h,s}(\hat{V}_{mc}(\tau;\gammab),\hat{V}_{lc}(\tau;\gammab) + \hat{V}_{h}(\tau;\gammab);\gammab,\vartheta) \dd \tau, \\
        \hat{V}_{h,u}(t;\gammab) & = - \int_t^\infty \tilde{S}_{h,u}^\eps(t - \tau;\gammab) \shortunderline{\hat{\Ncal}}_{h,u}(\hat{V}_{mc}(\tau;\gammab),\hat{V}_{lc}(\tau;\gammab) + \hat{V}_{h}(\tau;\gammab);\gammab,\vartheta) \dd \tau
    \end{split}
\end{equation}
has a fixed point in the weighted space
\begin{equation}\label{eq:weighted_space}
    \Zcal_\eta := \{V \in C^0(\R,\Zcal) \,:\, \|V\|_{\Zcal_\eta} := \sup_{t\in \R} e^{-\eta|t|} \|V(t)\|_{\Zcal} < \infty\}
\end{equation}
with a weight $\eta > 0$ chosen in the proof of \Cref{thm:centre-manifold}, which separates the more-central modes on the one hand from the less-central and hyperbolic modes on the other hand, so that the only remaining time scales relevant for the dynamics originate from the more-central modes. Similarly, we define the weighted space $\Ycal_\eta$.

We now establish bounds for the semigroups defined in \eqref{eq:projected-semigroups}. For this, we first prove estimates for the semigroups on each strip $S_{[j,j+1)}$ in \Cref{lem:semigroup-bounds} and then extend these bounds to the full semigroups in \Cref{lem:semigroup-bounds-full}.

\begin{figure}[H]
    \centering
    \begin{subfigure}[b]{0.45\textwidth}
        \includegraphics[width = \linewidth]{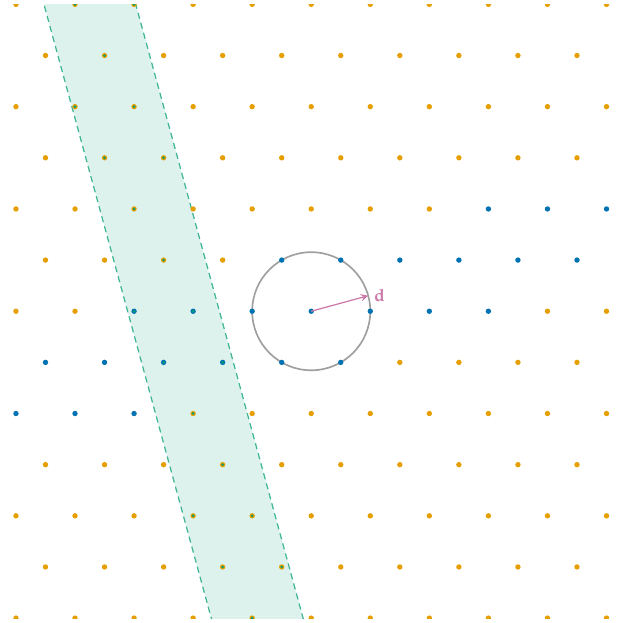}
        \subcaption{}
        \label{subfig:guido-1}
    \end{subfigure}
    \hfill
    \begin{subfigure}[b]{0.45\textwidth}
        \includegraphics[width = \linewidth]{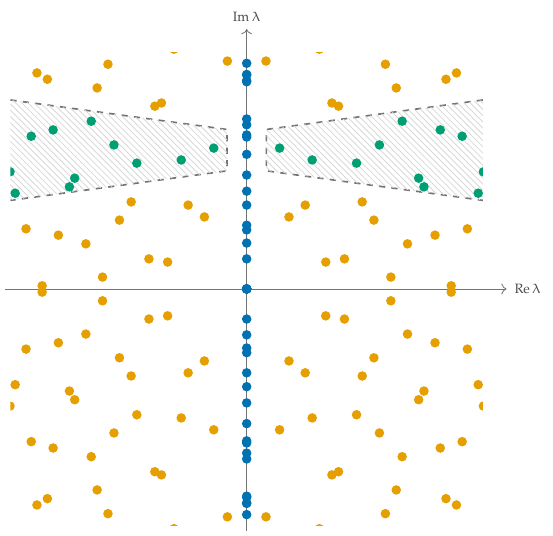}
        \subcaption{}
        \label{subfig:guido-2}
    \end{subfigure}    
    \caption{Panel \sref{subfig:guido-1} depicts a strip $S_{[j,j+1)}$ in Fourier space and panel \sref{subfig:guido-2} shows the corresponding hyperbolic eigenvalues of $\tilde{\Lcal}^0$, which are located in a bisectorial region of the complex plane.}
    \label{fig:spectral-slice}
\end{figure}

\begin{lemma}\label{lem:semigroup-bounds}
    Let $\ell > 0$ and $c_1 > 0$ be sufficiently small as in \Cref{lem:polynom-roots-1}. Then, for $\theta \neq \tfrac{\pi}{6}$ with $\cot(\theta) \in \sqrt{3}\Q$ exists a $\eps_0 > 0$ and constants $C_1, C_2, C_3, C_4 > 0$ such that the following estimates hold for all $\eps\in (0,\eps_0)$ and $j\in \Z$.
    \begin{enumerate}[label=(\alph*), ref=\thelemma(\alph*)]
        \item\label[lemma]{it:semigroup-bounds-1} The semigroup corresponding to the more-central eigenvalues satisfies
    \begin{equation*}
        \begin{split}
            \|\tilde{S}^\eps_{mc}(t)\|_{H^\ell_{S_{[j,j+1)}} \to H^{\ell+3}_{S_{[j,j+1)}}} & \leq C_1 \max(1,|t|) e^{C_2 \eps |t|}, \quad \text{for all } t \in \R, \\
        \end{split}
    \end{equation*}
        \item\label[lemma]{it:semigroup-bounds-2} The semigroups corresponding to the less-central eigenvalues that are, for given $\eps>0$, still $\sqrt{\eps}$-close to the imaginary axis satisfy
    \begin{equation*}
        \begin{split}
            \|\tilde{S}^\eps_{lc,s}(t)\|_{H^\ell_{S_{[j,j+1)}} \to H^{\ell+3}_{S_{[j,j+1)}}} & \leq \dfrac{C_1}{\sqrt{\eps}} e^{-C_3 \sqrt{\eps} t}, \quad \text{for all } t > 0 \text{ and } \eps |j| \leq c_1,  \\
            \|\tilde{S}^\eps_{lc,u}(t)\|_{H^\ell_{S_{[j,j+1)}} \to H^{\ell+3}_{S_{[j,j+1)}}} & \leq \dfrac{C_1}{\sqrt{\eps}} e^{C_3 \sqrt{\eps} t}, \quad \text{for all } t < 0 \text{ and } \eps |j| \leq c_1, 
        \end{split}
    \end{equation*}
    \item\label[lemma]{it:semigroup-bounds-3} The semigroups corresponding to the remaining less-central eigenvalues $\eqref{eq:projected-semigroups}_{2,3}$ and to the hyperbolic eigenvalues satisfy
    \begin{equation*}
        \begin{split}
            \|\tilde{S}^\eps_{lc,s}(t)\|_{H^\ell_{S_{[j,j+1)}} \to H^{\ell+3}_{S_{[j,j+1)}}}  & \leq C_1 e^{-C_4 t}, \quad \text{for all } t > 0 \text{ and } \eps |j| > c_1,  \\
            \|\tilde{S}^\eps_{lc,u}(t)\|_{H^\ell_{S_{[j,j+1)}} \to H^{\ell+3}_{S_{[j,j+1)}}} & \leq C_1 e^{C_4 t}, \quad \text{for all } t < 0 \text{ and } \eps |j| > c_1, \\
            \|\tilde{S}^\eps_{h,s}(t)\|_{H^\ell_{S_{[j,j+1)}} \to H^{\ell+3}_{S_{[j,j+1)}}} & \leq C_1 e^{-C_4 t}, \quad \text{for all } t > 0,  \\
            \|\tilde{S}^\eps_{h,u}(t)\|_{H^\ell_{S_{[j,j+1)}} \to H^{\ell+3}_{S_{[j,j+1)}}} & \leq C_1 e^{C_4 t}, \quad \text{for all } t < 0, 
        \end{split}
    \end{equation*}
    \end{enumerate}
\end{lemma}

\begin{proof}
    We first argue that there is $\eps_0>0$ so that $\hat{\Lcal}^{\eps}$ is a bisectorial operator for all $\eps \in (0,\eps_0)$ and the sector does not depend on $\eps$. Observe that $\hat{\Lcal}^0$ is bisectorial since its eigenvalues are given by $\pm \sqrt{|\d^\perp \cdot \gammab|^2 - 1}$ and therefore lie in a strip of width two around the real axis. Next, we recall the spectrum of $\hat{\Lcal}^{\eps}$ consists of the solutions to 
    \begin{equation*}
        g(\tilde{\lambda}) = -(\tilde{\lambda} + \sqrt{|\d^\perp \cdot \gammab|^2 - 1})^2(\tilde{\lambda} - \sqrt{|\d^\perp \cdot \gammab|^2 - 1})^2 + \eps^2 \mu_0 + \eps c_0 \tilde{\lambda} - i \eps c_0 (\d \cdot \gammab) = 0.
    \end{equation*}
    In particular, the $\eps$-perturbation does not depend on $\d^\perp \cdot \gammab$. Therefore, it is sufficient to consider the asymptotic behaviour of the eigenvalues for $|\d \cdot \gammab|$ large. Using the implicit function theorem, we find that
    \begin{equation}\label{eq:eigenval-expansion-at-infinity}
        \tilde{\lambda}_{1,2} = \pm \eps^{\tfrac{1}{4}} (\d\cdot \gammab)^{\tfrac{1}{4}} i^{\tfrac{3}{4}} + \Ocal(|\d\cdot \gammab|^\alpha), \quad \tilde{\lambda}_{3,4} = \pm \eps^{\tfrac{1}{4}} (\d\cdot \gammab)^{\tfrac{1}{4}} i^{\tfrac{7}{4}} + \Ocal(|\d\cdot \gammab|^\alpha),
    \end{equation}
    and hence, we find uniform sectors. In particular, $S^{\eps}_{mc}$, $S^{\eps}_{lc,s}$, $S^{\eps}_{lc,u}$, $S^{\eps}_{h,s}$ and $S^{\eps}_{h,u}$ are analytic semigroups. Note that the function spaces $H^\ell_{S_{[j,j+1)}}$ are invariant under the semigroups since they are invariant under the generator.

    Using \Cref{lem:polynom-roots-1,lem:polynom-roots-2}, we find that the eigenvalues of $\hat{\Lcal}^{\eps}_{lc}|_{S_{[j,j+1)}}$ for $\eps|j| > c_1$ and the hyperbolic eigenvalues are uniformly bounded away from the imaginary axis. By standard estimates for analytic semigroups, see e.g.~\cite{lunardi1995book}, this shows the estimates in \cref{it:semigroup-bounds-3} for $S^\eps$, the semigroup generated by $\hat{\Lcal}^{\eps}$. 

    We now prove \Cref{it:semigroup-bounds-1} and \Cref{it:semigroup-bounds-2} for $S^\eps$. To see this, note that on each strip $S_{[j,j+1)}$ there are only finitely many Fourier modes corresponding to more-central and less-central eigenvalues. Hence, it is sufficient to prove the estimates for a single Fourier mode.

    For \Cref{it:semigroup-bounds-1}, we note that, by \Cref{lem:estimate-projections}, the curves $J^\eps_{mc}(\gammab)$ can be chosen independently of $\eps$. Next, we recall that the more-central eigenvalues for $\eps = 0$ lie at $\lambda = 0$ and form a family of Jordan blocks of size two provided $\theta \neq \tfrac{\pi}{6}$.
    \begin{equation*}
    \begin{split}
        \|S^\eps_{mc}(t;\gammab)\| & = \left\| \int_{J_{mc}(\gammab)} \frac{\operatorname{adj}(\lambda- \hat{\Lcal}_{mc}^{\eps}(\gammab))}{\det(\lambda- \hat{\Lcal}_{mc}^{\eps}(\gammab))} e^{\lambda t} \de \lambda  \right\|_{H^\ell_\Gamma \to H^{\ell + 3}_\Gamma} \\
        & = \left\| \int_{J_{mc}(\gammab)} \frac{\operatorname{adj}(\lambda- \hat{\Lcal}_{mc}^{\eps}(\gammab))}{(\lambda - \lambda_1(\eps)) (\lambda-\lambda_2(\eps))} e^{\lambda t} \de \lambda  \right\|_{H^\ell_\Gamma \to H^{\ell + 3}_\Gamma} \\
        & = \left| \operatorname{Res}_{\lambda = \lambda_1(\eps)} \frac{\operatorname{adj}(\lambda- \hat{\Lcal}_{mc}^{\eps}(\gammab))}{(\lambda - \lambda_1(\eps)) (\lambda-\lambda_2(\eps))} e^{\lambda t}  + \operatorname{Res}_{\lambda = \lambda_2(\eps)} \frac{\operatorname{adj}(\lambda- \hat{\Lcal}_{mc}^{\eps}(\gammab))}{(\lambda - \lambda_1(\eps)) (\lambda-\lambda_2(\eps))} e^{\lambda t}  \right| \\
        & \leq C \left| \dfrac{e^{\lambda_1(\eps)t} - e^{\lambda_2(\eps)t}}{\lambda_1(\eps) - \lambda_2(\eps)} \right| + \left|\frac{\operatorname{adj}(\lambda_1(\eps)- \hat{\Lcal}_{mc}^{\eps}(\gammab)) - \operatorname{adj}(\lambda_2(\eps)- \hat{\Lcal}_{mc}^{\eps}(\gammab))}{\lambda_1(\eps) - \lambda_2(\eps)} e^{\lambda_1(\eps)t}\right|  \\
        & \leq C \left| \dfrac{e^{\eps \nu_+(\gammab) t} - e^{\eps \nu_-(\gammab)t}}{\eps (\nu_+(\gammab) - \nu_-(\gammab))} \right| + C e^{C_2 \eps|t|} \\
        & \leq C_1 \max(1,|t|) e^{C_2 \eps |t|},
    \end{split}
    \end{equation*}
    where we use that $\lambda_j(\eps) - \hat{\Lcal}_{mc}^{\eps}(\gammab)$ is bounded for $\eps\in (0,\eps_0)$, $\operatorname{adj}(\cdot)$ is locally Lipschitz. Additionally, we use the estimate
    \begin{equation*}
        \left| \dfrac{e^{\eps \nu_+(\gammab) t} - e^{\eps \nu_-(\gammab)t}}{\eps (\nu_+(\gammab) - \nu_-(\gammab))} \right| \leq C_1 |t| e^{C_2\eps |t|},
    \end{equation*}
    which we obtain by recalling \Cref{prop:spectrum-eps-positive}.

    For \Cref{it:semigroup-bounds-2}, we follow the same calculations as in the first part. The only difference is that since $J^\eps_{lc,s}(\gammab)$ and $J^\eps_{lc,u}(\gammab)$ are restricted to the stable and unstable eigenvalues, respectively, the sum of residues only contains singularities at stable or unstable eigenvalues, respectively. Therefore, we obtain
    \begin{equation*}
    \begin{split}
        \|S^\eps_{lc,s}(t;\gammab)\| & = \left\| \int_{J_{lc,s}(\gammab)} \frac{\operatorname{adj}(\lambda- \hat{\Lcal}_{lc,s}^{\eps}(\gammab))}{\det(\lambda- \hat{\Lcal}_{lc,s}^{\eps}(\gammab))} e^{\lambda t} \de \lambda  \right\| \\
        &\leq \left| \operatorname{Res}_{\lambda = \lambda_1(\eps)} \frac{\operatorname{adj}(\lambda- \hat{\Lcal}_{mc}^{\eps}(\gammab))}{(\lambda - \lambda_1(\eps)) (\lambda-\lambda_2(\eps))} e^{\lambda t}\right| \leq \dfrac{C_1}{\sqrt{\eps}} e^{-C_3 \sqrt{\eps}t},
    \end{split}
    \end{equation*}
    where we recall from \eqref{eq:expansion-imaginary-less-central} that the real parts of $\lambda_1(\eps)$ and $\lambda_2(\eps)$ have different signs.

    It remains to prove that the same estimates also hold for $\tilde{S}^\eps$. For this, we use perturbation theory for sectorial operators. Note that on each strip $S_{[j,j+1)}$, the operator generated by the Fourier multiplier $i \d \cdot \gammab$ is a bounded operator, which operates diagonally on Fourier modes. Therefore, on each strip $\tilde{\Lcal}^{\eps}$ is still bisectorial, see \Cref{fig:spectral-slice}. \cite[Prop.~2.4.1]{lunardi1995book} and therefore, $\tilde{S}^{\eps}_{mc}$, $\tilde{S}^{\eps}_{lc,s}$, $\tilde{S}^{\eps}_{lc,u}$, $\tilde{S}^{\eps}_{h,s}$ and $\tilde{S}^{\eps}_{h,u}$ are also analytic semigroups. In addition, the constants in the estimates do not change since the perturbation is purely imaginary. In particular, the constants do not depend on the strip. This concludes the proof.
\end{proof}

\begin{lemma}\label{lem:semigroup-bounds-full}
    Let $\ell_1, \ell_2 > 0$. Then, for $\theta \neq \tfrac{\pi}{6}$ with $\cot(\theta) \in \sqrt{3}\Q$ there exists an $\eps_0 > 0$ and constants $C_1, C_2, C_3, C_4 > 0$ such that the following estimates hold for all $\eps\in (0,\eps_0)$.
    \begin{enumerate}[label=(\alph*), ref=\thelemma(\alph*)]
        \item\label[lemma]{it:semigroup-bounds-1-full} The semigroup corresponding to the more-central eigenvalues satisfies
    \begin{equation*}
        \begin{split}
            \|\tilde{S}^\eps_{mc}(t)\|_{\Ycal \to \Zcal} & \leq C_1 \max(1,|t|) e^{C_2 \eps |t|}, \quad \text{for all } t \in \R, \\
        \end{split}
    \end{equation*}
        \item\label[lemma]{it:semigroup-bounds-2-full} The semigroups corresponding to the less-central eigenvalues satisfy
    \begin{equation*}
        \begin{split}
            \|\tilde{S}^\eps_{lc,s}(t)\|_{\Ycal \to \Zcal} & \leq \dfrac{C_1}{\sqrt{\eps}} e^{-C_3 \sqrt{\eps} t}, \quad \text{for all } t > 0,  \\
            \|\tilde{S}^\eps_{lc,u}(t)\|_{\Ycal \to \Zcal} & \leq \dfrac{C_1}{\sqrt{\eps}} e^{C_3 \sqrt{\eps} t}, \quad \text{for all } t < 0, 
        \end{split}
    \end{equation*}
    \item\label[lemma]{it:semigroup-bounds-3-full} The semigroups corresponding to the hyperbolic eigenvalues satisfy
    \begin{equation*}
        \begin{split}
            \|\tilde{S}^\eps_{h,s}(t)\|_{\Ycal \to \Zcal} & \leq C_1 e^{-C_4 t}, \quad \text{for all } t > 0,  \\
            \|\tilde{S}^\eps_{h,u}(t)\|_{\Ycal \to \Zcal} & \leq C_1 e^{C_4 t}, \quad \text{for all } t < 0, 
        \end{split}
    \end{equation*}
    \end{enumerate}
\end{lemma}

\begin{proof}
    The estimates follow directly from the previous \Cref{lem:semigroup-bounds} with the following calculation
    \begin{equation*}
    \begin{split}
        \|\tilde{S}_{mc}^\eps(t)\hat{V}\|_{\Zcal} & = \left(\sum_{j \in \Z} (1+j^2)^{\ell_1} \|\tilde{S}^\eps_{mc}(t)\hat{V}\|_{H^{\ell_2+3}_{S_{[j,j+1)}}}\right)^{\tfrac{1}{2}} \\
        & \leq \sup_{j \in \Z} \|\tilde{S}^\eps_{mc}(t)\|_{H^{\ell_2}_{S_{[j,j+1)}} \to H^{\ell_2+3}_{S_{[j,j+1)}}}  \left(\sum_{j \in \Z} (1+j^2)^{\ell_1} \|\hat{V}\|_{H^{\ell_2}_{S_{[j,j+1)}}}\right)^{\tfrac{1}{2}} \\
        &\leq C_1 \max(1,|t|) e^{C_2 \eps |t|} \|\hat{V}\|_{\Ycal}.
    \end{split}
    \end{equation*}
    The remaining estimates follow the same pattern.
\end{proof}

We now state the main theorem on the existence of a centre manifold.

\begin{theorem}\label{thm:centre-manifold}
    Let $\ell_1, \ell_2 \in \N$ and $\theta \neq \tfrac{\pi}{6}$ with $\cot(\theta) \in \sqrt{3}\Q$. For any $\delta \in (0,\tfrac{1}{4})$ there exists an $\eps_0 > 0$ such that for every $\eps \in (0,\eps_0)$ there exists a neighbourhood $O_{mc}^{\eps} \subset \Zcal_{mc}^\eps$ of the origin and a map $\Psi^{\eps} : O_{mc}^{\eps} \to (I - \Pcal_{mc}^{\eps})\Zcal$, which is at least quadratic, such that the following holds.
    \begin{enumerate}[label=(\alph*), ref=\thetheorem(\alph*)]
        \item The neighbourhood $O_{mc}^{\eps}$ is of size $\Ocal(\eps^{3/4 + \delta})$.
        \item The centre manifold
        \begin{equation*}
            \Mcal_{mc}^{\eps} = \{W = W_{mc} + \Psi^{\eps}(W_{mc}) \, : \, W_{mc} \in O_{mc}^{\eps} \}
        \end{equation*}
        contains all small bounded solutions to \eqref{eq:spat-dyn-Fourier}.
        \item Solutions to the reduced equation
        \begin{equation}\label{eq:reduced-equation-CM-theorem}
            \partial_{\xi} W_{mc} = \Lcal^{\eps} W_{mc}+ \Pcal_{mc}^{\eps} \Ncal(W_{mc} + \Psi^{\eps}(W_{mc});\vartheta)
        \end{equation}
        give rise to solutions $W$ to the full system \eqref{eq:spat-dyn-Fourier} via $W=W_{mc} + \Psi^{\eps}(W_{mc})$. 
        \item The symmetries of the system \eqref{eq:Swift-Hohenberg} are preserved by the reduction function $\Psi$.
    \end{enumerate}
\end{theorem}

\begin{proof}
    The proof is organised as follows. First, we argue that the system \eqref{eq:spatial-dynamics-after-normal-form-and-cut-off} has a centre manifold of size $\Ocal(1)$ as $\eps \to 0$. For this, we solve the fixed-point problem \eqref{eq:spatial-dynamics-after-normal-form-and-cut-off-variation-of-constants} in the weighted space $\Zcal_{\eta}$ defined in \eqref{eq:weighted_space} by the standard contraction mapping theorem. By reverting the rescaling \eqref{eq:rescaling} and the normal-form transform \eqref{eq:normal-form-lc}, we then obtain a centre manifold for \eqref{eq:spat-dyn-Fourier} with the claimed $\eps$-scaling.

    We first choose $\eta = \tfrac{\tilde{C}}{2} \sqrt{\eps}$, where $\tilde{C}$ is the minimum of $C_2, C_3, C_4$ in \Cref{lem:semigroup-bounds-full}. Then, for the self-mapping property and the Lipschitz estimate, we observe that for any nonlinear function $F \colon \Zcal \to \Ycal$ the following estimates hold:
    \begin{equation*}
    \begin{split}
        \left\| \int_0^t \tilde{S}_{mc}^\eps(t - \tau) F(V(\tau)) \de \tau \right\|_{\Zcal_{\eta}}  & \leq \sup_{t\in \R} e^{-\eta|t|} \int_0^t \|\tilde{S}_{mc}^\eps(t - \tau)\|_{\Ycal\to \Zcal} \|F(V(\tau))\|_{\Ycal} \de \tau \\
        & \leq \sup_{t\in \R} e^{-\eta|t|} \int_0^t C_1 \max(1,|t-\tau|) e^{C_2 \eps|t-\tau|} e^{\eta \tau} \de \tau  \|F(V)\|_{\Ycal_{\eta}} \\
        & \leq C \eps^{-1} \|F(V)\|_{\Ycal_{\eta}}, \\
        \left\| \int_0^t \tilde{S}_{lc,s}^\eps(t - \tau) F(V(\tau)) \de \tau \right\|_{\Zcal_{\eta}}  &  \leq \sup_{t\in \R} e^{-\eta|t|} \int_0^t \frac{C_1}{\sqrt{\eps}} e^{C_2 \sqrt{\eps}|t-\tau|} e^{\eta \tau} \de \tau  \|F(V)\|_{\Ycal_{\eta}} \\
        & \leq C \eps^{-1} \|F(V)\|_{\Ycal_{\eta}}, \\
        \left\| \int_{-\infty}^{t} \tilde{S}_{lc,u}^\eps(t - \tau) F(V(\tau)) \de \tau \right\|_{\Zcal_{\eta}}  &  \leq \sup_{t\in \R} e^{-\eta|t|} \int_{-\infty}^{t} \frac{C_1}{\sqrt{\eps}} e^{C_2 \sqrt{\eps}|t-\tau|} e^{\eta \tau} \de \tau  \|F(V)\|_{\Ycal_{\eta}} \\
        & \leq C \eps^{-1} \|F(V)\|_{\Ycal_{\eta}}, \\
        \left\| \int_0^t \tilde{S}_{h,s}^\eps(t - \tau) F(V(\tau)) \de \tau \right\|_{\Zcal_{\eta}}  &  \leq C \|F(V)\|_{\Ycal_{\eta}},\\
        \left\| \int_{-\infty}^{t} \tilde{S}_{h,u}^\eps(t - \tau) F(V(\tau)) \de \tau \right\|_{\Zcal_{\eta}}  & \leq C \|F(V)\|_{\Ycal_{\eta}}.
    \end{split}
    \end{equation*}
    Here, we use the semigroup estimates in \Cref{lem:semigroup-bounds-full}, $\int_0^\infty t e^{-\sqrt{\eps} t} \de t = \Ocal(\eps^{-1})$, and $0<\eta < \min(C_2,C_3,C_4) \sqrt{\eps}$. With this choice of $\eta$, we also get that $\tilde{S}_{mc}^\eps(t) V_{mc}(0) \in \Zcal_{\eta}$.
    
    Next, we note that the cut-off nonlinearities $\shortunderline{\Ncal}$ map from $\Zcal$ into $\Ycal$ and are globally Lipschitz with the same Lipschitz constant as in \Cref{lem:Lipschitz}. Therefore, we find that \eqref{eq:spatial-dynamics-after-normal-form-and-cut-off-variation-of-constants} is a self-map on $\Zcal_{\eta}$ and is a contraction if
    \begin{equation*}
        C \eps^{\min(\kappa_{mc} - 1, \kappa_{lc} - 1, \kappa_h)} < 1.
    \end{equation*}
    This is satisfied for $\eps$ sufficiently small if $\min(\kappa_{mc} - 1, \kappa_{lc} - 1, \kappa_h) > 0$, which is satisfied for the choice
    \begin{equation}\label{eq:beta-gamma-final}
        \beta = \dfrac{3}{4} + \delta, \quad \gamma = 1 + \delta
    \end{equation}
    for $\delta \in (0,\tfrac{1}{4})$. Since this does not impose a restriction on the cut-off radius $r$, we obtain a centre manifold of size $\Ocal(1)$ for the system \eqref{eq:spatial-dynamics-after-normal-form-and-cut-off} by following the standard proof for a centre manifold theorem \cite{haragus2011book}. In particular, the reduction map $\Psi_V$ is at least quadratic in its arguments.

    Next, we reverse the normal-form transform \eqref{eq:normal-form-lc} to obtain an $\Ocal(1)$ centre manifold for \eqref{eq:spatial-dynamics-split}. Note that since $\beta > \tfrac{1}{2}$, the normal-form transformation \eqref{eq:normal-form-W-to-V} is a near-identity transformation and thus invertible. In particular, we point out that the normal-form transformation acts only on the less-central part, leaving the more-central part invariant. Finally, we revert the rescaling \eqref{eq:rescaling} to obtain a centre manifold of size $\Ocal(\eps^\beta) = \Ocal(\eps^{3/4 + \delta})$ for the system \eqref{eq:spat-dyn-Fourier} with reduction map $\Psi$. This concludes the proof.
\end{proof}

\begin{remark}
    The restriction $\delta < \tfrac{1}{4}$ in the centre manifold theorem \ref{thm:centre-manifold} in particular guarantees that solutions of size $\eps$ are contained in the centre manifold $\Mcal_{mc}$.
\end{remark}

\subsection{The reduced equations}\label{sec:reduced-equations} 

We now derive the reduced equations on the centre manifold. We will postpone the discussion of the dynamics of the reduced equations to \Cref{sec:interfaces} after we have discussed the existence of a centre manifold and the derivation of the reduced equation for the special case $\theta = \frac{\pi}{6}$ in \Cref{sec:special-case}. Similarly, we assume that $c_0 \neq c_{\crit}(\k_j)$, which guarantees that $\lambda_{+}^{\eps}(\k_j) - \lambda_-^\eps(\k_j) = \Ocal(\eps)$, see \Cref{prop:spectrum-eps-positive}. The case where $c_0 = c_{\crit}(\k_j)$ for some $j = 1,2,3$ is then discussed in \Cref{sec:reduced-equations-critical}. The derivation follows the strategy in \cite{hilder2025-08JNonlinearSci}.

To derive a reduced equation on the centre manifold, we introduce the slow spatial scale $\Xi = \eps \xi$ and make the ansatz
\begin{equation}\label{eq:ansatz-reduced-equations}
    W = W_{mc} + \Psi(W_{mc}) =  \eps \sum_{j = 1}^3 (\tilde{A}_{j,+}(\Xi) \hat{\phib}_+^{\eps}(\k_j) + \tilde{A}_{j,-}(\Xi) \hat{\phib}_-^{\eps}(\k_j))e^{i \k_j \cdot \p} + c.c. + \eps^2 \tilde{\Psi}(\tilde{A}_{1,\pm},\tilde{A}_{2,\pm},\tilde{A}_{3,\pm};\eps),
\end{equation}
where $\tilde{A}_j(\Xi) \in \C$ and $\hat{\phib}_{\pm}^{\eps}(\k_j) \in \C^4$ are the eigenvectors of $\Lcal^{\eps}(\k_j)$ that belong to the more-central eigenvalues $\lambda_{\pm}^{\eps}(\k_j)$ and are normalised in the sense that their first component is equal to one. Since the reduction map $\Psi$ in \Cref{thm:centre-manifold} is at least quadratic, we find that $\tilde{\Psi}(A_{j,\pm};\eps) = \Ocal(1)$. 

To determine the reduced equations, it is necessary to obtain the leading-order contributions of $\tilde{\Psi}$. Following \cite[Corollary 2.12]{haragus2011book} and using that $\xi$-derivatives gain one power of $\eps$, we find that 
\begin{equation*}
    (I-\Pcal_{mc}^\eps(\gammab))(\tilde{\Lcal}^{\eps}(\gammab)\eps^2 \tilde{\Psi}_{\gammab} + \hat{\Ncal}(W_{mc} + \Psi; \gammab, \vartheta)) = \Ocal(\eps^3)
\end{equation*}
for all $\gammab \in \Gamma$. To solve this equation, we first point out that $\hat{\Ncal}$ is of the form $(0,0,0,\ast)$. Then, for $\gammab \in \Gamma \setminus \Gamma_0$, we recall that $\Pcal_{mc}^\eps(\gammab) = 0$ and that
\begin{equation*}
    (\hat{\Lcal}^{\eps}(\gammab) - i\d \cdot \gammab) \tilde{\Psi}_{\gammab} = \boldsymbol{f} \quad \text{holds if and only if} \quad \hat{L}^{\eps}(\gammab) (\tilde{\Psi}_{\gammab})_0 = f_3
\end{equation*}
whenever $f = (0,0,0,f_3)$ by explicit computation using \eqref{eq:spat-dyn-Fourier} and using that $\hat{\Lcal}^\eps(\gammab)$ is given by 
\begin{equation}\label{eq:spat-dyn-linear}
    \hat{\Lcal}^\eps(\gammab) = \begin{pmatrix}
        0 & 1 & 0 & 0 \\
        0 & 0 & 1 & 0 \\
        0 & 0 & 0 & 1 \\
        a_0 & a_1 & a_2 & a_3
    \end{pmatrix},
\end{equation}
where the matrix-coefficients $a_j$ for $j = 0,\dots,3$ are given explicitly by
\begin{equation*}
    \begin{split}
        a_0 & = -(1-|\d^{\perp}\cdot\gammab|^2)^2 + \eps^2\mu_0 - i \eps c_0 \d\cdot\gammab, &\qquad  a_1 & =  \eps c_0,\\
        a_2 & = - 2(1-|\d^{\perp}\cdot\gammab|^2), & \qquad a_3 & = 0.
    \end{split}
\end{equation*}
In particular, the other entries of $\Psi_{\gammab}$ are given by $(\Psi_{\gammab})_j = (i \d \cdot \gammab)^j (\Psi_{\gammab})_0$ for $j = 0,1,2,3$. Next, we note that nonlinear terms of order $\eps^2$ can arise only as quadratic combinations of $W_{mc}$ and therefore these terms only appear at Fourier modes with $|\gammab| \leq 2$. For these terms, we also use that for the new variable $A_j = \tilde{A}_{j,+} + \tilde{A}_{j,-}$ it holds
\begin{equation}\label{eq:N2-spatial-to-physical}
\begin{split}
    \hat{\Ncal}_2(W_{mc};\gammab,\vartheta)_3 & = \eps^2\hat{N}_2(\sum_{j = 1}^3 A_{j} e^{i\k_j \cdot \p} + c.c.;\gammab,\vartheta) + \Ocal(\eps^3) \\
    &= \eps^2 \sum_{\substack{j,\ell=-3 \\ j,\ell \neq 0}}^{3} \delta_{\k_j+\k_\ell = \gammab}   A_{j} A_{\ell}\hat{N}_2(e^{i\k_j \cdot \p},e^{i\k_\ell \cdot \p};\gammab,\vartheta) + \Ocal(\eps^3),
\end{split}
\end{equation}
which we obtain from the observation that
\begin{equation*}
\begin{split}
    \hat{\phib}_\pm^\eps(\k_j) & = \bigl((i \d \cdot \k_j + \lambda_\pm^{\eps}(\k_j))^0,(i \d \cdot \k_j + \lambda_\pm^{\eps}(\k_j))^1,(i \d \cdot \k_j + \lambda_\pm^{\eps}(\k_j))^2,(i \d \cdot \k_j + \lambda_\pm^{\eps}(\k_j))^3\bigr)^T \\
    & =  \bigl(1, i \d \cdot \k_j, (i \d \cdot \k_j)^2, (i \d \cdot \k_j)^3\bigr)^T + \Ocal(\eps),
\end{split}
\end{equation*}
see \cite{eckmann1991-02CommunMathPhys}. Therefore, for $\gammab \in \Gamma \setminus \Gamma_0$ with $d(\gammab,0) \leq 2$, we find that the leading-order contribution of $(\tilde{\Psi}_{\gammab})_0$ is given as the solution to
\begin{equation*}
    \hat{L}^0(\gammab) (\tilde{\Psi}_{\gammab})_0 = - \sum_{\substack{j,\ell=-3 \\ j,\ell \neq 0}}^{3} \delta_{\k_j+\k_\ell = \gammab}   A_{j} A_{\ell}\hat{N}_2(e^{i\k_j \cdot \p},e^{i\k_\ell \cdot \p};\gammab,\vartheta).
\end{equation*}
Here, we note that $\hat{L}^0(\gammab) \neq 0$ for $\gammab \in \Gamma \setminus \Gamma_0$. For later use, we thus introduce the notation
\begin{equation*}
\begin{split}
     (\tilde{\Psi}_\zerob)_0 & = - 2 \hat{L}^0(\zerob)^{-1} \sum_{j = 1}^3 |A_j|^2 \hat{N}_2(e^{i \k_j \cdot \p}, e^{-i \k_j \cdot \p}; \zerob, \vartheta) =: 3\nu_0 \sum_{j = 1}^3 |A_j|^2, \\
     (\tilde{\Psi}_{\k_j + \k_\ell})_0 & = - (2- \delta_{j\ell}) \hat{L}^0(\k_j + \k_\ell)^{-1}  \hat{N}_2(e^{i \k_j \cdot \p}, e^{i \k_\ell \cdot \p}; \k_j+\k_{\ell}, \vartheta) =: \nu_{\k_j + \k_\ell} A_j A_\ell \ \text{for } d(\k_j+\k_{\ell},0) = 2.
\end{split}
\end{equation*}

We now consider Fourier modes $\gammab \in \Gamma_0$, where it holds that
\begin{equation*}
    \begin{split}
        (I-\Pcal_{mc}^\eps(\k_j))((\hat{\Lcal}^{\eps}(\k_j) - i \d \cdot \k_j)\eps^2 \tilde{\Psi}_{\k_j} + \hat{\Ncal}_2(W_{mc} + \Psi; \k_j, \vartheta)) & = \Ocal(\eps^3)
    \end{split}
\end{equation*}
for the less-central contributions and
\begin{equation}\label{eq:more-central-equation}
\begin{split}
    &\eps^2\partial_\Xi (\tilde{A}_{j,+}(\Xi)\hat{\phib}_+^{\eps}(\k_j) + \tilde{A}_{j,-}(\Xi) \hat{\phib}_-^{\eps}(\k_j)) \\
    &\qquad = \eps(\hat{\Lcal}^{\eps}(\k_j) - i \d \cdot \k_j) (\tilde{A}_{j,+}(\Xi) \hat{\phib}_+^{\eps}(\k_j) + \tilde{A}_{j,-}(\Xi) \hat{\phib}_-^{\eps}(\k_j)) \\
    & \qquad \quad  + \Pcal_{mc}^\eps(\k_j) \hat{\Ncal}_2(W_{mc} + \Psi; \k_j, \vartheta) + \Pcal_{mc}^\eps(\k_j) \hat{\Ncal}_3(W_{mc} + \Psi; \k_j, \vartheta) +  \Ocal(\eps^4),
\end{split}
\end{equation}
We start by discussing the first equation. For this, we note that due to the resonance of the hexagonal lattice, there are quadratic interactions of the more-central modes that lie again at a Fourier mode in $\Gamma_0$. However, we assume that these nonlinear terms gain an additional order of $\eps$, see \Cref{ass:smallness}. Therefore, we find that $(I - \Pcal_{mc}^\eps) \tilde{\Psi}_{k_j} = \Ocal(\eps)$, and hence, it is irrelevant for the leading-order equation on the centre manifold.

Finally, we derive an equation for the more-central modes. To systematically exploit the Jordan-block structure at $\eps = 0$, we introduce the following basis for the span of $\Pcal_{mc}^\eps$ given by
\begin{equation}\label{eq:clever-basis}
    \phib_+^{\eps}(\k_j) = \dfrac{1}{2}\bigl(\hat{\phib}_+^{\eps}(\k_j) + \hat{\phib}_-^{\eps}(\k_j)\bigr) \quad \text{and} \quad \phib_-^{\eps}(\k_j) =  \dfrac{1}{2\eps}\bigl(\hat{\phib}_+^{\eps}(\k_j) - \hat{\phib}_-^{\eps}(\k_j)\bigr).
\end{equation}
In particular, we point out that $(\phib_+^\eps(\k_j), \phib_-^\eps(\k_j))$ converges as $\eps \to 0$ to $(\phib_+^0(\k_j), \phib_-^0(\k_j))$ where $\phib_+^0(\k_j)$ is the eigenvector corresponding to the zero eigenvalue of $\hat{\Lcal}^0(\k_j)-i\d\cdot\k_j$ with first component normalised to one and $\phib_-^0(\k_j)$ is a generalised eigenvector to the zero eigenvalue with first component equal to zero. We then write
\begin{equation}\label{eq:tildeA-to-other-tilde-A}
    \tilde{A}_{j,+}\hat{\phib}_+^{\eps}(\k_j) + \tilde{A}_{j,-} \hat{\phib}_-^{\eps}(\k_j) = \tilde{A}_j \phib_+^\eps(\k_j) + \eps \tilde{B}_j \phib_-^\eps(\k_j)
\end{equation}
and note that the transformation $(\tilde{A}_{j,+}, \tilde{A}_{j,-}) \mapsto (\tilde{A}_j, \tilde{B}_j)$ is uniformly bounded for all $\eps > 0$. Moreover, we also recover $\tilde{A}_{j,+} + \tilde{A}_{j,-} = \tilde{A}_j$. We now derive evolution equations for $\tilde{A}_j$ and $\tilde{B}_j$ by projecting \eqref{eq:more-central-equation} onto $\phib_+^\eps(\k_j)$ and $\phib_-^\eps(\k_j)$, respectively. For this, we first note that
\begin{equation*}
    \begin{split}
        &(\hat{\Lcal}^{\eps}(\k_j) - i \d \cdot \k_j) (\tilde{A}_j \phib_+^{\eps}(\k_j) + \eps \tilde{B}_j \phib_-^{\eps}(\k_j)) \\
        &\quad= \dfrac{1}{2}(\lambda_+^\eps(\k_j) \hat{\phib}_+^\eps(\k_j) + \lambda_-^\eps(\k_j) \hat{\phib}_-^\eps(\k_j)) \tilde{A}_j + \eps \dfrac{1}{2\eps}(\lambda_+^\eps(\k_j) \hat{\phib}_+^\eps(\k_j) -\lambda_-^\eps(\k_j) \hat{\phib}_-^\eps(\k_j)) \tilde{B}_j \\
        &\quad = \Big(\dfrac{1}{2} (\lambda_+^\eps(\k_j) + \lambda_-^\eps(\k_j)) \phib_+^\eps(\k_j) + \dfrac{\eps}{2} (\lambda_+^\eps(\k_j) - \lambda_-^\eps(\k_j)) \phib_-^\eps(\k_j)\Big) \tilde{A}_j \\
        &\qquad + \Big(\dfrac{1}{2} (\lambda_+^\eps(\k_j) - \lambda_-^\eps(\k_j)) \phib_+^\eps(\k_j) + \dfrac{\eps}{2} (\lambda_+^\eps(\k_j) + \lambda_-^\eps(\k_j)) \phib_-^\eps(\k_j)\Big) \tilde{B}_j \\
        &\quad = \eps \Big( - \dfrac{c_0}{8 (\d \cdot \k_j)^2} \tilde{A}_j + \dfrac{\sqrt{c_0^2 - 16 (\d \cdot \k_j)^2 \mu_0}}{8 (\d \cdot \k_j)^2} \tilde{B}_j\Big) \phib_+^\eps(\k_j) \\
        & \qquad+ \eps^2 \Big( \dfrac{\sqrt{c_0^2 - 16 (\d \cdot \k_j)^2 \mu_0}}{8 (\d \cdot \k_j)^2} \tilde{A}_j - \dfrac{c_0}{8 (\d \cdot \k_j)^2} \tilde{B}_j\Big) \phib_-^\eps(\k_j),
    \end{split}
\end{equation*}
where we use \eqref{eq:clever-basis} and that the more-central eigenvalues are given in \Cref{prop:spectrum-eps-positive}. For the nonlinearity, we note as before that $\Pcal_{mc}^\eps \hat{\Ncal}_2 = \Ocal(\eps^3)$ using \Cref{ass:smallness} and $\Pcal_{mc}^\eps \hat{\Ncal}_3 = \Ocal(\eps^3)$. Hence, projecting on $\phi_+^{\eps}(\k_j)$ and $\phi_-^{\eps}(\k_j)$ we obtain that the dynamics on the centre manifold are described to leading order by the system
\begin{equation*}
    \begin{split}
        \partial_\Xi \tilde{A}_j &=  - \dfrac{c_0}{8 (\d \cdot \k_j)^2} \tilde{A}_j + \dfrac{\sqrt{c_0^2 - 16 (\d \cdot \k_j)^2 \mu_0}}{8 (\d \cdot \k_j)^2} \tilde{B}_j, \\
        \partial_\Xi \tilde{B}_j &=  \dfrac{\sqrt{c_0^2 - 16 (\d \cdot \k_j)^2 \mu_0}}{8 (\d \cdot \k_j)^2} \tilde{A}_j - \dfrac{c_0}{8 (\d \cdot \k_j)^2} \tilde{B}_j + \Ncal_{\Mcal_{mc},j}(\tilde{A}_1,\tilde{A}_2,\tilde{A}_3,\bar{\tilde{A}}_1,\bar{\tilde{A}}_2,\bar{\tilde{A}}_3).
    \end{split}
\end{equation*}
Here, $\Ncal_{\Mcal_{mc},j}$ denotes the $\Ocal(\eps^3)$-contributions of the nonlinearity. Note that these only depend on $A_j$, see \eqref{eq:N2-spatial-to-physical}. Before determining the nonlinearities explicitly, we make the final coordinate change
\begin{equation*}
    A_j = \tilde{A}_j \quad B_j = - \dfrac{c_0}{8 (\d \cdot \k_j)^2} \tilde{A}_j + \dfrac{\sqrt{c_0^2 - 16 (\d \cdot \k_j)^2 \mu_0}}{8 (\d \cdot \k_j)^2} \tilde{B}_j,
\end{equation*}
cf.~\cite{eckmann1991-02CommunMathPhys} which transforms this system to
\begin{equation}\label{eq:reduced-equation-with-nonlinear-placeholder}
    \begin{split}
        \partial_\Xi A_j &=  B_j, \\
        %\frac{8(\d\cdot\k_j)^2}{\sqrt{c_0^2 - 16(\d\cdot\k_j)^2 \mu_0}}\partial_\Xi B_j + \frac{c_0}{\sqrt{c_0^2 - 16(\d\cdot\k_j)^2 \mu_0}}B_j &=  \dfrac{\sqrt{c_0^2 - 16 (\d \cdot \k_j)^2 \mu_0}}{8 (\d \cdot \k_j)^2} A_j - \dfrac{c_0}{8 (\d \cdot \k_j)^2} \Bigg(\frac{c_0}{\sqrt{c_0^2 - 16(\d\cdot\k_j)^2 \mu_0}}A_j + \frac{8(\d\cdot\k_j)^2}{\sqrt{c_0^2 - 16(\d\cdot\k_j)^2 \mu_0}} B_j\Bigg) + \Ncal_{\Mcal_{mc},j}(\tilde{A}_1,\tilde{A}_2,\tilde{A}_3,\bar{\tilde{A}}_1,\bar{\tilde{A}}_2,\bar{\tilde{A}}_3). \\
        \partial_\Xi B_j &= -\dfrac{\mu_0}{4 (\d \cdot \k_j)^2} A_j - \dfrac{c_0}{4(\d \cdot \k_j)^2}B_j + \dfrac{\sqrt{c_0^2 - 16 (\d \cdot \k_j)^2 \mu_0}}{8(\d \cdot \k_j)^2} \Ncal_{\Mcal_{mc},j}(A_1,A_2,A_3,\bar{A}_1,\bar{A}_2,\bar{A}_3).
    \end{split}
\end{equation}
We point out that this second transformation is effectively 'normalising' the eigenvector $\phib_-^{\eps}(\k_j)$ so that $\hat{\Lcal}^{\eps}(\k_j) \phib_-^\eps(\k_j) = \phib^\eps_+(\k_j) + \Ocal(\eps)$.

It remains to derive an expression for the nonlinear terms. Since we have already determined the $\Ocal(\eps^2)$-terms of $\tilde{\Psi}_{\gammab}$, we obtain the nonlinear terms by balancing at $\eps^3 e^{i\k_j \cdot \p}$. For this, the relevant contributions are quadratic combinations $A_{\ell_1}, A_{\ell_2}$ such that $\k_{\ell_1} + \k_{\ell_2} = \k_j$, quadratic combinations of $A_\ell$ and $\tilde{\Psi}_{\gammab}$ with $\k_\ell + \gammab = \k_j$, and cubic combinations of $A_{\ell_1}, A_{\ell_2}, A_{\ell_3}$ such that $\k_{\ell_1} + \k_{\ell_2} + \k_{\ell_3} = \k_j$. In addition, to determine the normalisation of the projection $\Pcal_{mc}^\eps(\k_j)$, we recall that the projection is explicitly given by
\begin{equation}
    \Pcal_{mc}^\eps(\k_j) W = \dfrac{\hat{\phib}^{\eps, \ast}_+(\k_j) \cdot W}{\hat{\phib}^{\eps,\ast}_+(\k_j)\cdot \hat{\phib}^{\eps}_+(\k_j)} \hat{\phib}^{\eps}_+(\k_j) + \dfrac{\hat{\phib}^{\eps, \ast}_-(\k_j) \cdot W}{\hat{\phib}^{\eps,\ast}_-(\k_j)\cdot \hat{\phib}^{\eps}_-(\k_j)} \hat{\phib}^{\eps}_-(\k_j),
\end{equation}
where $\hat{\phib}^{\eps,\ast}_{\pm}(\k_j)$ are the eigenvectors of the adjoint matrix $\hat{\Lcal}^{\eps}(\k_j)^\ast$ corresponding to the eigenvalue $\lambda_\pm^{\eps}(\k_j) + i\d \cdot \k_j$, which are normalised such that the last component is equal to one. A direct computation yields that
\begin{equation*}
    \hat{\phib}^{\eps,\ast}_\pm(\k_j)_i = \sum_{j = i+1}^4 \bar{a}_j (\bar{\lambda}_{\pm}^{\eps}(\k_j)- i \d\cdot \k_j)^{j-i-1},
\end{equation*}
where we have set $a_4 = -1$, and so
\begin{equation}\label{eq:scalar-product-projection}
\begin{split}
    \hat{\phib}^{\eps,\ast}_\pm(\k_j)\cdot \hat{\phib}^{\eps}_\pm(\k_j) & = -\sum_{i = 0}^{3} (\lambda_\pm^{\eps}(\k_j) + i\d \cdot \k_j)^i \sum_{j = i+1}^4 a_j (\lambda_\pm^{\eps}(\k_j) + i\d \cdot \k_j)^{j-i-1} \\
    & = - \sum_{k=0}^{4} k a_k (\lambda_\pm^{\eps}(\k_j) + i\d \cdot \k_j)^{k-1} = \mp \eps \sqrt{c_0^2 - 16(\d\cdot\k_j)^2 \mu_0} + \Ocal(\eps^2),
\end{split}
\end{equation}
see \cite{eckmann1991-02CommunMathPhys}. Next, we write the projection $\Pcal_{mc}^\eps(\k_j)$ in the new basis $(\phib_+^\eps(\k_j), \phib_-^\eps(\k_j))$, which yields
\begin{equation*}
\begin{split}
        \Pcal_{mc}^\eps(\k_j) \hat{\Ncal}(W;\k_j, \vartheta) & = \left(\dfrac{\hat{\phib}^{\eps, \ast}_+(\k_j) \cdot \hat{\Ncal}(W;\k_j, \vartheta) }{\hat{\phib}^{\eps,\ast}_+(\k_j)\cdot \hat{\phib}^{\eps}_+(\k_j)}  + \dfrac{\hat{\phib}^{\eps, \ast}_-(\k_j) \cdot \hat{\Ncal}(W;\k_j, \vartheta)}{\hat{\phib}^{\eps,\ast}_-(\k_j)\cdot \hat{\phib}^{\eps}_-(\k_j)}\right) \phib_+^\eps(\k_j) \\ 
        & \quad + \eps \left(\dfrac{\hat{\phib}^{\eps, \ast}_+(\k_j) \cdot \hat{\Ncal}(W;\k_j, \vartheta)}{\hat{\phib}^{\eps,\ast}_+(\k_j)\cdot \hat{\phib}^{\eps}_+(\k_j)} - \dfrac{\hat{\phib}^{\eps, \ast}_-(\k_j) \cdot \hat{\Ncal}(W;\k_j, \vartheta)}{\hat{\phib}^{\eps,\ast}_-(\k_j)\cdot \hat{\phib}^{\eps}_-(\k_j)}\right) \phib_-^\eps(\k_j) \\
        &= \alpha_+^{\eps}(\k_j)\hat{N}(W;\k_j,\vartheta) \phib_+^\eps(\k_j) - \dfrac{2 + \Ocal(\eps)}{\sqrt{c_0^2 - 16(\d\cdot\k_j)^2 \mu_0}}\hat{N}(W;\k_j,\vartheta) \phib_-^\eps(\k_j),
\end{split}
\end{equation*}
where $\alpha_+^{\eps}(\k_j) = \Ocal(1)$ and we used that $\hat{\Ncal} = (0,0,0,\hat{N})$ and the last component of $\hat{\phib}_\pm^{\eps,\ast}(\k_j)$ is normalised to one.

Again, we notice that $\hat{N}(W_{mc} + \Psi(W_{mc});\k_j,\vartheta) = \Ocal(\eps^3)$ and therefore, the nonlinearity only adds higher-order terms to the $A_j$-equation. For the $B_j$-equation, we find with that the usual combinatorics, see e.g.~\cite{hilder2025-08JNonlinearSci} for details, yields the nonlinear terms of order $\Ocal(\eps^3)$
\begin{equation}\label{eq:reduced-nonlinearity}
\begin{split}
    & \dfrac{\sqrt{c_0^2 - 16 (\d \cdot \k_1)^2 \mu_0}}{8(\d \cdot \k_1)^2} \Ncal_{\Mcal_{mc},1}(A_1,A_2,A_3,\bar{A}_1,\bar{A}_2,\bar{A}_3) \\
    &= -\dfrac{\eps^{-3}}{4(\d \cdot \k_1)^2} \hat{N}(W_{mc} + \Psi(W_{mc});\k_1, \vartheta) = -\dfrac{1}{4(\d \cdot \k_1)^2}(\beta_2 \bar{A}_2 \bar{A}_3 + K_0 |A_1|^2 A_1 + K_2 (|A_2|^2 + |A_3|^2) A_1).
\end{split}
\end{equation}
Here, the quadratic coefficient is given as
\begin{equation*}
    \begin{split}
        \beta_2 := 2 \frac{N_2(e^{-i\k_2\cdot \p},e^{-i\k_3\cdot \p};\k_1,\vartheta)}{\eps},
    \end{split}
\end{equation*}
which is of order one for $\vartheta\in \Theta_0$ due to \Cref{ass:smallness}. The self-interaction and the cross-interaction coefficients are given as
\begin{equation}\label{eq:K-0-K-2}
    \begin{split}
        K_0 & := 2 \bigl[\hat{N}_2(e^{i \zerob \cdot \p},e^{i\k_1\cdot \p};\k_1, \vartheta)\nu_{0} + \hat{N}_2(e^{2i\k_1\cdot\p},e^{-i\k_1\cdot \p};\k_1, \vartheta) \nu_{2\k_1} \bigr] + 3 \hat{N}_3(e^{i\k_1\cdot \p},e^{i\k_1\cdot \p},e^{-i\k_1\cdot \p};\k_1,\vartheta),  \\
        K_2 & := 2 \bigl[\hat{N}_2(e^{i \zerob \cdot \p},e^{i\k_1\cdot \p};\k_1, \vartheta)\nu_{0} + \hat{N}_2(e^{i(\k_1-\k_2)\cdot\p},e^{i\k_2\cdot \p};\k_1, \vartheta) \nu_{\k_1-\k_2}\bigr] + 6 \hat{N}_3(e^{i\k_1\cdot \p},e^{i\k_2\cdot \p},e^{-i\k_2\cdot \p};\k_1,\vartheta).
    \end{split}
\end{equation}
Here, we recall that $\tilde{\Psi}_{\k_j} = \Ocal(\eps)$ and therefore, there are no quadratic combinations of $\tilde{\Psi}_{\k_j}$ and $A_\ell$ with $\k_j + \k_\ell = \k_1$ at leading order. While we have only derived the leading-order term at $\k_1$, we may obtain the leading-order terms at $\k_2$ and $\k_3$ by cyclic permutation of the $\k_i$. We point out that the coefficients are the same in all equations due to the rotational invariance of the nonlinearity and are all real-valued due to the reflectional invariance of the nonlinearity, cf.~\Cref{ass:rotation-invariant}.

Combining equations \eqref{eq:reduced-equation-with-nonlinear-placeholder} and \eqref{eq:reduced-nonlinearity}, we obtain the reduced equations on the centre manifold 
\begin{equation}\label{eq:reduced-equations-explicit}
    \begin{split}
        \partial_\Xi A_1 &=  B_1, \\
        \partial_\Xi B_1 &= -\dfrac{\mu_0}{4 (\d \cdot \k_1)^2} A_1 - \dfrac{c_0}{4(\d \cdot \k_1)^2}B_1 - \dfrac{1}{4(\d \cdot \k_1)^2} (\beta_2\bar{A}_2\bar{A}_3 + (K_0|A_1|^2 + K_2(|A_2|^2 + |A_3|^2)) A_1) + \Ocal(\eps),\\
        \partial_\Xi A_2 &=  B_2, \\
        \partial_\Xi B_2 &= -\dfrac{\mu_0}{4 (\d \cdot \k_2)^2} A_2 - \dfrac{c_0}{4(\d \cdot \k_2)^2}B_2 - \dfrac{1}{4(\d \cdot \k_2)^2} (\beta_2\bar{A}_1\bar{A}_3 + (K_0|A_2|^2 + K_2(|A_1|^2 + |A_3|^2)) A_2) + \Ocal(\eps),\\
        \partial_\Xi A_3 &=  B_3, \\
        \partial_\Xi B_3 &= -\dfrac{\mu_0}{4 (\d \cdot \k_3)^2} A_3 - \dfrac{c_0}{4(\d \cdot \k_3)^2}B_3 - \dfrac{1}{4(\d \cdot \k_3)^2} (\beta_2\bar{A}_1\bar{A}_2 + (K_0|A_3|^2 + K_2(|A_1|^2 + |A_2|^2)) A_3) + \Ocal(\eps),\\
    \end{split}
\end{equation}
which can also be written as the system of second-order equations
\begin{equation}\label{eq:reduced-equations-second-order}
\begin{split}
    4(\d\cdot\k_1)^2 \partial_{\Xi}^2 A_1 + c_0 \partial_{\Xi} A_1 + \mu_0 A_1 + \beta_2\bar{A}_2\bar{A}_3 + (K_0|A_1|^2 + K_2(|A_2|^2 + |A_3|^2)) A_1 + \Ocal(\eps) & = 0, \\
    4(\d\cdot\k_2)^2 \partial_{\Xi}^2 A_2 + c_0 \partial_{\Xi} A_2 + \mu_0 A_2 + \beta_2\bar{A}_1\bar{A}_3 + (K_0|A_2|^2 + K_2(|A_1|^2 + |A_3|^2)) A_2 + \Ocal(\eps) & = 0, \\
    4(\d\cdot\k_3)^2 \partial_{\Xi}^2 A_3 + c_0 \partial_{\Xi} A_3 + \mu_0 A_3 + \beta_2\bar{A}_1\bar{A}_2 + (K_0|A_3|^2 + K_2(|A_1|^2 + |A_2|^2)) A_3 + \Ocal(\eps) & = 0. \\
\end{split}
\end{equation}

\begin{remark}
    We note that the quadratic, self-interaction and cross-interaction coefficients are the same as in \cite{hilder2025-08JNonlinearSci}. In addition, they are also consistent with \cite{doelman2003-02EuropeanJournalofAppliedMathematics} by noting that in this paper, the quadratic nonlinearity is assumed to be of higher order, while in this paper, we only assume that the quadratic nonlinearity is small at the critical Fourier modes, cf.~\Cref{ass:smallness}. 
\end{remark}

\begin{remark}\label{rem:higher-order-terms}
    As in \cite[Rem.~4.5]{hilder2025-08JNonlinearSci} we highlight that the higher-order terms in the $A_1$-equation $\eqref{eq:reduced-equations-second-order}_1$ are of the form 
    \begin{equation*}
        A_1p_1(|A_1|^2,|A_2|^2,|A_3|^2,q) + \bar{A}_2\bar{A}_3 p_2(|A_1|^2,|A_2|^2,|A_3|^2,q)
    \end{equation*}
    with $q = A_1 A_2 A_3 + \bar{A}_1 \bar{A}_2 \bar{A}_3$ for polynomials $p_{j}(|A_1|^2,|A_2|^2,|A_3|^2,q) = \Ocal((|A_1|^2+|A_2|^2+|A_3|^2+q)^2)$, $j=1,2$. This follows from the rotational invariance of the nonlinearity, cf.~\Cref{ass:rotation-invariant}, from translation invariance and the fact that $\k_1 + \k_2 + \k_3 = 0$, cf.~\cite[Prop.~3.1]{buzano1983PhilosTransAMathPhysEngSci}.
\end{remark}

\begin{remark}
    We point out that the reduced equations on the centre manifold \eqref{eq:reduced-equations-explicit} are the same as the ones obtained in equation \eqref{eq:amplitude-system-hex} via the asymptotic expansion in physical variables \eqref{eq:amplitude-ansatz} in the Swift–Hohenberg equation \eqref{eq:Swift-Hohenberg}. This is, of course, not a coincidence. In fact, one can show that the leading-order systems are equivalent in the following sense. If $U$ is a solution to the physical system \eqref{eq:Swift-Hohenberg} of the form
    \begin{equation*}
        U(\xi,\p) = \eps \sum_{j = 1}^3 A_j(\eps \xi) e^{i \k_j \cdot \p} + c.c. + \Ocal(\eps^2),
    \end{equation*}
    one obtains a solution to the reduced equation on the centre manifold via $W = (U, (\partial_\xi+\d\cdot \nabla_p)U,(\partial_\xi+\d\cdot \nabla_p)^2U,(\partial_\xi+\d\cdot \nabla_p)^3U)$. Indeed, one observes that $W = \eps W_{mc} + \Ocal(\eps^2)$ with $W_{mc} \in \Zcal_{mc}^\eps$ since $\xi$-derivatives yield additional powers of $\eps$ and
    \begin{equation*}
        (1, (\d\cdot \nabla_p),(\d\cdot \nabla_p)^2,(\d\cdot \nabla_p)^3) e^{i\k_j \cdot \p} = (1,(i\d \cdot \k_j), (i\d \cdot \k_j)^2,(i\d \cdot \k_j)^3) e^{i\k_j\cdot \p} = \phib_{\pm}^{\eps}(\k_j) e^{i\k_j\cdot \p} + \Ocal(\eps).
    \end{equation*}
    On the other hand, solutions $(A_1,A_2,A_3,B_1,B_2,B_3)$ to the reduced equations \eqref{eq:reduced-equations-explicit} yield solutions $W$ to the spatial-dynamics system via
    \begin{equation*}
        W(\xi,\p) = \eps \sum_{j = 1}^3 (\tilde{A}_{j,+}(\eps \xi) \hat{\phib}_+^\eps(\k_j) + \tilde{A}_{j,-}(\eps \xi) \hat{\phib}_-^\eps(\k_j)) e^{i \k_j \cdot \p} + c.c. + \Ocal(\eps^2),
    \end{equation*}
    see \Cref{thm:centre-manifold}, where $\tilde{A}_{j,\pm}$ is obtained through \eqref{eq:tildeA-to-other-tilde-A}. Finally, we obtain a solution to the physical system by restricting the first component and note that $\hat{\phib}_\pm^\eps(\k_j)_1 = 1 + \Ocal(\eps)$ which yields $A_j = \tilde{A}_{j,+} + \tilde{A}_{j,-}$, cf.~\cite{eckmann1991-02CommunMathPhys,hilder2020-08JournalofDifferentialEquations}.
\end{remark}

\subsection{The centre manifold reduction and reduced equations for $\theta = \tfrac{\pi}{6}$}\label{sec:special-case}

We now consider the case that $\theta = \tfrac{\pi}{6}$, where $\hat{\Lcal}^0(\pm \k_2)$ has a Jordan block of size four instead of size two, cf.~\Cref{prop:imaginary-spectrum-discrete}. For $\eps > 0$ this Jordan block then splits into three less-central eigenvalues with real part of order $\eps^{1/3}$ and one more-central eigenvalue of the form $\lambda_{\perp}^{\eps}(\pm\k_2) = - \eps\tfrac{\mu_0}{c_0} + \Ocal(\eps^2)$, cf.~\Cref{prop:spectrum-eps-positive}. This makes it more complicated to separate the more-central and less-central eigenvalues. Indeed, \Cref{lem:estimate-projections} and \Cref{rem:projection-is-optimal} demonstrate that the relevant projection $\Pcal_{mc}^\eps(\k_2)$ behaves like $\eps^{-1}$. Therefore, we need to modify the centre manifold construction above. For this, we introduce the maps
\begin{equation}\label{eq:not-projections}
    \begin{split}
        \tilde{\Pcal}_{mc}^{\eps}(\k_2) W &= (\hat{\phib}_{\perp}^{\eps,*}(\k_2) \cdot W) \hat{\phib}^{\eps}_{\perp}(\k_2), \\
        \tilde{\Pcal}_{lc,j}^{\eps}(\k_2) W &= (\hat{\phib}_{lc,j}^{\eps,*}(\k_2) \cdot W) \hat{\phib}_{lc,j}^{\eps}(\k_2),
    \end{split}
\end{equation}
where $\phib_{\perp}^{\eps}(\k_2)$ is the eigenvector of $\hat{\Lcal}^{\eps}(\k_2)$ corresponding to the more-central eigenvalue $\lambda_{\perp}^{\eps}(\k_2)$, $\phib_{lc,j}^{\eps}(\k_2)$ are the eigenvectors of $\hat{\Lcal}^{\eps}(\k_2)$ corresponding to the less-central eigenvalues $\lambda_{lc,j}^{\eps}(\k_2)$, and $\phib_{\perp}^{\eps,\ast}(\k_2)$ and $\phib_{lc,j}^{\eps,\ast}(\k_2)$ are the eigenvectors of the corresponding adjoint eigenvalue problems. As above, $\phib_{\perp}^{\eps}(\k_2)$ and $\phib_{lc,j}^{\eps}(\k_2)$ are normalised such that their first component equals one, and $\phib_{\perp}^{\eps,\ast}(\k_2)$ and $\phib_{lc,j}^{\eps,\ast}(\k_2)$ are normalised such that their last component equals one. We point out that $\tilde{\Pcal}_{mc}^{\eps}(\k_2) $ and $\tilde{\Pcal}_{lc,j}^{\eps}(\k_2)$ are not projections. In particular, the more-central map satisfies
\begin{equation*}
    \begin{split}
        \tilde{\Pcal}_{mc}^{\eps}(\k_2)\tilde{\Pcal}_{mc}^{\eps}(\k_2)W = (\hat{\phib}_{\perp}^{\eps,*}(\k_2) \cdot \hat{\phib}^{\eps}_{\perp}(\k_2))  \tilde{\Pcal}_{mc}^{\eps}(\k_2)W = - \eps c_0 \tilde{\Pcal}_{mc}^{\eps}(\k_2)W,
    \end{split}
\end{equation*}
where we used \eqref{eq:scalar-product-projection} and $\d \cdot \k_2 = 0$. We note that $\tilde{\Pcal}_{mc}^{\eps}(\k_2) $ and $\tilde{\Pcal}_{lc,j}^{\eps}(\k_2)$ map onto the more-central and less-central eigenspaces, respectively, and, in particular, are uniformly bounded for $\eps > 0$.

Using these maps, we define, similarly to \eqref{eq:projected-variables},
\begin{equation*}
    \hat{W}_{mc}(\k_2) := \tilde{\Pcal}_{mc}^{\eps}(\k_2) \hat{W}(\k_2), \quad \hat{W}_{lc,j}(\k_2) := \tilde{\Pcal}_{lc,j}^{\eps}(\k_2) \hat{W}(\k_2)
\end{equation*}
and introduce the rescaled variables
\begin{equation*}
    \shortunderline{\hat{W}}_{mc}(\k_2) = \eps^{-\beta} \hat{W}_{mc}(\k_2), \quad \shortunderline{\hat{W}}_{lc,j}(\k_2) = \eps^{-\gamma} \hat{W}_{lc,j}(\k_2)
\end{equation*}
for $j = 1,2,3$ and $0 < \beta < 1 < \gamma < 2$ as in \eqref{eq:rescaling} with explicit values given in \eqref{eq:beta-gamma-final}. We then obtain the spatial-dynamics system
\begin{equation}\label{eq:spatial-dynamics-k2}
    \begin{split}
        \partial_\xi \shortunderline{\hat{W}}_{mc}(\pm \k_2) &= \lambda_{\perp}^{\eps}(\pm \k_2)\shortunderline{\hat{W}}_{mc}(\pm \k_2) + \shortunderline{\hat{\Ncal}}_{mc} (\shortunderline{\hat{W}}_{mc},\shortunderline{\hat{W}}_{lc} + \shortunderline{\hat{W}}_{h};\pm \k_2,\vartheta), \\
        \partial_\xi \shortunderline{\hat{W}}_{lc,j}(\pm \k_2) &= \lambda_{lc,j}^{\eps}(\pm \k_2)\shortunderline{\hat{W}}_{lc,j}(\pm \k_2) + \shortunderline{\hat{\Ncal}}_{lc,j} (\shortunderline{\hat{W}}_{mc},\shortunderline{\hat{W}}_{lc} + \shortunderline{\hat{W}}_{h};\pm \k_2,\vartheta),
    \end{split}
\end{equation}
with $j = 1,2,3$, at $\gammab=\k_2$ combined with system \eqref{eq:spatial-dynamics-split} for all other modes $\gammab\neq\k_2$. Here, $\shortunderline{\hat{\Ncal}}_{mc}$ and $\shortunderline{\hat{\Ncal}}_{lc,j}$ are defined like the rescaled nonlinearities after system \eqref{eq:spatial-dynamics-split} replacing the projections by $\tilde{P}^{\eps}_{mc}(\pm\k_2)$ and $\tilde{P}^{\eps}_{lc,j}(\pm\k_2)$, respectively. Then, repeating the normal-form analysis of \Cref{sec:rescaling-normal-form} and the proof of the centre manifold theorem in \Cref{sec:proof-centre-manifold-theorem} line by line, using that $\tilde{\Pcal}_{mc}^{\eps}(\k_2) $ and $\tilde{\Pcal}_{lc,j}^{\eps}(\k_2)$ are uniformly bounded for $\eps > 0$ and that the semigroup bounds for \eqref{eq:spatial-dynamics-k2} are better than for the rest of the more-central and less-central spectrum (cf.~\Cref{lem:semigroup-bounds}), we obtain the following result.

\begin{theorem}\label{thm:centre-manifold-pi-over-6}
    Let $\ell_1, \ell_2 \in \N$ and $\theta = \tfrac{\pi}{6}$. For any $\delta \in (0,\tfrac{1}{4})$ there exists an $\eps_0 > 0$ such that for every $\eps \in (0,\eps_0)$ there exists a neighbourhood $O_{mc}^{\eps} \subset \Zcal_{mc}^\eps$ of the origin and a map $\Psi^{\eps} : O_{mc}^{\eps} \to (I - \Pcal_{mc}^{\eps})\Zcal$, which is at least quadratic, such that the following holds.
    \begin{enumerate}[label=(\alph*), ref=\thetheorem(\alph*)]
        \item The neighbourhood $O_{mc}^{\eps}$ is of size $\Ocal(\eps^{3/4 + \delta})$.
        \item The centre manifold
        \begin{equation*}
            \Mcal_{mc}^{\eps} = \{W = W_{mc} + \Psi^{\eps}(W_{mc}) \, : \, W_{mc} \in O_{mc}^{\eps} \}
        \end{equation*}
        contains all small bounded solutions to \eqref{eq:spat-dyn-Fourier}.
        \item Solutions to the reduced equation
        \begin{equation}\label{eq:reduced-equation-CM-theorem-pi-over-6}
            \partial_{\xi} W_{mc} = \Lcal^{\eps} W_{mc}+ \Pcal_{mc}^{\eps} \Ncal(W_{mc} + \Psi^{\eps}(W_{mc});\vartheta)
        \end{equation}
        give rise to solutions $W$ to the full system \eqref{eq:spat-dyn-Fourier} via $W=W_{mc} + \Psi^{\eps}(W_{mc})$.
        \item The symmetries of the system \eqref{eq:Swift-Hohenberg} are preserved by the reduction function $\Psi$.
    \end{enumerate}
\end{theorem}

Finally, we derive the reduced equation on the centre manifold. For this, we make the modified ansatz
\begin{equation*}
\begin{split}
    W = W_{mc} + \Psi(W_{mc}) & =  \eps \sum_{j \in \{1,3\}} (\tilde{A}_{j,+}(\Xi) \hat{\phib}_+^{\eps}(\k_j) + \tilde{A}_{j,-}(\Xi) \hat{\phib}_-^{\eps}(\k_j))e^{i \k_j \cdot \p} + \eps A_2(\Xi) \phib_{\perp}^\eps(\k_j) e^{i \k_2 \cdot \p} + c.c. \\
    &\quad + \eps^2 \tilde{\Psi}(\tilde{A}_{1,\pm},A_2,\tilde{A}_{3,\pm};\eps),
\end{split}
\end{equation*}
with $\Xi = \eps \xi$. We then obtain an equation for $A_2$ by inserting the ansatz into \eqref{eq:reduced-equation-CM-theorem-pi-over-6} and applying the map $\tilde{\Pcal}^\eps_{mc}(\k_2)$. This yields
\begin{equation*}
    -\eps^3 c_0 \partial_\Xi A_2 = -\eps^2 c_0 \lambda_{\perp}(\k_2) A_2 + \Ncal_{\Mcal_{mc},2}(A_1,A_2,A_3,\bar{A}_1,\bar{A}_2,\bar{A}_3)
\end{equation*}
with $A_1 = A_{1,+} + A_{1,-}$ and $A_3 = A_{3,+} + A_{3,-}$ as in \Cref{sec:reduced-equations}. To determine the nonlinearity, we note that
\begin{equation*}
    \tilde{\Pcal}_{mc}^\eps \hat{\Ncal}(W; \k_2, \vartheta) = \hat{N}(W; \k_2, \vartheta) \phib_{\perp}^\eps(\k_2),
\end{equation*}
where we use again that $\hat{\Ncal} = (0,0,0,\hat{N})$ and that the last component of $\hat{\phib}_{\perp}^{\eps,*}(\k_2)$ is normalised to one. Recalling that $\lambda_{\perp}^\eps(\k_2) = - \eps \tfrac{\mu_0}{c_0} + \Ocal(\eps^2)$, see \Cref{it:spectrum-pi-over-six}, and repeating the derivation of the equations for $A_1$ and $A_3$ as in \Cref{sec:reduced-equations}, we find
\begin{equation}\label{eq:reduced-equations-pi-over-six-second-order}
\begin{split}
    3 \partial_{\Xi}^2 A_1 + c_0 \partial_{\Xi} A_1 + \mu_0 A_1 + \beta_2\bar{A}_2\bar{A}_3 + (K_0|A_1|^2 + K_2(|A_2|^2 + |A_3|^2)) A_1 + \Ocal(\eps) & = 0, \\
    c_0 \partial_{\Xi} A_2 + \mu_0 A_2 + \beta_2\bar{A}_1\bar{A}_3 + (K_0|A_2|^2 + K_2(|A_1|^2 + |A_3|^2)) A_2 + \Ocal(\eps) & = 0, \\
    3 \partial_{\Xi}^2 A_3 + c_0 \partial_{\Xi} A_3 + \mu_0 A_3 + \beta_2\bar{A}_1\bar{A}_2 + (K_0|A_3|^2 + K_2(|A_1|^2 + |A_2|^2)) A_3 + \Ocal(\eps) & = 0,
\end{split}
\end{equation}
since $\d\cdot \k_1 = \tfrac{\sqrt{3}}{2}$ and $\d\cdot \k_3 = -\tfrac{\sqrt{3}}{2}$. Note that \eqref{eq:reduced-equations-pi-over-six-second-order} is consistent with \eqref{eq:reduced-equations-second-order}. Indeed, we make this connection rigorous through a fast-slow analysis, see \Cref{sec:fast-slow-results}.

\subsection{The reduced equations for $c_0 = c_{\crit}(\k_j)$}\label{sec:reduced-equations-critical}

In \Cref{sec:reduced-equations}, we assumed that $c_0 \neq c_{\crit}(\k_j)$ for all $j = 1,2,3$, which yields that the more-central eigenvalues split at order $\eps$. We now discuss the case that $c_0 = c_{\crit}(\k_j)$ for one $j \in \{1,2,3\}$. Since the eigenvalues $\lambda^\eps_\pm(\k_j)$ are now equal at leading order, we have to modify the derivation of the reduced equation for $A_j$ and $B_j$. For this, we recall that $\hat{\phib}^\eps_\pm(\k_j)$ are the eigenvectors of $\hat{\Lcal}^{\eps}(\k_j)$ corresponding to $\lambda_\pm^\eps(\k_j)$ and we introduce the new basis
\begin{equation*}
    \phib_+^{\eps}(\k_j) = \dfrac{1}{2}\bigl(\hat{\phib}_+^{\eps}(\k_j) + \hat{\phib}_-^{\eps}(\k_j)\bigr) \quad \text{and} \quad \phib_-^{\eps}(\k_j) =  \dfrac{1}{2\eps^{3/2}}\bigl(\hat{\phib}_+^{\eps}(\k_j) - \hat{\phib}_-^{\eps}(\k_j)\bigr)
\end{equation*}
similar to \eqref{eq:clever-basis}. Using the expansion of the eigenvalues provided in \Cref{prop:spectrum-eps-positive}, we again find that $(\phib_+^\eps(\k_j), \phib_-^\eps(\k_j))$ converges to $(\phib_+^0(\k_j), \phib_-^0(\k_j))$, where $\phi_+^0(\k_j)$ is the eigenvector corresponding to the zero eigenvalue of $\hat{\Lcal}^0(\k_j)$ with first component normalised to one and $\phib_-^0(\k_j)$ is a generalised eigenvalue with first component equal to zero. As in the non-critical case, we then write
\begin{equation*}
    \tilde{A}_{j,+}\hat{\phib}_+^{\eps}(\k_j) + \tilde{A}_{j,-} \hat{\phib}_-^{\eps}(\k_j) = \tilde{A}_j \phib_+^\eps(\k_j) + \eps^{3/2} \tilde{B}_j \phib_-^\eps(\k_j),
\end{equation*}
which defines a bounded transformation $(\tilde{A}_{j,+}, \tilde{A}_{j,-}) \to (\tilde{A}_j,\tilde{B}_j)$. Then the linear terms are given by
\begin{equation*}
    \begin{split}
        &(\hat{\Lcal}^{\eps}(\k_j) - i \d \cdot \k_j) (\tilde{A}_j \phib_+^{\eps}(\k_j) + \eps^{3/2} \tilde{B}_j \phib_-^{\eps}(\k_j)) \\
        &\quad  = \eps \Big( - \dfrac{c_0}{8 (\d \cdot \k_j)^2} \tilde{A}_j + \sqrt{\eps} \frac{i^{3/2}\mu_0^{3/4}}{2\sqrt{2}(\d\cdot\k_j)^2} \tilde{B}_j\Big) \phib_+^\eps(\k_j) + \eps^{\tfrac{5}{2}} \Big( \sqrt{\eps}\frac{i^{3/2}\mu_0^{3/4}}{2\sqrt{2}(\d\cdot\k_j)^2} \tilde{A}_j - \dfrac{c_0}{8 (\d \cdot \k_j)^2} \tilde{B}_j\Big) \phib_-^\eps(\k_j),
    \end{split}
\end{equation*}
where we used that $\lambda_+^\eps(\k_j) + \lambda_-^\eps(\k_j) = - \eps \tfrac{c_0}{8(\d \cdot \k_j)}$ and $\lambda_+^\eps(\k_j) - \lambda_-^\eps(\k_j) = \eps^{3/2} \tfrac{i^{3/2}\mu_0^{3/4}}{2\sqrt{2}(\d\cdot\k_j)^2}$. Combining the linear and nonlinear terms, we therefore find that
\begin{equation*}
    \begin{split}
        \eps^2 \partial_\Xi \tilde{A}_j &= \eps^2 \Big( - \dfrac{c_0}{8 (\d \cdot \k_j)^2} \tilde{A}_j + \sqrt{\eps} \frac{i^{3/2}\mu_0^{3/4}}{2\sqrt{2}(\d\cdot\k_j)^2} \tilde{B}_j\Big) + \Ocal(\eps^3) \\
        \eps^{\tfrac{7}{2}} \partial_\Xi \tilde{B}_j &= \eps^{\tfrac{7}{2}} \Big( \sqrt{\eps}\frac{i^{3/2}\mu_0^{3/4}}{2\sqrt{2}(\d\cdot\k_j)^2} \tilde{A}_j - \dfrac{c_0}{8 (\d \cdot \k_j)^2} \tilde{B}_j\Big) + \eps^3 \Ncal_{\Mcal_{mc},j}(\tilde{A}_1,\tilde{A}_2,\tilde{A}_3,\bar{\tilde{A}}_1,\bar{\tilde{A}}_2,\bar{\tilde{A}}_3).
    \end{split}
\end{equation*}
In particular, we find that the $\phib_-^\eps(\k_j)$-contribution of the projected nonlinearity is of order $\eps^3$ since $\hat{N}(W;k_j,\vartheta) = \Ocal(\eps^3)$. This is due to the computations
\begin{equation*}
    \hat{\phib}_\pm^{\eps,\ast}(\k_j) \cdot \hat{\phib}_\pm^{\eps}(\k_j) = \mp \eps^{\tfrac{3}{2}} \dfrac{i^{3/2}\mu_0^{3/4}}{2\sqrt{2}(\d\cdot\k_j)^2} 8(\d \cdot \k_j)^2 + \Ocal(\eps^2)
\end{equation*}
and therefore
\begin{equation*}
\begin{split}
        \Pcal_{mc}^\eps(\k_j) \hat{\Ncal}(W;\k_j, \vartheta) = \alpha_+^{\eps}(\k_j)\hat{N}(W;\k_j,\vartheta) \phib_+^\eps(\k_j) - \dfrac{1+\Ocal(\eps)}{4(\d\cdot\k_j)^2} \left(\dfrac{i^{3/2}\mu_0^{3/4}}{2\sqrt{2}(\d\cdot\k_j)^2}\right)^{-1} \hat{N}(W; \k_j,\vartheta) \phib_-^\eps(\k_j),
\end{split}
\end{equation*}
with $\alpha_+^{\eps}(\k_j) = \Ocal(1)$. We now proceed as in \Cref{sec:reduced-equations} and introduce the transformation
\begin{equation*}
    A_j = \tilde{A}_j, \quad B_j = - \dfrac{c_0}{8 (\d \cdot \k_j)^2} \tilde{A}_j + \sqrt{\eps} \frac{i^{3/2}\mu_0^{3/4}}{2\sqrt{2}(\d\cdot\k_j)^2} \tilde{B}_j.
\end{equation*}
Notably, $B_j$ is of the same order as $\partial_\Xi A_j$ since the $\sqrt{\eps}$ in the $\tilde{B}_j$-coefficient compensates for the different $\eps$-behaviour in the definition of $\tilde{B}_j$. Using this transformation, we then find
\begin{equation*}
    \begin{split}
        \partial_\Xi A_j &= B_j + \Ocal(\eps), \\
        \partial_\Xi B_j &= \eps \Bigl(\frac{i^{3/2}\mu_0^{3/4}}{2\sqrt{2}(\d\cdot\k_j)^2}\Bigr)^2 A_j - \Bigl(\dfrac{c_0}{8 (\d \cdot \k_j)^2}\Bigr)^2 A_j - \dfrac{c_0}{4 (\d \cdot \k_j)^2} B_j + \left(\dfrac{i^{3/2}\mu_0^{3/4}}{2\sqrt{2}(\d\cdot\k_j)^2}\right) \Ncal_{\Mcal_{mc},j}(\tilde{A}_1,\tilde{A}_2,\tilde{A}_3,\bar{\tilde{A}}_1,\bar{\tilde{A}}_2,\bar{\tilde{A}}_3).
    \end{split}
\end{equation*}
As before, $\Ncal_{\Mcal_{mc},j}$ denotes the $\eps^3$-contributions of the (projected) nonlinearity. Since the combinatorics for the nonlinear terms are not affected by the value of $c_0$, and using that $\bigl(\tfrac{c_0}{8 (\d \cdot \k_j)^2}\bigr)^2 = \tfrac{\mu_0}{4(\d\cdot\k_j)^2}$ for $c_0 = c_{\crit}(\k_j)$, the reduced equations are given by \eqref{eq:reduced-equations-explicit} or, equivalently, \eqref{eq:reduced-equations-second-order}.

\begin{remark}
    The calculation above shows that, although the algebraic manipulations are different, the reduced equations on the centre manifold are, to leading order, the same for $c_0 = c_{\crit}(\k_j)$ and $c_0 \neq c_{\crit}(\k_j)$. Indeed, this is not surprising since the reduced equations can also be obtained from the amplitude equations \eqref{eq:amplitude-system-hex} by inserting a travelling-wave ansatz. Since the amplitude equations are derived in the stationary frame, they are independent of the chosen speed for the travelling wave. However, it turns out that the critical speed is still hidden in the dynamics of the reduced equations, which change when $c_0$ passes through $c_{\crit}(\k_j)$, see \Cref{sec:interfaces} and \Cref{fig:different-front-behaviour}. Indeed, we observe that fronts at supercritical speeds have a monotone profile, while at subcritical speeds, fronts have an oscillatory leading edge. This is consistent with monostable fronts in reaction-diffusion systems.
\end{remark}

\begin{remark}
    We only need to consider the case $c_0 = c_{\crit}(\k_j)$ for $\d \cdot \k_j \neq 0$. Otherwise, we find that $c_{\crit}(\k_j) = 0$, however, we have assumed that $c_0 > 0$. 
\end{remark}

\section{Analysis of reduced equations and pattern interfaces}\label{sec:interfaces}

We now discuss the dynamics of the reduced equations. Specifically, we consider the existence of equilibria, their stability, and the existence of heteroclinic orbits since these correspond to planar patterns and pattern interfaces in the full system \eqref{eq:Swift-Hohenberg}, respectively. The general strategy is to first analyse the leading-order dynamics given by the truncated system for $\theta \in [0,\tfrac{\pi}{6})$, which reads as
\begin{equation}\label{eq:leading-order-reduced-equations}
        \begin{split}
        \partial_\Xi A_1 &=  B_1, \\
        \partial_\Xi B_1 &= -\dfrac{\mu_0}{4 (\d \cdot \k_1)^2} A_1 - \dfrac{c_0}{4(\d \cdot \k_1)^2}B_1 - \dfrac{1}{4(\d \cdot \k_1)^2} (\beta_2\bar{A}_2\bar{A}_3 + (K_0|A_1|^2 + K_2(|A_2|^2 + |A_3|^2)) A_1),\\
        \partial_\Xi A_2 &=  B_2, \\
        \partial_\Xi B_2 &= -\dfrac{\mu_0}{4 (\d \cdot \k_2)^2} A_2 - \dfrac{c_0}{4(\d \cdot \k_2)^2}B_2 - \dfrac{1}{4(\d \cdot \k_2)^2} (\beta_2\bar{A}_1\bar{A}_3 + (K_0|A_2|^2 + K_2(|A_1|^2 + |A_3|^2)) A_2),\\
        \partial_\Xi A_3 &=  B_3, \\
        \partial_\Xi B_3 &= -\dfrac{\mu_0}{4 (\d \cdot \k_3)^2} A_3 - \dfrac{c_0}{4(\d \cdot \k_3)^2}B_3 - \dfrac{1}{4(\d \cdot \k_3)^2} (\beta_2\bar{A}_1\bar{A}_2 + (K_0|A_3|^2 + K_2(|A_1|^2 + |A_2|^2)) A_3),\\
    \end{split}
\end{equation}
and is obtained by setting $\eps = 0$ in \eqref{eq:reduced-equations-explicit}. Notably, this system is the travelling-wave ODE corresponding to the Ginzburg–Landau system obtained as a formal amplitude system close to a Turing instability, see \eqref{eq:amplitude-system-hex}. In the analysis, we restrict to real-valued solutions since these are sufficient to describe the dynamics typically observed in experiments. We then show the existence of equilibria in \eqref{eq:leading-order-reduced-equations} corresponding to roll waves, hexagons and mixed modes in \eqref{eq:Swift-Hohenberg}, cf.~\Cref{cor:stationary-patterns}. We note that, as already noticed in equivariant bifurcation theory, these are all possible real-valued equilibria.

To analyse the stability of the equilibria, we show that the stability structure of the linearisation in the spatial system \eqref{eq:leading-order-reduced-equations} can be characterised through the temporal stability of the corresponding equilibria in a system of Landau equations which governs the dynamics of spatially homogeneous solutions to the Ginzburg–Landau system \eqref{eq:amplitude-system-hex}. Specifically, if an equilibrium is temporally stable, the corresponding equilibrium in the spatial system is a saddle with even splitting between stable and unstable directions. Moreover, a temporally-unstable equilibrium corresponds to a stable equilibrium in the spatial system. Therefore, understanding a three-dimensional eigenvalue problem is sufficient to determine the stability structure in the six-dimensional spatial system.

Subsequently, we discuss the existence of heteroclinic orbits between equilibria analytically. Outside of the invariant subspace of roll waves $\{A_1, B_1 \in \R, A_2 = B_2 = A_3 = B_3 = 0\}$, where the leading-order system \eqref{eq:leading-order-reduced-equations} reduces to a damped Duffing oscillator, this requires understanding the dynamics in the full six-dimensional phase space. Exploiting that the leading-order system \eqref{eq:leading-order-reduced-equations} has a strictly decreasing Lyapunov function, for $\mu_0>0$, we can rigorously show the existence of heteroclinic orbits connecting the non-trivial equilibrium with the lowest energy (either down-hexagons or roll waves) to the trivial equilibrium, which persist in the reduced equations \eqref{eq:reduced-equations-explicit} for $\eps>0$. In the pattern-forming system \eqref{eq:Swift-Hohenberg}, these correspond to pattern interfaces, where the spatially homogeneous state is invaded by a planar pattern. Using fast-slow analysis, we then show that these orbits also exist in the reduced system for $\theta = \tfrac{\pi}{6}$, which can be interpreted as the slow subsystem of \eqref{eq:reduced-equations-explicit} as $\theta \to \tfrac{\pi}{6}$.

\subsection{Equilibria and their stability}

\subsubsection{Equilibria}

First, we consider the equilibria of the reduced equations \eqref{eq:leading-order-reduced-equations} and their stability. We point out that the existence of equilibria does not depend on the direction of the front since equilibria of the reduced equations correspond to steady patterns in the Swift–Hohenberg equation \eqref{eq:Swift-Hohenberg}, which is rotationally invariant. In particular, the results also apply to the leading-order system of the reduced equations \eqref{eq:reduced-equations-pi-over-six-second-order}, when $\theta = \tfrac{\pi}{6}$. Therefore, equilibria are given as solutions to 
\begin{equation}\label{eq:reduced-equations-stationary}
    \begin{split}
        \mu_0 A_1 + \beta_2\bar{A}_2\bar{A}_3 + (K_0|A_1|^2 + K_2(|A_2|^2 + |A_3|^2)) A_1  & = 0, \\
        \mu_0 A_2 + \beta_2\bar{A}_1\bar{A}_3 + (K_0|A_2|^2 + K_2(|A_1|^2 + |A_3|^2)) A_2  & = 0, \\
        \mu_0 A_3 + \beta_2\bar{A}_1\bar{A}_2 + (K_0|A_3|^2 + K_2(|A_1|^2 + |A_2|^2)) A_3  & = 0,
    \end{split}
\end{equation}
with $A_1, A_2, A_3 \in \C$. The system of equations \eqref{eq:reduced-equations-stationary} has been well-studied in the literature on equivariant bifurcation theory, and we refer to \cite{buzano1983PhilosTransAMathPhysEngSci,golubitsky1984-03PhysicaDNonlinearPhenomena,hoyle2007book} for a full characterisation of solutions to reduced equations obtained from a hexagonal symmetry with translation invariance. In the context of this paper, we note that \eqref{eq:reduced-equations-stationary} is cubic to leading order, with small higher-order terms. Therefore, we can follow the discussion in \cite[Sec.~5.4]{hoyle2007book} on cubic truncations, and we find that, to leading order, the system \eqref{eq:reduced-equations-stationary} has only
\begin{enumerate}[label=(\alph*)]
    \item the trivial solution $A_1 = A_2 = A_3 = 0$;
    \item roll waves with $|A_1| > 0$ and $A_2 = A_3 = 0$;
    \item hexagons with $A_1 = A_2 = A_3 \neq 0$;
    \item mixed modes (or rectangles) with $A_2 = A_3$ and $A_2 \neq A_1 \neq 0$
\end{enumerate}
as well as solutions of the same form obtained from $ D_6$ symmetry. A simple calculation shows that no solution satisfies $|A_2| < |A_1| < |A_3|$. These solutions can be calculated explicitly, see e.g.~\cite{hilder2025-08JNonlinearSci} and we summarise the results in the following corollary.

\begin{corollary}\label{cor:stationary-patterns}
    The equation \eqref{eq:reduced-equations-stationary} has the following solutions:
    \begin{enumerate}[label=(\alph*)]
        \item The trivial solution
        \begin{equation*}
            \Ab_T := (A_{1,T},A_{2,T},A_{3,T}) = (0,0,0).
        \end{equation*}
        \item Roll waves
        \begin{equation}\label{eq:roll-waves}
            \Ab_R = (A_{1,R},A_{2,R},A_{3,R}) = \Big(\pm \sqrt{-\frac{\mu_0}{K_0}},0,0\Big)
        \end{equation}
        if $\mu_0 K_0 < 0$.
        \item Hexagons
        \begin{equation}\label{eq:hexagons}
            \begin{split}
                \Ab_{H_\pm} & := (A_{1,H_{\pm}},A_{2,H_{\pm}},A_{3,H_{\pm}}) \\
                & = \Big( \frac{-\beta_2 \mp \sqrt{\beta_2^2 - 4\mu_0 (K_0 + 2 K_2)}}{2 (K_0 + 2K_2)}, \frac{-\beta_2 \mp \sqrt{\beta_2^2 - 4\mu_0 (K_0 + 2 K_2)}}{2 (K_0 + 2K_2)}, \frac{-\beta_2 \mp \sqrt{\beta_2^2 - 4\mu_0 (K_0 + 2 K_2)}}{2 (K_0 + 2K_2)}\Big)
            \end{split}
        \end{equation}
        if $K_0 + 2 K_2 \neq 0$ and $\beta_2^2 - 4\mu_0 (K_0 + 2 K_2) > 0$. If $A_{1,H_{\pm}} > 0$ the solutions are called \emph{up-hexagons} and if $A_{1,H_{\pm}} <  0$, they are called \emph{down-hexagons}.
        \item Mixed modes and false hexagons
        \begin{equation*}
            \begin{split}
               \Ab_{\MM} & := (A_{1,\MM}, A_{2,\MM}, A_{3,\MM})\\
               & = \Big(\frac{\beta_2}{K_0 - K_2}, \sqrt{-\frac{K_0 \beta_2^2 + (K_0 - K_2)^2 \mu_0}{(K_0 + K_2)(K_0-K_2)^2}}, \sqrt{-\frac{K_0 \beta_2^2 + (K_0 - K_2)^2 \mu_0}{(K_0 + K_2)(K_0-K_2)^2}}\Big)
            \end{split}
        \end{equation*}
        if $\tfrac{K_0 \beta_2^2 + (K_0 - K_2)^2 \mu_0}{K_0 + K_2} < 0$. In addition, if $- \tfrac{K_0 \beta_2^2}{(K_0 - K_2)^2} < \mu_0 < - \tfrac{\beta_2^2 (2K_0 + K_2)}{(K_0 - K_2)^2}$ these solutions are called \emph{mixed modes}. For larger $\mu_0$ they are called \emph{false hexagons}.
    \end{enumerate}
\end{corollary}

\begin{remark}
    Note that the equilibria of the reduced equations \eqref{eq:reduced-equation-CM-theorem} and \eqref{eq:reduced-equation-CM-theorem-pi-over-6} correspond to steady patterns of the Swift–Hohenberg equation \eqref{eq:Swift-Hohenberg} which are, to leading order, given by
    \begin{equation*}
        u(t,x) = \eps \sum_{j = 1}^3 A_j e^{i\k_j\cdot \x} + c.c. + \Ocal(\eps^2) = \eps \sum_{j = 1}^3 2 A_j \cos(\k_j \cdot \x) + \Ocal(\eps^2),
    \end{equation*}
    see Figure \ref{fig:stationary-patterns}.
\end{remark}

\begin{figure}[H]
    \centering
    \begin{subfigure}[p]{0.3\textwidth}
    \includegraphics[width=\linewidth]{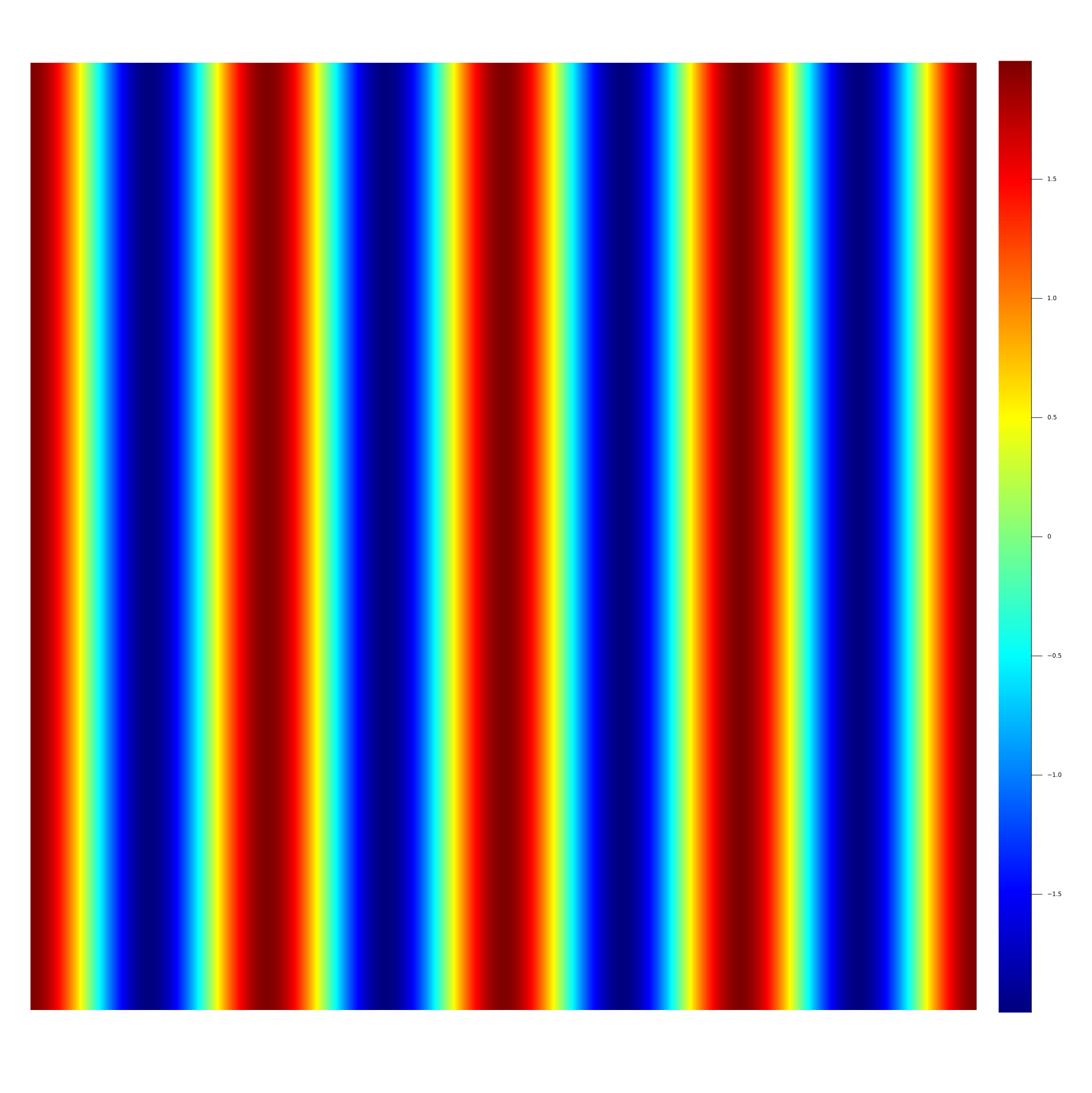}
    \subcaption{roll waves}
    \end{subfigure}
    \hfill
    \begin{subfigure}[p]{0.3\textwidth}
    \includegraphics[width=\linewidth]{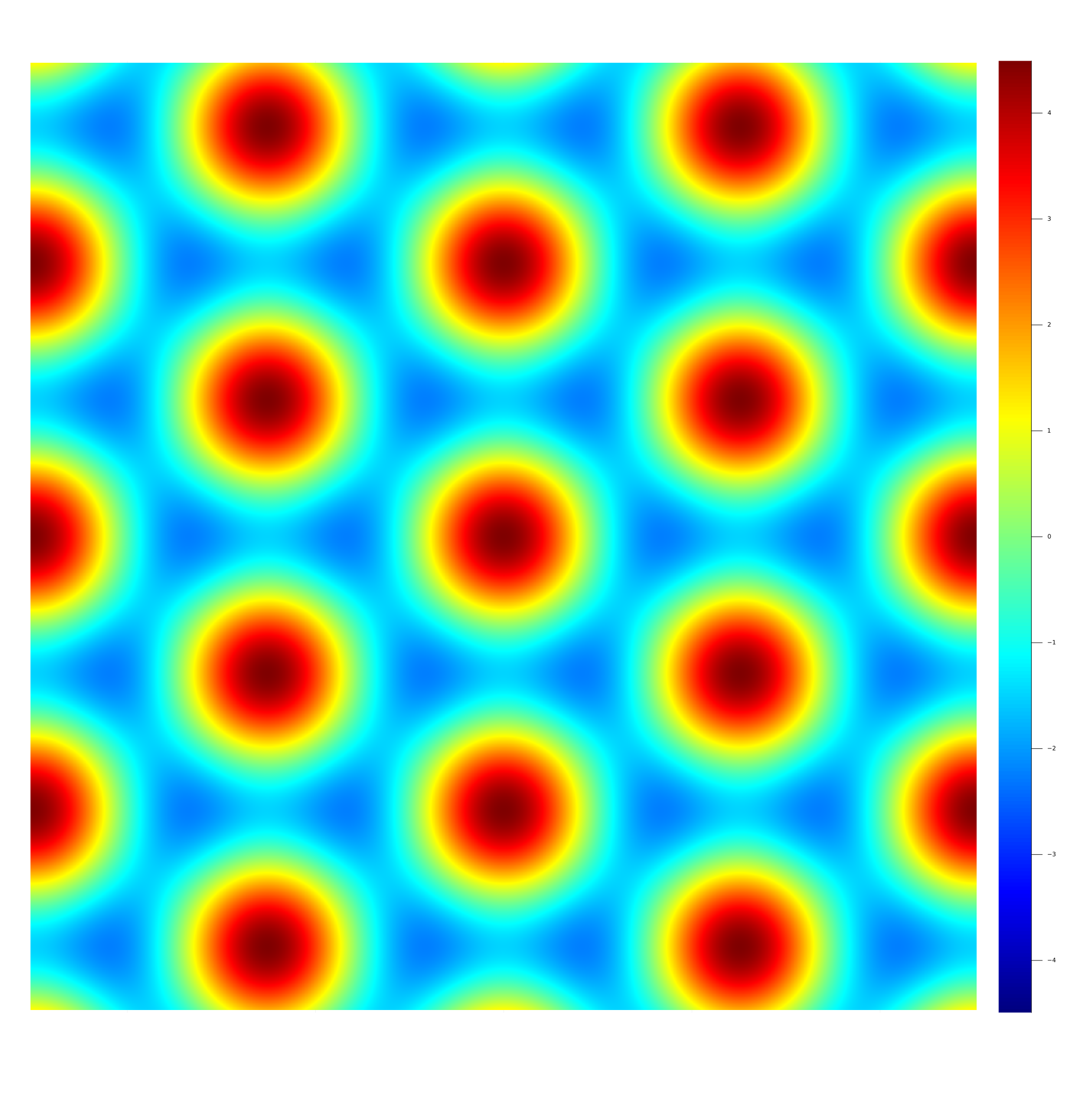}
    \subcaption{up-hexagons}
    \end{subfigure}
    \hfill
    \begin{subfigure}[p]{0.3\textwidth}
    \includegraphics[width=\linewidth]{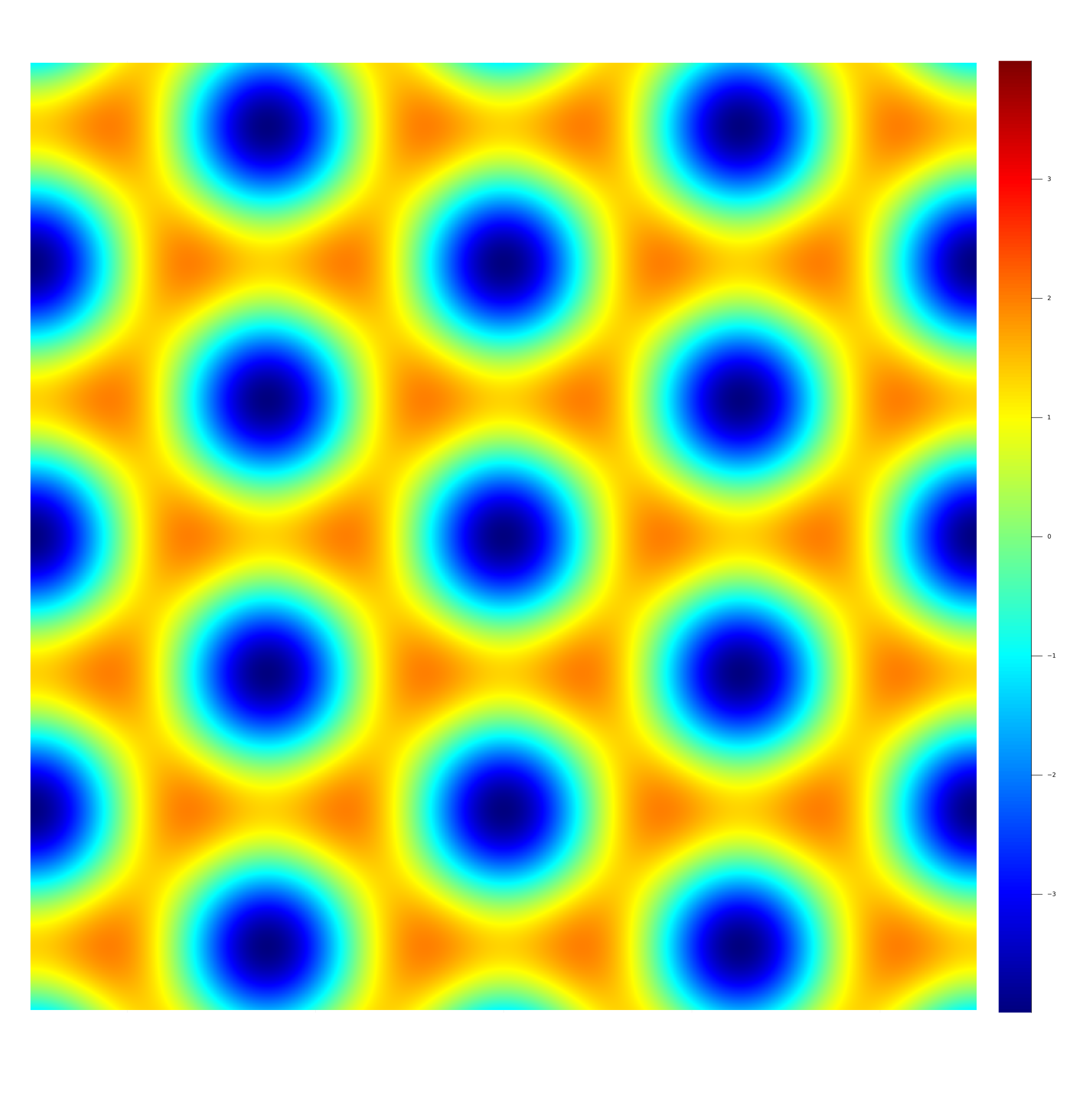}
    \subcaption{down-hexagons}
    \end{subfigure}

    \vspace{0.2cm}

    \begin{subfigure}[p]{0.3\textwidth}
    \includegraphics[width=\linewidth]{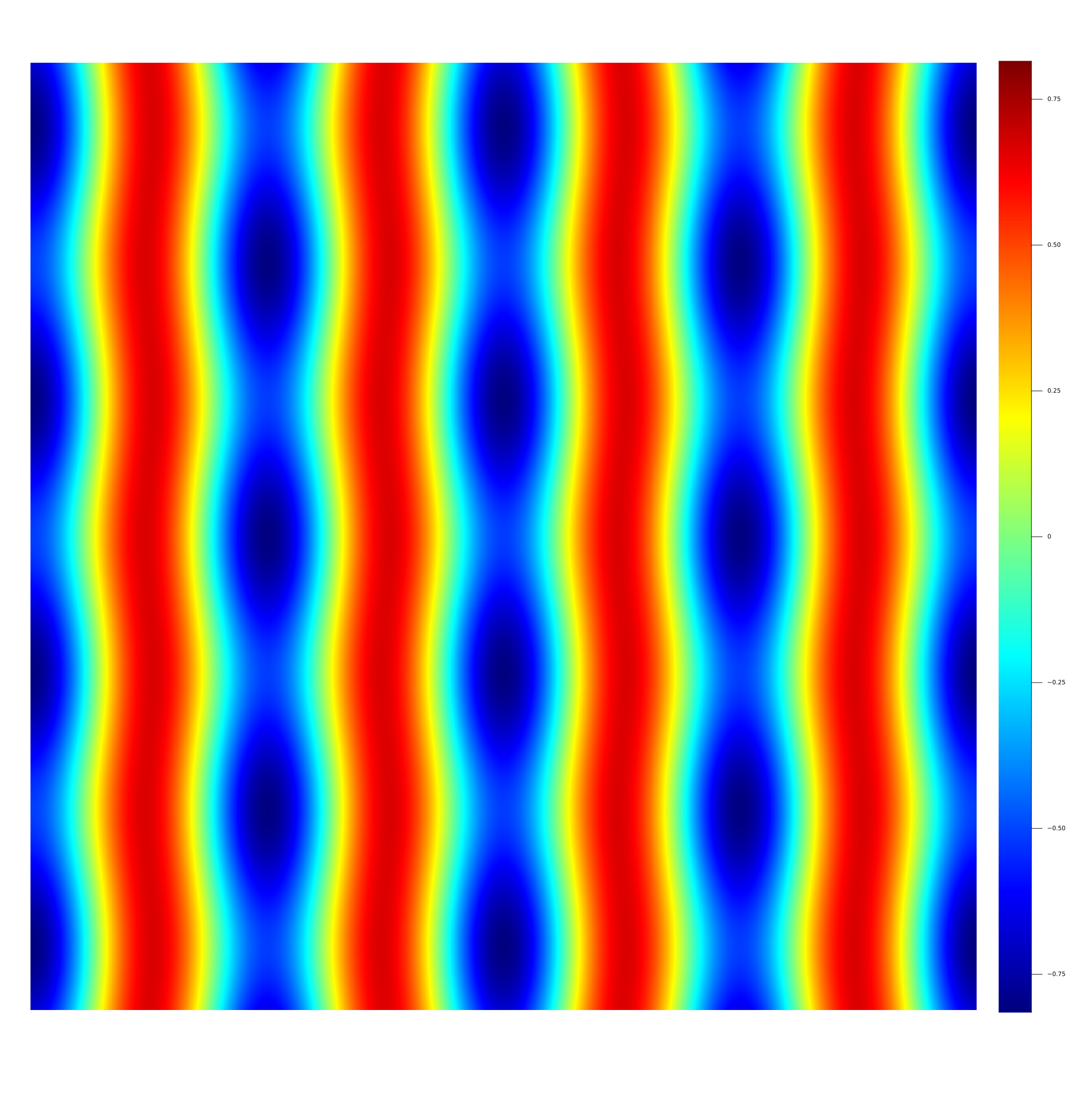}
    \subcaption{mixed modes}
    \label{subfig:mm-1}
    \end{subfigure}
    \hfill
    \begin{subfigure}[p]{0.3\textwidth}
    \includegraphics[width=\linewidth]{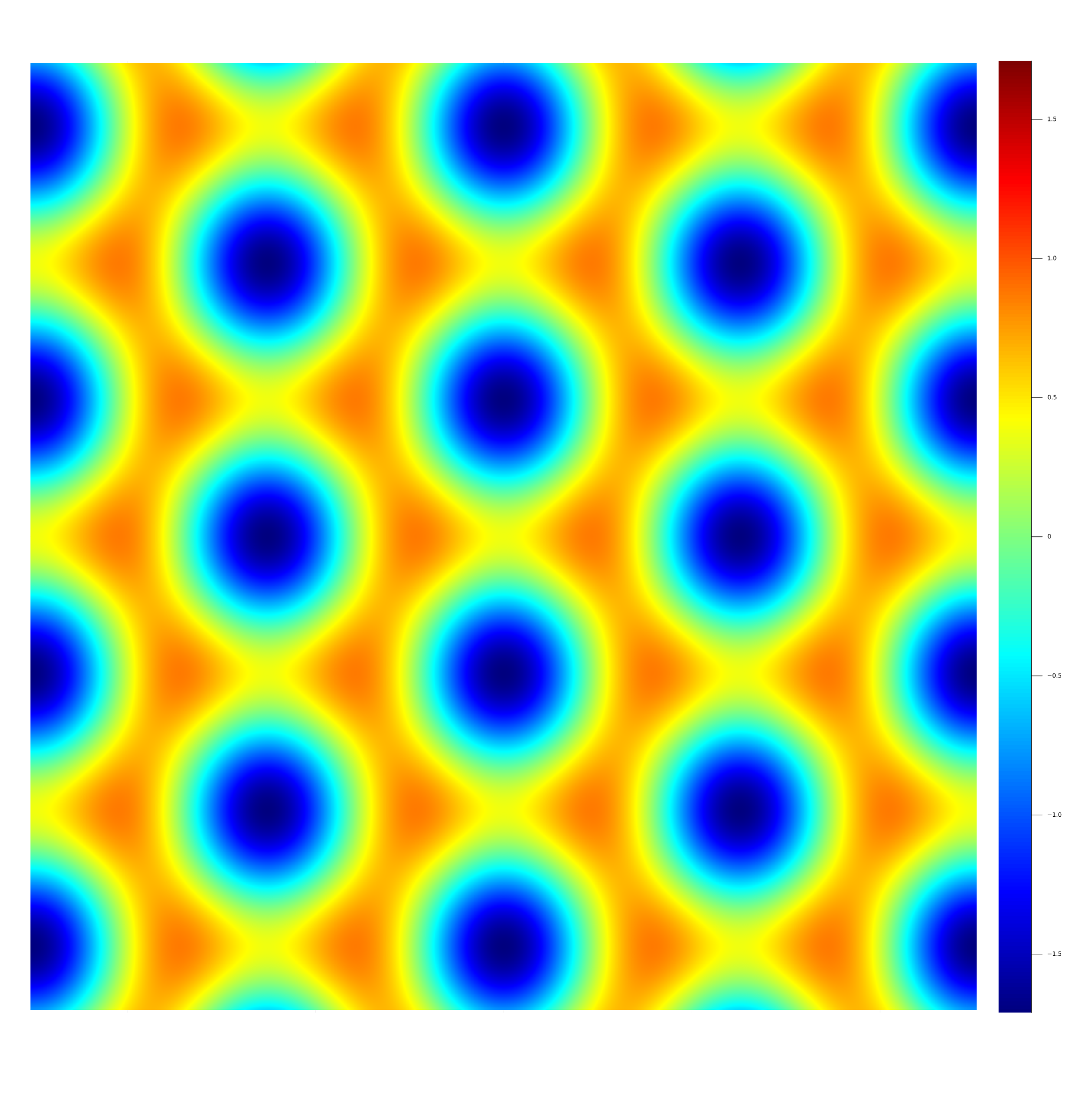}
    \subcaption{mixed modes}
    \label{subfig:mm-2}
    \end{subfigure}
    \hfill
    \begin{subfigure}[p]{0.3\textwidth}
    \includegraphics[width=\linewidth]{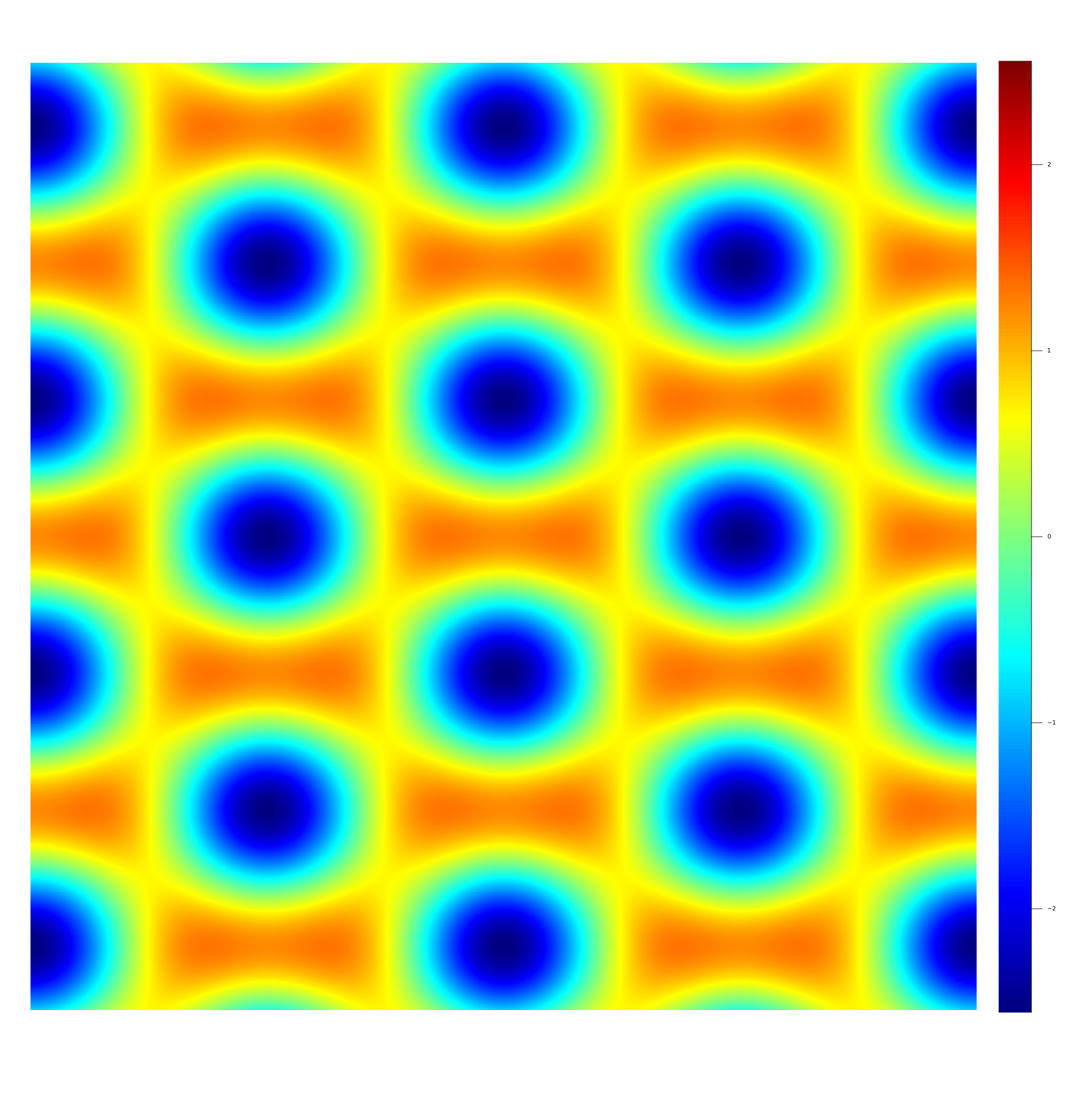}
    \subcaption{false hexagons}
    \label{subfig:mm-3}
    \end{subfigure}
   
    \caption{Depiction of all stationary patterns obtained in \Cref{cor:stationary-patterns}. Panels \sref{subfig:mm-1}, \sref{subfig:mm-2} and \sref{subfig:mm-3} show different variations of patterns on the mixed modes and false hexagons branch for increasing values of $\mu_0$.}
    \label{fig:stationary-patterns}
\end{figure}

\begin{remark}
    We point out that, due to translational and reflectional invariance, for any solution $(A_1,A_2,A_3)$ to \eqref{eq:reduced-equations-stationary}, we find a corresponding solution given by $(A_1,-A_2,-A_3)$. This invariance is obtained by shifting by a full hexagon in $x_2$-direction, which maps $\k_3 \mapsto \k_2$ and $-\k_2 \mapsto -\k_3$, and then applying reflection symmetry along the $x_2$-axis.
\end{remark}

\subsubsection{Temporal vs.~spatial stability}

We now discuss the stability of the equilibria listed in \Cref{cor:stationary-patterns}. For this, we introduce some notation. Let $\Ab = (A_1,A_2,A_3)$ be a solution to \eqref{eq:reduced-equations-stationary}. Then, we denote the linearisation of the corresponding equilibrium $(A_1,0,A_2,0,A_3,0)$ of the first-order spatial-dynamics system \eqref{eq:reduced-equations-explicit} by $\LinSD(\Ab)$. The eigenvalue structure of $\LinSD(\Ab)$ is characterised by the spectral stability of the equilibrium as a solution to the system of Landau equations 
\begin{equation}\label{eq:Landau-equations}
    \begin{split}
        \partial_T A_1 &= \mu_0 A_1 + \beta_2\bar{A}_2\bar{A}_3 + (K_0|A_1|^2 + K_2(|A_2|^2 + |A_3|^2)) A_1, \\
        \partial_T A_2 &= \mu_0 A_2 + \beta_2\bar{A}_1\bar{A}_3 + (K_0|A_2|^2 + K_2(|A_1|^2 + |A_3|^2)) A_2, \\
        \partial_T A_3 &= \mu_0 A_3 + \beta_2\bar{A}_1\bar{A}_2 + (K_0|A_3|^2 + K_2(|A_1|^2 + |A_2|^2)) A_3,
    \end{split}
\end{equation}
which is the system of amplitude equations \eqref{eq:amplitude-system-hex} without the spatial derivatives. Note that since the (degenerate) diffusion operator in \eqref{eq:amplitude-system-hex} is diagonal, spectral stability in \eqref{eq:Landau-equations} implies spectral stability in \eqref{eq:amplitude-system-hex}. Therefore, we refer to this as \emph{PDE stability}. The following lemma then relates the PDE stability of an equilibrium to its stability in the leading-order system \eqref{eq:leading-order-reduced-equations} on the centre manifold. This relation was already noted in \cite{doelman2003-02EuropeanJournalofAppliedMathematics}, and we provide a proof for completeness.

\begin{lemma}\label{lem:characterisation-spat-dyn-stable-vs-pde-stable}
    Let $\theta \in (0,\tfrac{\pi}{6})$ and $\Ab = (A_1,A_2,A_3) \in \R^3$ be an equilibrium of the spatial-dynamics system \eqref{eq:leading-order-reduced-equations}. If the linearisation of the Landau equations \eqref{eq:Landau-equations} has $n$ unstable eigenvalues with $0 \leq n \leq 3$ and $3 - n$ stable eigenvalues, the linearisation of the spatial-dynamics system $\LinSD(\Ab)$ has $3+ n$ stable eigenvalues and $3-n$ unstable eigenvalues. In particular, if $\Ab$ is PDE-stable, then $\LinSD(\Ab)$ has three eigenvalues with negative real parts and three with positive real parts.
\end{lemma}
\begin{proof}
    We first note that $\lambda \in \C$ is an eigenvalue of $\LinSD(\Ab)$ with eigenvector $\hat{\psib} \in \C^6$ if and only if the quadratic matrix equation
    \begin{equation}\label{eq:quadratic-eigenvalue-problem}
        (\Dcal \lambda^2 + c_0 I \lambda + \LinGL(\Ab))\psib = 0,
    \end{equation}
    where the matrix $\Dcal \in \R^{3\times 3}$ encodes the diffusion coefficients, i.e.,
    \begin{equation*}
        \Dcal = \begin{pmatrix}
            4 (\d \cdot \k_1)^2 & 0 & 0 \\
            0 & 4 (\d \cdot \k_2)^2 & 0 \\
            0 & 0 & 4 (\d \cdot \k_3)^2
        \end{pmatrix},
    \end{equation*}
    and the matrix $\LinGL(\Ab) \in \R^{3 \times 3}$ is the linearisation of the Landau system \eqref{eq:Landau-equations} about the equilibrium $\Ab$ and given by
    \begin{equation}\label{eq:linearisation-landau-system}
        \LinGL(\Ab) = \begin{pmatrix}
            \mu_0 + 3 K_0 A_1^2 + K_2 (A_2^2 + A_3^2) & \beta_2 A_3 + 2 K_2 A_1 A_2 & \beta_2 A_2 + 2 K_2 A_1 A_3 \\
            \beta_2 A_3 + 2 K_2 A_1 A_2 & \mu_0 + 3 K_0 A_2^2 + K_2 (A_1^2 + A_3^2) & \beta_2 A_1 + 2 K_2 A_2 A_3 \\
            \beta_2 A_2 + 2 K_2 A_1 A_3 & \beta_2 A_1 + 2 K_2 A_2 A_3 & \mu_0 + 3 K_0 A_3^2 + K_2 (A_1^2 + A_2^2)
        \end{pmatrix}.
    \end{equation}
    Additionally, $\psib \in \C^3$ is related to $\hat{\psib}$ via
    \begin{equation*}
            \hat{\psib} = \begin{pmatrix}
                                1 & 0 & 0 \\
                                \lambda & 0 & 0 \\
                                0 & 1 & 0 \\
                                0 & \lambda & 0 \\
                                0 & 0 & 1 \\
                                0 & 0 & \lambda
                            \end{pmatrix}\psib.
    \end{equation*}
    The equation \eqref{eq:quadratic-eigenvalue-problem} is known as a quadratic eigenvalue problem, see e.g.~\cite{tisseur2001-01SIAMRev}, and the equivalence to the spatial-dynamics eigenvalue problem follows with the same algebraic manipulations as in \cite[Lem.~5.3]{hilder2025-08JNonlinearSci}.

    We first note that for $0<\theta < \tfrac{\pi}{6}$, $\Dcal$ is positive-definite. Next, we note some properties of the quadratic eigenvalue problem \eqref{eq:quadratic-eigenvalue-problem}, which follow from results in \cite{tisseur2001-01SIAMRev}. Since the matrix $\Dcal$ is positive-definite, \eqref{eq:quadratic-eigenvalue-problem} is non-singular. In particular, this yields that \eqref{eq:quadratic-eigenvalue-problem} has six finite eigenvalues, as expected. Additionally, $\LinGL(\Ab)$ is symmetric. Therefore, \eqref{eq:quadratic-eigenvalue-problem} is self-adjoint, and the eigenvalues $\lambda$ are either real-valued or appear in complex-conjugate pairs. Finally, multiplying \eqref{eq:quadratic-eigenvalue-problem} from the left by $\psib^*$ and normalising the eigenvector $\psib$ such that $\psib^* \psib = 1$, we obtain that $\lambda$ satisfies the quadratic polynomial
    \begin{equation}\label{eq:qep-polynomial}
        \psib^* \Dcal \psib \lambda^2 + c_0 \lambda + \psib^* \LinGL(\Ab) \psib = 0,
    \end{equation}
    where $c_0 > 0$ by assumption and $\psib^* \Dcal \psib > 0$ since $\Dcal$ is positive-definite.

    We now prove both extreme cases of the lemma, i.e. if $\LinGL(\Ab)$ has $n = 0$ or $n = 3$ unstable eigenvalues. In the former case ($n = 0$), we find that $\LinGL(\Ab)$ is negative-definite and therefore $\psib^* \LinGL(\Ab) \psib < 0$. Therefore, solutions to the polynomial \eqref{eq:qep-polynomial} come in pairs with positive and negative real parts. Using the equivalence of \eqref{eq:quadratic-eigenvalue-problem} to the linear eigenvalue problem for $\LinSD$, we obtain that $\LinSD$ has three stable and three unstable eigenvalues. This proves the case $n = 0$. The case $n = 3$ follows with the same arguments by noting that $\LinGL(\Ab)$ is now positive-definite and thus $\psib^* \LinGL(\Ab) \psib > 0$, which yields that all solutions to \eqref{eq:qep-polynomial} have negative real parts. Therefore, $\LinSD$ has six stable and no unstable eigenvalues.

    It thus remains to prove the intermediate cases $n = 1,2$ where $\LinGL(\Ab)$ has both stable and unstable eigenvalues and is thus indefinite. For this, we define the shifted quadratic eigenvalue problem 
    \begin{equation}\label{eq:shifted-quadratic-eigenvalue-problem}
        (\Dcal \lambda^2 + c_0 \lambda I  + \LinGL(\Ab) + \kappa I)\psib = 0,
    \end{equation}
    with $\kappa \in \R$. Since $\kappa I$ shifts the eigenvalues of $\LinGL(\Ab)$ by $\kappa$, we find that for $\kappa \ll 0$, all eigenvalues of $\LinGL(\Ab)$ have negative real part ($n = 0$) while for $\kappa \gg 0$ all eigenvalues of $\LinGL(\Ab)$ have positive real part ($n = 3$). Therefore, \eqref{eq:shifted-quadratic-eigenvalue-problem} interpolates between the two extreme cases. Again, the solutions $\lambda$ to \eqref{eq:shifted-quadratic-eigenvalue-problem} are also solutions to
    \begin{equation}\label{eq:shifted-qep-polynomial}
        \psib^*\Dcal \psib \lambda^2 + c_0 \lambda I + \psib^*\LinGL(\Ab)\psib + \kappa = 0.
    \end{equation}
    Note that there are at most three values of $\kappa$ where \eqref{eq:shifted-qep-polynomial} has solutions on the imaginary axis, i.e. where at least one eigenvalue crosses from the right to the left complex half-plane as $\kappa$ increases through the critical value. Since $c_0 \in \R \setminus \{0\}$, the explicit solution formula for \eqref{eq:shifted-qep-polynomial} shows that \eqref{eq:shifted-qep-polynomial} cannot have solutions in $i \R \setminus \{0\}$. Therefore, all eigenvalues have to cross through $\lambda = 0$.

    Next, we note that \eqref{eq:shifted-quadratic-eigenvalue-problem} has an eigenvalue $\lambda = 0$ if and only if $\LinGL + \kappa I$ has an eigenvalue $\lambda = 0$. In this case, $\psib$ is also an eigenvector of $\LinGL + \kappa I$ corresponding to the eigenvalue $\lambda = 0$. Therefore, as $\kappa$ increases, there is a one-to-one correspondence between eigenvalues of $\LinGL + \kappa I$ crossing from the right to the left complex half-plane and solutions to \eqref{eq:shifted-quadratic-eigenvalue-problem} crossing from the right to the left complex half-plane. Therefore, increasing $\kappa$ from $-\infty$ to $\kappa = 0$ and counting the eigenvalues in $\LinGL + \kappa I$ crossing the imaginary axis shows that if there are $n$ unstable eigenvalues of $\LinGL$, there are $3+n$ stable eigenvalues of the quadratic eigenvalue problem \eqref{eq:quadratic-eigenvalue-problem}, and, by equivalence, $3+ n$ stable eigenvalues of $\LinSD(\Ab)$. This completes the proof.
\end{proof}

\begin{remark}
    The proof of \Cref{lem:characterisation-spat-dyn-stable-vs-pde-stable} shows that the assumption $c_0 > 0$ plays a crucial role for the eigenvalue-structure of $\LinSD(\Ab)$. In particular, in the case $c_0 < 0$ we find that the eigenvalues of $\LinSD(\Ab)$ cross from stable to unstable (instead from unstable to stable) as $\Ab$ becomes PDE-unstable. In particular, if $\Ab$ is fully PDE-unstable, i.e.~$\LinGL(\Ab)$ has only unstable eigenvalues, then $(A_1,0,A_2,0,A_3)$ is an unstable equilibrium in the leading-order system \eqref{eq:leading-order-reduced-equations}.
\end{remark}

\begin{remark}\label{rem:spectral-characterisation-for-RD}
    Although \cref{lem:characterisation-spat-dyn-stable-vs-pde-stable} is formulated for the leading-order problem \eqref{eq:leading-order-reduced-equations}, we note that the result applies more generally to travelling-wave solutions in reaction-diffusion systems of the form
    \begin{equation*}
        \partial_t u = \Dcal \partial_x^2 u + f(u)
    \end{equation*}
    with $u(t,x) \in \R^{m}$, $\Dcal \in \R^{m \times m}$ and a smooth function $f : \R^{m} \to \R^{m}$, for $m \in \N$. Assume that $u^\ast \in \R^m$ is an equilibrium. The key ingredients to extend the proof of \cref{lem:characterisation-spat-dyn-stable-vs-pde-stable} are that the diffusion matrix $\Dcal$ is positive-definite and that the linearisation about $u^\ast$ given by $\DD f(u^\ast)$ is symmetric. While the first condition is typically satisfied since it is equivalent to parabolicity of the system, the second condition is restrictive. Moreover, the result can fail without this symmetry condition, as the following example shows. Let
    \begin{equation*}
        \Dcal = \begin{pmatrix}
            1 & 0 \\
            0 & 2
        \end{pmatrix}, \quad \text{and} \quad \DD f(u^\ast) = \dfrac{1}{2} \begin{pmatrix}
            1 & -1 \\ 3 & 1
        \end{pmatrix}.
    \end{equation*}
    The eigenvalues of $Df(u^\ast)$ are complex conjugate with real part $\tfrac{1}{2}$. Therefore, $u^\ast$ is PDE-unstable, and we would expect that the linearisation of the corresponding spatial-dynamics system for travelling waves with speed $c = 1$ has only stable eigenvalues. However, an explicit computation shows that the corresponding quadratic eigenvalue problem $(\Dcal \lambda^2 + I \lambda + \DD f(u^\ast))\psib$ has two complex-conjugate eigenvalues with negative real part and two complex-conjugate eigenvalues with positive real part. We point out that the different diffusion coefficients in this counter-example are necessary. Indeed, if $\Dcal = \alpha I$ for $\alpha > 0$, then eigenvectors of the quadratic eigenvalues problem $(\Dcal \lambda^2 + I \lambda + \DD f(u^\ast))\psib$ and eigenvectors of $\DD f(u^\ast)$ coincide and the proof of \Cref{lem:characterisation-spat-dyn-stable-vs-pde-stable} applies.
\end{remark}

\begin{remark}
    Note that the mechanism obtained in \Cref{lem:characterisation-spat-dyn-stable-vs-pde-stable} resembles the behaviour of spatial eigenvalues in \Cref{prop:spectrum-eps-positive}, where Fourier modes which are 'PDE-stable' induce an even splitting of the spatial eigenvalues into stable and unstable eigenvalues for $\eps > 0$, whereas the more-central eigenvalues connected to 'PDE-unstable' Fourier modes all have negative real part for $\eps > 0$. The mechanism for the more-central modes is indeed very similar, in fact, the quadratic polynomial for the $\Ocal(\eps)$-terms of the more-central eigenvalues is the same as the quadratic eigenvalue problem obtained for the trivial equilibrium from \Cref{lem:characterisation-spat-dyn-stable-vs-pde-stable}, cf.~\Cref{rem:analog-qep-trivial-more-central}. However, the mechanism for the even splitting of the less-central modes is different and relies on the fact that the less-central eigenvalues emerge from a non-trivial point on the imaginary axis, cf.~\Cref{app:general-spectral-analysis}.
\end{remark}

We now handle the case $\theta = \tfrac{\pi}{6}$. Here, we recall that the reduced equation on the centre manifold is given by \eqref{eq:reduced-equations-pi-over-six-second-order} and consider its leading-order system. This can be rewritten as a five-dimensional first-order system for the unknowns $(A_1, B_1, A_2, A_3, B_3) \in \R^5$. For an equilibrium $\Ab = (A_1,A_2,A_3)$ given in \Cref{cor:stationary-patterns} we then denote, by a slight abuse of notation, the linearisation about the corresponding equilibrium $(A_1,0,A_2,A_3,0)$ in the first-order spatial-dynamics system by $\LinSD(\Ab) \in \R^{5 \times 5}$.

\begin{lemma}\label{lem:characterisation-spat-dyn-stable-vs-pde-stable-pi-over-six}
    Let $\theta = \tfrac{\pi}{6}$ and $\Ab = (A_1,A_2,A_3) \in \R^3$ be an equilibrium of the leading-order system to \eqref{eq:reduced-equations-pi-over-six-second-order}. If the linearisation of the Landau equations \eqref{eq:Landau-equations} has $n$ unstable eigenvalues with $0 \leq n \leq 3$ and $3 - n$ stable eigenvalues, the linearisation of the spatial-dynamics system $\LinSD(\Ab)$ has $2 + n$ stable eigenvalues and $3-n$ unstable eigenvalues. In particular, if $\Ab$ is PDE-stable, then $\LinSD(\Ab)$ has three eigenvalues with negative and three eigenvalues with positive real part.
\end{lemma}

\begin{proof}
    By the same arguments as in \Cref{lem:characterisation-spat-dyn-stable-vs-pde-stable} $\lambda\in \C$ is an eigenvalue of $\LinSD(\Ab)$ with eigenvector $\hat{\psib} \in \C^{5}$ if and only if $(\lambda,\psib)$ is a solution to the quadratic eigenvalue problem \eqref{eq:qep-polynomial} with 
    \begin{equation*}
        \Dcal = \begin{pmatrix}
                3 & 0 & 0 \\
                0 & 0 & 0 \\
                0 & 0 & 3
            \end{pmatrix},
    \end{equation*}
    and $\psib \in \C^3$ is related to $\hat{\psib}$ via
    \begin{equation*}
            \hat{\psib} = \begin{pmatrix}
                                1 & 0 & 0 \\
                                \lambda & 0 & 0 \\
                                0 & 1 & 0 \\
                                0 & 0 & 1 \\
                                0 & 0 & \lambda
                            \end{pmatrix}\psib.
    \end{equation*}
    Note that since $\Dcal$ is singular, we cannot guarantee that $\psib^\ast \Dcal \psib > 0$ for an eigenvector $\psib$. Thus, while the structure of the proof is similar to \Cref{lem:characterisation-spat-dyn-stable-vs-pde-stable}, we need to modify the specific arguments.
    
    Computing $\deg \det (\Dcal \lambda^2 + c_0 I \lambda + \LinGL(\Ab)) = 5$ and following \cite{tisseur2001-01SIAMRev}, there are five finite eigenvalues and one infinite eigenvalue. In the case $n=3$ in which $\LinGL(\Ab)$ is positive-definite, we obtain again that solutions to \eqref{eq:qep-polynomial} need to have negative real parts and there are five stable eigenvalues.

    If $n=0$ and $\LinGL(\Ab)$ is negative-definite, we first transform the problem: there is an orthogonal matrix $S$ such that $S^T\LinGL(\Ab)S  = \Lambda$ where $\Lambda$ is diagonal and negative-definite. Applying this to \eqref{eq:qep-polynomial}, we obtain the quadratic matrix equation
    \begin{equation}\label{eq:qep-polynomial-after-change-of-basis}
        (\tilde{\Dcal} \lambda^2 + c_0 I \lambda + \Lambda) \tilde{\psib}= 0.
    \end{equation}
    where $\tilde{\psib} = S^T \psib$ and $\tilde{\Dcal}=S^T\Dcal S$ has rank two and is positive semidefinite and symmetric. With the same argument as in the proof of \Cref{lem:characterisation-spat-dyn-stable-vs-pde-stable}, for any negative-definite matrix $\Lambda$, \eqref{eq:qep-polynomial-after-change-of-basis} cannot have solutions on the imaginary axis. We now define the interpolation $\Lambda_t = - t I + (1-t) \Lambda$, which is strictly contained in the set of negative-definite matrices. Hence, the number of eigenvalues with positive and negative real parts to \eqref{eq:qep-polynomial-after-change-of-basis} with $\Lambda$ replaced by $\Lambda_t$  is independent of $t$, and so, we can also study
    \begin{equation}\label{eq:qep-polynomial-after-change-of-basis-and-homotopy}
        (\tilde{\Dcal} \lambda^2 + c_0 I \lambda -I) \check{\psib}= 0.
    \end{equation}
    Now, let $P_{\ker \tilde{\Dcal}}$ be the orthogonal projection onto the kernel of $\tilde{\Dcal}$ and assume $(\lambda,\check{\psib})$ is a solution to \eqref{eq:qep-polynomial-after-change-of-basis-and-homotopy}. If $P_{\ker \tilde{\Dcal}} \tilde{\psib} \neq 0$, then applying $P_{\ker \tilde{\Dcal}}$ to \eqref{eq:qep-polynomial-after-change-of-basis-and-homotopy} one obtains
    \begin{equation*}
        0=\lambda^2 P_{\ker \tilde{\Dcal}} \tilde{\Dcal}\tilde{\psib} + (c_0 \lambda - I)  P_{\ker \tilde{\Dcal}}\tilde{\psib} = (c_0 \lambda - I)  P_{\ker \tilde{\Dcal}}\tilde{\psib},
    \end{equation*}
    since the $\operatorname{image}(\tilde{\Dcal}) \perp \ker \tilde{\Dcal}$, using that $\tilde{\Dcal}$ is symmetric. But this gives $\lambda = \tfrac{1}{c_0} > 0$ and $\tilde{\Dcal} \tilde{\psib} = 0$ follows, so that $\tilde{\psib}\in \ker \tilde{\Dcal}$, which is one-dimensional, so that there is precisely one solution of this type. If $P_{\ker \tilde{\Dcal}} \tilde{\psib} = 0$, we can argue as before that there are two eigenvalues for each eigenvector $\tilde{\psib}$, one with positive and one with negative real part. So there are three unstable and two stable eigenvalues. The general case $0<n<3$ now again follows by the same argument as in the proof of \Cref{lem:characterisation-spat-dyn-stable-vs-pde-stable}.
\end{proof}

\subsubsection{Stability of equilibria}\label{sec:stability-of-equilibria}

As motivated above, \Cref{lem:characterisation-spat-dyn-stable-vs-pde-stable,lem:characterisation-spat-dyn-stable-vs-pde-stable-pi-over-six} allow for a simpler characterisation of the eigenvalue structure of the linearisation in the leading-order system \eqref{eq:leading-order-reduced-equations}. From this, we make the immediate observation that the eigenvalue structure does not depend on the direction $\d$. To obtain more detailed information about the stability structure of the equilibria in \Cref{cor:stationary-patterns}, we now consider their linearisation in the Landau system \eqref{eq:Landau-equations}.

We recall that the linearisation in the Landau system \eqref{eq:Landau-equations} about an equilibrium $\Ab$, denoted as above by $\LinGL(\Ab)$, is given by \eqref{eq:linearisation-landau-system}. First, we note that $\LinGL(\Ab_T) = \mu_0 I$ is diagonal. Therefore, for $\mu_0 < 0$, the corresponding equilibrium in the spatial-dynamics system is a saddle with 3 stable and 3 unstable eigenvalues, whereas for $\mu_0 > 0$ it is a stable equilibrium. More specifically, the quadratic eigenvalue problem decouples into three independent equations, given by
\begin{equation}\label{eq:eigenvalues-trivial-eq-reduced-equations}
    0 = 4 (\d \cdot \k_j)^2 \lambda^2 + c_0 \lambda + \mu_0, \quad j=1,2,3.
\end{equation}
This implies that the eigenvalues of $\LinSD(\Ab_T)$ are real for all $c_0 > c_{\crit}(\k_j)$. As $c_0$ decreases through $c_{\crit}(\k_j)$, two eigenvalues collide and then split into a pair of complex-conjugate eigenvalues. In particular, this changes the qualitative dynamics of orbits close to the trivial equilibrium. For $c_0 > c_{\crit}(\k_j)$ for all $j=1,2,3$, orbits decay into the trivial state monotonically, whereas for $c_0 < c_{\crit}(\k_j)$, the trivial equilibrium is a stable centre and oscillations occur, see \Cref{fig:different-front-behaviour}. This is well-known in the one-dimensional case, see e.g.~\cite{hilder2020-08JournalofDifferentialEquations}, and similar behaviour occurs for other monostable fronts, for example, in the Fisher–KPP equation \cite[Sec.~7.1]{schneider2017book}.

\begin{figure}[H]
    \centering
    \begin{subfigure}[b]{0.33\textwidth}
        \includegraphics[width=\textwidth]{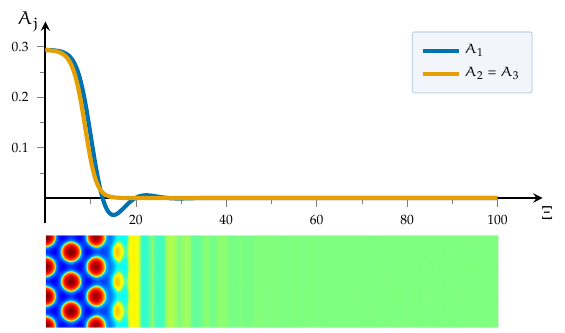}
        \subcaption{Subcritical front speed}
        \label{subfig:subcrit}
    \end{subfigure}
    \hfill
    \begin{subfigure}[b]{0.33\textwidth}
        \includegraphics[width=\textwidth]{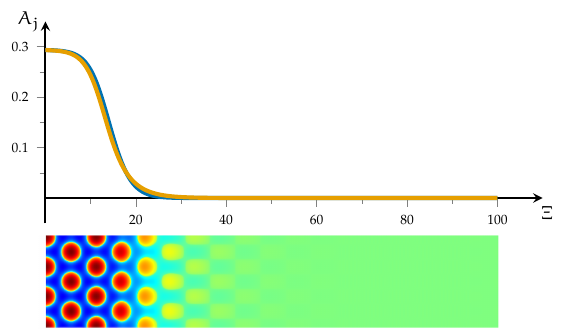}
        \subcaption{Critical front speed}
        \label{subfig:crit}
    \end{subfigure}
    \hfill
    \begin{subfigure}[b]{0.33\textwidth}
        \includegraphics[width=\textwidth]{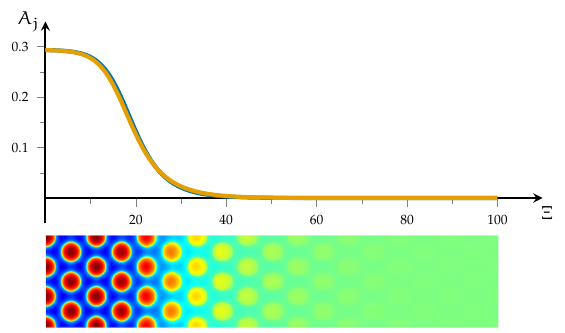}
        \subcaption{Supercritical front speed}
        \label{subfig:supercrit}
    \end{subfigure}
    \caption{Numerical simulations of heteroclinic orbits between $(\Ab_{H_+},\zerob)$ and $(\Ab_{T},\zerob)$ and the corresponding pattern interfaces connecting up-hexagons to the trivial state for increasing values of $c_0$. \sref{subfig:subcrit} depicts the subcritical regime $c_0 < \max_j(c_{\crit}(\k_j))$ and the orbit oscillates close to $(\Ab_{T},\zerob)$ which is a stable centre in this regime. \sref{subfig:crit} depicts the critical speed $c_0 = \max_j(c_{\crit}(\k_j))$, where the front decays monotonically. Increasing $c$ into the supercritical regime leads to a less steep monotone front profile. This is consistent with monostable fronts in reaction-diffusion systems, where the critical front is the steepest, see \cite{avery2025-12preprint}. Observe that, for illustration purposes, we have amplified the pattern interfaces close to the trivial state.}
    \label{fig:different-front-behaviour}
\end{figure}

\begin{remark}\label{rem:analog-qep-trivial-more-central}
    We highlight that \eqref{eq:eigenvalues-trivial-eq-reduced-equations} recovers the leading-order equation for the $\Ocal(\eps)$-contribution of the more-central eigenvalues in \Cref{prop:spectrum-eps-positive}, cf.~\eqref{eq:equation-for-correction-more-central-eigenvalues}.
\end{remark}

For the other equilibria, we note that these can be grouped into two cases: $\Ab = (A,A,A)$ with $A \in \R \setminus \{0\}$ and $\Ab = (A,B,B)$ with $A \in \R \setminus \{0\}$ and $B \in \R$. The former group consists of hexagons, and the latter encompasses roll waves, mixed modes, and false hexagons. Thus, the corresponding linearisations have the general structure
\begin{equation*}
    \LinGL((A,A,A)) = \begin{pmatrix}
        a & b & b \\ b & a & b \\ b & b & a
    \end{pmatrix}
    \quad \text{and} \quad \LinGL((A,B,B)) = \begin{pmatrix}
        a & e & e \\ e & b & f \\ e & f & b
    \end{pmatrix},
\end{equation*}
with $a,b,e,f \in \R$, respectively. From these general forms, we can directly obtain that the eigenvectors of $\LinGL((A,A,A))$ are given by
\begin{equation*}
    \begin{pmatrix}
        1 \\ 1 \\ 1
    \end{pmatrix}, \begin{pmatrix}
        0 \\ 1 \\ -1
    \end{pmatrix}, \text{ and } \begin{pmatrix}
        2 \\ -1 \\ -1
    \end{pmatrix} \quad \text{with eigenvalues } a+2b, a-b, a-b
\end{equation*}
and the eigenvectors of $\LinGL((A,B,B))$ are given by
\begin{equation*}
    \begin{pmatrix}
        1 \\ -\frac{2 e}{\sqrt{(-a+b+f)^2+8
   e^2}-a+b+f} \\ -\frac{2 e}{\sqrt{(-a+b+f)^2+8
   e^2}-a+b+f}
    \end{pmatrix}, \text{ and } \begin{pmatrix}
        \frac{\sqrt{(-a+b+f)^2+8 e^2}+a-b-f}{2 e} \\ 1 \\ 1
    \end{pmatrix}, \begin{pmatrix}
        0 \\ 1 \\ -1
    \end{pmatrix}
\end{equation*}
with eigenvalues $\tfrac{1}{2} \left(-\sqrt{(-a+b+f)^2+8 e^2}+a+b+f\right)$, $\tfrac{1}{2} \left(\sqrt{(-a+b+f)^2+8 e^2}+a+b+f\right)$, and $b-f$, which for roll waves reduces to the eigenvalues $a$, $b-f$, and $b+f$. Recalling that we can choose $\beta_2 > 0$ without loss of generality and computing these eigenvalues for the different parameter ranges, we obtain the bifurcation diagrams shown in \Cref{fig:bifurcation-diagrams-hex}.
\begin{figure}[H]
    \centering
    \begin{subfigure}[p]{0.99\textwidth}
        \includegraphics[width=\linewidth]{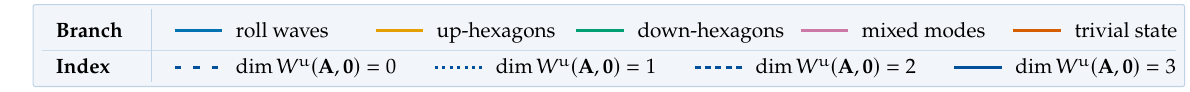}
    \end{subfigure}

    \vspace{0.4cm}

    \begin{subfigure}[p]{0.45\textwidth}
        \includegraphics[width=\linewidth]{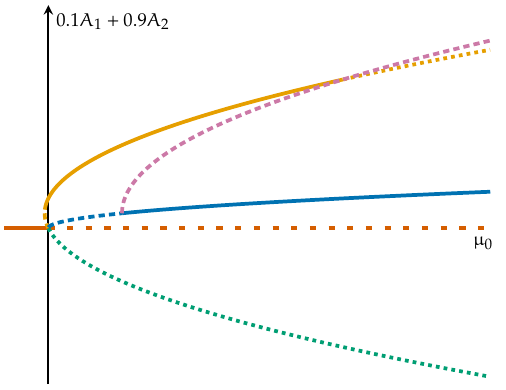}
        \subcaption{$K_2 < K_0 <0$}
    \end{subfigure}
    \hfill
    \begin{subfigure}[p]{0.45\textwidth}
        \includegraphics[width=\linewidth]{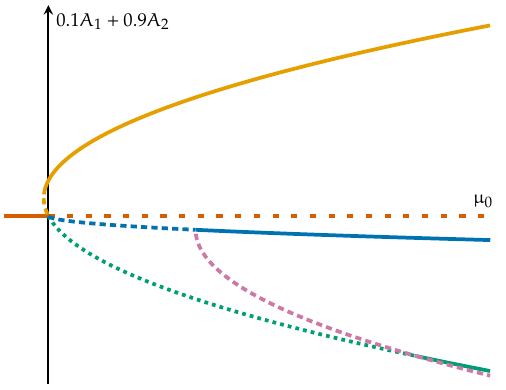}
        \subcaption{$K_0 < K_2 <0$}
    \end{subfigure}
    \caption{Two examples of bifurcation diagrams for stationary patterns on the hexagonal lattice including the dimension of the unstable manifold, which is calculated using the spectrum of the corresponding linearisation $\LinGL$ of the Landau equations \eqref{eq:Landau-equations}.}
    \label{fig:bifurcation-diagrams-hex}
\end{figure}

\subsection{Analytic construction of some heteroclinic orbits}\label{sec:heteroclinic}

We recall from the beginning of this section that our goal is to construct heteroclinic orbits in the leading-order system \eqref{eq:leading-order-reduced-equations}, which correspond to planar pattern interfaces in the full problem \eqref{eq:Swift-Hohenberg}. We note that since \eqref{eq:leading-order-reduced-equations} has a six-dimensional phase space, understanding its global dynamics analytically is a non-trivial problem. Therefore, we present here selected orbits that can be constructed using well-understood analytical methods. While it might be possible to obtain a deeper understanding using topological methods such as Conley index theory, see e.g.~\cite{mischaikow2002HandbookofDynamicalSystems}, this is beyond the scope of this paper.

We restrict the following analysis to the parameter set $\beta_2 \geq 0$, $K_0 < 0$ and $K_2 < 0$, for which we obtain a bifurcation diagram as in Figure \ref{fig:bifurcation-diagrams-hex} and which is also found in \cite{doelman2003-02EuropeanJournalofAppliedMathematics} where $\beta_2 > 0$, $K_0 = -3$ and $K_2 = -6$. Additionally, we restrict to the case $\theta \in (0, \tfrac{\pi}{6})$ and note that the results can be extended to the case $\theta = \tfrac{\pi}{6}$ using the fast-slow analysis presented in \Cref{sec:fast-slow-results}. 

We first note that $\{A_1, B_1 \in \R, A_2 = B_2 = A_3 = B_3 = 0\}$ is an invariant subspace of the leading-order reduced equations \eqref{eq:leading-order-reduced-equations}, which contains roll waves. In fact, it is an invariant subspace including the higher-order terms in \eqref{eq:reduced-equations-explicit} due to \Cref{rem:higher-order-terms}. In this space, the reduced equations on the centre manifold simplify to
\begin{equation*}
        \begin{split}
        \partial_\Xi A_1 &=  B_1, \\
        \partial_\Xi B_1 &= -\dfrac{\mu_0}{4 (\d \cdot \k_1)^2} A_1 - \dfrac{c_0}{4(\d \cdot \k_1)^2}B_1 - \dfrac{K_0 A_1^3}{4(\d \cdot \k_1)^2},
    \end{split}
\end{equation*}
which is the first-order formulation of a homogeneous, damped Duffing oscillator. For this system, when $\mu_0 > 0$, heteroclinic orbits connecting the non-trivial equilibrium to the trivial equilibrium can be constructed via a straightforward phase-plane analysis, see \Cref{fig:phase-plane-roll-waves}.

\begin{figure}[h!]
    \centering
    \begin{subfigure}[p]{0.4\textwidth}
    \includegraphics[width=\linewidth]{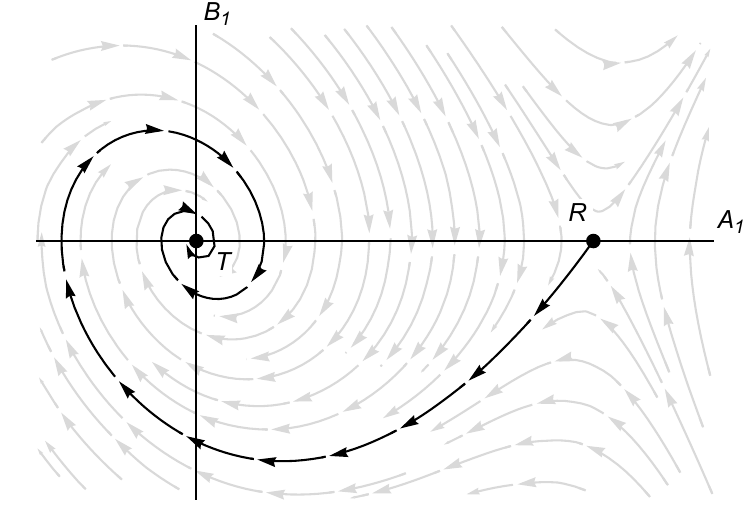}
    \subcaption{}\label{subfig:duffing-subcrit}
    \end{subfigure}
    \hspace{2cm}
    \begin{subfigure}[p]{0.4\textwidth}
    \includegraphics[width=\linewidth]{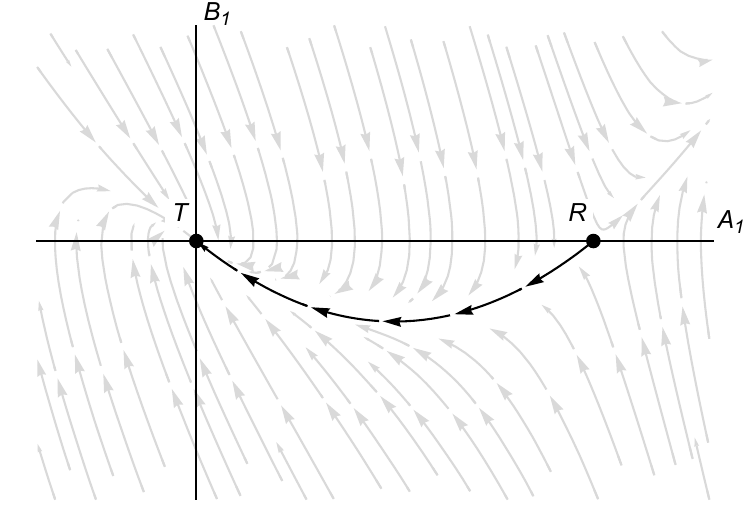}
    \subcaption{}\label{subfig:duffing-supercrit}
    \end{subfigure}
    \caption{Phase plane of the damped Duffing oscillator for $\mu_0 = 1$, $K_0 = -1$ and $(\d \cdot \k_1)^2 = 0.9$. Here, \sref{subfig:duffing-subcrit} depicts the case of subcritical speed $c_0 = 0.8 < c_{\crit}(\k_1) \approx 3.795$, and \sref{subfig:duffing-supercrit} depicts the supercitical case $c_0 = 4 > c_{\crit}(\k_1)$.}
    \label{fig:phase-plane-roll-waves}
\end{figure}

While in the case $\d = (1,0)$ there is another invariant subspace $\{A_2 = A_3 \text{ and } B_2 = B_3\}$, this is not invariant for any other $\d$ with $\theta \in (0,\tfrac{\pi}{6})$ since $(\d \cdot \k_2)^2 \neq (\d \cdot \k_3)^2$. Even in the case $\d = (1,0)$, the phase space is still four-dimensional. Nevertheless, we are still able to construct selected heteroclinic orbits analytically. Since the construction works the same way, we directly consider the full leading-order system \eqref{eq:leading-order-reduced-equations} for general directions $\d$ and prove that for any $c_0 > 0$ there exists a heteroclinic orbit connecting either down-hexagons $\Ab_{H_-}$ or roll waves $\Ab_{R}$ to the trivial state $\Ab_T$ depending on the parameter values $(\mu_0,K_0,K_2)$. 

The main observation for the analysis is that the leading-order system \eqref{eq:leading-order-reduced-equations} has a Lyapunov function given by
\begin{equation}\label{eq:hamiltonian}
    \Hcal(\Ab, \Bb) = \sum_{j = 1}^3 2 (\d \cdot \k_j)^2 |B_j|^2 + \dfrac{\mu_0}{2} \sum_{j = 1}^3 |A_j|^2 + \beta_2 A_1 A_2 A_3 + \dfrac{K_0}{4} \sum_{j = 1}^3 |A_j|^4 + \dfrac{K_2}{2} (A_1^2 A_2^2 + A_1^2 A_3^2 + A_2^2 A_3^2).
\end{equation}
Indeed, a direct calculation gives that along a solution $(\Ab(\Xi),\Bb(\Xi))$
\begin{equation}\label{eq:derivative-Lyapunov}
    \partial_\Xi \Hcal(\Ab(\Xi), \Bb(\Xi)) = - c_0 |\Bb(\Xi)|^2 \leq 0.
\end{equation}
From this, we note that for $c_0 > 0$, the Lyapunov function $\Hcal$ is strictly decreasing along non-stationary solutions. For $c_0 = 0$, we find that $\Hcal$ is conserved along solutions. In fact, the leading-order system \eqref{eq:leading-order-reduced-equations} is a Hamiltonian system with Hamiltonian $\Hcal$. This observation can be used to show that there exist $\mu_{H_+ \leftrightarrow T} < 0$ and $\mu_{H_+ \leftrightarrow R} > 0$ and corresponding reversible heteroclinic orbits in \eqref{eq:leading-order-reduced-equations} from $\Ab_{H_+}$ to $\Ab_T$ and $\Ab_{H_+}$ to $\Ab_R$, respectively, using variational techniques for Hamiltonian systems, see e.g.~\cite{stefanopoulos2008-12ProcRoySocEdinburghSectA}. However, we note that the centre manifold analysis in \cref{thm:centre-manifold} does not apply for $c_0 = 0$. While a spatial centre manifold approach, see e.g.~\cite{haragus2012-12EuropeanJournalofAppliedMathematics}, can be used to construct a centre manifold in the case $c_0 = 0$, the resulting manifold can be of arbitrary dimension depending on the direction $\d$. Although the leading-order system is expected to have an invariant manifold where the dynamics are given by \eqref{eq:reduced-equations-explicit}, this does not extend to the full centre manifold. In particular, due to the presence of highly oscillatory higher-order terms, the persistence of the orbits is non-trivial. We plan to treat this problem in a future paper.

Using the Lyapunov function, we now construct heteroclinic orbits for $\mu_0>0$ and $c_0 > 0$. More specifically, the orbit connects the non-trivial equilibrium with the lowest energy, which is either down-hexagons $\Ab_{H_-}$ or roll waves $\Ab_R$ as shown below, to the trivial equilibrium. The construction uses the following strategy. Let $\Ab_\ell$ be the non-trivial equilibrium with the lowest energy. Then, we define the trapping region $\Omega$ as the connected subset of the sublevel set $\{(\Ab,\Bb) \in \R^6 \,:\, \Hcal(\Ab,\Bb) < \Hcal(\Ab_\ell,\zerob)\}$ that contains $(\Ab_T,\zerob)$ and show that $\Omega$ is a bounded subset of the stable manifold of the trivial equilibrium. Note here that since $\mu_0 > 0$, the trivial equilibrium is stable due to \Cref{lem:characterisation-spat-dyn-stable-vs-pde-stable}. Thus, the corresponding stable manifold is six-dimensional. Moreover, we show that $(\Ab_{\ell}, \zerob) \in \partial\Omega$. To complete the construction, it remains to show that the unstable eigenspace of the equilibrium $(\Ab_{\ell}, \zerob)$ intersects $\Omega$. This yields the following result.

\begin{theorem}\label{thm:heteroclinic}
    Let $\theta\in (0,\tfrac{\pi}{6})$, $c_0>0$.
    \begin{thmenum}
        \item\thmitemlabel{it:heteroclinic-1} If $K_2 < K_0$, for every $\mu_0>0$ there exists a heteroclinic orbit $(\Ab_{H_-\to T},\Bb_{H_- \to T})$ in \eqref{eq:leading-order-reduced-equations} such that
        \begin{equation*}
            \lim_{\Xi \to -\infty} (\Ab_{H_-\to T},\Bb_{H_- \to T}) = (\Ab_{H_-},\zerob) \quad\text{and}\quad \lim_{\Xi \to +\infty} (\Ab_{H_-\to T},\Bb_{H_- \to T}) = (\Ab_{T},\zerob).
        \end{equation*}
        \item\thmitemlabel{it:heteroclinic-2}  If $K_2 \in (K_0, 0)$ there exists a $\mu_1 > 0$, which is given by
        \begin{equation}\label{eq:mu-one}
            \mu_1 = -\frac{\beta_2^2 K_0}{2 \sqrt{2} \sqrt{K_0 (K_0+K_2)^3} - 2 K_0 (K_0+3 K_2) }
        \end{equation}
        such that for every $\mu_0 \in (0,\mu_1)$ there exists a heteroclinic orbit $(\Ab_{H_-\to T},\Bb_{H_- \to T})$ as in \iref{it:heteroclinic-1}, and for $\mu_0 > \mu_1$ there exists a heteroclinic orbit $(\Ab_{R\to T},\Bb_{R \to T})$ to \eqref{eq:leading-order-reduced-equations} such that
        \begin{equation*}
            \lim_{\Xi \to -\infty} (\Ab_{R\to T},\Bb_{R \to T}) = (\Ab_{R},\zerob) \quad\text{and}\quad \lim_{\Xi \to +\infty} (\Ab_{R\to T},\Bb_{R \to T}) = (\Ab_{T},\zerob).
        \end{equation*}
    \end{thmenum}
\end{theorem}

\begin{figure}[ht]
    \centering
    \begin{subfigure}[b]{0.99\textwidth}
        \includegraphics[width=\linewidth]{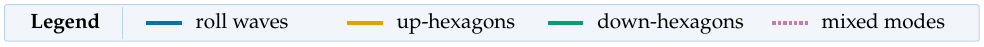}
    \end{subfigure}
    \vspace{0.1cm}
    
    \begin{subfigure}[b]{0.45\textwidth}
        \includegraphics[width=\linewidth]{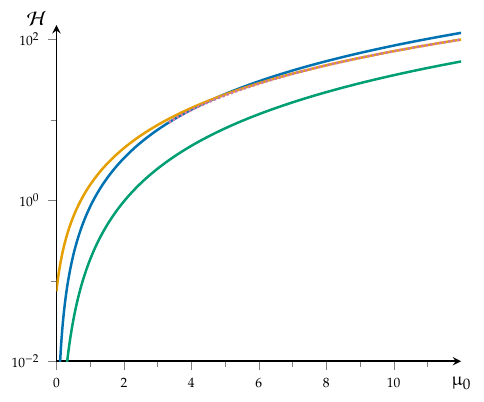}
        \subcaption{}\label{subfig:energy-a}
    \end{subfigure}
    \hfill
    \begin{subfigure}[b]{0.45\textwidth}
        \includegraphics[width=\linewidth]{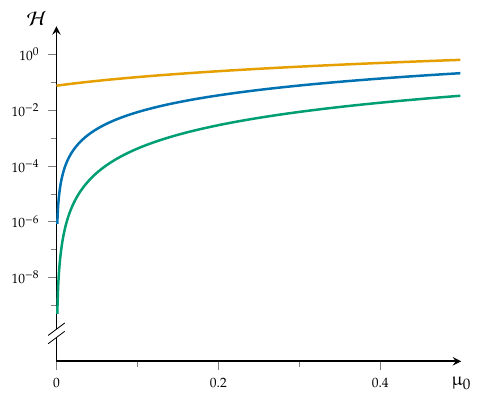}
        \subcaption{}\label{subfig:energy-b}
    \end{subfigure}   
    
    \begin{subfigure}[b]{0.45\textwidth}
        \includegraphics[width=\linewidth]{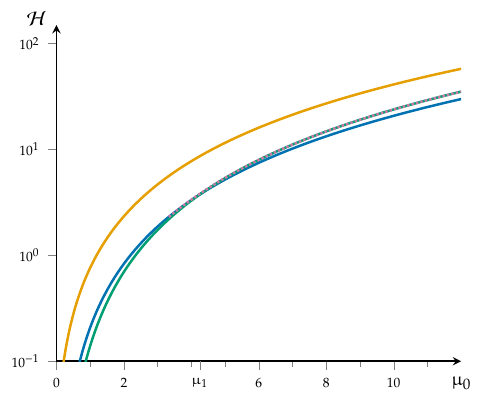}
        \subcaption{}\label{subfig:energy-c}
    \end{subfigure}
    \hfill
    \begin{subfigure}[b]{0.45\textwidth}
        \includegraphics[width=\linewidth]{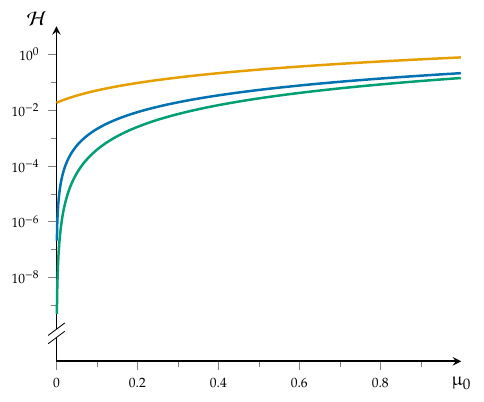}
        \subcaption{}\label{subfig:energy-d}
    \end{subfigure}
    
    \caption{Logarithmic plot of energy levels for different parameter regimes. 
        Panels~\sref{subfig:energy-a} and~\sref{subfig:energy-b} correspond to 
        the case $K_2 < K_0 < 0$, with $K_0 = -0.3$ and $K_2 = -0.6$. 
        Panels~\sref{subfig:energy-c} and~\sref{subfig:energy-d} correspond to 
        the case $K_0 < K_2 < -\tfrac{K_0}{2}$, with $K_0 = -1.2$ and $K_2 = -0.6$.}
    \label{fig:energy-levels}
\end{figure}

\begin{proof}
    For $\beta_2 = 0$, the energies can be computed explicitly as
    \begin{equation*}
        \Hcal(\Ab_R,\zerob) = -\dfrac{\mu_0^2}{4 K_0}, \quad \Hcal(\Ab_{H_\pm},\zerob) = - \dfrac{3 \mu_0^2}{4 (K_0 + 2 K_2)}, \quad \Hcal(\Ab_{\MM},\zerob) = - \dfrac{\mu_0^2}{2 (K_0 + K_2)}
    \end{equation*}
    which implies that down-hexagons have the lowest energy if $K_2 < K_0$, whereas roll waves have the lowest energy if $K_2 \in (K_0,0)$.
    
    We now establish the energy hierarchy for $\beta_2 > 0$, see \Cref{fig:energy-levels}. We first check that for $K_2 < K_0$, the down-hexagons $\Ab_{H_-}$ are the non-trivial equilibrium with the lowest energy. Then, we may compute that
    \begin{equation*}
        \Hcal(\Ab_{H_-},\zerob)\sim \frac{\mu_0^3}{2\beta_2^2} \quad \text{and} \quad  \Hcal(\Ab_{R},\zerob)\sim -\frac{\mu_0^2}{4K_0} \quad \text{as } \mu_0 \searrow 0.
    \end{equation*}
    In addition, we have
    \begin{equation*}
        \partial_{\mu_0}\Hcal(\Ab_{H_-},\zerob) = -\frac{3}{2(K_0+2K_2)} \mu_0 - \frac{3 (-\beta_2^2 + \sqrt{\beta_2^4-4\beta_2^2(K_0+2K_2) \mu_0})}{4 (K_0+2 K_2)^2} \quad \text{and} \quad  \partial_{\mu_0}\Hcal(\Ab_{R},\zerob) = -\frac{\mu_0}{2K_0}.
    \end{equation*}
    We note that $\partial_{\mu_0}\Hcal(\Ab_{H_-},\zerob)$ is strictly convex and $\tfrac{\partial_{\mu_0}\Hcal(\Ab_{H_-},\zerob) }{\partial_{\mu_0}\Hcal(\Ab_{R},\zerob)}$ is asymptotic to $\tfrac{3K_0}{K_0+2K_2}$ as $\mu_0 \to + \infty$. Hence, if $K_2 < K_0$, we have $\partial_{\mu_0}\Hcal(\Ab_{H_-},\zerob) < \partial_{\mu_0}\Hcal(\Ab_{R},\zerob)$ and so the down-hexagon energy never exceeds the roll-wave energy. If $0 > K_2 > K_0$, then $\partial_{\mu_0}\Hcal(\Ab_{H_-},\zerob) - \partial_{\mu_0}\Hcal(\Ab_{R},\zerob)$ has exactly one root and due to the asymptotic behaviour there is exactly one $\mu_1 >0$ such that $\Hcal(\Ab_{H_-},\zerob)|_{\mu_0=\mu_1} = \Hcal(\Ab_{R},\zerob)|_{\mu_0=\mu_1}$. $\mu_1$ can be explicitly calculated and is given by \eqref{eq:mu-one}.

    A direct computation shows that for any $K_2 < 0$, it holds
    \begin{equation*}
        \Hcal(\Ab_{H_+},\zerob) - \Hcal(\Ab_{H_-},\zerob) =-\frac{\beta_2 \left(\beta_2^2-4 \mu_0 (K_0+2
   K_2)\right)^{3/2}}{4 (K_0+2 K_2)^3} > 0.
    \end{equation*}

    For the mixed modes, we may also compute
    \begin{equation*}
        \partial_{\mu_0}\Hcal(\Ab_{\MM},\zerob) = - \frac{1}{K_0+K_2} \mu_0 - \frac{\beta_2^2}{2(K_0-K_2)(K_0+K_2)} ,\quad \mu_0 > -\frac{K_0\beta_2^2}{(K_0-K_2)^2}.
    \end{equation*}
    If $0>K_2 > K_0$, it immediately follows that $\partial_{\mu_0}\Hcal(\Ab_{\MM},\zerob) > \partial_{\mu_0}\Hcal(\Ab_{R},\zerob)$ and hence $\Hcal(\Ab_{\MM},\zerob) > \Hcal(\Ab_{R},\zerob) > \Hcal(\Ab_{H_-},\zerob)$ for $\mu_0 > -\frac{K_0\beta_2^2}{(K_0-K_2)^2}$. If $K_2 < K_0$, we find
    \begin{equation*}
        \partial_{\mu_0}\Hcal(\Ab_{\MM},\zerob) - \partial_{\mu_0}\Hcal(\Ab_{H_-},\zerob) \geq \frac{K_0-K_2}{2(K_0+2K_2)(K_0+K_2)} \mu_0 > 0
    \end{equation*}
    and since $\Hcal(\Ab_{\MM},\zerob)|_{\mu = \mu_{\MM}} = \Hcal(\Ab_{R},\zerob)|_{\mu = \mu_{\MM}} > \Hcal(\Ab_{H_-},\zerob)|_{\mu = \mu_{\MM}} $, we find that $\Hcal(\Ab_{\MM},\zerob)>\Hcal(\Ab_{H_-},\zerob)$.

    We now again denote the non-trivial equilibrium with the smallest energy by $(\Ab_\ell,\zerob)$. Since $\Hcal(\Ab_T,\zerob) = 0$ and $\Hcal(\Ab_\ell,\zerob) > 0$ as shown, $(\Ab_T,\zerob)$ is in the sublevel set $\{(\Ab,\Bb) \in \R^6 \,:\, \Hcal(\Ab,\Bb) < \Hcal(\Ab_\ell,\zerob)$. Therefore, there is a connected component of the sublevel set which contains $(\Ab_T,\zerob)$. As noted above, we denote this component as $\Omega$. Using that $\Hcal$ is strictly decreasing along non-stationary orbits, we find that $\Omega$ is forward invariant under the dynamics of \eqref{eq:leading-order-reduced-equations}. 

    Next, we show that $\Omega$ is bounded. To show this, we argue by contradiction and assume that it is unbounded. Then, there exists a path $(\tilde{\Ab}(s),\tilde{\Bb}(s))_{s \geq 0}$ with $\|(\tilde{\Ab}(s),\tilde{\Bb}(s))\| \to \infty$ as $s \to \infty$ such that $\Hcal(\tilde{\Ab}(s),\tilde{\Bb}(s)) < \Hcal(\Ab_\ell,\zerob)$ for all $s \geq 0$ and $(\tilde{\Ab}(0),\tilde{\Bb}(0)) = (\Ab_T,\zerob)$. Since $\Hcal(\Ab,\Bb) \geq \Hcal(\Ab,\zerob)$, the projected path $(\tilde{\Ab}(s),\zerob)_{s \geq 0}$ also satisfies $\Hcal(\tilde{\Ab}(s),\zerob) < \Hcal(\Ab_\ell,\zerob)$ and $\|\Ab(s)\| \to \infty$ as $s \to \infty$. The latter follows by contradiction. Assume that $\|\tilde{\Ab}(s)\|$ is uniformly bounded for $s \geq 0$, then $\|\tilde{\Bb}(s)\| \to \infty$ and hence, $\Hcal(\tilde{\Ab}(s),\tilde{\Bb}(s)) \to +\infty$ as $s \to \infty$. This contradicts the assumption that $\Hcal(\tilde{\Ab}(s),\tilde{\Bb(s)}) < \Hcal(\Ab_\ell,\zerob)$. We now show that there exists a critical point $(\Ab,\zerob) \neq \zerob$ with $\Hcal(\Ab,\zerob) < \Hcal(\Ab_\ell,\zerob)$, which contradicts the assumption that $(\Ab_\ell,\zerob)$ is the non-trivial equilibrium with the lowest energy. To obtain the existence of this critical point, we note that $\Ab \mapsto \Hcal(\Ab,\zerob)$ satisfies the assumptions of the Mountain Pass theorem, cf.~\cite[Thm.~2.2]{rabinowitz1986-07book}. Specifically, the Palais–Smale condition is satisfied since $\Hcal(\Ab,\zerob) \to -\infty$ as $\|\Ab\| \to \infty$ due to $K_0 < 0$ and $K_2 < 0$. Hence, a critical value of $\Ab \mapsto \Hcal(\Ab,\zerob)$ is given by the infimum of the maximal values along all paths connecting $\Ab = \zerob$ to $\Ab = \tilde{\Ab}(\tilde{s})$ for $\tilde{s}$ sufficiently large with $\Hcal(\tilde{\Ab}(\tilde{s}),\zerob) <0$. This infimum is, in particular, bounded above by $\max_{s\in [0,\tilde{s}]} \Hcal(\tilde{\Ab}(s),\zerob) < \Hcal(\Ab_\ell,\zerob)$, the desired contradiction.
    
    Using this, we now show that $\Omega$ is a subset of the stable manifold of $(\Ab_T,\zerob)$. For this, we need to show that the $\omega$-limit set of every point in $\Omega$ consists of $(\Ab_T,\zerob)$ only. We note that $(\Ab_T,\zerob)$ is the only equilibrium in $\Omega$ as $(\Ab_\ell,\zerob)$ is the non-trivial equilibrium with the lowest energy by assumption. In addition, $(\Ab_T,\zerob)$ is a minimum of $\Hcal$ since $\mu_0 > 0$. Since the Lyapunov function is strictly decreasing along orbits, there are no periodic orbits and thus, any solution with initial value in $\Omega$ converges to $(\Ab_T,\zerob)$ since $\Omega$ is bounded.

    As outlined above, it now remains to show that the unstable manifold of $(\Ab_\ell,\zerob)$ intersects with $\Omega$. For this, we argue that $(\Ab_\ell,\zerob) \in \partial\Omega$, which follows from the fact that $s \mapsto \Hcal(s \Ab_\ell,\zerob)$ is strictly increasing for $s \in (0,1)$ since $\Omega$ contains no equilibria except the trivial one. This yields that the line $\{(s \Ab_\ell,\zerob) \,:\, s \in (0,1)\}$ is fully contained in $\Omega$. It is therefore sufficient to show that the unstable eigenspace of the linearisation $\LinSD(\Ab_\ell)$ of \eqref{eq:leading-order-reduced-equations} about $(\Ab_\ell)$ intersects with $\Omega$.

    Let $\vb \in \R^6$ be an element of the unstable eigenspace. Since the unstable manifold is tangential to the unstable eigenspace and $\Hcal$ is a strictly decreasing Lyapunov function, it holds that $\Hcal((\Ab_\ell,\zerob) + t \vb) < \Hcal(\Ab_\ell,\zerob)$ for all $|t| \ll 1$. Now, we consider the function
    \begin{equation*}
        s \mapsto \omega_{\eps,\vb}(s):=\Hcal(s((\Ab_\ell,\zerob) + \eps \vb)).
    \end{equation*}
    We want to show that there exists $\eps>0$ such that $\max_{s\in[0,1]}\omega_{\eps,\vb}(s) < \Hcal((\Ab_{\ell},\zerob))$ holds true either for $\vb$ or $-\vb$. This ensures that $(\Ab_\ell,\zerob) + \eps \vb$ lies in the correct component of the sublevel set, i.e. $(\Ab_\ell,\zerob) + \eps \vb \in \Omega$. Let $s^*(\eps)\in [0,1]$ be the maximal value, and we may assume that $s^*(\eps)<1$ since otherwise there is nothing to prove. Since $(\Ab_\ell,\zerob) + \eps \vb \to (\Ab_\ell,\zerob)$ for $\eps \to 0$, we obtain that $s^*(\eps) \to 1$. Note that $s^*$ is smooth in a neighbourhood of $\eps=0$ due to the implicit function theorem. We thus make the $\eps$-expansion about $\eps=0$ given by $s^*(\eps) = 1 + \eps s_1 + \Ocal(\eps^2)$ for some $s_1\leq 0$. Inserting this expansion into $\omega_{\eps,\vb}$ and expanding again, we obtain
    \begin{equation*}
        \omega_{\eps,\vb}(s^*(\eps)) = \Hcal((\Ab_{\ell},\zerob)) + \frac{\eps^2}{2}\bigl(s_1^2 \kappa_1 + \kappa_2 + 2s_1(\Ab_{\ell},\zerob) D^2\Hcal((\Ab_{\ell},\zerob))\vb\bigr) + \Ocal(\eps^3), 
    \end{equation*}
    where $\kappa_1 := (\Ab_{\ell},\zerob) D^2\Hcal((\Ab_{\ell},\zerob)) (\Ab_{\ell},\zerob) < 0$ and $\kappa_2 = \vb D^2\Hcal((\Ab_{\ell},\zerob)) \vb < 0$. By potentially switching the sign of $\vb$, we can ensure that $s_1^2 \kappa_1 + \kappa_2 + 2s_1(\Ab_{\ell},\zerob) D^2\Hcal((\Ab_{\ell},\zerob))\vb <0$ and then choose $\eps$ sufficiently small to obtain $\omega_{\eps,\vb}(s^*(\eps)) < \Hcal((\Ab_{\ell},\zerob))$. Therefore, $(\Ab_\ell,\zerob) + \eps \vb \in \Omega$ for $\eps$ sufficiently small. Since the unstable manifold is tangential to the unstable direction $\vb$ this establishes an intersection of the unstable manifold of $(\Ab_\ell,\zerob)$ with $\Omega$, which completes the proof.
\end{proof}

\begin{remark}
    We point out that the proof of \Cref{thm:heteroclinic} in particular shows that $\Omega$ is uniformly bounded in $\d$.
\end{remark}

Before considering the persistence of the orbits established in \Cref{thm:heteroclinic}, we briefly discuss the existence of other heteroclinic orbits. A fully rigorous construction of such orbits is highly non-trivial due to the high-dimensional phase space and the lack of symmetries, which prevents a direct further reduction to smaller invariant subspaces, with the exception of the invariant subspace containing roll waves as the only non-trivial pattern. However, we conjecture, for example, the existence of heteroclinic orbits connecting the other non-trivial equilibria to the trivial one, as well as a heteroclinic orbit connecting up-hexagons to roll waves for $\mu_0 > 0$ before the bifurcation of the mixed modes from the roll waves. In the fast-front case treated in \cite{hilder2025-08JNonlinearSci}, such orbits can be found rigorously using phase plane analysis. Here, the former orbits can be found using a numerical shooting algorithm, see \Cref{fig:planar-front} for an invasion of up-hexagons into the trivial state and \Cref{subfig:mixed-modes-invasion,subfig:false-hexagons-invasion} for an invasion of mixed modes and false hexagons into the trivial state. The construction of the latter orbit is difficult even numerically, and we leave a detailed numerical analysis for future research. However, by performing systematic shooting in directions of the unstable manifold of the equilibrium corresponding to up-hexagons, we find orbits that connect to the trivial equilibrium but come close to the roll-wave equilibrium. The corresponding pattern interface is a two-stage invasion of up-hexagons into the trivial state with a large plateau of roll waves as an intermediate state, see \Cref{subfig:two-stage}. As discussed in \Cref{sec:front-speed}, this is a typical mechanism of how hexagons arise in experiments.

\begin{figure}[h]
    \centering
    \begin{subfigure}[b]{0.99\textwidth}
        \includegraphics[width=\linewidth]{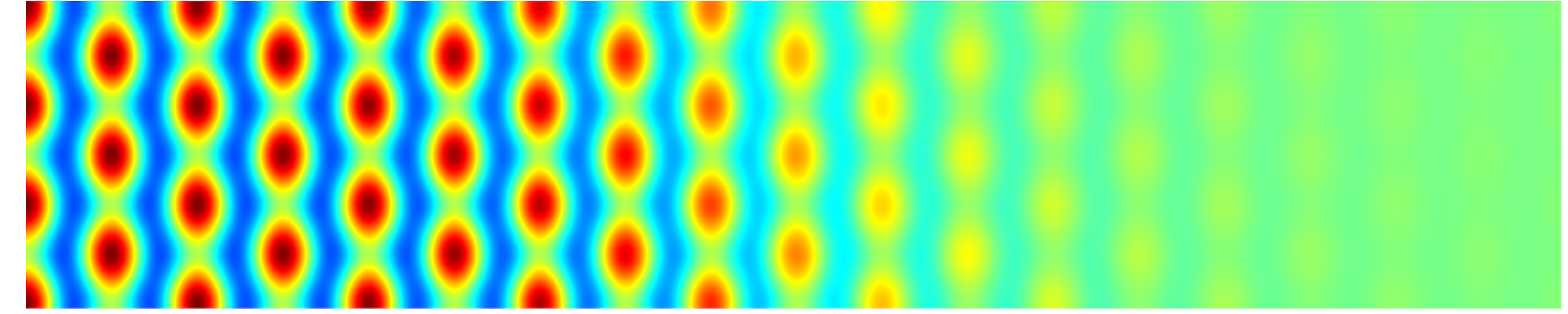}
        \subcaption{Pattern interface connecting mixed modes to the trivial state.}\label{subfig:mixed-modes-invasion}
    \end{subfigure}

    \vspace{0.2cm}

    \begin{subfigure}[b]{0.99\textwidth}
        \includegraphics[width=\linewidth]{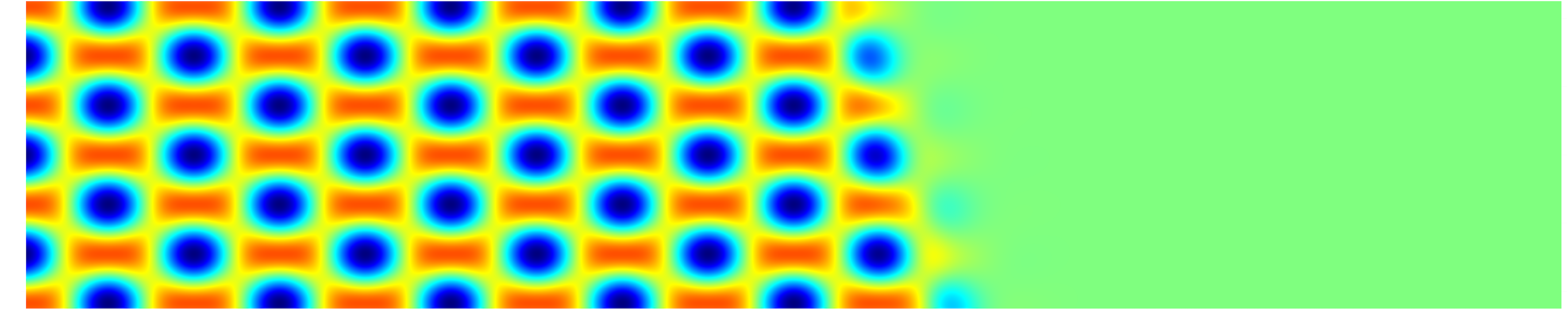}
        \subcaption{Pattern interface connecting false hexagons to the trivial state.}\label{subfig:false-hexagons-invasion}
    \end{subfigure}

    \vspace{0.2cm}
    
    \begin{subfigure}[b]{0.99\textwidth}
        \includegraphics[width=\linewidth]{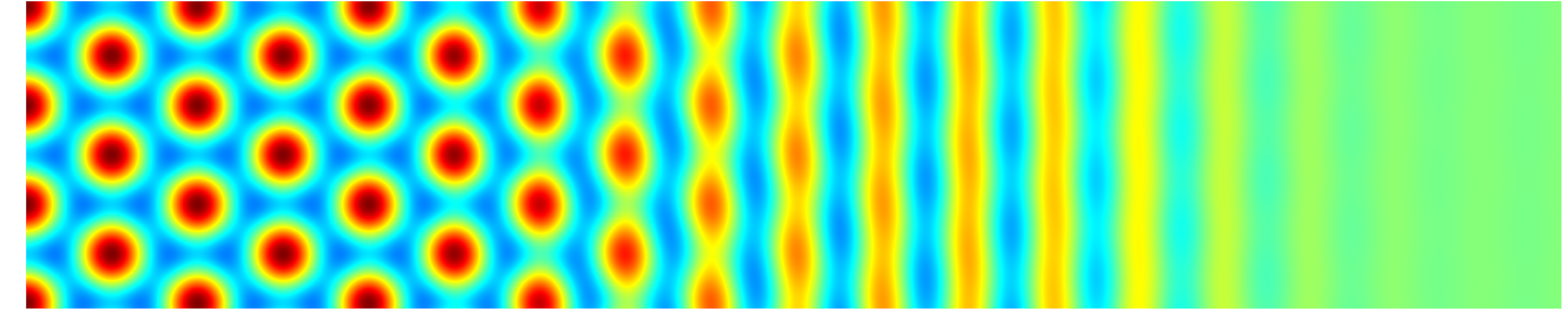}
        \subcaption{Two-stage invasion of up-hexagons into the trivial state via roll waves.}\label{subfig:two-stage}
    \end{subfigure}
    \caption{Example of a two-dimensional planar pattern interface corresponding to numerically obtained heteroclinic orbits.}
    \label{fig:planar-fronts-numerics}
\end{figure}

\subsection{Persistence of heteroclinic orbits}

We now discuss the persistence of heteroclinic orbits in the leading-order system \eqref{eq:leading-order-reduced-equations} as solutions to the full reduced equation on the centre manifold \eqref{eq:reduced-equations-explicit}. For this, we first note that the equilibria obtained in \Cref{cor:stationary-patterns} are hyperbolic outside of bifurcation points and therefore, as discussed in \cite{hilder2025-08JNonlinearSci}, the bifurcation diagram persists. Next, for $\mu_0 > 0$, we consider heteroclinic orbits in the leading-order system \eqref{eq:leading-order-reduced-equations}, which connect a non-trivial equilibrium $(\Ab_\ell,\zerob)$ to the trivial equilibrium $(\Ab_T,\zerob)$. For this, we note the heteroclinic orbit lies in both the stable manifold of $(\Ab_T,\zerob)$ and the unstable manifold of $(\Ab_\ell,\zerob)$ and thus, they intersect. Since the stable manifold of the trivial equilibrium $(\Ab_T,\zerob)$ is a six-dimensional manifold in $\R^6$, this intersection is transversal and persists under perturbation. This yields the following result.

\begin{theorem}\label{thm:persistence}
    Let $\mu_0$ and $\theta\in (0,\tfrac{\pi}{6})$ with $\cot(\theta) \in \sqrt{3}\Q$. Any heteroclinic orbit of the leading-order system \eqref{eq:leading-order-reduced-equations}, which connects to the trivial equilibrium, persists as a solution to the full reduced equation \eqref{eq:reduced-equations-explicit}. In particular, the heteroclinic orbits constructed in \Cref{thm:heteroclinic} persist.
\end{theorem}

We point out that the persistence of other heteroclinic orbits, which do not connect to the trivial state, is less obvious. For example, we consider a potential connection between up-hexagons and roll waves for $\mu_0 > 0$ before the bifurcation of mixed modes from the roll waves. Recalling the dimensions of the respective stable and unstable manifolds depicted in \Cref{fig:bifurcation-diagrams-hex}, we find that, in this regime, the unstable manifold of the up-hexagons is three-dimensional and the stable manifold of the roll waves is four-dimensional (since the equilibria are hyperbolic). Therefore, we must show that the intersection of a three-dimensional manifold with a four-dimensional manifold persists in a six-dimensional phase space. Counting dimensions, the intersections between these manifolds will generically be transversal and thus persistent. However, to the best of our knowledge, there is no general result that excludes the occurrence of degenerate situations in which the intersection is not transversal. Therefore, the persistence of such orbits, even if they can be constructed in the leading-order system, remains an open question.

\subsection{The singular limits $\theta \to \tfrac{\pi}{6}$ and $c_0 \to \infty$}\label{sec:fast-slow-results}

We now study the behaviour of the reduced system \eqref{eq:reduced-equations-explicit} when $\theta \to \tfrac{\pi}{6}$. For this, we define the small parameter $\delta := 4 (\d \cdot \k_2)^2$. Then, we can write the reduced equations \eqref{eq:reduced-equations-explicit} for $\theta < \tfrac{\pi}{6}$ as
\begin{equation}\label{eq:reduced-equations-fast-slow}
    \begin{split}
        \partial_\Xi A_1 &=  B_1, \\
        \partial_\Xi B_1 &= -\dfrac{\mu_0}{4 (\d \cdot \k_1)^2} A_1 - \dfrac{c_0}{4(\d \cdot \k_1)^2}B_1 - \dfrac{1}{4(\d \cdot \k_1)^2} (\beta_2\bar{A}_2\bar{A}_3 + (K_0|A_1|^2 + K_2(|A_2|^2 + |A_3|^2)) A_1) + \Ocal(\eps),\\
        \partial_\Xi A_2 &=  B_2, \\
        \delta \partial_\Xi B_2 &= -\mu_0A_2 - c_0 B_2 - \beta_2\bar{A}_1\bar{A}_3 - (K_0|A_2|^2 + K_2(|A_1|^2 + |A_3|^2)) A_2 + \Ocal(\eps),\\
        \partial_\Xi A_3 &=  B_3, \\
        \partial_\Xi B_3 &= -\dfrac{\mu_0}{4 (\d \cdot \k_3)^2} A_3 - \dfrac{c_0}{4(\d \cdot \k_3)^2}B_3 - \dfrac{1}{4(\d \cdot \k_3)^2} (\beta_2\bar{A}_1\bar{A}_2 + (K_0|A_3|^2 + K_2(|A_1|^2 + |A_2|^2)) A_3) + \Ocal(\eps),\\
    \end{split}
\end{equation}
which can be written in the abstract form
\begin{equation}\label{eq:abstract-fast-slow-system}
    \begin{split}
        \delta \partial_\Xi B_2 & = F(B_2,Y;\delta), \\
        \partial_\Xi Y & = G(B_2,Y;\delta),
    \end{split}
\end{equation}
with $Y = (A_1,B_1,A_2,A_3,B_3)$, which is a fast-slow system, see e.g.~\cite{kuehn2015book}. We then define the critical manifold $C_0$, given by
\begin{equation*}
    C_0 := \{(B_2,Y) \in \R^6 \,:\, F(B_2,Y;0) = 0\}
\end{equation*}
which is obtained by setting $\delta = 0$ in \eqref{eq:abstract-fast-slow-system}. The flow on $C_0$ is obtained by solving $F(B_2,Y;0)$ for $B_2$ and substituting this result into $\partial_\Xi A_2 = B_2$. The resulting system then reads as
\begin{equation}\label{eq:reduced-equations-slow-subsystem}
    \begin{split}
        \partial_\Xi A_1 &=  B_1, \\
        \partial_\Xi B_1 &= -\dfrac{\mu_0}{4 (\d \cdot \k_1)^2} A_1 - \dfrac{c_0}{4(\d \cdot \k_1)^2}B_1 - \dfrac{1}{4(\d \cdot \k_1)^2} (\beta_2\bar{A}_2\bar{A}_3 + (K_0|A_1|^2 + K_2(|A_2|^2 + |A_3|^2)) A_1) + \Ocal(\eps),\\
        c_0 \partial_\Xi A_2 &= -\mu_0A_2 - \beta_2\bar{A}_1\bar{A}_3 - (K_0|A_2|^2 + K_2(|A_1|^2 + |A_3|^2)) A_2 + \Ocal(\eps),\\
        \partial_\Xi A_3 &=  B_3, \\
        \partial_\Xi B_3 &= -\dfrac{\mu_0}{4 (\d \cdot \k_3)^2} A_3 - \dfrac{c_0}{4(\d \cdot \k_3)^2}B_3 - \dfrac{1}{4(\d \cdot \k_3)^2} (\beta_2\bar{A}_1\bar{A}_2 + (K_0|A_3|^2 + K_2(|A_1|^2 + |A_2|^2)) A_3) + \Ocal(\eps),\\
    \end{split}
\end{equation}
which is equivalent to the reduced equation \eqref{eq:reduced-equations-pi-over-six-second-order} for $\theta = \tfrac{\pi}{6}$.

Let $(\Ab_{\mathrm{het},\delta}, \Bb_{\mathrm{het},\delta})$ a uniformly bounded sequence of heteroclinic orbits with $\delta$ such that the corresponding angle $\theta$ satisfies $\cot(\theta) \in \sqrt{3} \Q$. Then, our goal is to show that this sequence converges to a heteroclinic orbit $(\Ab_{\mathrm{het},0}, \Bb_{\mathrm{het},0})$ solving \eqref{eq:reduced-equations-pi-over-six-second-order}. Fix any such sequence and let $\Scal_0$ be a compact submanifold of $C_0$ which is sufficiently large to contain $Y_\delta$ for $\delta > 0$. We now check that $\Scal_0$ is normally hyperbolic since $D_{B_2} F(B_2,Y;0) = -c_0 + \Ocal(\eps) \neq 0$ for fixed $c_0 > 0$ and $\eps$ sufficiently small. Therefore, standard Fenichel theory applies, see e.g.~\cite[Thm.~3.1.4]{kuehn2015book}, and we obtain a sequence of slow manifolds $\Scal_\delta$ with the following properties. For every $\delta > 0$, $\Scal_\delta$ is a locally invariant manifold of \eqref{eq:reduced-equations-fast-slow}, and it is diffeomorphic to $\Scal_0$. Moreover, the flow on $\Scal_\delta$ converges to the slow flow generated by \eqref{eq:reduced-equations-slow-subsystem}.

Finally, we note that any solution $\Xi \mapsto (\Ab_\delta(\Xi),\Bb_\delta(\Xi))$ to \eqref{eq:reduced-equations-fast-slow} that is uniformly bounded in $\Xi \in \R$ is also contained in $\Scal_\delta$, after potentially choosing a larger compact submanifold of $C_0$. For this, we first note that slow manifolds are centre manifolds of the fast system, i.e.~\eqref{eq:reduced-equations-fast-slow} after a change of variables to the fast time $\tilde{\Xi} = \delta^{-1} \Xi$, extended with $\partial_{\tilde{\Xi}} \delta = 0$. Then, we can lift the typical observation that solutions that are globally contained in a small neighbourhood lie on the centre manifold, see e.g.~\cite{haragus2011book} to slow manifolds by using a finite covering exploiting that $S_0$ is a compact submanifold. Note that, although centre manifolds (or slow manifolds) are not unique, bounded solutions are contained in all centre manifolds, see e.g.~\cite{sijbrand1985TransAmerMathSoc}, and therefore, we write `the' centre manifold. In particular, this shows that equilibria and heteroclinic orbits of \eqref{eq:reduced-equations-fast-slow} are contained in $\Scal_\delta$.

Since the heteroclinic orbits from \Cref{thm:heteroclinic} are uniformly bounded in $\theta$ and persistent by \Cref{thm:persistence}, we obtain the following result for $\theta = \tfrac{\pi}{6}$.

\begin{theorem}\label{thm:heteroclinic-pi-over-six}
    Let $c_0 > 0$. There exist heteroclinic orbits in \eqref{eq:reduced-equations-slow-subsystem}, or equivalently in \eqref{eq:reduced-equations-pi-over-six-second-order}, with the same properties as in \Cref{thm:heteroclinic}.
\end{theorem}

To conclude this subsection, we discuss the case of solutions for $c_0$ large. For this, we treat $\tfrac{1}{c_0^2}$ as a small parameter. Then, the corresponding slow subsystem for $c_0 \to \infty$ reads as
\begin{equation}\label{eq:reduced-equations-c-to-infinity}
    \begin{split}
        \partial_{\Xi} A_1 + \mu_0 A_1 + \beta_2\bar{A}_2\bar{A}_3 + (K_0|A_1|^2 + K_2(|A_2|^2 + |A_3|^2)) A_1 + \Ocal(\eps) & = 0, \\
        \partial_{\Xi} A_2 + \mu_0 A_2 + \beta_2\bar{A}_1\bar{A}_3 + (K_0|A_2|^2 + K_2(|A_1|^2 + |A_3|^2)) A_2 + \Ocal(\eps) & = 0, \\
        \partial_{\Xi} A_3 + \mu_0 A_3 + \beta_2\bar{A}_1\bar{A}_2 + (K_0|A_3|^2 + K_2(|A_1|^2 + |A_2|^2)) A_3 + \Ocal(\eps) & = 0.
    \end{split}
\end{equation}
We highlight that the system \eqref{eq:reduced-equations-c-to-infinity} is independent of the direction $\d$, which was already pointed out in \cite[Rem.~5.1(1)]{hilder2025-08JNonlinearSci}, where the system \eqref{eq:reduced-equations-c-to-infinity} was obtained as a reduced equation on a centre manifold in the construction of fast-moving pattern interfaces in an asymptotic model for the three-dimensional Bénard–Marangoni problem. A detailed analysis of heteroclinic orbits occurring in the slow subsystem can be found in \cite[Sec.~5.5]{hilder2025-08JNonlinearSci} by exploiting that the system \eqref{eq:reduced-equations-c-to-infinity} has a two-dimensional invariant subspace $\{A_2 = A_3 \in \R\}$, which contains all equilibria. As shown in \cite{doelman2003-02EuropeanJournalofAppliedMathematics}, these orbits persist in the full system \eqref{eq:reduced-equations-explicit} with $c_0$ large, using similar arguments as above.

\subsection{Pattern interfaces corresponding to heteroclinic orbits}\label{sec:pattern-interfaces}

Before analysing the dynamics of the reduced equations numerically, we analyse the solutions in the full pattern-forming system \eqref{eq:Swift-Hohenberg} that correspond to heteroclinic orbits in the reduced equations \eqref{eq:reduced-equations-explicit} and \eqref{eq:reduced-equations-pi-over-six-second-order}. For this, we find from the ansatz \eqref{eq:ansatz-reduced-equations} in the derivation of the reduced equations and using the centre manifold theorem \ref{thm:centre-manifold} that the spatial-dynamics system \eqref{eq:spat-dyn-Fourier} has a solution $W = W(\xi,p)$ with first component given by
\begin{equation*}
    W_0(\xi,\p) = \eps \sum_{j = 1}^3 A_j(\eps \xi) e^{i\k_j \cdot \p} + c.c. + \Ocal(\eps^2)
\end{equation*}
where we used that the eigenvectors $\hat{\phib}^\eps_\pm$ are normalised such that their first component is equal to one and that $A_j = \tilde{A}_{j,+} + \tilde{A}_{j,-}$. Furthermore, we recall that $W(\xi,\p) = (U, (\partial_\xi + \d \cdot \nabla_\p) U(\xi,\p), (\partial_\xi + \d \cdot \nabla_\p)^2 U(\xi,\p),(\partial_\xi + \d \cdot \nabla_\p)^3 U(\xi,\p))$ and that $u(t,\x) = U(\xi, \p)$ is a solution to \eqref{eq:Swift-Hohenberg}. We therefore obtain the following result.

\begin{theorem}\label{thm:main-theorem}
    Fix $\d = (\cos(\theta), \sin(\theta)) \in S^1$ such that $\cot(\theta)\in \sqrt{3}\Q$ and let $c_0 > 0$. Then, for any bounded orbit $\Xi \mapsto (\Ab, \Bb)(\Xi)$ of \eqref{eq:reduced-equations-explicit} there exists an $\eps_0 > 0$ such that for all $\eps \in (0,\eps_0)$ the pattern-forming equation \eqref{eq:Swift-Hohenberg} has solutions of the form
    \begin{equation*}
        u(t,\x) = 2\eps \sum_{j = 1}^3 A_j(\eps (\d \cdot \x - \eps c_0 t)) \cos(\k_j \cdot \x) + \Ocal(\eps^2).
    \end{equation*}
    Specifically, there exist open classes of suitable polynomial nonlinearities such that \eqref{eq:Swift-Hohenberg} has solutions $u(t,\x) = U(\d \cdot \x + \eps c_0 t, \x)$ with 
    \begin{equation*}
        \lim_{\xi \to - \infty} U(\xi,\p) = - 2 \eps A_\mathrm{hex} \sum_{j = 1}^3 \cos(\k_j \cdot \p) + \Ocal(\eps^2) \quad \text{and} \lim_{\xi \to \infty} U(\xi,\p) = 0
    \end{equation*}
    for some $A_\mathrm{hex} > 0$, cf.~\eqref{eq:hexagons}, or 
    \begin{equation*}
        \lim_{\xi \to - \infty} U(\xi,\p) = 2 \eps A_\mathrm{roll} \cos(x_1) + \Ocal(\eps^2) \quad \text{and} \lim_{\xi \to \infty} U(\xi,\p) = 0
    \end{equation*}
    for some $A_\mathrm{roll} > 0$, cf.~\eqref{eq:roll-waves}.
\end{theorem}

% \subsection{Numerical analysis}\label{sec:numerics}

\section{Pattern interfaces on a square lattice}\label{sec:square}

So far, the construction was restricted to patterns which lie on a hexagonal Fourier lattice generated by $\k_1$, $\k_2$ and $\k_3$ given in \eqref{eq:hex-lattice-generators}. Notably, this lattice does not support square patterns, which are generated by the Fourier wave vectors $\k_1 = (1,0)$ and $\k_2 = (0,1)$. Therefore, we now discuss how to obtain pattern interfaces on this square Fourier lattice $\Gamma_{\square}$. Since the construction largely follows that of the hexagonal lattice and is significantly easier because the square lattice is non-resonant, we only report the differences. In particular, the full analysis can be carried out without using \Cref{ass:smallness}.

\subsection{Spectral analysis}

Consider now a direction $\d = (\cos(\theta),\sin(\theta)) \in S^1$ with $\theta \in [0,\tfrac{\pi}{4})$. We reiterate that it is sufficient to consider only $0 \leq \theta < \tfrac{\pi}{4}$ since all other angles can be obtained by using the $D_4$-symmetry of \eqref{eq:Swift-Hohenberg}. We note that the ansatz \eqref{eq:modfront-ansatz} and the resulting spatial dynamics formulation \eqref{eq:spat-dyn-Fourier} do not depend on the specific choice of Fourier lattice and therefore, do not change. Similarly, the geometric interpretation of the eigenvalue problem given by \Cref{lem:eigenvalue-correspondence,lem:multiplicity} is unchanged, see \Cref{fig:spectral-eps=0-square}. A direct consequence of the geometric interpretation of the spectrum is that a more-central Jordan block of size four now occurs for $\theta = 0$ instead of $\theta = \tfrac{\pi}{6}$ since $\k_2 \perp (1,0)$. Therefore, \Cref{prop:imaginary-spectrum-discrete} holds verbatim with the condition $\theta = \tfrac{\pi}{6}$ replaced by $\theta = 0$. Next, we note that the angle condition for the presence of a spectral gap given in \Cref{prop:real-spectral-gap} changes from $\cot(\theta) \in \sqrt{3} \Q$ to $\cot(\theta) \in \Q$, which follows from the same algebraic computations using that $\d^{\perp}\cdot \gammab = -n_1 \sin(\theta) + n_2 \cos(\theta)$ for $\gammab = n_1 \k_1 + n_2\k_2\in \Gamma_{\square}$. To conclude the spectral analysis, we find that there are eight more-central eigenvalues (for $\theta \neq 0$) or six (for $\theta = 0$) which still have the expansions \eqref{eq:spectrum-eps-pos-perturbation-block-size-two} and \eqref{eq:spectrum-eps-pos-perturbation-block-size-four}, respectively. The remaining statements of \Cref{lem:polynom-roots-1,lem:polynom-roots-2} and \Cref{prop:spectrum-eps-positive} hold unchanged under the angle condition $\cot(\theta) \in \Q$.

\begin{figure}[H]
    \centering
    \begin{subfigure}[b]{0.45\textwidth}
        \includegraphics[width = \linewidth]{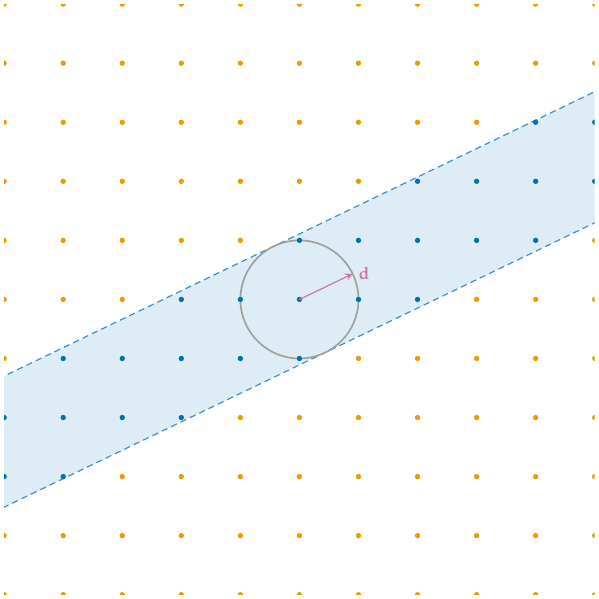}
        \subcaption{}
        \label{subfig:geo-square-1}
    \end{subfigure}
    \hfill
    \begin{subfigure}[b]{0.45\textwidth}
        \includegraphics[width = \linewidth]{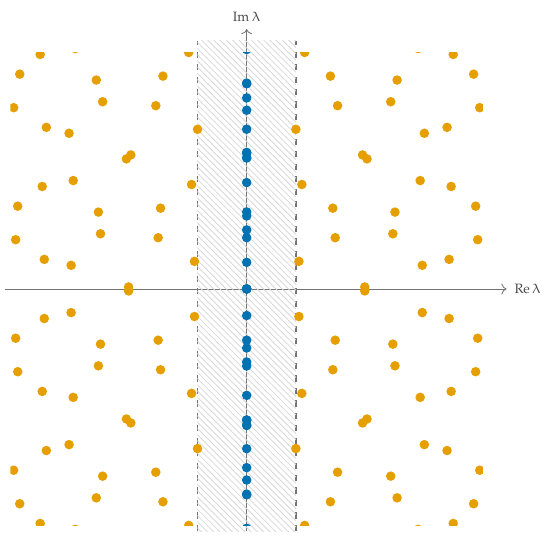}
        \subcaption{}
        \label{subfig:geo-square-2}
    \end{subfigure}    
    \caption{Geometric characterisation of the spectrum of $\tilde{\Lcal}^0$ on a square lattice \sref{subfig:geo-square-1} for an angle $\theta$ satisfying $\cot(\theta) \in \Q$. Therefore, the hyperbolic lattice points have positive distance to the blue shaded strip and $\tilde{\Lcal}^0$ has a spectral gap around the imaginary axis \sref{subfig:geo-square-2}.}
    \label{fig:spectral-eps=0-square}
\end{figure}

\subsection{Centre manifold theory and reduced equations}

In the centre manifold theory, we first remark that the functional-analytic setup is independent of a specific choice of Fourier lattice, and thus holds unchanged. The projections can then be defined in the same way, and the estimates in \Cref{lem:estimate-projections} hold with the natural changes to the angle. We then define the splitting of $W$ into more-central, less-central, and hyperbolic parts as in \eqref{eq:projected-variables} and introduce the same rescaled variables \eqref{eq:rescaling}, as well as the resulting system \eqref{eq:spatial-dynamics-split}. 

The main change compared to the hexagonal case is in the following normal-form analysis. First, since the square lattice is non-resonant, we find $\Pcal_{mc}^\eps(\k_j) \hat{\Ncal}(\eps^\beta \shortunderline{\hat{W}}_{mc};\k_j,\vartheta) = 0$. In particular, this implies that \Cref{ass:smallness} is not necessary in the square case. Second, we find that in $\shortunderline{\hat{\Ncal}}_{lc} (\shortunderline{\hat{W}}_{mc},\shortunderline{\hat{W}}_{lc} + \shortunderline{\hat{W}}_{h};\gammab,\vartheta)$ the same problematic nonlinearities appear. Therefore, we make the near-identity transformation, which is still given by \eqref{eq:normal-form-lc} and results in the equations \eqref{eq:normal-form-bilinear} for $\Bcal_1, \Bcal_2, \Bcal_3$. While we obtain a $\Ocal(1)$-bound for $\Bcal_1$ and $\Bcal_3$ already in the hexagonal case, $\Bcal_2$ did not have a uniform bound due to the resonance of the hexagonal lattice, cf.~\Cref{lem:normal-form-analysis} and \Cref{fig:geometric-argument}. This simplifies in the square case since there are no two lattice points in $\Gamma_{\square}$ that generate less-central eigenvalues with the same imaginary part and have distance one. Therefore, we find that $\Bcal_2$ is also uniformly bounded as $\eps \to 0$.

With this normal-form transformation, \Cref{lem:Lipschitz} holds unchanged, and in particular, the nonlinearities in the transformed equations satisfy the Lipschitz bounds \eqref{eq:Lipschitz-bounds} with the same scaling in $\eps$. Additionally, the semigroup bounds hold unchanged after replacing the angle condition with $\cot(\theta) \in \Q$ and distinguishing the case $\theta = 0$. Therefore, we obtain a centre manifold theorem as in \Cref{thm:centre-manifold} for $\theta \neq 0$ and $\Cref{thm:centre-manifold-pi-over-6}$ for $\theta = 0$. However, the resulting centre manifolds are of size $\Ocal(\eps^{2/3 + \delta})$ for any chosen $\delta \in (0,\tfrac{1}{3})$ as in the one-dimensional case \cite{haragus-courcelle1999-01ZangewMathPhys}.

Next, we consider the reduced equations on the centre manifold. Recalling that there are eight, respectively six, more-central eigenvalues, and repeating the same derivation as in \Cref{sec:reduced-equations}, we obtain the reduced equations in the case $\theta \neq 0$ as
\begin{equation}\label{eq:reduced-equations-explicit-square}
    \begin{split}
        \partial_\Xi A_1 &=  B_1, \\
        \partial_\Xi B_1 &= -\dfrac{\mu_0}{4 (\d \cdot \k_1)^2} A_1 - \dfrac{c_0}{4(\d \cdot \k_1)^2}B_1 - \dfrac{1}{4(\d \cdot \k_1)^2}(K_0|A_1|^2 + K_1|A_2|^2) A_1 + \Ocal(\eps),\\
        \partial_\Xi A_2 &=  B_2, \\
        \partial_\Xi B_2 &= -\dfrac{\mu_0}{4 (\d \cdot \k_2)^2} A_2 - \dfrac{c_0}{4(\d \cdot \k_2)^2}B_2 - \dfrac{1}{4(\d \cdot \k_2)^2} (K_0|A_2|^2 + K_1|A_1|^2) A_2 + \Ocal(\eps),
    \end{split}
\end{equation}
where $K_0$ is given by \eqref{eq:K-0-K-2} and $K_1$ is given by
\begin{equation*}
\begin{split}
    K_1 & := 2 \left[\hat{N}_2(e^{i\zerob\cdot \p}, e^{i\k_1\cdot\p};\k_1)\nu_0
    + \hat{N}_2(e^{i(\k_1+\k_2)\cdot \p}, e^{-i\k_2\cdot \p};\k_1) \nu_{\k_1+\k_2} + \hat{N}_2(e^{i(\k_1-\k_2)\cdot \p},e^{i\k_2\cdot \p};\k_1) \nu_{\k_1-\k_2} \right] \\
    & \quad
    + 6 \hat{N}_3(e^{i\k_1\cdot\p},e^{i\k_2\cdot\p}, e^{-i\k_2\cdot\p};\k_1).
\end{split}
\end{equation*}
Note that \eqref{eq:reduced-equations-explicit-square} can be written as the second-order system
\begin{equation}\label{eq:reduced-equations-second-order-square}
\begin{split}
    4(\d\cdot\k_1)^2 \partial_{\Xi}^2 A_1 + c_0 \partial_{\Xi} A_1 + \mu_0 A_1 + (K_0|A_1|^2 + K_1|A_2|^2) A_1 + \Ocal(\eps) & = 0, \\
    4(\d\cdot\k_2)^2 \partial_{\Xi}^2 A_2 + c_0 \partial_{\Xi} A_2 + \mu_0 A_2 + (K_0|A_2|^2 + K_1|A_1|^2) A_2 + \Ocal(\eps) & = 0, \\
\end{split}
\end{equation}
In the case $\theta = 0$, this reduces to the three-dimensional first-order system
\begin{equation}\label{eq:slow-subsystem-square}
    \begin{split}
        \partial_\Xi A_1 &=  B_1, \\
        \partial_\Xi B_1 &= -\dfrac{\mu_0}{4} A_1 - \dfrac{c_0}{4}B_1 - \dfrac{1}{4}(K_0|A_1|^2 + K_1|A_2|^2) A_1 + \Ocal(\eps),\\
        c_0 \partial_\Xi A_2 &= -\mu_0 A_2 - (K_0|A_2|^2 + K_1|A_1|^2) A_2 + \Ocal(\eps),
    \end{split}
\end{equation}
which can again be recovered from the system \eqref{eq:reduced-equations-explicit-square} in the limit $\theta \to 0$ via a fast-slow analysis as in \Cref{sec:fast-slow-results}. In particular, we note that there is no quadratic term since the square lattice is non-resonant.

\subsection{Analysis of reduced equations and pattern interfaces}

We conclude this section by discussing the dynamics of the reduced equations in the square case, specifically, the existence and stability of equilibria and heteroclinic connections between them. As in the hexagonal case, we restrict to the case $\theta \neq 0$ since the results for $\theta = 0$ can be recovered via a fast-slow analysis. Additionally, we first analyse the leading order equation obtained from \eqref{eq:reduced-equations-explicit-square} for $\eps = 0$. A direct computation shows that the following equilibria exist, see \Cref{fig:stationary-patterns-square} for typical patterns.

\begin{corollary}\label{cor:stationary-patterns-square}
    The system \eqref{eq:reduced-equations-explicit-square} at $\eps = 0$ has the following equilibrium solutions:
    \begin{thmenum}
        \item The trivial solution
        \begin{equation*}
            \Ab_T := (A_{1,T},A_{2,T}) = (0, 0).
        \end{equation*}
        \item Roll waves
        \begin{equation}\label{eq:roll-waves-squares}
            \Ab_R = (A_{1,R},A_{2,R}) = \Big(\pm \sqrt{-\frac{\mu_0}{K_0}},0\Big)
        \end{equation}
        if $\mu_0 K_0 < 0$.
        \item Squares
        \begin{equation}\label{eq:squares}
            \Ab_{S} = (A_{1,S},A_{2,S}) \quad \text{with} \quad |A_{j,S}| = \sqrt{-\frac{\mu_0}{K_0+K_1}}
        \end{equation}
        if $\mu_0 (K_0 + K_1) < 0$.
    \end{thmenum}
\end{corollary}

\begin{figure}[H]
    \centering
    \begin{subfigure}[p]{0.3\textwidth}
        \includegraphics[width=\textwidth]{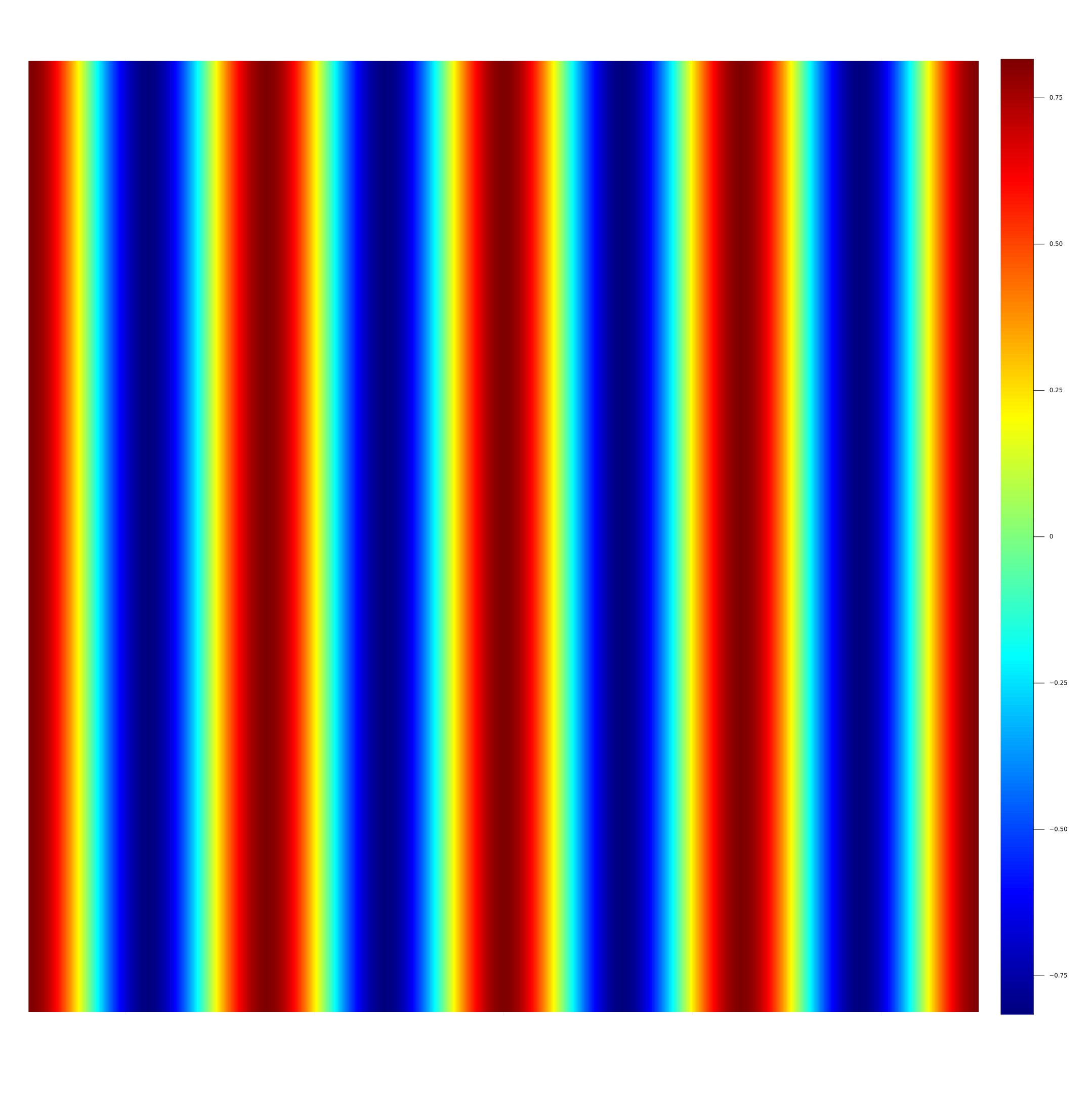}
        \subcaption{Roll waves}
    \end{subfigure}
    \hspace{2cm}
    \begin{subfigure}[p]{0.3\textwidth}
        \includegraphics[width=\textwidth]{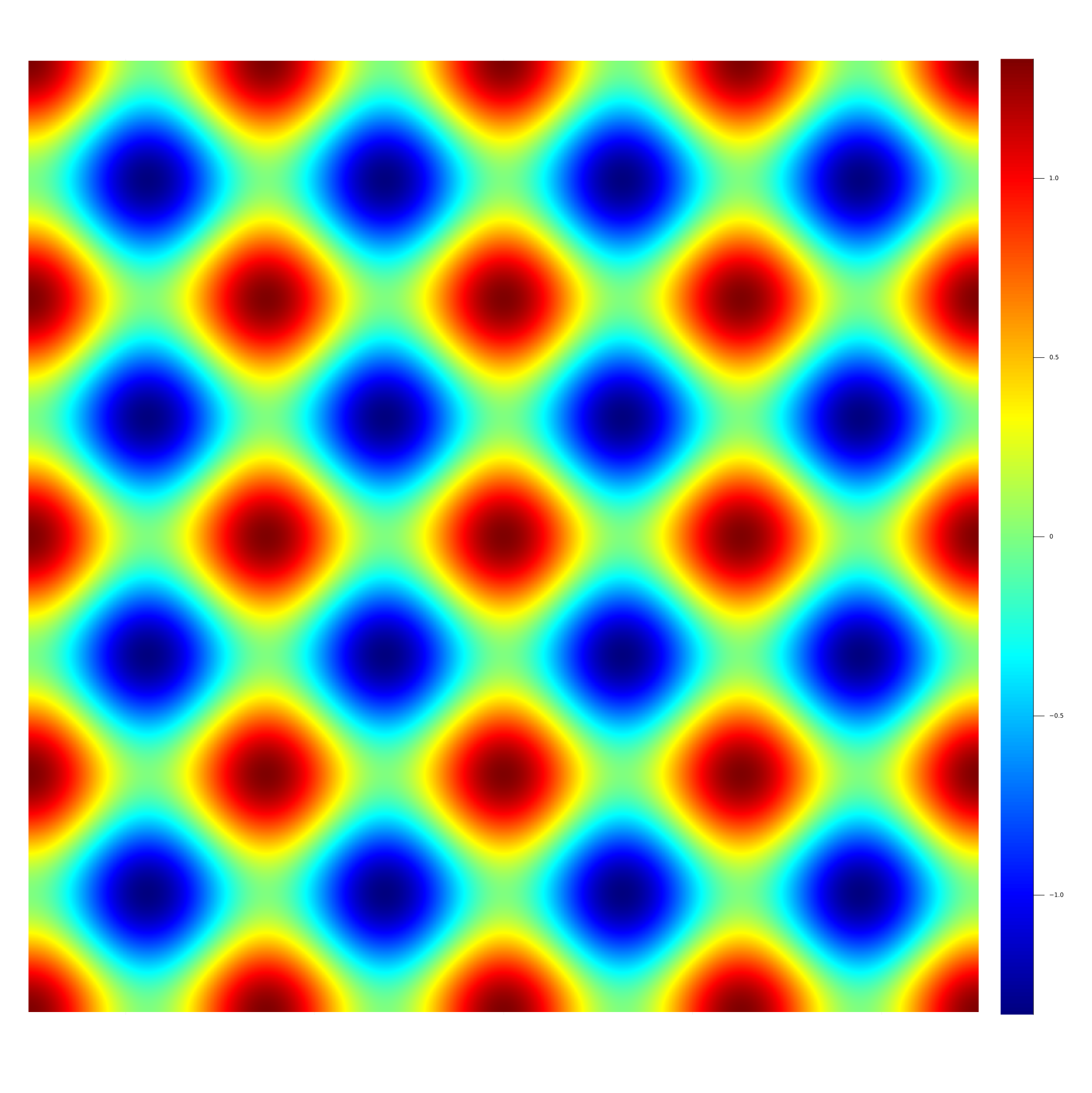}
        \subcaption{Squares}
    \end{subfigure}

    \caption{Stationary pattern on a square lattice obtained in \Cref{cor:stationary-patterns-square}.}
    \label{fig:stationary-patterns-square}
\end{figure}

Their stability properties in the spatial-dynamics system \eqref{eq:reduced-equations-explicit-square} can again be characterised by the linear stability in the Landau equations
\begin{equation*}
    \begin{split}
        \partial_T A_1 &= \mu_0 A_1 + (K_0|A_1|^2 + K_1|A_2|^2) A_1, \\
        \partial_T A_2 &= \mu_0 A_2 + (K_0|A_2|^2 + K_1|A_1|^2) A_2,
    \end{split}
\end{equation*}
see also \Cref{rem:spectral-characterisation-for-RD}. Restricting to the case $K_0 < 0$ and $K_1 < 0$, we find the bifurcation diagram of \Cref{fig:bifurcation-diagram-square}.

\begin{figure}[H]
    \centering
    \centering
    \begin{subfigure}[p]{0.99\textwidth}
        \includegraphics[width=\linewidth]{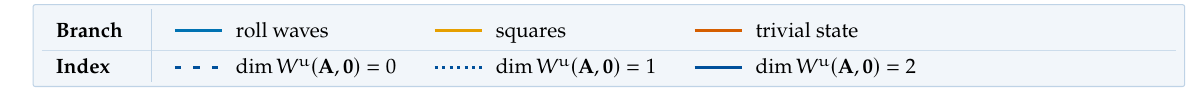}
    \end{subfigure}

    \vspace{0.4cm}

    \begin{subfigure}[p]{0.45\textwidth}
        \includegraphics[width=\linewidth]{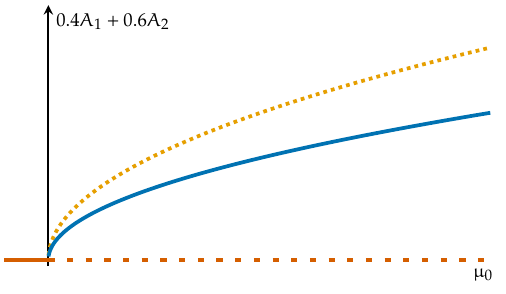}
        \subcaption{$K_2 < K_0 <0$}
    \end{subfigure}
    \hfill
    \begin{subfigure}[p]{0.45\textwidth}
        \includegraphics[width=\linewidth]{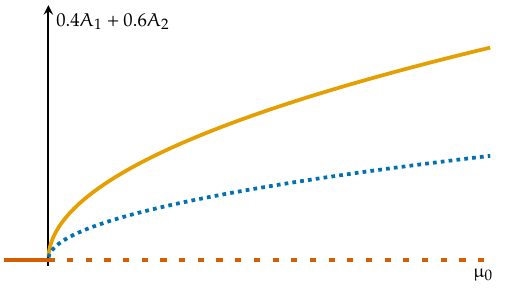}
        \subcaption{$K_0 < K_2 <0$}
    \end{subfigure}
    \caption{Bifurcation diagram of stationary patterns on the square lattice for different parameter regimes.}
    \label{fig:bifurcation-diagram-square}
\end{figure}

For the heteroclinic orbits, we first point out that also the leading-order system given by \eqref{eq:reduced-equations-explicit-square} with $\eps = 0$ has a Lyapunov function given by
\begin{equation*}
    \Hcal(\Ab, \Bb) = (\d \cdot \k_1)^2 |B_1|^2 + (\d \cdot \k_2)^2 |B_2|^2  + \dfrac{\mu_0}{2} (|A_1|^2 + |A_2|^2) + \dfrac{K_0}{4} (|A_1|^4 + |A_2|^4) + \dfrac{K_1}{2} A_1^2 A_2^2,
\end{equation*}
which is strictly decreasing along orbits since it satisfies \eqref{eq:derivative-Lyapunov} along a solution. Calculating the energies at the non-trivial equilibria, we find that
\begin{equation*}
    \Hcal(\Ab_R,\zerob) = - \dfrac{\mu_0^2}{4K_0} \quad \text{and}\quad \Hcal(\Ab_S,\zerob) = -\dfrac{\mu_0^2}{2(K_0+K_1)}.
\end{equation*}
For $\mu_0 > 0$, we therefore can follow the same strategy as in \Cref{thm:heteroclinic} to obtain heteroclinic orbits between roll waves and the trivial state for $K_0 < K_1 < 0$ and between squares and the trivial state for $K_1 < K_0$. These orbits persist for $\eps > 0$ sufficiently small again using the fact that, for $\mu_0 > 0$, the trivial equilibrium is stable and thus the orbits lie in an intersection of the unstable manifold of either roll waves or squares with the four-dimensional stable manifold of the trivial equilibrium. Since the phase space is $\R^4$, these intersections are transversal. Summarising the above discussion, including the fast-slow structure for $\theta \to 0$, we thus obtain the following result.

\begin{theorem}\label{thm:heteroclinic-square}
    Let $\theta\in [0,\tfrac{\pi}{4})$, $c_0>0$ and $K_0<0$.
    \begin{thmenum}
        \item If $K_1 < K_0$, for every $\mu_0>0$ there exist heteroclinic orbits $(\Ab_{S\to T},\Bb_{S\to T})$ in \eqref{eq:reduced-equations-explicit-square} and in \eqref{eq:slow-subsystem-square} such that
        \begin{equation*}
            \lim_{\Xi \to -\infty} (\Ab_{S\to T},\Bb_{S \to T}) = (\Ab_{S},\zerob) \quad\text{and}\quad \lim_{\Xi \to +\infty} (\Ab_{S\to T},\Bb_{S \to T}) = (\Ab_{T},\zerob).
        \end{equation*}
        \item  If $K_0 < K_1 < 0$, for every $\mu_0>0$ there exist heteroclinic orbits $(\Ab_{R\to T},\Bb_{R\to T})$ in \eqref{eq:reduced-equations-explicit-square} and in \eqref{eq:slow-subsystem-square} such that
        \begin{equation*}
            \lim_{\Xi \to -\infty} (\Ab_{R\to T},\Bb_{R \to T}) = (\Ab_{R},\zerob) \quad\text{and}\quad \lim_{\Xi \to +\infty} (\Ab_{R\to T},\Bb_{R \to T}) = (\Ab_{T},\zerob).
        \end{equation*}
    \end{thmenum}
\end{theorem}

Following the discussion in \Cref{sec:pattern-interfaces} we obtain the existence of pattern interfaces on a square lattice in \eqref{eq:Swift-Hohenberg} from heteroclinic orbits in \eqref{eq:reduced-equations-explicit-square}, cf.~\Cref{thm:main-theorem}. Specifically, there exist $\eps_0 > 0$ and open classes of suitable polynomial nonlinearities such that for any $\eps \in (0,\eps_0)$, equation \eqref{eq:Swift-Hohenberg} has solutions $u(t,\x) = U(\d \cdot \x + \eps c_0 t, \x)$ with 
    \begin{equation*}
        \lim_{\xi \to - \infty} U(\xi,\p) = 2 \eps A_\mathrm{sq} \sum_{j = 1}^2 \cos(\k_j \cdot \p) + \Ocal(\eps^2) \quad \text{and} \lim_{\xi \to \infty} U(\xi,\p) = 0
    \end{equation*}
    for some $A_\mathrm{sq} > 0$, cf.~\eqref{eq:squares}, or 
    \begin{equation*}
        \lim_{\xi \to - \infty} U(\xi,\p) = 2 \eps A_\mathrm{roll} \cos(x_1) + \Ocal(\eps^2) \quad \text{and} \lim_{\xi \to \infty} U(\xi,\p) = 0
    \end{equation*}
    for some $A_\mathrm{roll} > 0$, cf.~\eqref{eq:roll-waves-squares}.

\begin{figure}
    \centering
    \begin{subfigure}[p]{0.99\textwidth}
    \includegraphics[width=\textwidth]{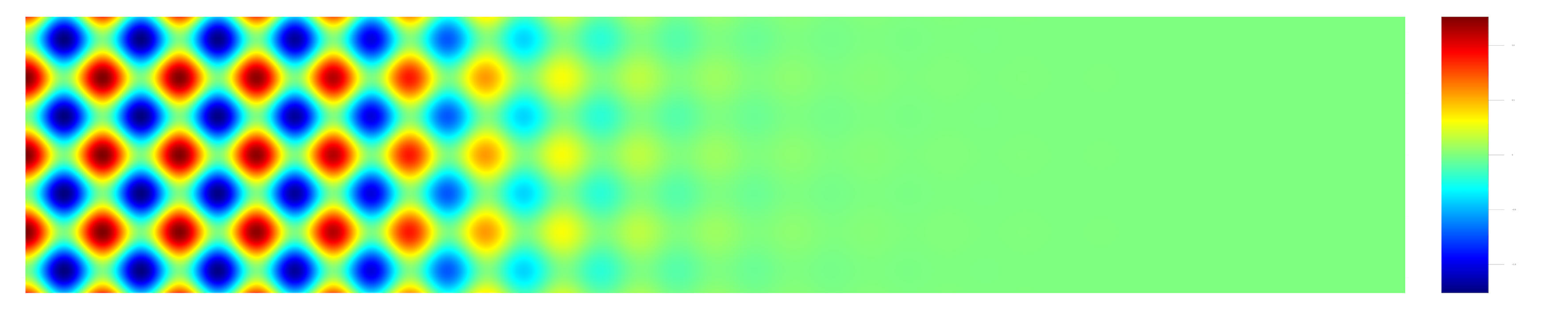}
    \subcaption{Pattern interface for angle $\theta = 0$.}
    \label{subfig:square-front-1}
    \end{subfigure}
     
    \begin{subfigure}[p]{0.99\textwidth}
    \includegraphics[width=\textwidth]{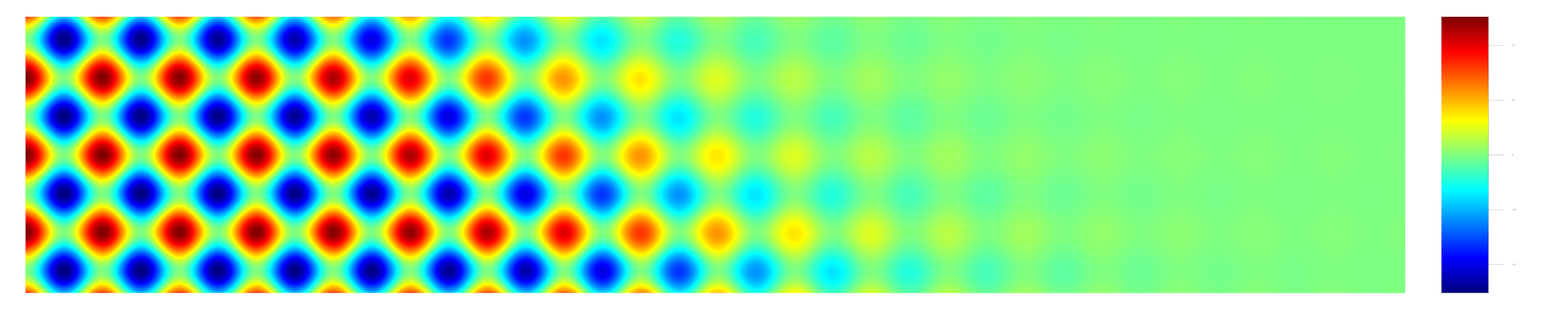}
    \subcaption{Pattern interface for angle $\theta = \tfrac{\pi}{4}$.}
    \label{subfig:square-front-2}
    \end{subfigure}
    
    \caption{Panels \sref{subfig:square-front-1} and \sref{subfig:square-front-2} show pattern interfaces describing the invasion of square patterns into the unstable homogeneous state for different angles.}
    \label{fig:square-fronts}
\end{figure}

\section{Discussion}\label{sec:discussion}

In this paper, we construct slow-moving pattern interfaces with general spreading direction in a two-dimensional Swift–Hohenberg-type equation. These interfaces describe the invasion of a linearly unstable, spatially homogeneous state by either roll waves or planar hexagonal patterns. We also present evidence for the existence of two-stage invasion processes, where the homogeneous state is first invaded by roll waves, which in turn are invaded by hexagonal patterns through a secondary front interface.

We conclude this discussion with a brief overview of related questions and potential pathways for further investigation.

\paragraph{Front selection}

While it is well-understood that planar patterns, specifically hexagonal patterns, arise through planar or radial pattern interfaces, the exact mechanism is not understood rigorously. In the spatially one-dimensional case, some rigorous selection results are available. In particular, for fronts in reaction-diffusion equations such as the Fisher–KPP equation, a rigorous theory for the selected speed has been developed recently, and we refer to \cite{avery2022-07CommAmerMathSoc} and to the recent overview \cite{avery2025-12preprint}. Moreover, rigorous selection results for pattern-forming fronts specifically have only been obtained very recently \cite{avery2026-03preprint}. However, in the two-dimensional case considered here, many open problems remain regarding (a) the selected invasion speed and (b) the selected pattern in the wake when starting from a compactly supported perturbation of the linearly unstable trivial state.

The two-dimensional situation is even more complicated since, in experiments, the selection of planar patterns is facilitated through a cascade of invading fronts, which typically do not travel with the same speed. An example for this is that the hexagonal pattern in thermal convection typically forms as a secondary invasion to roll waves, which are the selected pattern in the leading edge of the front \cite{pismen1994-08EurophysLett,csahok1999-08EurophysLett}. The main challenge in understanding these structures rigorously is the fact that the connecting state is unstable and that the amplitude solution is not a travelling wave. This is an open question even for one-dimensional reaction-diffusion systems, and we refer to \cite{garenaux2025-03preprint} for first linear stability results.

\paragraph{Rigorous heteroclinics}

While we construct selected heteroclinic orbits analytically in \Cref{sec:heteroclinic}, these only cover parts of the orbits that we find numerically, cf.~\Cref{fig:planar-front,fig:pattern-interfaces,fig:planar-fronts-numerics}. Specifically, the rigorous existence of two-stage invasion processes corresponding to heteroclinic orbits from up-hexagons to the trivial state that pass close to the roll waves remains open. The main problem is that the phase space of the reduced equations is still six-dimensional, which makes any global analysis challenging even though the reduced equations have a strictly decreasing Lyapunov function. The partial results of this paper rely on the observation that the energy landscape is topologically very simple close to the trivial equilibrium allowing for the construction of a sufficiently large trapping region that has an intersection with the unstable manifold of the non-trivial equilibrium with the lowest energy. Beyond that trapping region, the energy landscape is topologically more complicated and more involved methods have to be used. One such topological method to construct heteroclinic orbits is Conley index theory \cite{conley1978book,mischaikow2002HandbookofDynamicalSystems}, which has already been mentioned in \cite{doelman2003-02EuropeanJournalofAppliedMathematics} as a potential pathway.

\paragraph{Extending the set of allowed angles}

The construction of planar interfaces in the paper is restricted to directions satisfying the condition $\cot(\theta) \in \sqrt{3} \Q$, which guarantees the existence of a spectral gap around the imaginary axis at $\eps = 0$, see \Cref{prop:real-spectral-gap}. The case $\cot(\theta) \notin \sqrt{3}\Q$ is much more challenging since there exists a sequence of hyperbolic eigenvalues which accumulate at the imaginary axis at infinity. This suggests the presence of a small-divisor problem, which could potentially be solved using Kolmogorov–Arnold–Moser (KAM) theory to obtain pattern interfaces in directions satisfying a Diophantine condition. However, we expect that this is a technical challenge and that the reduced dynamics to leading order are still given by \eqref{eq:travelling-wave-equation} obtained from the amplitude equations. Nevertheless, the proof of this remains an open problem.

\section*{Data availability statement}

The explicit computations for the front speed via a marginal stability analysis were performed using Mathematica \cite{mathematicacomputerProgram}. The numerical simulations of orbits in the leading-order system \eqref{eq:leading-order-reduced-equations} were done using Julia \cite{bezanson2017SIAMReview}. This was also used to plot the pattern interfaces using the leading-order expressions in \Cref{thm:main-theorem} with the numerically obtained amplitudes. Finally, the simulations of the selected front speeds presented in \Cref{app:front-speed-numerics} were done using FEniCSx \cite{baratta2023computerProgram}. The code used to generate the corresponding data is available under \cite{hilder2026-04webpage}.

\section*{Acknowledgments}

B.H.~acknowledges the support by the Deutsche Forschungsgemeinschaft (DFG, German Research Foundation) -- Project-ID 543917644.

\appendix

\crefalias{section}{appsec}

\section{Spectral analysis in generic rotation-symmetric pattern-forming systems close to a Turing instability}\label{app:general-spectral-analysis}

Although we have hinted the genericity of the spectral results obtained in \Cref{prop:imaginary-spectrum-discrete,prop:real-spectral-gap,prop:spectrum-eps-positive} in the above analysis, we now discuss this in more detail. The main observation is that the spectral properties are generic for systems which are (i) rotation-symmetric and (ii) close to a Turing instability. More precisely, we consider a system of the form $\partial_t u = \Lambda_\mu(\nabla) u + N(u)$, where $\Lambda_\mu(\nabla)$ is a linear operator with polynomial symbol depending on a parameter $\mu$ and $N$ is a generic, smooth nonlinearity. 
We assume that the system is rotation-symmetric and in particular, the corresponding Fourier symbol $\hat{\Lambda}_{\mu}(\k) = e^{-i \k \cdot \x} \Lambda_\mu(\nabla) e^{i \k \cdot \x}$ satisfies $\hat{\Lambda}_\mu(\k) = \hat{\Lambda}_\mu(|\k|)$ for all $\k \in \R^2$. At every wave vector $\k$, the Fourier symbol $\hat{\Lambda}_\mu(\k)$ we denote the eigenvalue with the largest real part by $\lambda_{\max}(\k;\mu) = \lambda_{\max}(|\k|;\mu)$. We further assume that the system is close to a Turing instability, which means that there is a critical wave number $\tilde{k} > 0$ and a critical parameter $\tilde{\mu} \in \R$ such that 
\begin{enumerate}[label=(\alph*)]
    \item all eigenvalues of $\hat{\Lambda}_{\tilde{\mu}}(\k)$ have negative real parts for all $\k$ with $|\k| \neq \tilde{k}$;
    \item $\lambda_{\max}(\k;\tilde{\mu}) = 0$ if and only if $|\k| = \tilde{k}$ and zero is a simple eigenvalue;
    \item $\partial_\mu \lambda_{\max}(\k; \tilde{\mu}) = \kappa > 0$ for all $|\k| = \tilde{k}$;
    \item $\lambda_{\max}(\k;\tilde{\mu}) = -\alpha(|\k| - \tilde{k})^2 + \mathcal{O}((|\k|-\tilde{k})^3)$ for $\alpha = -\tfrac{1}{2} \partial_{|\k|}^2 \lambda_{\max}(|\k|; \tilde{\mu})|_{|\k| = \tilde{k}} > 0$ and all $\k \in \R^2$ in a neighbourhood of $|\k| = \tilde{k}$.
\end{enumerate} 
We refer to \Cref{fig:turing-instab} for the typical behaviour of the principal eigenvalue curve $\lambda_{\max}$. Without loss of generality, we can assume that $\tilde{k} = 1$ and $\tilde{\mu} = 0$. In addition, we set the scaling $\mu = \eps^2 \mu_0$ and $c = \eps c_0$ as above.

\begin{figure}[h]
    \centering
    \begin{subfigure}[p]{0.3\textwidth}
        \includegraphics[width=\textwidth]{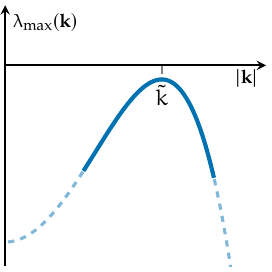}
        \subcaption{$\mu < \tilde{\mu}$}
    \end{subfigure}
    \hfill
    \begin{subfigure}[p]{0.3\textwidth}
        \includegraphics[width=\textwidth]{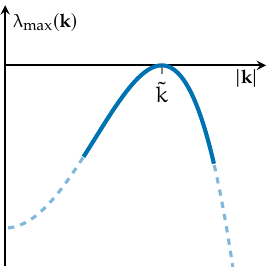}
        \subcaption{$\mu = \tilde{\mu}$}
    \end{subfigure}
    \hfill
    \begin{subfigure}[p]{0.3\textwidth}
        \includegraphics[width=\textwidth]{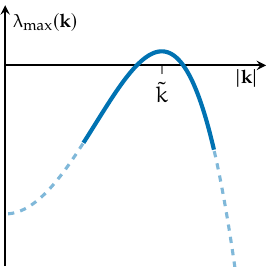}
        \subcaption{$\mu > \tilde{\mu}$}
    \end{subfigure}
    
    \caption{Schematic depiction of the eigenvalue curve with the largest real part close to the Turing instability as $\mu$ transitions through $\tilde{\mu}$.}
    \label{fig:turing-instab}
\end{figure}

As in \Cref{lem:eigenvalue-correspondence}, we find that the eigenvalues $\nu$ of the corresponding spatial-dynamics formulation are given by the roots of the 
\begin{equation*}
    p(\nu; \gammab, \eps) = \det(\Lambda_{\eps^2\mu_0}(i \gammab + \d \nu) + \eps c_0 \nu I) = \det(\hat{\Lambda}_{\eps^2 \mu_0}(\gammab - i \d \nu) + \eps c_0 \nu I),
\end{equation*}    
cf.~\cite[Lem.~5.3]{hilder2025-08JNonlinearSci}. Due to the simplicity of the zero eigenvalue $\lambda_{\max}$ at $\eps = 0$ and $|\k| = 1$, we can write
\begin{equation*}
    p(\nu; \gammab, \eps) = f(\nu,\eps) (-\alpha (|\gammab - i \d \nu|_{\R^2} - 1)^2 + \eps^2 \kappa \mu_0 + \eps c_0 \nu) + \Ocal((|i \gammab + \d \nu|_{\R^2} - 1)^3 + \eps^3)),
\end{equation*}
where $|\omega|_{\R^2} = \sqrt{\omega^T \omega}$ for $\omega \in \C^2$ and $f(\nu,\eps)$ is given by the product of the non-critical eigenvalues of $\hat{\Lambda}_{\eps^2\mu_0}(\gammab - i \d \nu) + \eps c_0 \nu I$. This product is non-zero if $\nu$ is contained in a sufficiently small strip around the imaginary axis, since the non-critical eigenvalues of $\hat{\Lambda}_0(\k)$ have strictly negative real parts for all $\k \in \R^2$.

For $\eps = 0$, it is straightforward to check that the geometric characterisation of the spectrum depicted in \Cref{fig:spectral-eps=0} holds true. Therefore, the spectral properties obtained in \Cref{prop:imaginary-spectrum-discrete,prop:real-spectral-gap} are generic for rotation-symmetric pattern-forming systems close to a Turing instability.

To obtain the spectral properties for $\eps > 0$ we consider how the eigenvalues on the imaginary axis at $\eps = 0$ perturb for $\eps > 0$. For this, we assume that the spatial eigenvalue $\nu(\eps) = i \nu_0 + \nu_1(\eps)$ with $\nu_0 \in \R$. From the geometric interpretation for $\eps = 0$ we know that $\tilde{\gammab} := \gammab + \d \nu_0$ lies on the unit circle, see \Cref{fig:geometric-argument}. We first consider the case of lattice points $\gammab$ such that $\ell_{\gammab} = \{\gammab + \d s \,:\, s \in \R\}$ intersects transversally with the unit circle, that is, $\d \cdot \tilde{\gammab} \neq 0$. Using the expansion
\begin{equation*}
    |\gammab - i \d (i \nu_0 + \nu_1(\eps))|_{\R^2} = |\gammab + \d \nu_0|_{\R^2} - i (\d \cdot \tilde{\gammab}) \nu_1(\eps) + \Ocal(|\nu_1(\eps)|^2), 
\end{equation*}
and using again that $|\gammab + \d \nu_0|_{\R^2} = 1$, we find that the correction $\nu_1(\eps)$ is given by a solution to the equation
\begin{equation*}
    \alpha (\d \cdot \gammab)^2 \nu_1(\eps)^2 + \eps c_0 (i \nu_0 + \nu_1(\eps)) + \eps^2 \kappa \mu_0 + \Ocal(|\nu_1(\eps)|^3 + \eps^3) = 0.
\end{equation*}
Depending on $\nu_0 = 0$ or $\nu_0 \neq 0$, we thus find two different mechanisms for the perturbation of the eigenvalues. 

If $\nu_0 = 0$, the natural scaling for $\nu_1(\eps)$ is $\nu_1(\eps) = \eps \delta + h.o.t.$ and $\delta \in \C$ is given by the roots of the quadratic equation
\begin{equation*}
    \alpha (\d \cdot \gammab)^2 \delta^2 + c_0 \delta + \kappa \mu_0 = 0
\end{equation*}
since $\tilde{\gammab} = \gammab$ for $\nu_0 = 0$. Noting that $\alpha > 0$ and $\kappa \mu_0 > 0$ since we are in the Turing unstable regime, and $c_0 > 0$ by assumption, we thus find that the roots $\delta$ have negative real part, which implies that the eigenvalues perturb into the left half-plane for $\eps > 0$. In fact, this is fully in line with the results for the stability of equilibria in the reduced equation on the centre manifold given in \Cref{lem:characterisation-spat-dyn-stable-vs-pde-stable}, where PDE instability implies that the linearisation of the spatial dynamics system only has stable eigenvalues.

If $\nu_0 \neq 0$, the natural scaling for $\nu_1(\eps)$ is $\nu_1(\eps) = \eps^{1/2} \delta + h.o.t.$ and $\delta \in \C$ is given by the roots of the quadratic equation
\begin{equation*}
    \alpha (\d \cdot \tilde{\gammab})^2 \delta^2 + c_0 i \nu_0 = 0.
\end{equation*}
Since $\alpha, c_0 \in \R$, we find that $\delta^2$ is purely imaginary and thus a spatial eigenvalue with non-zero imaginary part at $\eps = 0$ splits into an eigenvalue with negative real part and an eigenvalue with positive real part for $\eps > 0$. Notably, the leading order correction of a less-central eigenvalue is therefore independent of the distance of the instability encoded by $\mu_0$, and the splitting only relies on the advection term generated by the front speed $\eps c_0$.

In the degenerate case $\d \cdot \tilde{\gammab} = 0$, we find that
\begin{equation*}
    |\gammab - i \d (i \nu_0 + \nu_1(\eps))|_{\R^2} = |\gammab + \d \nu_0|_{\R^2} + \nu_1(\eps)^2 + \Ocal(|\nu_1(\eps)|^3).
\end{equation*}
Therefore, the correction $\nu_1(\eps)$ is given by a solution to the equation
\begin{equation*}
    -\alpha \nu_1(\eps)^4 + \eps c_0 (i \nu_0 + \nu_1(\eps)) + \eps^2 \kappa \mu_0 + \Ocal(|\nu_1(\eps)|^5 + \eps^3) = 0.
\end{equation*}
If $\nu_0 = 0$, this equation has one solution $\nu_1 = -\eps \tfrac{\kappa \mu_0}{c_0}$ and three solutions with $\nu_1 = \Ocal(\eps^{1/3})$. If $\nu_0 \neq 0$, there are four solutions with $\nu_1 = \Ocal(\eps^{1/4})$.

In both cases, the persistence of these solutions follows as in the proof of \Cref{prop:spectrum-eps-positive} by assuming that $|\eps \nu_1|$ is sufficiently small. We therefore fully recover the spectral properties obtained in \Cref{prop:spectrum-eps-positive} for generic rotation-symmetric pattern-forming systems close to a Turing instability. Specifically, in the Swift–Hohenberg equation, we have $\alpha = 4$ and $\kappa = 1$, and therefore, the general calculations recover the expansions \eqref{eq:spectrum-eps-pos-perturbation-block-size-two} and \eqref{eq:spectrum-eps-pos-perturbation-block-size-four}.

\section{Numerical front speed selection}\label{app:front-speed-numerics}

We conclude the paper by showing the results of some numerical experiments to estimate the selected front speed in the full dynamical problem for a Swift–Hohenberg-type equation \eqref{eq:Swift-Hohenberg}. Therefore, we simulate the quadratic-cubic Swift–Hohenberg equation
\begin{equation}\label{eq:Swift-Hohenberg-numerics}
    \partial_t u = -(1+\Delta)^2 u + \mu u - \beta |\nabla u|^2 - u^3,
\end{equation}
proposed in \cite{doelman2003-02EuropeanJournalofAppliedMathematics}, with $\mu = \eps^2\mu_0$, $\beta = \eps\beta_2$, $\mu_0 = \beta_2 = 1$, and $\eps = 0.3$ on the rectangle $\Omega = [-4\pi,36\pi]\times [-3\sqrt{3}\pi,3\sqrt{3}\pi]$ subject to homogeneous Neumann boundary conditions using a finite-element method. 

The equation \eqref{eq:Swift-Hohenberg-numerics} is recast as the second-order system
\begin{equation*}
    \begin{split}
        \partial_t u & = \mu u - \beta |\nabla u|^2 -u^3 -w - \Delta w, \\
        w & = u + \Delta u
    \end{split}
\end{equation*}
in a mixed finite-element formulation, and then discretised with continuous piecewise-linear finite elements on an unstructured triangular mesh with $1884\times 326$ mesh points, which is a spatial resolution of $\Delta x = \tfrac{1}{15}$. The time-discretisation is then done using a Crank-Nicolson scheme with fixed time step 
$\Delta t = \frac{1}{50}$ over $T=100$ time units. The resulting nonlinear system at each step is solved with Newton–Raphson line search with residual tolerance of $10^{-10}$. The implementation uses the FEniCSx software stack (DOLFINx/UFL/Basix) for weak-form assembly and mesh management, with PETSc providing the linear algebra backend and nonlinear solver; see~\cite{baratta2023computerProgram}.

Since we are interested in planar fronts, we choose the initial condition
\begin{align*}
    u_0(\x) = \eps \Bigl(\sum_{j=1}^{3} A_{\mathrm{hex}}\cos(\k_j \cdot \x)\Bigr) \frac{1-\tanh(\tfrac{\x\cdot \d - \phi}{\ell})}{2} 
\end{align*}
with a localisation width $\ell =3$ and amplitude
\begin{align*}
    A_{\mathrm{hex}} = \frac{-\beta_2 - \sqrt{\beta_2^2 - 4\mu_0 (K_0 + 2 K_2)}}{2 (K_0 + 2K_2)},
\end{align*}
which imposes a steep profile in direction $\d=(\cos(\theta),\sin(\theta))$ connecting hexagons to the spatially homogeneous state. For the choice of parameters, we expect hexagons to be stable with respect to all other stationary patterns: indeed, for \eqref{eq:Swift-Hohenberg-numerics} it holds $K_0 = -3$ and $K_2 = -6$, cf.~\cite{doelman2003-02EuropeanJournalofAppliedMathematics}, and with $\mu_0 = 1$ we are in the regime before the bifurcation of the mixed-mode branch in \Cref{fig:bifurcation-diagrams-hex}. We consider two initial profiles: the first for angle $\theta = 0$ and the second for an angle $\theta = \tfrac{\pi}{6}$. Note that we shift the second initial condition to the right by setting $\phi=2\pi$ to make sure no boundary effects from the left-hand boundary affect the dynamics.

\begin{figure}[h]
    \centering
    \begin{subfigure}[p]{0.99\textwidth}
        \includegraphics[width=\linewidth]{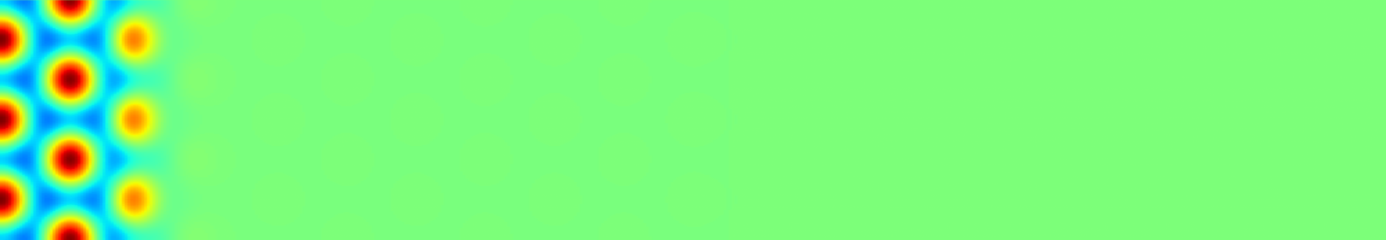}
        \subcaption{$\theta = 0$}
    \end{subfigure}
    
    \vspace{0.2cm}
    
    \begin{subfigure}[p]{0.99\textwidth}
        \includegraphics[width=\linewidth]{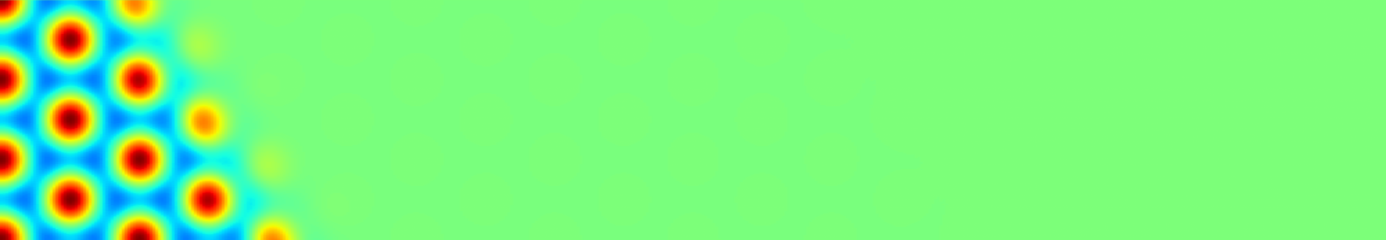}
        \subcaption{$\theta = \tfrac{\pi}{6}$}
    \end{subfigure}
    \caption{The initial conditions for the finite-element simulations. Note that the envelope of the second initial condition has been shifted to the right by $2\pi$.}
    \label{fig:numerics-initial-condition}
\end{figure}

The front position $x_f(t)$ is then determined by tracking the $L^2$-norm of a rolling $2\pi$-wide strip in front direction and defining $x_f(t)$ to be the rightmost position, where this norm exceeds the threshold $\delta = \frac{0.05}{\sqrt{2}}$. From the discussion of the selected front speed in \Cref{sec:front-speed}, cf.~\eqref{eq:results-marginal-stability-analysis}, we obtain the linearly predicted speed
\begin{align*}
    c_{\mathrm{pred}}(\theta) = 4\eps \sqrt{(1-(\d^{\perp}\cdot\k_1)^2) \mu_0},
\end{align*}
where $\d^{\perp} = (-\sin(\theta),\cos(\theta)) \perp \d$. Inserting the parameters gives the predicted front speeds
\begin{align*}
    c_{\mathrm{pred}}(0) = 1.2 \quad \text{and} \quad c_{\mathrm{pred}}(\tfrac{\pi}{6}) \approx 1.039.
\end{align*}

Since we expect a relaxation of the initial condition to the selected front profile and hence a relaxation of the front speed, we fit a linear curve to $x_f(t)$ in the interval $[t_0,t_1]$. Note that we also restrict from above as the plots indicate boundary effects on the right-hand-side boundary to affect the front. For $\theta=0$, the front speed has relaxed at $t_0=20$ and boundary effects are visible after $t_1=80$, while for $\theta = \tfrac{\pi}{6}$ the front settles at $t_0=30$ and boundary effects are visible at $t_1=80$. With this curve, we find the numerical front speeds
\begin{align*}
    c_{\mathrm{num}}(0) \approx 1.1134 \quad \text{and} \quad c_{\mathrm{num}}(\tfrac{\pi}{6}) = 0.9922 
\end{align*}
being a reasonable match for the predicted front speed, with a relative error 
of approximately $7.2\%$ and $4.5\%$ respectively. The remaining discrepancy is consistent with logarithmic corrections of the position not captured by the leading-order marginal stability prediction, which lead to a correction of the front speed of order $\tfrac{1}{t}$. These corrections have been rigorously found in monostable fronts in reaction-diffusion systems, see \cite{avery2025-12preprint}. We also observe that the leading edge of the front exhibits a stripe-like structure before the pattern settles into the hexagonal bulk state, see 
\Cref{fig:numerics-snapshots,fig:numerics-snapshots-theta}. This is consistent with experimental observations, see e.g.~\cite{pismen1994-08EurophysLett,csahok1999-08EurophysLett}.

The data of the front speeds is shown in \Cref{fig:numerics-front-speed}, and snapshots of the solutions at selected times are shown in \Cref{fig:numerics-snapshots}.

\begin{figure}[h]
    \centering
    \begin{subfigure}[p]{0.45\textwidth}
        \includegraphics[width=\linewidth]{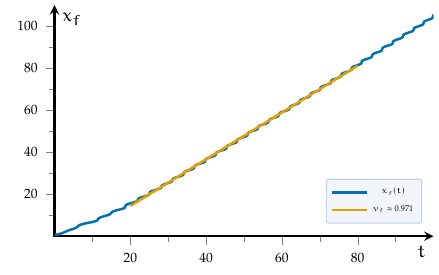}
        \subcaption{Front position for $\theta = 0$}
    \end{subfigure}
    \hfill
    \begin{subfigure}[p]{0.45\textwidth}
        \includegraphics[width=\linewidth]{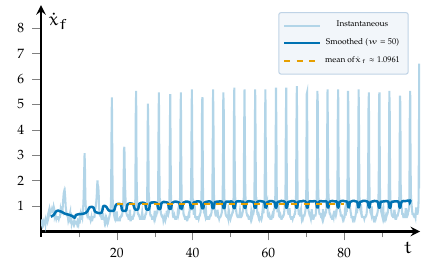}
        \subcaption{Front speed for $\theta = 0$}
    \end{subfigure}

    \begin{subfigure}[p]{0.45\textwidth}
        \includegraphics[width=\linewidth]{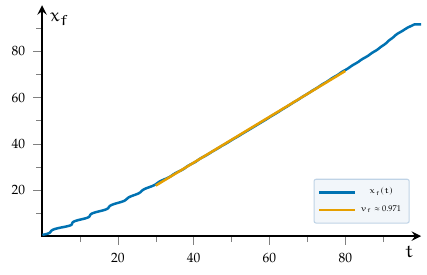}
        \subcaption{Front position for $\theta = \tfrac{\pi}{6}$}
    \end{subfigure}
    \hfill
    \begin{subfigure}[p]{0.45\textwidth}
        \includegraphics[width=\linewidth]{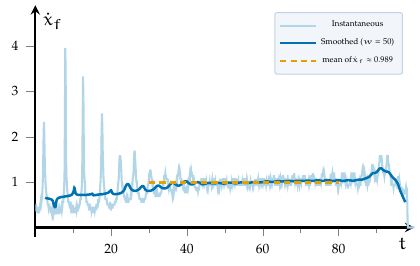}
        \subcaption{Front speed for $\theta = \tfrac{\pi}{6}$}
    \end{subfigure}

    \caption{Front position $x_f(t)$, with linear fit $v_f$, and instantaneous front speed 
    $\dot{x}_f(t)$. Note that the relaxation time is visible in the smoothed version of the front speeds.}
    \label{fig:numerics-front-speed}
\end{figure}

\begin{figure}[H]
    \centering
    \includegraphics[width=0.99\linewidth]{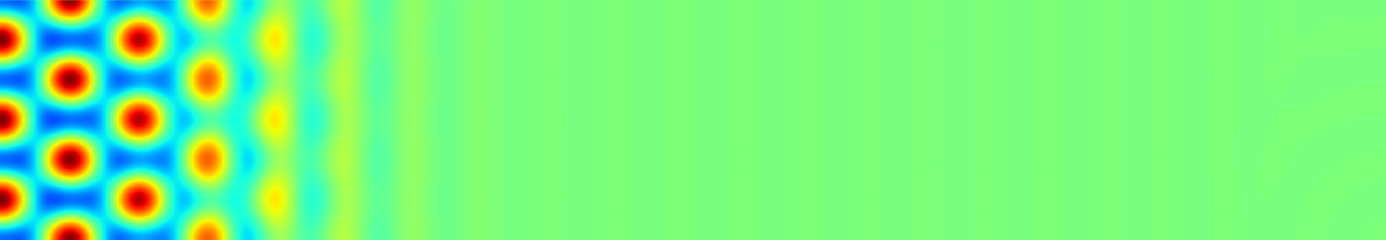}
    
    \vspace{0.2cm}
    
    \includegraphics[width=0.99\linewidth]{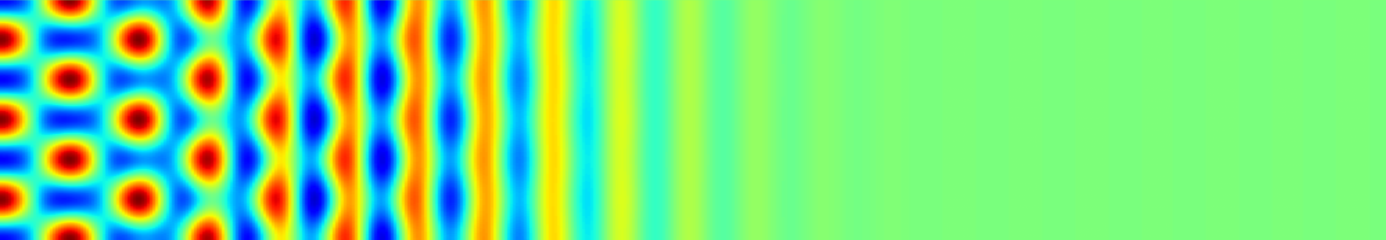}
    
    \vspace{0.2cm}
    
    \includegraphics[width=0.99\linewidth]{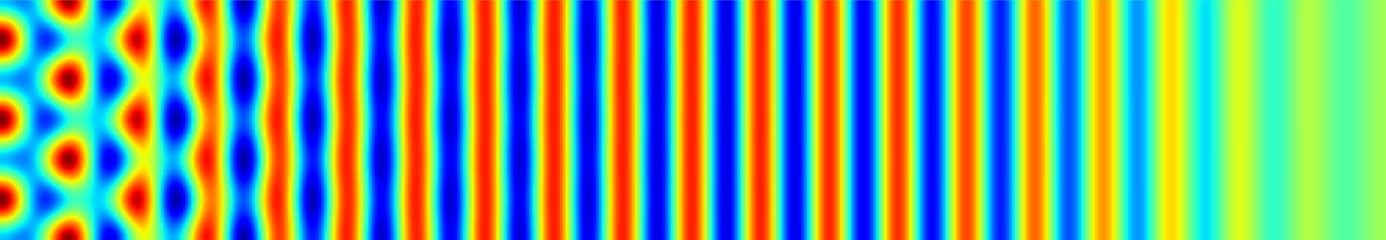}
    
    \caption{Snapshots of the solution $u(t,\cdot)$ for $\theta = 0$ at times 
    $t = 20, 50, 100$. The front propagates to the right, leaving a hexagonal 
    pattern in its wake. Note the stripe-like structure at the leading edge of 
the front. Boundary effects near the right-hand boundary are 
    visible at $t = 100$.}
    \label{fig:numerics-snapshots}
\end{figure}

\begin{figure}[H]
    \centering

    \includegraphics[width=0.99\linewidth]{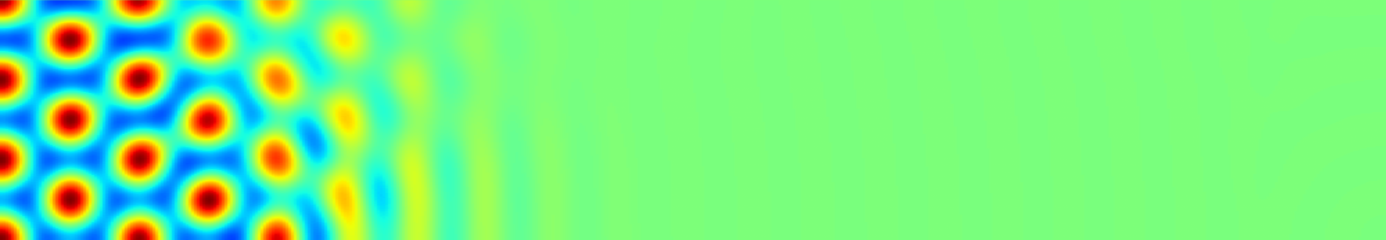}
    
    \vspace{0.2cm}
    
    \includegraphics[width=0.99\linewidth]{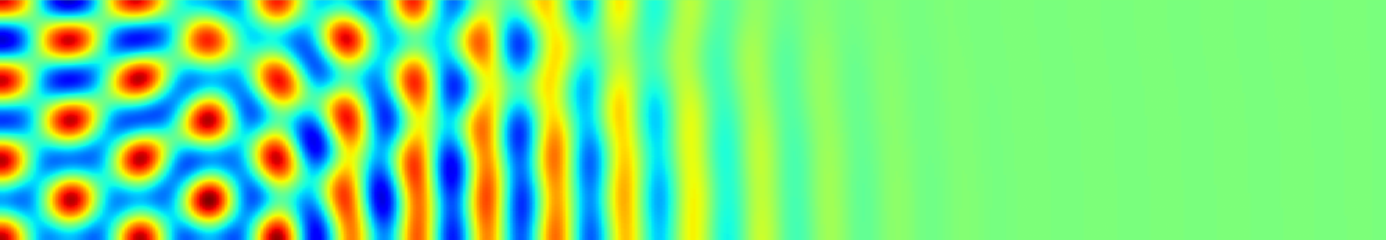}
    
    \vspace{0.2cm}
    
    \includegraphics[width=0.99\linewidth]{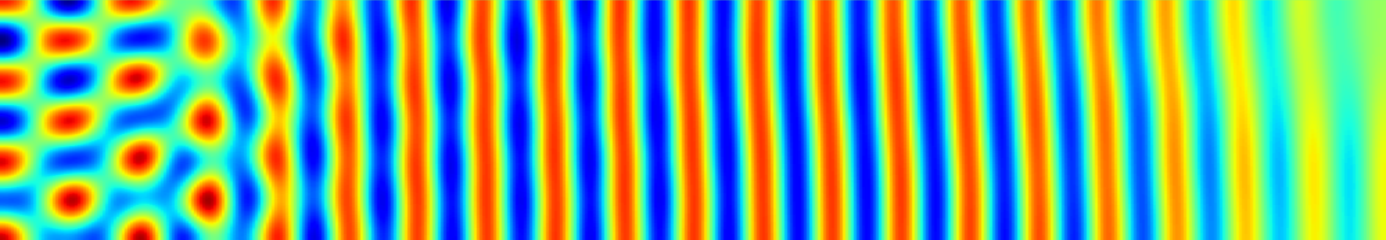}

    \caption{Snapshots of the solution $u(t,\cdot)$ for $\theta = \tfrac{\pi}{6}$ at times 
    $t = 20, 50, 100$. Note the stripe-like structure at the leading edge of the front, which still carries information about the initial angle. Furthermore, the bulk state becomes unordered over time. Whether this is a genuine effect of the dynamics or an artefact of the boundary condition on top and bottom will remain open. Boundary effects near the right-hand boundary are visible at $t = 100$.}
    \label{fig:numerics-snapshots-theta}
\end{figure}

\emergencystretch=2em
\printbibliography

\authordetails

\end{document}